%% file: Fr_Kopt_fract_apost_revised_arxiv.tex
\documentclass[12pt]{article}

\usepackage[utf8]{inputenc}
\usepackage{amsmath}
\usepackage{amssymb}
\usepackage{amsthm}
\usepackage{cite}
\usepackage{graphicx}
\usepackage{tikz}
\usepackage{pgfplots}
\usepackage{listings}

\pgfplotsset{compat=1.16}
\usetikzlibrary{external}
\tikzsetfigurename{figure_}
\tikzset{external/system call={pdflatex \tikzexternalcheckshellescape -interaction=batchmode -jobname "\image" "\texsource"; rm "\image".dpth "\image".log "\image".nlo "\image".atfi "\image".spl}}

\makeatletter
\renewcommand*\env@matrix[1][r]{\hskip -\arraycolsep
  \let\@ifnextchar\new@ifnextchar
  \array{*\c@MaxMatrixCols #1}}
\makeatother
\newcommand{\ord}[1]{\mathcal{O}(#1)}
\newcommand{\pt}{\partial}
\newcommand{\ds}{\mathrm{d}s}
\newcommand{\dsigma}{\mathrm{d}\sigma}

\newcommand{\R}{\mathbb{R}}
\newcommand{\PS}{\mathcal{P}}

\newcommand{\dt}{d}

\newcommand{\LL}{{L}}
\newcommand{\EE}{{\mathcal E}}
\newcommand{\RR}{{\mathcal R}}

\newcommand {\beq} {\begin{equation}}
\newcommand {\eeq} {\end{equation}}
\newcommand {\beqa} {\begin{eqnarray}}
\newcommand {\eeqa} {\end{eqnarray}}
\newcommand {\beqann} {\begin{eqnarray*}}
\newcommand {\eeqann} {\end{eqnarray*}}

     \definecolor{pass}{rgb}{0,0,0.8}
     \definecolor{pass1}{rgb}{0,0,0.5}

\DeclareMathOperator{\expm}{\texttt{expm1}}
\DeclareMathOperator{\logp}{\texttt{log1p}}

\theoremstyle{plain}
\newtheorem{theorem}{Theorem}[section]
\newtheorem{lemma}[theorem]{Lemma}

\newtheorem{corollary}[theorem]{Corollary}
\newtheorem{remark}[theorem]{Remark}

\newcounter{EXAMPLEcounter}[section]
\renewcommand{\theEXAMPLEcounter}{\thesection.\arabic{EXAMPLEcounter}}
\newcommand{\example}{\refstepcounter{EXAMPLEcounter}%
  \textbf{Example~\theEXAMPLEcounter:}}

\begin{document}
\title{Pointwise-in-time a posteriori error control for higher-order discretizations of time-fractional parabolic equations}
\author{Sebastian Franz\footnote{
          Institute of Scientific Computing, Technische Universit\"at Dresden, Germany.
          \mbox{e-mail}: sebastian.franz@tu-dresden.de}
        \and
        Natalia Kopteva\footnote{corr. author, Department of Mathematics and Statistics, University of Limerick, Ireland.
        \mbox{e-mail}: natalia.kopteva@ul.ie}
       }
\date{\today}
\maketitle
    \begin{abstract}
    Time-fractional parabolic equations with a Caputo time derivative are considered.
    For such equations, we
    explore and further develop the new methodology of the a-posteriori error estimation and adaptive time stepping proposed in \cite{Kopteva22}.
        We improve the earlier time stepping algorithm based on this theory, and specifically address its stable and efficient implementation
        in the context of high-order methods.
    The considered methods include an L1-2 method and continuous collocation methods of arbitrary order,
    for which adaptive temporal meshes are shown to yield optimal convergence rates  in the presence of solution singularities.
    \end{abstract}

  \textit{AMS subject classification (2010): 65M15} 

  \textit{Key words: time-fractional, subdiffusion, a posteriori error estimation, adaptive time stepping algorithm, higher order, collocation,
          L1-2 method, stable implementation}

  \section{Introduction}\label{sec:intro}

  We address the numerical solution of fractional-order parabolic equations, of order $\alpha\in(0,1)$, of the form
\beq\label{problem}
D_t^{\alpha}u+\LL u=f(x,t)\qquad\mbox{for}\;\;(x,t)\in\Omega\times(0,T],
\eeq
subject to an initial condition $u(\cdot,0)=u_0$ in $\Omega$, and the boundary condition $u=0$ on $\pt\Omega$ for $t>0$.
This problem is posed in a bounded Lipschitz domain  $\Omega\subset\R^d$ (where $d\in\{1,2,3\}$), and involves
a spatial linear second-order elliptic operator~$\LL=\LL(t)$ of the form
\beq\label{LL_def}
\LL u := -\sum_{i,j=1}^d \pt_{x_i}\!\bigl(a_{ij}(x,t)\,\pt_{x_j}\!u\bigr) + \sum_{i}^d b_i(x,t)\, \pt_{x_i}\!u+c(x,t)\,u,
\eeq
with a symmetric positive definite coefficient matrix $\{a_{ij}(x,t)\}_{i,j=1}^d$  $\forall (x,t)\in\Omega\times(0,T]$.
The Caputo fractional derivative in time, denoted here by $D_t^\alpha$, is  defined \cite{Diet10},
for $t>0$, by
\begin{equation}\label{CaputoEquiv}
D_t^{\alpha} u := J_t^{1-\alpha}(\pt_t u),\qquad
J_t^{1-\alpha} v(\cdot,t) :=  \frac1{\Gamma(1-\alpha)} \int_{0}^t(t-s)^{-\alpha}\, v(\cdot, s)\, ds,
\end{equation}
where $\Gamma(\cdot)$ is the Gamma function, and $\pt_t$ denotes the partial derivative in~$t$.

  The purpose of this paper is to explore and further develop the new methodology of the a-posteriori error estimation and adaptive time stepping proposed in \cite{Kopteva22} (see also a recent extension of this approach in \cite{Kopt_Stynes_apost22}).
  One distinctive feature of the present article is that we employ the approach of \cite{Kopteva22,Kopt_Stynes_apost22}
  in a wider context, to wide classes of temporal discretizations for \eqref{problem} of arbitrarily high order.
In comparison, only the L1 method was considered in \cite{Kopteva22,Kopt_Stynes_apost22}, while now we also address an L1-2 method proposed in \cite{GSZZ14} and a family of continuous collocation methods of arbitrary order.
It should be noted that despite a substantial literature on the a-priori error bounds for problem of type~\eqref{problem}, both
 on uniform
  and graded temporal meshes---see, e.g., \cite{JLZ19,Kopteva19,Kopteva21,Kopteva_Meng,Liao_etal_sinum2018,Liao_etal_sinum2019,LX16,SORG17} and references therein---%
the a-priori error analysis of the collocation methods appears very problematic on reasonably general meshes.
The adaptive algorithm based on our theory, by contrast,
yields reliable computed solutions and attains optimal convergence rates  in the presence of solution singularities
for all numerical approximations that we consider.

  We also note an interesting alternative approach
  to the a-posteriori error estimation for problems of type~\eqref{problem}
  recently proposed in \cite{B_Mark_apost}; however, the latter approach has been tested mainly
  on a-priori chosen meshes, and it remains unclear how it may be implemented in an adaptive time stepping algorithm (in view of the nonlocal nature of the estimators).

To give a flavour of the advantages in the accuracy of numerical approximations offered by our adaptive approach,
we compare the errors of 5 numerical methods on uniform temporal meshes (see Fig.~\ref{fig:apriorierr} left)
vs. adaptive meshes (Fig.~\ref{fig:apriorierr} right), with a striking improvement in the accuracy due
to the adaptive time stepping.
Here we consider the L1 method, an L1-2 method from \cite{GSZZ14}, and the continuous collocation methods of order 2, 4, and 8
(for details on the algorithm and the test problem, the reader is referred to Sections~\ref{sec_algorithm}--\ref{sec:numerics},
in particular, Section~\ref{ssec_numerics}).
Overall, here and in Section~\ref{sec:numerics}, we observe that the algorithm is capable of adapting the time steps to the initial singularity, as well to solution spikes away from the initial time.

    \begin{figure}[htb]
    \begin{center}
      \input{src/errors_eq_err}
      \input{src/errors_grad_err_pi2}
    \end{center}\vspace{-0.3cm}
    \caption{$L_\infty(0,T;\, L_\infty(\Omega))$ errors for various methods vs. number of time steps $M$ for Example~\ref{Ex1}, $\alpha=0.4$
    on uniform meshes (left), and adaptive meshes (right)
    with residual barrier $\RR_0$, $\lambda=\pi^2$, and $\omega=\lambda/8$\label{fig:apriorierr}
    }
  \end{figure}
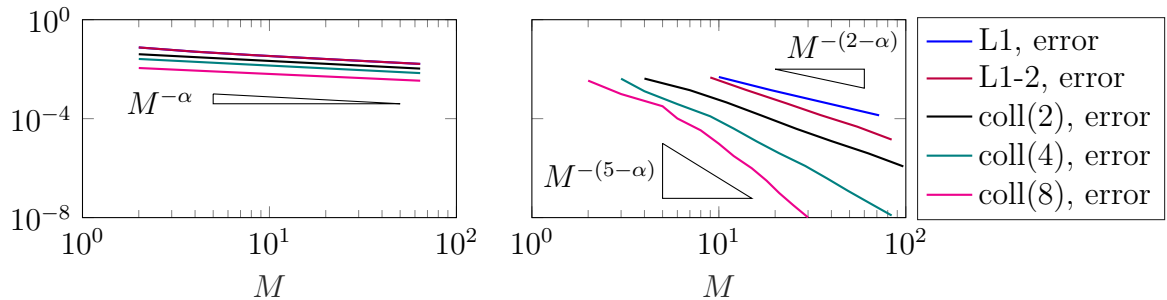

  The main findings of the paper are as follows.

  \begin{itemize}

  \item
  The considered adaptive technology is reliable in the sense that it is based on theoretical pointwise-in-time a-posteriori error bounds.
    Importantly, our adaptive algorithm is essentially independent of the method
    (or its order) and, additionally, does not require a preliminary a-priori error analysis either of the exact solution or  its numerical approximation.
    The latter may be important if the a-priori error analysis is lacking (such as for collocation methods)
     or limited to, e.g., uniform meshes.

    \item We demonstrate that high-order methods (of order up to as high as 8)
    exhibit a huge improvement in the accuracy when the time steps are chosen adaptively.
    In fact, our algorithm yields optimal convergence rates of order $q-\alpha$,
    where $q$ denotes the order of the method,
    either globally in time or in positive time
    (depending on the desired error profile used by the algorithm).
      At the same time, the algorithm is  capable of capturing
      {both initial singularities  and}
      local shocks/peaks in the solution.

      \item
    We make
  a few subtle improvements in the original version of the time stepping algorithm  \cite{Kopteva22}
  that substantially reduce the computational time.
    In particular, we modify the choice and search for a suitable initial time step,
    and also numerically test the algorithm parameters.

  \item  We provide clear and specific recommendations on the stable and efficient implementation of the resulting algorithm,
 which are essential, and not at all straightforward,  in the context of higher-order methods.
  Hence, we obtain
  numerically stable and efficient implementations for all considered methods
  (including computations of their residuals)
   with $\alpha$ at least within the range between $0.1$ and $0.999$ and
  for values of $TOL$ (used in the target bound for the error) as small as $10^{-8}$.



  \end{itemize}

  The paper is organised as follows.
  In Section~\ref{sec_apost} we recall a posteriori error estimates from \cite{Kopteva22} and
  give a few generalizations, such as for the  semilinear case.
  Next, in Section~\ref{sec_methods}, we introduce  numerical approximations for our problem \eqref{problem}
  and describe the evaluation of their residuals (which are required by the algorithm).
  The computationally stable implementation of these methods, as well the stable computation of the residuals,
  is addressed in Section~\ref{ssec_implem}, while
  our adaptive algorithm is described in Section~\ref{sec_algorithm}.
  Finally,
  in Section~\ref{sec:numerics}, we perform extensive numerical experiments to demonstrate the effectiveness and reliability of our adaptive approach.
  \medskip

\noindent{\it Notation.} We use the standard inner product $\langle\cdot,\cdot\rangle$ and the norm $\|\cdot\|$
in the space $L_2(\Omega)$, as well as the standard spaces 
$L_\infty(\Omega)$, $H^1_0(\Omega)$,
$L_{\infty}(0,t;\,L_2(\Omega))$, and $W^1_\infty(t',t'';\,L_2(\Omega))$
(see \cite[Section 5.9.2]{Evans10} for the notation used for functions of $x$ and~$t$).
The notation $v^+:=\max\{0,\,v\}$ is used for the positive part of a generic function $v$.

  \section{A posteriori error estimates}\label{sec_apost}

We start by recalling a few results from \cite{Kopteva22} and
then give a few generalizations, such as Lemma~\ref{lem_Linfty_new}, which will allow for more efficient algorithms,
and Corollary~\ref{cor_semi} for the semilinear case.
%
%
Define the operator $(D_t^\alpha+\lambda)^{-1}$ by
\beq\label{Dlammbda_inv}
(D_t^\alpha+\lambda)^{-1}v(t) :=
\int_0^t\! (t-s)^{\alpha-1}\,
E_{\alpha,\alpha}(-\lambda [t-s]^\alpha)\,v(s)\,ds
\quad\;\; \forall\,t>0.
\eeq
Here $E_{\alpha,\beta}(s)=\sum_{k=0}^\infty\{\Gamma(\alpha k+\beta)\}^{-1}s^k$
is a generalized Mittag-Leffler function.
The notation $(D_t^\alpha+\lambda)^{-1}$ reflects 
\cite[Remark~7.1]{Diet10}
that
\eqref{Dlammbda_inv} gives a solution $w$ of the equation
$(D_t^\alpha+\lambda)w(t)=v(t)$ for $t>0$ subject to $w(0)=0$.

\begin{theorem}[{\cite[Theorem~2.2]{Kopteva22}}]\label{the_L2}
Let $\LL$ in \eqref{problem}, for some $\lambda\in\R $, satisfy $\langle \LL v,v\rangle\ge \lambda\|v\|^2$ $\forall\,v\in H_0^1(\Omega)$.
Suppose a unique solution $u$ of \eqref{problem} and its approximation $u_h$ are
in $C([0,T]
;\,L_2(\Omega)) \cap W^1_\infty(\epsilon,t;\,L_2(\Omega))$ for any $0<\epsilon<t\le T$,
and also in
$H^1_0(\Omega)$ for any $t>0$,
while
$u_h(\cdot,0)= u_0$.
Then
the error of the latter is bounded in terms of its residual $R_h(\cdot,t)=(D_t^\alpha+\LL)u_h(\cdot,t)-f(\cdot,t)$
as follows:
\begin{align}
\label{L2_error}
\|(u_h-u)(\cdot,t)\|&\le (D_t^\alpha+\lambda)^{-1}\|R_h(\cdot, t)\|\qquad\forall\, t>0.
\end{align}
\end{theorem}

Note that the key ingredient in the proof of the above result is the bound \cite[Lemma~2.8]{Kopteva22}
$$
\langle  D_t^\alpha v(\cdot,t),\,v(\cdot,t) \rangle \ge \bigl (D_t^\alpha \|v(\cdot,t)\|\bigr)\|v(\cdot,t)\|\qquad\forall\,t>0,
$$
valid for any
$v\in L_{\infty}(0,t;\,L_2(\Omega)) \cap W^1_\infty(\epsilon,t;\,L_2(\Omega))$ for any $0<\epsilon<t\le T$,
subject to $v(\cdot,0)=0$.
Hence, one  gets
$(D_t^\alpha+\lambda) \|(u_h-u)(\cdot,t)\|\le \|R_h(\cdot,t)\|$ $\forall\, t>0$, which then yields~\eqref{L2_error}.

Furthermore, one gets a version of Theorem~\ref{the_L2} for the $L_\infty(\Omega)$ norm.

\begin{theorem}[{\cite[Theorem~3.2]{Kopteva22}}]\label{the_Linfty}
Suppose that the coefficients of $\LL$ in \eqref{problem} are sufficiently smooth, and $c\ge\lambda$ for some $\lambda\in\R $.
Let a unique solution $u$ of \eqref{problem} and its approximation $u_h$
 be in $ C^2(\Omega)$ for each~$t>0$, and, for each $x\in\Omega$, belong to $W^1_\infty(\epsilon,t)$ for any $0<\epsilon<t\le T$, while
$u_h(\cdot,0)= u_0$.
Then the error bound \eqref{L2_error} of Theorem~\ref{the_L2} 
remains true with $\|\cdot\|=\|\cdot\|_{L_2(\Omega)}$
replaced by $\|\cdot\|_{L_\infty(\Omega)}$.
\end{theorem}

\begin{remark}
Note that in \cite{Kopteva22}, the above Theorem~\ref{the_Linfty} was given for the case $\lambda\ge0$ (and also for $L$ without mixed derivatives). But in view of a more recent paper \cite{Kopteva_ML_max_pr} addressing the maximum principle for the case of
a reaction coefficient of arbitrary sign (see also \cite{Luchko_Yam_2017} for a self-adjoint time-independent $\LL$), the proof in \cite{Kopteva22} also applies to this more general case.
\end{remark}

While Theorems~\ref{the_L2} and~\ref{the_Linfty} give computable a-posteriori error estimates on any given temporal mesh,
it is not immediately clear
 how the time steps may be chosen adaptively to attain a certain solution accuracy, or, more ambitiously, a certain pointwise-in-time error profile. This is addressed by the next result, which is a version of  \cite[Corollary~2.3]{Kopteva22}.

\begin{corollary}[residual barrier]\label{cor1_L2}
Suppose that $p\in\{2,\infty\}$, and  for some non-negative barrier function $\EE \in W^1_\infty(\epsilon,t)$ for any $0<\epsilon<t\le T$,
such that  $\lim_{t\to 0^+}\EE(t)\ge 0$ exists,
one
has
\beq\label{R_h_bound}
\|R_h(\cdot,t)\|_{L_p(\Omega)}\le 
(D_t^\alpha+\lambda)\EE(t)\qquad\forall\,t>0.
\eeq
Then,
under the conditions of Theorem~\ref{the_L2} if $p=2$, and under the conditions of Theorem~\ref{the_Linfty} if $p=\infty$,
one has $\|(u_h-u)(\cdot,t)\|_{L_p(\Omega)}\le \EE(t)$ $\forall\,t\ge0$.
\end{corollary}

Possible choices of $\EE$ are discussed in Section~\ref{ss_EE}.
Note that while $\EE \in W^1_\infty(\epsilon,t)$ for any $0<\epsilon<t\le T$ implies that $\EE\in C(0,T]$, it is convenient to choose
$\EE$ such that $\EE(0)=0$ and $\EE(0^+):=\lim_{t\to 0^+}\EE(t)> 0$, i.e. discontinuous at $t=0$.
{(To be more precise, setting $\EE(0)=0$ yields the least restrictive barrier on the residual, while retaining $\EE(t)\ge 0$ $\forall\,t\ge0$ \cite{Kopteva22}.)}

Note that for some operators, such as $\LL=-\triangle=-\sum_{i=1}^d\partial_{x_i}^2$, Theorem~\ref{the_Linfty} is applicable with $\lambda=0$,
while a negative reaction coefficient $c$ in \eqref{LL_def} would imply that $\lambda<0$, which would limit the applicability of our results in
Section~\ref{ss_EE}. To rectify this, we now establish an improved version of
Corollary~\ref{cor1_L2} for $p=\infty$.

\begin{lemma}[improved residual barrier for $L_\infty(\Omega)$]\label{lem_Linfty_new}
Suppose that for $\lambda,{\omega}\in\R$, where $\omega\ge 0$, there exists a function $g\in C^2(\Omega)$ such that
$\LL g\ge \lambda$ and $1\le g\le 1+\omega$ in $\Omega$
(if $\LL=\LL(t)$, then $\LL(t) g\ge \lambda$ $\forall\,t>0$).
Also, suppose that
for some non-negative barrier function $\EE \in W^1_\infty(\epsilon,t)$ for any $0<\epsilon<t\le T$,
such that  $\lim_{t\to 0^+}\EE(t)\ge 0$ exists and $\omega D_t^\alpha\EE(t)\ge 0$  $\forall\,t>0$,
one
has
\beq\label{R_h_bound_new}
\|R_h(\cdot,t)\|_{L_\infty}\le 
\frac{(D_t^\alpha+\lambda)\EE(t)}{1+\omega}\qquad\forall\,t>0.
\eeq
Then,
under the conditions  of Theorem~\ref{the_L2},
one has $\|(u_h-u)(\cdot,t)\|_{L_\infty}\le \EE(t)$ $\forall\,t\ge0$.
\end{lemma}

\begin{proof}
Set $\hat\EE(x,t):= g(x)\, \EE(t)$. Then
{
$D_t^\alpha \hat\EE= g D_t^\alpha\EE\ge D_t^\alpha\EE$
(in view of $(1-g) D_t^\alpha\EE\ge 0$ whether $\omega=0$ or $\omega>0$), so}
$$
(D_t^\alpha+\LL)\hat\EE(x,t)\ge (D_t^\alpha+\lambda) \EE(t)\ge (1+\omega) |R_h(x,t)|=(1+\omega) |(D_t^\alpha+\LL)(u_h-u)(x,t)|.
$$
Now, an application of the maximum principle for the operator $D_t^\alpha+\LL$ (see, e.g., \cite{Kopteva_ML_max_pr}) yields
$$
(1+\omega)|(u_h-u)(x,t)|\le \hat\EE(x,t)=g(x)\, \EE(t)\le (1+\omega)\, \EE(t).
$$
This immediately implies \eqref{R_h_bound_new}.
%
\end{proof}

\begin{remark}[$\lambda$ and $\omega$ in Lemma~\ref{lem_Linfty_new}]\label{rem_lem_Linfty_new1}
(i) Setting $\omega:=0$ and $g:=1$ in Lemma~\ref{lem_Linfty_new} immediately yields Corollary~\ref{cor1_L2} for $p=\infty$
with $\lambda:=\inf_{\Omega\times(0,T)}c$.\\
(ii) If $c\ge 0$, for any $\lambda>0$, one may choose $g$ such that $\LL g=\max\{\lambda,\,c\}$ in $\Omega$, subject to $g=1$ on $\pt\Omega$, and $1+\omega:=\sup_{\Omega}g$ (then
$L(g-1)\ge 0$, so,
by the maximum principle, $g\ge 1$).
For example, if $\Omega=(0,1)$ and $\LL:=-\pt^2_x$, then $g=1+\frac12 \lambda\,x(1-x)$ yields $\omega=\frac18\lambda$ for any $\lambda\ge 0$.
{For the same $L$ on a more general $\Omega=(0,\bar x)$ one similarly gets $g=1+\frac12 \lambda\,x(\bar x-x)$, so $\omega=\frac18\lambda\bar x^2$ for any $\lambda\ge 0$.}
\\
(iii) Even if $c<0$, in some cases one may still use Lemma~\ref{lem_Linfty_new} with positive $\lambda$ and $\omega$.
For example, if $\Omega=(0,1)$ and $\LL:=-\pt^2_x-c_0$ for some constant $0<c_0<8$, then $g=1+4\omega\,x(1-x)\le 1+\omega$ yields
$\LL g\ge 8\omega-c_0(1+\omega)= :\lambda$. So for any $\lambda>0$, we can choose $\omega=\omega(\lambda)$ to be used in \eqref{R_h_bound_new}.
\end{remark}

\begin{remark}[flexibility of \eqref{R_h_bound_new} vs. \eqref{R_h_bound}]\label{rem_lem_Linfty_new2}
It may appear that
the new residual barrier \eqref{R_h_bound_new} is more restrictive compared to \eqref{R_h_bound}.
In fact,  \eqref{R_h_bound_new} is {not only} more general (as it reduces to \eqref{R_h_bound} in a particular case of $\omega=0$).
Importantly, by allowing larger values of $\lambda$, \eqref{R_h_bound_new} weakens the restriction on the residual (albeit with an additional factor $(1+\omega)^{-1}$).
This additional flexibility allows for more efficient time stepping algorithms.
\end{remark}

\subsection{Residual profiles for $\lambda\ge 0$}\label{ss_EE}
Corollary~\ref{cor1_L2} seems to imply 
that there is abundant flexibility
in the choice of a desirable
 pointwise-in-time error profile $\EE(t)$.
However, one needs to ensure that the non-local inequality $(D_t^\alpha+\lambda)\EE(t)>0$ holds true  $\forall\,t>0$.
Furthermore, one should avoid a positive $(D_t^\alpha+\lambda)\EE(t)$ becoming too small at any time $t=t^*>0$, as the latter,
combined with a suitable
adaptive time stepping algorithm attempting to attain \eqref{R_h_bound},
 may lead to the local time step near $t^*$ becoming unpractically small, or, even worse, the adaptive
  algorithm failing
to satisfy the required bound \eqref{R_h_bound}
(as $R_h$ is also non-local).

The following lemma describes
two possible error profiles, which are motivated by the pointwise-in-time a-priori error analyses
\cite{Kopteva_Meng,Kopteva21}; see also a discussion in Remark~\ref{rem_apriori}.


\begin{lemma}[{\cite[Corollary~2.4]{Kopteva22}}]\label{lem_L2_bounds}
Suppose that $p\in\{2,\infty\}$ and $\lambda\ge0$.
Then,
under the conditions of Theorem~\ref{the_L2} if $p=2$, and under the conditions of Theorem~\ref{the_Linfty} if $p=\infty$,
for the error $e=u_h-u$
 one has
\begin{subequations}\label{barrier_bounds}
\begin{align}
&\|e(\cdot,t)\|_{L_p(\Omega)}\le \sup_{0<s\le t}\!\left\{ \frac{ \|R_h(\cdot,s)\|_{L_p(\Omega)}}{\RR_0(s) }\right\},
&&
\RR_0(t):=\{\Gamma(1-\alpha)\}^{-1}\,t^{-\alpha}+\lambda,
\label{barrier_bounds_a}
\\[0.3cm]
&\|e(\cdot,t)\|_{L_p(\Omega)}\le t^{\alpha-1}\!
\sup_{0<s\le t}\!\left\{ \frac{ \|R_h(\cdot,s)\|_{L_p(\Omega)}}{\RR_1(s) }\right\},
&&\RR_1(t):=\{\Gamma(1-\alpha)\}^{-1}\,t^{-1}\rho(\tau/t)+\lambda\,\EE_1(t),
\label{barrier_bounds_b}
\end{align}
\beq
\EE_1(t):=\max\{\tau, t\}^{\alpha-1},\qquad
\rho(s):=s^{-\beta}[1-((1-s)^+)^\beta]
,
\qquad\beta:=1-\alpha,
\eeq
where $\tau>0$ is an arbitrary parameter (and $t^{\alpha-1}$ in \eqref{barrier_bounds_b} can be replaced by $\EE_1(t)$).
\end{subequations}
\end{lemma}

The above lemma may be reformulated for the purpose of a possible
adaptive time stepping algorithm with some desirably small positive  $TOL$,
as follows:
\begin{subequations}\label{barrier_bounds_alg}
\begin{align}
&
\|R_h(\cdot,t)\|_{L_p(\Omega)}\le TOL\cdot\RR_0(t)\;\;\;\forall\,t>0
&\Rightarrow\quad&
\|e(\cdot,t)\|_{L_p(\Omega)}\le TOL,
\label{barrier_bounds_alg_a}
\\[0.3cm]
&
\|R_h(\cdot,t)\|_{L_p(\Omega)}\le TOL\cdot\RR_1(t)\;\;\;\forall\,t>0
&\Rightarrow\quad&
\|e(\cdot,t)\|_{L_p(\Omega)}\le TOL\cdot  t^{\alpha-1}.
\label{barrier_bounds_alg_b}
\end{align}
\end{subequations}
Hence, $\|R_h(\cdot,t)\|_{L_p(\Omega)}\le TOL\cdot\RR_l(t)$, with $l\in\{0,1\}$ and $p\in\{2,\infty\}$, can be employed as a criterion for the adaptive time stepping
(see Section~\ref{sec_algorithm} for further details on such algorithms).

Furthermore,
an inspection of the proof of Lemma~\ref{lem_L2_bounds} (given in \cite{Kopteva22})
shows that under the conditions of Lemma~\ref{lem_Linfty_new} one
immediately gets
 more general versions of
\eqref{barrier_bounds_alg_a} and \eqref{barrier_bounds_alg_b} for $p=\infty$; see below.
 These new versions 
 are of interest
 since they are valid for possibly larger values of $\lambda$ in the definitions of $\RR_0(t)$ and $\RR_1(t)$
 (see Remarks~\ref{rem_lem_Linfty_new1} and~\ref{rem_lem_Linfty_new2}).

 \begin{corollary}\label{cor_new}
 Under the conditions of Lemma~\ref{lem_Linfty_new}, for the error $e=u_h-u$, one has
\begin{subequations}\label{barrier_bounds_alg_new}
\begin{align}
&
\|R_h(\cdot,t)\|_{L_\infty(\Omega)}\le \frac{TOL\cdot\RR_0(t)}{1+\omega}\;\;\;\forall\,t>0
&\Rightarrow\quad&
\|e(\cdot,t)\|_{L_\infty(\Omega)}\le TOL,
\label{barrier_bounds_alg_new_a}
\\[0.3cm]
&
\|R_h(\cdot,t)\|_{L_\infty(\Omega)}\le \frac{TOL\cdot\RR_1(t)}{1+\omega}\;\;\;\forall\,t>0
&\Rightarrow\quad&
\|e(\cdot,t)\|_{L_\infty(\Omega)}\le TOL\cdot  t^{\alpha-1}.
\label{barrier_bounds_alg_new_b}
\end{align}
\end{subequations}
\end{corollary}


\begin{remark}[error profiles v pointwise a-priori error bounds]\label{rem_apriori}
{Suppose that $u$ exhibits an initial singularity of type $t^\alpha$, typical for this problem, with the derivative bounds $\|\pt_t^lu(\cdot, t)\|_{L_p(\Omega)}\le C t^{\alpha-l}$ $\forall\,t>0$, $1\le l\le q$, with some integer $q\ge2$, constant $C>0$, and $p\in\{2,\infty\}$.}
Then the error bounds of type \cite[(3.2)]{Kopteva_Meng} and
 \cite[(4.2)]{Kopteva21} imply that
given a method of order $q$
on a graded mesh $\{T(j/M)^r\}_{j=0}^M$
(with $q=2$ for the L1 method), depending on the degree of grading, the error is either proportional to
$t^{\alpha-1}$ or (if the grading parameter $r$ exceeds $q-\alpha$) to  $t^{\alpha-(q-\alpha)/r}$.
Hence, the two error profiles of interest that we consider are proportional to $\EE_0(t)=1$ for $t>0$, or $\EE_1(t)=t^{\alpha-1}$ for $t>\tau>0$; see \eqref{barrier_bounds_alg_a} and \eqref{barrier_bounds_alg_b}, respectively.
(To be more precise, $\EE_1(t):=\max\{\tau, t\}^{\alpha-1}$, while $\EE_0(0)=\EE_1(0)=0$.
{In fact, $\RR_0$ and $\RR_1$ in \eqref{barrier_bounds} and, hence, \eqref{barrier_bounds_alg}
are obtained simply by an application of $D_t^\alpha+\lambda$
to respectively $\EE_0$ and $\EE_1$ \cite{Kopteva22}.})

With these two choices,
the a-priori error bounds from \cite{Kopteva_Meng,Kopteva21} suggest that
 the error is expected to be respectively $\lesssim M^{q-\alpha}$ or $\lesssim M^{q-\alpha}t^{\alpha-1}$ $\forall\,t\in[0,T]$, which
agrees, and surprisingly well, with numerical results in Section~\ref{ssec_numerics}.
 Note also that these convergence rates are consistent with those
 in  \cite{Kopteva_Meng,Kopteva21}
 on a-priori chosen graded meshes
 with $r=(q-\alpha)/\alpha$
 and $r=2-\alpha$, respectively (in the latter case, up to the logarithmic term $\ln M$).
\end{remark}

\begin{remark}[$\lambda<0$]
Strictly speaking, Lemma~\ref{lem_L2_bounds} also applies to the case $\lambda<0$.
However, in this case both $\RR_0$ and $\RR_1$ become negative at some $t>0$, so the residual bound
$ \|R_h(\cdot,s)\|_{L_p(\Omega)}\le TOL\cdot\RR_l(t)$ (with $l=0,1$)
cannot be attained.
 One possible remedy is to replace $\RR_l$  by $\RR^*_l:=\max\{\RR_l,\varepsilon\}$ with some small parameter $\varepsilon$,
 in which case the error will be bounded by $TOL\cdot\EE^*(t)$, where
 $\EE^*:=(D_t^\alpha+\lambda)^{-1}\RR^*_l$. Clearly, one will enjoy  $\EE^*=\EE_l$ for $t\le t^*$ as long as $\RR^*_l= \RR_l$ $\forall\,t\in(0,t^*]$. Afterwards $\EE^*$ may be computed with sufficiently high accuracy by solving the fractional ODE $(D_t^\alpha+\lambda)\EE^*=\RR^*$ numerically on a very fine mesh.
\end{remark}

\subsection{Generalization for the semilinear case}\label{ss_semi}
One can easily extend the above results to the following semilinear version of \eqref{problem}:
\beq\label{semi_problem}
D_t^{\alpha}u+\LL u+g(x,t,u)=f(x,t)\qquad\mbox{for}\;\;(x,t)\in\Omega\times(0,T],
\eeq
assuming that $g$ is sufficiently smooth and, with some $\mu\in\R$, satisfies
$$
\partial_v g(x,t,v)\ge\mu\qquad\forall (x,t,v)\in\Omega\times(0,T]\times\R.
$$
Then, in view of  the standard linearization
$$
g(x,t,u_h)-g(x,t, u)=\hat c(x,t)\,(u_h-u),\qquad
\hat c:=\int_0^1 \pt_v g(x,t,u+s(u_h-u))\,ds\ge\mu,
$$
the error satisfies $(D^\alpha_t+\LL+\hat c)(u_h-u)=R_h$, with the updated definition of the residual
$$
R_h:=D_t^{\alpha}u_h+\LL u_h+g(x,t,u_h)-f(x,t).
$$
\begin{corollary}[semilinear case]\label{cor_semi}
Assume that $\langle \LL v,v\rangle\ge \lambda^*\|v\|^2$ in Theorem~\ref{the_L2} for some $\lambda^*\in\R$
(instead of $\langle \LL v,v\rangle\ge \lambda\|v\|^2$), or, similarly, $c\ge\lambda^*$ in Theorem~\ref{the_Linfty}
(instead of $c\ge\lambda$). Then one gets
the error bound~\eqref{L2_error} with  $\lambda:=\lambda^*+\mu$. 
In the latter $\|\cdot\|$ is
understood as $\|\cdot\|_{L_p(\Omega)}$ with $p=2$ or $p=\infty$, respectively.
A version of Corollary~\ref{cor1_L2}, as well as a version of Lemma~\ref{lem_L2_bounds}, is also valid for the semilinear equation~\eqref{semi_problem}.
\end{corollary}

  \section{Numerical approximations and their residuals}\label{sec_methods}
  In this Section
  we describe several numerical approximations for our time-fractional problem \eqref{problem} and also discuss the evaluation of their residuals
  (the latter are to be used by the adaptive algorithm considered in Section~\ref{sec_algorithm}).
  All numerical methods are presented relative to an arbitrary temporal mesh $\{t_k\}_{k=0}^M$ covering $[0,T]$ with intervals $(t_{k-1},t_k]$ of width $\tau_k$,
  and are conveniently described using certain continuous piecewise-polynomial functions in time.

  \subsection{L1 method}

  We start with the popular L1 method; see, e.g., \cite{JLZ19,Stynes21} and references therein.
  Defining the numerical approximation $u_h$  in $\bar\Omega\times[0,T]$ as continuous piecewise-linear in time,
  one can describe the L1 method by
  \beq\label{L1method}
(D_t^\alpha+\LL)\,u_h(x, t_k)=f(x, t_k)\qquad \mbox{for}\;\; x\in\Omega,\;\;  k=1,\ldots, M,
\eeq
subject to $u_h^0:=u_0$ and $u_h=0$ on $\pt \Omega$.

To be more precise, with the notation $U_k:=u_h(\cdot,t_k)$,
  \[
    u_h(t)\big|_{[t_{k-1},t_k]}:=U_{k-1}\,\phi_k^0(t)+
    U_k\,\phi_k^1(t),
  \]
  where
 \beq\label{phi12}
    \phi_k^0(t):=\frac{t_k-t}{\tau_k},\qquad
    \phi_k^1(t):=\frac{t-t_{k-1}}{\tau_k}=1-\phi_k^0(t).
  \eeq

 To implement the L1 method, one needs to evaluate the non-local $D^\alpha_t u_h(\cdot,t_k)$ in terms of $\{U_k\}$.
 More generally, to compute the residual $R_h$ (to be used by the adaptive algorithm), one needs to compute $D^\alpha_t u_h(\cdot, t)$ for any $t>0$.
 For $t_{k-1}< t\leq t_k$ $\forall$\,$k\ge1$, a straightforward calculation using \eqref{CaputoEquiv} yields
  \begin{align}\label{D_alph_L1}
    D^\alpha_t u_h(t)
      &= \frac{1}{\Gamma(2-\alpha)}\left(-\sum_{j=1}^{k-1}\frac{U_j-U_{j-1}}{\tau_j}(t-s)^{1-\alpha}\big|_{s=t_{j-1}}^{s=t_j}
                                                         +\frac{U_k-U_{k-1}}{\tau_k}(t-t_{k-1})^{1-\alpha}\right).
  \end{align}
  Stable implementations of this method, as well as the other considered methods, will be discussed in Section~\ref{ssec_implem}, while
  the efficient computation of the residuals for all considered methods will be addressed in Section~\ref{ssec_residuals}.

  For the latter, note that \eqref{L1method} immediately implies
  for the residual that $R_h(\cdot, t_k)=0$ for $k\ge 1$;
  hence on each $(t_{k-1},t_k)$ for $k>1$, the residual is a non-symmetric bubble.
  This is illustrated by Figure~\ref{fig:bubblesL1},
  which shows a typical behaviour for the residual of the L1-method on an equidistant temporal mesh of four 
  cells.

  \begin{figure}[t]
    \begin{center}
      \input{src/residualsL1}
    \end{center}\vspace{-0.6cm}
    \caption{The residual
     $\|R_h\|_{L_\infty(\Omega)}$ of the L1-method\label{fig:bubblesL1} for problem \eqref{problem} with $\alpha=0.8$, $L=-\triangle$, $f=1+t$, $\Omega=(0,1)$, $u_0=0$.}
  \end{figure}
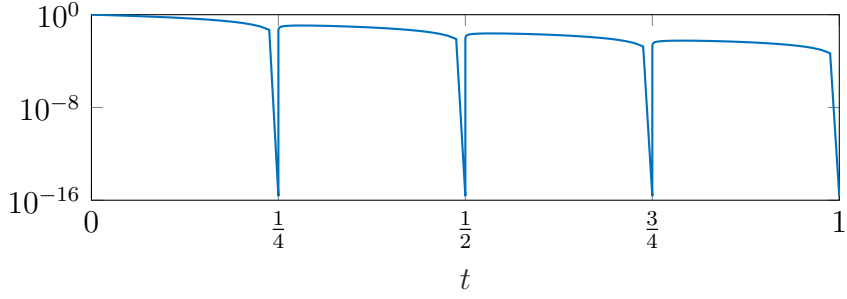

  \subsection{L1-2 method}
  A natural improvement over the L1-method is to use a piecewise quadratic $u_h$.
  Assuming the associated degrees of freedom are $\{U_k=u_h(\cdot,t_k)\}_{k=0}^M$,
  there are several possibilities of defining such a method; see, e.g., \cite{LX16,Kopteva21}.
  We will use an alternative method proposed in \cite{GSZZ14} as it employs a more natural (backward quadratic) interpolation of the computed solution
  between time layers, which allows for a simpler evaluation of the residuals (see a discussion in Section~\ref{ssec_residuals}).
  Note that, in contrast with \cite{LX16,Kopteva21}, we are not aware of any a-priori error analysis for the L1-2 method of \cite{GSZZ14}; nevertheless, our a-posteriori error analysis applies seamlessly to this method.

  Let $u_h$ be continuous in time, linear on the first interval $[0,t_1]$, and
  piecewise-quadratic on $[t_1,T]$ as follows:
  \begin{align*}
    u_h(t)\big|_{[0,t_1]}&:=U_0\,\phi_1^0(t)+U_1\,\phi_1^1(t),\\[4pt]
    u_h(t)\big|_{[t_{k-1},t_k]}&:=U_{k-1}\,\phi_k^0(t)+U_k\,\phi_k^1(t)+y_k\,\phi_k^2(t)\quad\forall\,k\ge2.
  \end{align*}
  Here
  \[
    y_k:=\frac{\frac{U_k-U_{k-1}}{\tau_k}-\frac{U_{k-1}-U_{k-2}}{\tau_{k-1}}}{\tau_k+\tau_{k-1}}\,,
  \]
  while $\phi_k^0$ and $\phi_k^1$ are from \eqref{phi12}, and
  \[
    \phi_k^2(t):=(t-t_{k-1})(t-t_k).
  \]

  With the above piecewise-quadratic $u_h$, the L1-2 method of \cite{GSZZ14}
  can be described, similarly to the L1 method, by
  \eqref{L1method}.
  Hence, for the residual one again gets $R_h(\cdot, t_k)=0$ for $k\ge 1$,
  i.e. the residual remains a non-symmetric bubble
  on each $(t_{k-1},t_k)$ for $k>1$.
 Figure~\ref{fig:bubblesL12}
 shows a typical behaviour for the residuals of the L1-2 method on an equidistant mesh of four cells.
  \begin{figure}[tb]
    \begin{center}
      \input{src/residualsL12}
    \end{center}\vspace{-0.6cm}
    \caption{Typical behaviour of the residual of the L1-2 method\label{fig:bubblesL12}
    (shown in the same setting as in Figure~\ref{fig:bubblesL1}).}
  \end{figure}
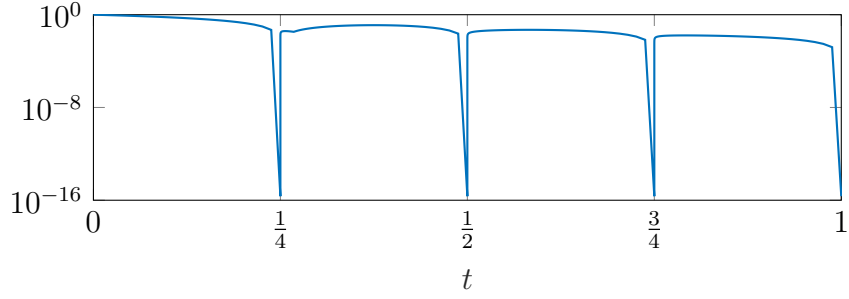

   To implement this method, one needs to reevaluate $D^\alpha_t u_h(\cdot,t_k)$, while to
   compute the residual $R_h$, one, more generally,  needs to compute the non-local $D^\alpha_t u_h(\cdot,t)$ for any $t>0$.
   For the first time interval, i.e. for $0<t\le t_1$ we recall
   \eqref{D_alph_L1}, which simplifies to
     \begin{align*}
    D_t^\alpha u_h(t)
      &= \frac{1}{\Gamma(2-\alpha)}\,\frac{U_1-U_{0}}{\tau_1}\,t^{1-\alpha}.
  \end{align*}
  For $t_{k-1}< t\leq t_k$ $\forall\,k\ge2$, a more elaborate calculation using \eqref{CaputoEquiv} yields
  \begin{align*}
    D_t^\alpha u_h(t)
      = &\frac{1}{\Gamma(2-\alpha)}\left(-\sum_{j=1}^{k-1}\frac{U_j-U_{j-1}}{\tau_j}(t-s)^{1-\alpha}\big|^{t_j}_{t_{j-1}}
                        +\frac{U_k-U_{k-1}}{\tau_k}(t-t_{k-1})^{1-\alpha}\right)\\
          &+\frac{1}{\Gamma(1-\alpha)}\,\sum_{j=2}^{k-1}y_j\left( \frac{2}{2-\alpha}(t-s)^{2-\alpha}\big|_{t_{j-1}}^{t_j}
                                                                     -\frac{2t-{t_{j-1}-t_j}}{1-\alpha}(t-s)^{1-\alpha}\big|_{t_{j-1}}^{t_j}
                                                                     \right)\\
          &+\frac{1}{\Gamma(1-\alpha)}\,y_k\left( -\frac{2}{2-\alpha}(t-t_{k-1})^{2-\alpha}
                                                                     +\frac{2t-t_{k-1}-t_k}{1-\alpha}(t-t_{k-1})^{1-\alpha}
                                                                     \right).
  \end{align*}

  \subsection{Continuous collocation methods}\label{ssec_collocation}
  To define higher-order methods,
  it is convenient to define computed solutions $u_h$ as piecewise polynomials of higher degrees
  within the framework of continuous collocation methods; see, e.g.,~\cite{Brunner04}.
  Note that we are not aware of any a-priori error analysis for high-order collocation methods in the context of time-fractional parabolic equation~\eqref{problem}; nevertheless, our a-posteriori error analysis and our adaptive algorithm are immediately applicable to such methods,
  as is also demonstrated by numerical experiments in Section~\ref{sec:numerics}.
 A word of caution should be added: one needs to ensure a stable implementation of the method itself and of the adaptive time stepping algorithm. While a careful implementation is essential even for the L1 method,
  higher-order methods particularly require a more sophisticated treatment; see Sections~\ref{ssec_residuals} and \ref{ssec_implem} for further discussion.

 Let $u_h$ be a piecewise polynomial of degree $m\ge1$, defined on a subgrid  of
  collocation points $\{t_k^\ell\}_{\ell\in\{0,\dots,m\}}$ on each time interval $[t_{k-1},t_k]$ with $t_k^\ell:=t_{k-1}+c_\ell\cdot(t_k-t_{k-1})$,
  $\{c_\ell\}\subset[0,1]$, $c_0=0$, and $c_m=1$.
 While the  choice of the collocation
  points is quite arbitrary, in our experiments we shall simply use equidistant points.

  Now,
with any set $\{\phi_k^\ell\}$ of basis functions of $\PS_m(t_{k-1},t_k)$, the polynomial space of degree $m$ over $[t_{k-1},t_k]$,
on which, for convenience, we impose
  \beq\label{basis_delta}
  \phi_k^\ell(t^0_k)=\phi_k^\ell(t_{k-1})=\delta_{\ell,0}\quad\mbox{and}\quad \phi_k^\ell(t^m_k)=\phi_k^\ell(t_k)=\delta_{\ell,m}
    \qquad\forall\,
    \ell\in\{0,\dots,m\},
\eeq
  let
  \begin{align*}
    u_h(t)\big|_{(t_{k-1},t_k)}
      &=\sum_{\ell=0}^m U_k^\ell\,\phi_k^\ell(t)= U_{k-1}^m\,\phi_k^0(t)+\sum_{\ell=1}^m U_k^\ell\,\phi_k^\ell(t)\,.
  \end{align*}
Here, in agreement with \eqref{basis_delta}, $U_k^0:=u_h(\cdot,t_{k-1})$ and $U_k^m:=u_h(\cdot,t_k)$,
  and we additionally impose the continuity of $u_h$ in time, which is equivalent to $U_{k-1}^m=U_{k}^0$, while $U_{0}^m:=U_{1}^0$.

  With the above definitions, a continuous collocation method is given by
  \beq\label{Col_method}
    (D_t^\alpha+\LL)\,u_h(x, t^\ell_k)=f(x, t_k^\ell)\qquad \mbox{for}\;\; x\in\Omega,\;\;  \ell\in\{1,\ldots, m\},\;\;k=1\ldots, M,
  \eeq
  subject to $u_h^0:=u_0$ and $u_h=0$ on $\pt \Omega$.
  A comparison with \eqref{L1method} shows that for $m=1$ the above collocation method is identical with the L1 method.

  To implement the above method, one needs to evaluate $D_t^\alpha u_h(\cdot,t_k^\ell)$, while to compute the residual, one requires
  a more general $D_t^\alpha u_h(\cdot,t)$. For the latter, for $t_{k-1}< t\leq t_k$ $\forall\,k\ge 1$, one gets
  \begin{align}\notag
    D_t^\alpha u_h(t)
      &=\frac{1}{\Gamma(1-\alpha)}\bigg(\sum_{j=1}^{k-1}\int_{t_{j-1}}^{t_j}\!\!\bigg[ U_{j-1}^m\,\partial\phi_j^0(s)
                                                                                +\sum_{\ell=1}^m
                                                                                U_j^\ell\,\partial\phi_j^\ell(s) \bigg](t-s)^{-\alpha}\ds+\\&\hspace{3cm}
                                        \int_{t_{k-1}}^t\!\!
                                        \bigg[ U_{k-1}^m\,\partial\phi_k^0(s)
                                                                                +\sum_{\ell=1}^m
                                                                                U_k^\ell\,\partial\phi_k^\ell(s) \bigg](t-s)^{-\alpha}\ds\
                                        \bigg).
\label{D_alp_coll}
  \end{align}
  A stable implementation of this formula will be discussed in the Section~\ref{ssec_implem}.

  Thus, to implement the collocation method \eqref{Col_method}, on each time interval $(t_{k-1},t_k]$ one needs to solve a linear system in
  $\{U_k^\ell\}_{\ell=1}^m$; the right-hand side of this linear system is computed using the values of $u_h$ for $t\le t_{k-1}$.

  Note that the definition of the method \eqref{Col_method}  immediately implies that $R_h(\cdot, t_k^\ell)=0$, i.e. the residual vanishes at all collocation points, except for $t=0$.
  A typical behaviour of the residuals is illustrated in
  Figure~\ref{fig:bubblesColl}. 
  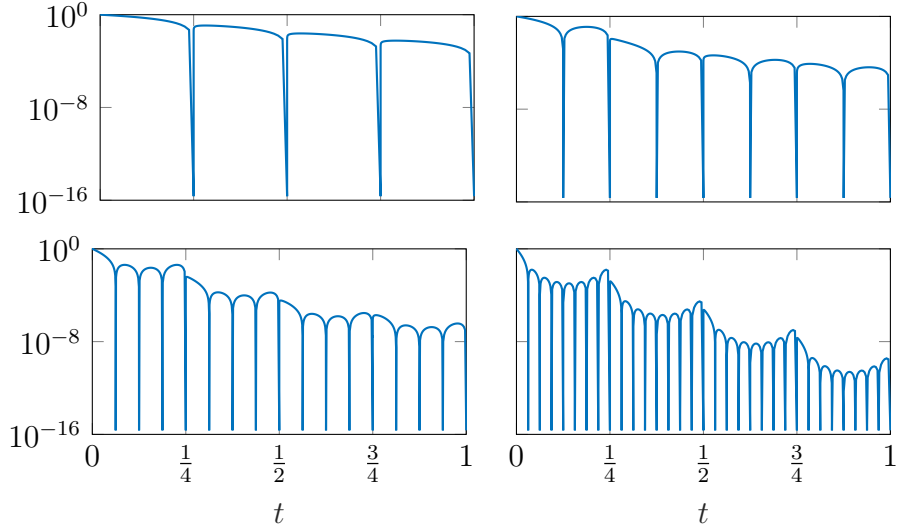
\begin{figure}[htb]
    \begin{center}
      \input{src/residualsColl1}
      \input{src/residualsColl2}\\
      \input{src/residualsColl4}
      \input{src/residualsColl8}
    \end{center}\vspace{-0.6cm}
    \caption{Typical behaviour of the residuals of collocation methods on an equidistant temporal mesh of four cells
    for $m\in\{1,2,4,8\}$ from top left to bottom right\label{fig:bubblesColl}  (shown in the same setting as in Figure~\ref{fig:bubblesL1}).}
  \end{figure}
  Importantly, in view of the bubble nature of the residuals, when computing the residual $R_h$, one should employ sufficiently many well distributed
  sampling points between any two consecutive collocation points.
  For example, in our experiments we have used $20$ sampling points per time interval; see also Figure~\ref{fig:sampling}.

    \subsection{Computation of the residuals. Discussion of alternative discretizations}\label{ssec_residuals}
  Recall that the residuals of all considered methods exhibit a bubble-type behaviour. More formally,
  let  $R_h^I$ be the continuous piecewise-polynomial interpolant of  $R_h$ in time,
  defined using exactly the same interpolation points and definitions as used for the corresponding computed solution $u_h$.
  Then the definitions of the above methods immediately imply that $R_h^I(t)=0$ for $t\ge t_1$.
  More generally,
  for the L1 and L1-2 methods, note that $D_t^\alpha u_h^0(\cdot, 0)=0$ (as $u_h$ is linear on the first time interval);
  so  $R_h(\cdot,0)=\LL u_0-f(\cdot,0)$ implies that
  $R_h^I(t)=[\LL u_0-f(\cdot, 0)](1-t/t_1)^+$ for $t>0$.
  For the considered collocation methods with $m>1$, one additionally has
  $R_h^I(t_1^\ell)=0$ at all interior collocation points over the first time interval, so
  $R_h^I(t)=R_h^I(\cdot,0)\,\varphi_1^0(t)$ for $t>0$, where
  $\varphi_1^0\in\PS_m(0,t_1)$, subject to $\varphi_1^0(t_1^\ell)=\delta_{0,\ell}$ $\forall\,\ell=0,\ldots,m$
  (i.e., in general, $\varphi_1^0$ may be different from $\phi_1^0$ in \eqref{basis_delta}), and
  $\varphi_1^0$ vanishes for $t\ge t_1$.

  Now, suppose that the spatial operator $\LL$ in \eqref{problem} is time-independent.
  Then $(\LL u_h)^I=\LL (u_h)^I=\LL u_h$ implies
  $R_h-R_h^I=(D_t^\alpha u_h-f)-(D_t^\alpha u_h-f)^I$.
  Hence,  $R_h$ can be computed by sampling, using parallel/vector evaluations and  without a direct application of $\LL$
  (or its discrete version)
  to $u_h$.

  Next, suppose that $\LL=\LL(t)$ in \eqref{problem}. Then  $R_h-R_h^I$ includes an additional ingredient
  $\LL u_h-(\LL u_h)^I$, which can also be approximated by sampling, but with fewer points.
  For example, for the L1 method, $(\LL u_h)^I$ is linear on each time interval $(t_{k-1},t_k)$ and equal to $\LL u_h$ at the end points.
  Depending on the coefficients of $\LL(t)$, a quadratic approximation  may be adequate for $\LL u_h$, in which case only one sample point per time interval
  will be required for this term.

  Finally, note that some discretizations of $D_t^\alpha$ do not naturally lead to residuals vanishing at the nodes of the temporal mesh.
  Recall, for example, the alternative L2 method considered
  in \cite{LX16,Kopteva21}. For the latter, assuming that $\LL$ is time-independent, one still enjoys
  $R_h-R_h^I=(D_t^\alpha u_h-f)-(D_t^\alpha u_h-f)^I$,
  Although now $R_h^I$ does not vanish at the mesh nodes, while the values  of $R_h(\cdot, t_k)$ (required to compute $R_h^I$) involve
  $\LL u_h(\cdot, t_k)$, the latter can be computed from the definition of the numeral method without a single additional application of $\LL$
  to the computed solution.

  {Another popular discretization that we do not consider in this paper is
  Alikhanov's L2-1${}_\sigma$ scheme (described, e.g., in \cite[Section~4]{Kopteva_Meng}).
  Recall that this scheme is similar to the above L2 methods with the main difference in that, when computing $u_h(\cdot,t_k)$, one assumes that $u_h$ is piecewise-quadratic on $(0,t_{k-1})$ and linear
  on $(t_{k-1},t_k)$, while the higher-order accuracy is ensured by replacing
  \eqref{L1method} with
$(D_t^\alpha+\LL)\,u_h(x, t^*_k)=f(x, t^*_k)$, where $t_k^*:=t_k-\frac12\alpha\tau_k$.
Despite $u_h$ being linear in time on $(t_{k-1},t_k)$, this choice of $t_k^*$
yields higher-order accuracy of order $3-\alpha$ at $t=t_{k}$.
Consequently, the general error estimation methodology still applies, but
for various a-posteriori error estimates of type~\eqref{L2_error} to remain sharp, an appropriate quadratic reconstruction of $u_h$
is to be used on $(t_{k-1},t_k)$.

More generally, our error estimation methodology is applicable to essentially any continuous-in-time computed
solution independently of the method (the continuity in time,
while being formally  required
by Theorems~\ref{the_L2} and~\ref{the_Linfty}, is essential for bounded residuals).
As such, it is also applicable for discontinuous
collocation discretizations \cite[p.\,84]{Brunner04} and discontinuous Galerkin methods
\cite{McLean_2020}, once an appropriate continuous-in-time reconstruction of the computed solution is generated
(see, e.g., \cite[Section 6]{McLean_2020}) and used as $u_h(\cdot, t)$.
A further numerical study in this direction is certainly warranted and will be presented elsewhere.
  }

  \section{Computationally stable implementation}\label{ssec_implem}

 Note that the above formulas for $D_t^\alpha u_h(t)$, such as \eqref{D_alph_L1} for the L1 method, are used both in the implementation of the considered method itself and the computation of relevant residuals.
  A direct implementation of such formulas using exact integration generally yields  numerically unstable solutions.
  In this Section, we will comment on some useful improvements that can be attained using certain reformulations.
  \begin{itemize}
  \item
  Hence, we obtain
  numerically stable and efficient implementations for all considered methods
  (including computations of their residuals)
   with $\alpha$ at least within the range between {$0.1$} and $0.999$ and
  for values of $TOL$ as small as $10^{-8}$.
  \end{itemize}

  Such improvements will be described here by means of MatLab commands, while similar strategies may be employed in all other scientific programming languages. In particular, we rely on the two MatLab commands,
  \texttt{expm1} and \texttt{log1p}, which, being mathematically equivalent to
    \beq\label{MatLab}
   \texttt{expm1}(s)=\exp(s)-1,\qquad \texttt{log1p}(s)=\ln(1+s),
  \eeq
  are used for a more robust evaluation near $s=0$.

  Other possible strategies include higher-precision computations
  (e.g., with Multiprecision Computing Toolbox for MatLab used in \cite{JLZ19})
   and adaptive quadrature routines.
In fact, for higher-order methods, we shall combine
 \eqref{MatLab} with
an adaptive quadrature rule  in the form of MatLab  function \texttt{integral}
  (which employs adaptive quadrature using a 7-point Gauß- with an 15-point Kronrod-quadrature rule to vector-valued functions; see \cite{Shampine08}).

 However, the reader should be cautioned against applying  adaptive quadrature routines to directly compute integrals in \eqref{D_alph_L1}
 or \eqref{D_alp_coll} (in view of singular integrals over $(t_{k-1},t)$, as well as over $(t_{j-1},t_j)$ when the sampling point $t\approx t_j$).
  For example, a simple computational test shows that, despite its versatility, \texttt{integral} becomes appallingly inaccurate when applied
  to a simple singular integral $\int_0^1 s^{-\alpha}ds$ as $\alpha\rightarrow 1^-$.

 \subsubsection*{L1 method}

  The reason for numerical instabilities can be easily understood by examining
 the explicit formula \eqref{D_alph_L1} for $D_t^\alpha u_h$ of the simplest  L1 method.
 The latter formula involves the evaluation of
 $(t-s)^{1-\alpha}\big|_{t_{j-1}}^{t_j}$, i.e. the difference of
 two nearly equal numbers (assuming that $(t-t_j)\approx (t-t_{j-1})$), which
 leads to noticeable round-off errors.
The following simple reformulation using~\eqref{MatLab} immediately rectifies this instability.
For $t> t_j$, set
  \[
    \dt_j(t):=t-t_j,\qquad
    \kappa_j(t):=\ln\left(\frac{\dt_{j}(t)}{\dt_{j-1}(t)}\right)
              =\texttt{log1p}\left(-\frac{\tau_j}{t-t_{j-1}}\right).
  \]
Now one gets
  \begin{align}
    (t-s)^{1-\alpha}\big|_{t_{j-1}}^{t_j}
      &=\dt_j(t)^{1-\alpha}-\dt_{j-1}(t)^{1-\alpha}
       = \dt_{j-1}(t)^{1-\alpha}\left(\left(\frac{\dt_j(t)}{\dt_{j-1}(t)}\right)^{1-\alpha}-1\right)\notag\\
      &= \dt_{j-1}(t)^{1-\alpha}\,\texttt{expm1}\left((1-\alpha)\,\kappa_j(t)\right).\label{eq:stable_L1}
  \end{align}
  Thus, \eqref{D_alph_L1}, for $t_{k-1}<t\le t_k$, allows a computationally stable reformulation
\begin{align}
      &D^\alpha_t u_h(t)=\notag\\
      &= \frac{1}{\Gamma(2-\alpha)}\left(-\sum_{j=1}^{k-1}\frac{U_j-U_{j-1}}{\tau_j}
      \,\dt_{j-1}(t)^{1-\alpha}\,\texttt{expm1}\left((1-\alpha)\,\kappa_j(t)\right)
                                                         +\frac{U_k-U_{k-1}}{\tau_k}{d_{k-1}(t)}^{1-\alpha}\right).
                                                         \label{L1_stable_Dt}
\end{align}
 One may worry that, due to the summation, each $U_j$ for $j<k$ is still multiplied by a difference of two possibly close numbers.
Nevertheless, we observed stable performance of the above reformulation
in all our experiments.

  \subsubsection*{Higher-order methods. Adaptive quadrature}

  Numerical instabilities become even more intractable in the context of higher-order methods, and even more so
  since the higher accuracy, offered by such methods, is availed only if the computations are performed with higher precision.
  Below we shall describe a stable implementation
  for the collocation methods of arbitrary order $m$.
  Note that for $m=1$, this approach reduces to the above \eqref{L1_stable_Dt}.
  We also used
  a version of this approach for a stable implementation of the L1-2 method.

   For the collocation methods, recall that \eqref{D_alp_coll} for $D_t^\alpha u_h(t)$, where $t_j\le t_{k-1}< t\leq t_k$, involves the integrals
   of two types:
  \[
      I_{hist}^{j,\ell}:=\int_{t_{j-1}}^{t_j}\partial\phi_j^\ell(s)\,(t-s)^{-\alpha}\,\ds
      \quad\text{and}\quad
      I_{sing}^{\ell}:=\int_{t_{k-1}}^t\partial\phi_k^\ell(s)\,(t-s)^{-\alpha}\,\ds,
      \quad \ell\in\{0,\ldots,m\},
  \]
  which we shall respectively refer to as the history integrals and the singular integrals.

  It is convenient to describe a set of basis functions $\{\phi_j^\ell\}_{\ell\in\{0,\ldots,m\}}$ on each $[t_{j-1},t_j]$
  using the reference interval $[0,1]$ by
\beq\label{phi_psi}
    \phi_j^\ell(s)=\psi^\ell\left( \frac{s-t_{j-1}}{\tau_j} \right)=:\psi^\ell(\sigma)
    ,\qquad\ell\in\{0,\dots,m\}.
  \eeq
  Here, in agreement with \eqref{basis_delta}, we also impose $\psi^\ell(0)=\delta_{\ell,0}$ and $\psi^\ell(1)=\delta_{\ell,m}$.

  In the case of the L1 method (which corresponds to $m=1$), the history integrals led to a possibly unstable evaluation of
  $(t-s)^{1-\alpha}\big|_{t_{j-1}}^{t_j}$, and, unsurprisingly, similar instabilities may occur when computing $I_{hist}^{j,\ell}$.
We rectify these by splitting $\partial\phi_j^\ell(s)$ in $I_{hist}^{j,\ell}$  as $\partial\phi_j^\ell(t_j)+[\partial\phi_j^\ell(s)-\partial\phi_j^\ell(t_j)]$.
Now
  $ I_{hist}^{j,\ell}$ can be split as $\bar I_{hist}^{j,\ell}+[I_{hist}^{j,\ell}-\bar I_{hist}^{j,\ell}]$,
  where
\begin{align*}
  \bar I_{hist}^{j,\ell}&:=\partial\phi_j^\ell(t_j)
\int_{t_{j-1}}^{t_j}(t-s)^{-\alpha}\,\ds
=\frac{\partial\psi^\ell(1)}{(1-\alpha)\,\tau_j}\,
(t-s)^{1-\alpha}\big|_{t_{j-1}}^{t_j}
\\
&=\frac{\partial\psi^\ell(1)}{(1-\alpha)\,\tau_j}\,
\dt_{j-1}(t)^{1-\alpha}\,\texttt{expm1}\left((1-\alpha)\,\kappa_j(t)\right).
\end{align*}
Here we used $\frac{d}{ds}\phi_j^\ell(t_j)=\tau_j^{-1}\frac{d}{d\sigma}\psi^\ell(1)$ and the stable-implementation formula \eqref{eq:stable_L1}.

For the remaining component $I_{hist}^{j,\ell}-\bar I_{hist}^{j,\ell}$ of $I_{hist}^{j,\ell}$ one gets
  \begin{align}\notag
   I_{hist}^{j,\ell}-\bar I_{hist}^{j,\ell}
  &=\int_{t_{j-1}}^{t_j}[\partial\phi_j^\ell(s)-\partial\phi_j^\ell(t_j)]\,(t-s)^{-\alpha}\,\ds\\
  &={d_{j-1}(t)}^{-\alpha}\int_{0}^{1}[\partial\psi^\ell(\sigma)-\partial\psi^\ell(1)]\,
\exp\left( -\alpha\logp\left( -\frac{\tau_j}{t-t_{j-1}}\sigma \right) \right)\,
  \dsigma\,,
  \label{integral_adaptive}
  \end{align}
 where we used
 $$
 (t-s)^{-\alpha}=(t-t_{j-1}-\tau_j\sigma)^{-\alpha}={d_{j-1}(t)}^{-\alpha}\,\exp\left( -\alpha\logp\left( -\frac{\tau_j}{t-t_{j-1}}\sigma \right) \right).
 $$
 Note that  the above integral is non-singular, as
 $(t-s)^{-\alpha}\lesssim (t-t_{j-1})^{-\alpha}\,(1-\sigma)^{-\alpha}$,
 while $|\partial\psi^\ell(\sigma)-\partial\psi^\ell(1)|\lesssim 1-\sigma$.
 Hence, an adaptive quadrature routine yields a fast and efficient evaluation of the latter integral.
 For example, we employed \texttt{integral}
 with an appropriate integrand $\chi(\sigma)$ from \eqref{integral_adaptive} in the form
  \[
    \texttt{integral(@($\sigma$)}
    \chi(\sigma)
    \texttt{, 0,1,'ArrayValued','true')}.
  \]
  In addition we supply the options \texttt{('RelTol',1e-16,'WayPoints',pts)} with increased relative tolerance for evaluating the
  integral, and also a hint on how to subdivide the interval of integration in the form of a vector
  \texttt{pts} of sampling points (specified by \eqref{sampling} below).
  This routine was essential for the
  evaluation of the residual at sampling points.

To give an example,
    in our experiments, we used
hierarchical basis functions  $\{\psi^\ell\}_{\ell\in\{0,\ldots,m\}}$
defined by
\beq\label{our_psi}
\psi^0(\sigma):=1-\sigma,\quad \psi^m(\sigma):=\sigma,
\quad
\psi^\ell(\sigma):=\sigma^{\ell}(1-\sigma)\;\;\forall\,\ell\in\{1,\dots,m-1\}.
\eeq
Then
$-\partial\psi^0(\sigma)=\partial\psi^m(\sigma)=1$, and for $\ell\in\{1,\ldots, m-1\}$ from
$\partial\psi^\ell(\sigma)=\ell\sigma^{\ell-1}(1-\sigma)-\sigma^{\ell}$ one easily gets
$$
\partial\psi^\ell(1)=-1,\qquad 0\le \partial\psi^\ell(\sigma)-\partial\psi^\ell(1)=\ell\sigma^{\ell-1}(1-\sigma)-\sigma^{\ell}+1\lesssim 1-\sigma.
$$

  The remaining integral $I_{sing}^{\ell}$ is singular, but can be evaluated analytically as follows.
A~transformation using \eqref{phi_psi} and $t_{loc}:=\frac{t-t_{k-1}}{\tau_k}$ yields
 $(t-s)^{-\alpha}=(t-t_{k-1}-\tau_k\sigma)^{-\alpha}=\tau_k^{-\alpha}(t_{loc}-\sigma)^{-\alpha}$, and then
  \[
    I_{sing}^{\ell}
       = \tau_k^{-\alpha}\int_0^{t_{loc}}\partial \psi^\ell(\sigma)\,(t_{loc}-\sigma)^{-\alpha}\,\dsigma
       =:\tau_k^{-\alpha}\,\hat I_{sing}^{\ell}.
  \]
  We observe, that the evaluation is reduced to finding an integral
  $\hat I_{sing}^{\ell}=\hat I_{sing}^{\ell}(t_{loc})$ that depends only on $t_{loc}$.
  For example, with our choice \eqref{our_psi}, a calculation yields
    \begin{align*}
    \hat I_{sing}^m &= -\hat I_{sing}^0 = 
    t_{loc}^{1-\alpha}\cdot\frac{1}{1-\alpha},\\
    \hat I_{sing}^1 &= 
    t_{loc}^{1-\alpha}\cdot\frac{2 t_{loc}-(2-\alpha)}{(1-\alpha)(2-\alpha)},
  \end{align*}
  while $\hat I_{sing}^\ell$ for other values of $\ell$ can be easily evaluated in a similar way.
  Furthermore, the sampling points for the evaluation of the residuals on each $[t_{k-1},t_k]$
  are typically chosen as $t=t_{k-1}+\tau_k\,t_{loc}$, where $t_{loc}$ takes values from a certain predefined set; hence,
  all $m+1$ integrals $\hat I_{sing}^{\ell}(t_{loc})$ can be pre-computed offline for all $\ell$ and all $t_{loc}$ of interest.

  {For the L1-2 method we obtain by the same reasoning the stable formulation
  \begin{align*}
    D^\alpha_t u_h(t)
       =&  \frac{1}{\Gamma(2-\alpha)}\sum_{j=1}^{k-1}-\frac{U_j-U_{j-1}}{\tau_j}d_{j-1}(t)^{1-\alpha}\expm((1-\alpha)\kappa_j(t))\\
        &+ \frac{1}{\Gamma(2-\alpha)}\sum_{j=2}^{k-1}y_j\tau_j d_{j-1}(t)^{1-\alpha}\expm((1-\alpha)\kappa_j(t))\\
        &+\frac{2}{\Gamma(1-\alpha)}\sum_{j=2}^{k-1}y_j\tau_j^2d_{j-1}(t)^{-\alpha}\int_0^1(\sigma-1)\exp\left(-\alpha\logp\left(-\frac{\tau_j}{t-t_{j-1}}\sigma\right)\right)\dsigma\\
        &+\frac{1}{\Gamma(1-\alpha)}\left((U_k-U_{k-1})\tau_k^{-\alpha}(-\hat I_{sing}^0)+y_k\tau_k^{2-\alpha}\hat I_{sing}^1 \right),
  \end{align*}
  where $\hat I_{sing}^0$ and $\hat I_{sing}^1$ are as above.
  Note that the parts without $y$ are exactly the same as for the L1 method.
  }

 \subsubsection*{Sampling points}
 On each interval, the residuals $R_h$ (and, hence, $D_t^\alpha u_h$) were evaluated using
 \beq\label{sampling}
    t_{loc}\in\left\{\left( \frac{i}{n} \right)^p \right\} \qquad\mbox{for}\;\; i\in\{1,\dots,n-1\},
\eeq
  i.e. this set forms a graded grid on $[0,1]$. In our computations, we set $n+1:=21$
and {heuristically choose} the grading exponent $p:=\min\{\frac{1}{1-\alpha},5\}$, which gives a sufficiently strong
  sampling near the maximal residual values, as we shall now discuss. 

Indeed, for our algorithm to be reliable, it is essential that the set of sampling points reaches
 the maximal value of the residual, while the residual
 itself behaves on each $[t_{k-1},t_k]$ like a left-shifted bubble with a sharp layer near $t_{k-1}$ as $\alpha\rightarrow 1^-$.
 The latter behaviour is easily understood, e.g., in the case of the L1 method, as in the extreme case $\alpha=1$, the residual involves
 a piecewise-constant $D_t^\alpha u_h=\frac{d}{dt}u_h$ and, thus, has discontinuities at each $t_{k-1}^+$.
 With above choice of $p$ and $n$, we observe
  a good distribution of sampling points as shown in Fig.~\ref{fig:sampling}.
    catching the maximum value. {Interestingly, for $\alpha=0.1$ we 
    also observe 
    a sharp layer in the residual bubble, but now only for a few initial time intervals.}
  \begin{figure}[htb]
    \begin{center}
      \input{src/bubbles_4_01}
      \input{src/bubbles_4_04}
      \input{src/bubbles_4_08}
      \input{src/bubbles_4_099}
    \end{center}\vspace{-0.3cm}
    \caption{Distribution of sampling points for a collocation method with $m=4$ inside the second interval for {$\alpha\in\{0.1,\,0.4,\,0.8,\,0.99\}$ (from left to right)}\label{fig:sampling}}
  \end{figure}
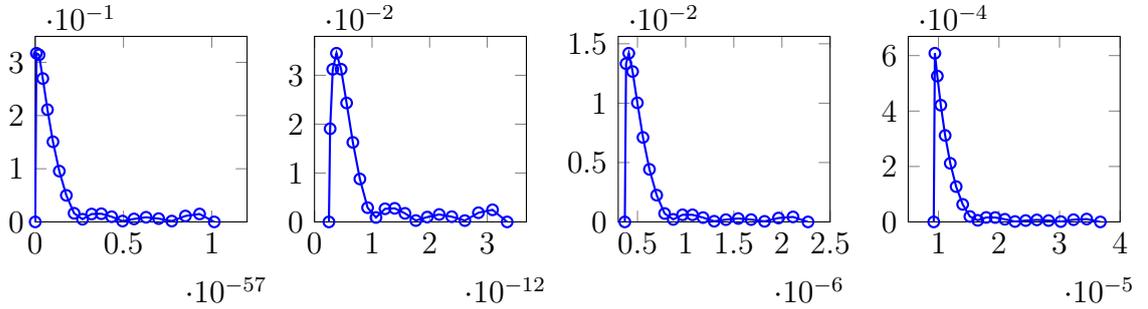
%


 \section{Adaptive algorithm}\label{sec_algorithm}


  \begin{lstlisting}[float,caption=Adaptive Algorithm,  basicstyle=\footnotesize,label=alg:mesh]
m         := 1;
Q         := Q0;
uh(1)     := u0;
mesh(1:2) := [0,tau_init];
while mesh(m)<T
  if (m=2 && Q=Q0)                        % found initial guess for mesh(1)
    Q := Q1;
  else
    m := m+1;
  end
  flag := 0;
  while mesh(m)-mesh(m-1) > tau_min
    uh(m)  := computeSolution(mesh(1:m),uh);
    Res    := computeResidual(uh,mesh);
    ResBarrier := computeResidualBarrier(mesh);
    if all(Res<TOL*ResBarrier)            % residual small enough
      if mesh(m)>=T                       % finished
        break
      else                                % ok
        if flag = 2                       % coming from larger stepsizes
          mesh(m+1) := min(mesh(m)+(mesh(m)-mesh(m-1)),T);
          break;                          % continue with next step
        end
        tmpuh   := uh(m);                 % save data
        tmptm   := mesh(m);
        mesh(m) := min(mesh(m-1)+Q*(mesh(m)-mesh(m-1)),T);
        flag    := 1;                     % try again with larger step
      end
    else                                  % residual not small enough
      if flag = 1                         % previous step was good
        uh(m)     := tmpuh;               % recall saved data
        mesh(m)   := tmptm;
        mesh(m+1) := min(mesh(m)+(mesh(m)-mesh(m-1)),T);
        break;                            % continue with next step
      else                                % previous residual not good
        mesh(m) := mesh(m-1)+(mesh(m)-mesh(m-1))/Q;
        flag    := 2;                     % try again with smaller step
      end
    end
  end
  if mesh(m)-mesh(m-1) < tau_min
    mesh(m)   := min(mesh(m-1)+tau_min,T);
    mesh(m+1) := min(mesh(m-1)+2*tau_min,T);
  end
end
  \end{lstlisting}

  Here we present a version of the adaptive time stepping algorithm from \cite{Kopteva22,Kopt_Stynes_apost22}, in which we made a few subtle improvements
  that substantially reduce the computational time; see Algorithm~\ref{alg:mesh}.

  Our algorithm yields a mesh with mesh nodes $\{t_j\}$, such that $0=t_0<t_1<\dots<t_M=T$ (where a suitable $M$ is chosen by the algorithm)
  and a computed solution on this mesh, such
  that the residual of the computed solution, measured in the $L_p(\Omega)$ norm, does not exceed $TOL\cdot \RR_0(t)$ or $TOL\cdot \RR_1(t)$,
  the residual barriers suggested by \eqref{barrier_bounds}, or, equivalently, by \eqref{barrier_bounds_alg}.

  Similarly to \cite{Kopteva22,Kopt_Stynes_apost22}, the algorithm calls three functions:
  \begin{itemize}
    \item \lstinline|computeSolution(mesh,oldSolution)| that employs a numerical method from Section~\ref{sec_methods} to
    compute the solution $u_h|_{(t_{m-1},t_m]}$ using a given mesh $\{t_j\}_{j=0}^m$
          and  the solution $u_h|_{[0,t_{m-1}]}$ as initial data;

    \item \lstinline|computeResidual(solution,mesh)| that computes the $L_p(\Omega)$-norm of the residual
    (as described in Section~\ref{ssec_residuals})
    at the prescribed set of sampling points in $(t_{m-1},t_m]$ using a given mesh $\{t_j\}_{j=0}^m$
          and  the solution $u_h|_{[0,t_{m}]}$ as initial data; 

    \item \lstinline|computeResidualBarrier(mesh)| that computes the barriers $\RR_0$ or $\RR_1$
     at the prescribed set of sampling points on $(t_{m-1},t_m]$. 
  \end{itemize}

  Furthermore, the following parameters are used
  \begin{itemize}

    \item $TOL$ is a sufficiently small positive number used in the guaranteed estimate for the error of the computed solution.
    \item $\tau_{init}=T/2$ is a very crude initial guess for the first time step $\tau_1$. This will typically be shrunken by the factor $Q_0$
    (see below) in the first few iterations.

    \item $\tau_{min}\ge 0$ is the lower bound on any time step enforced by the algorithm.
    We set it to $0$, but the algorithm allows for any small positive value.

        \item $Q_0>Q_1>1$ are two factors by which the size of the current time step is increased or reduced. Here the larger factor $Q_0$
          is used to find a crude size for the starting time step $\tau_1$,
          while the smaller factor $Q_1\approx 1$ is used to find the final size of $\tau_1$, as well as to compute all other time steps.
          In our experiments we chose $Q_0=5$ and the initial search was done within 10 iterations.
  \end{itemize}

\noindent
  Note that, compared to \cite{Kopteva22,Kopt_Stynes_apost22}, the introduction of the larger factor $Q_0$ (used to compute a crude size of $\tau_1$), as well as \lstinline|flag = 2|,
  significantly reduces the computational times.

  \begin{remark}[$\omega>0$]\label{rem_alg_oemga}
    Note that when the bounds of Corollary~\ref{cor_new} are employed, in view of
    \eqref{barrier_bounds_alg_new} (used instead of \eqref{barrier_bounds_alg}),
    the residual of the computed solution, measured in the $L_\infty(\Omega)$ norm, should not exceed $TOL\cdot \RR_0(t)/(1+\omega)$ or
    $TOL\cdot \RR_1(t)/(1+\omega)$, which requires an obvious minor change in line 15 of the algorithm
    to\\
    \centerline{
    \textnormal{\lstinline|ResBarrier := computeResidualBarrier(mesh)/(1+omega);|}}
  \end{remark}

  \section{Numerical experiments}\label{sec:numerics}
  We are mainly interested in the adaptive time stepping and, therefore, we consider only simple problems in
  the spatial direction. We use a conforming finite element method with piecewise polynomials of a fixed degree
  in space on a sufficiently fine equidistant mesh.
 For simplicity, throughout this Section,  all test problems posed in $(0,T]\times\Omega$ will be of the form
  \[
    D^\alpha_t u-\Delta u=f,
  \]
  subject to homogeneous boundary conditions.
  \medskip

  \example \label{Ex1}
  In order to compare the residual to the actual error we consider a test problem with a given exact solution.
  Here we take $\Omega=(0,1)$ and
  \[
    u(x,t)=(t^\alpha-t^2+1)\,x\,(1-x)
  \]
  that satisfies the homogeneous boundary conditions and the initial condition with $u_0(x)=x(1-x)$, and exhibits a typical weak singularity of
  type $t^\alpha$ near $t=0$.
  Note that the solution is  a quadratic polynomial in space for each $t\geq0$. Therefore, using
  piecewise quadratic elements in space, on a coarse spatial grid of just 10 cells, resolves it exactly and the error obtained is purely due to time
  discretisation.
    \medskip

  \example \label{Ex2}
  In order to investigate the behaviour of the adaptive algorithm and its parameters we
  consider a second test problem with an unknown solution posed in $\Omega=(0,1)$, but
  a given right-hand side
  \[
    f(x,t)=(t^\gamma-t)\cdot\sin((x\pi)^2)+t\cdot\exp(-100\cdot(2t-1)^2)
  \]
  for $\gamma\in[0,\alpha]$. Note, that this function has a very localised Gaussian pulse in addition
  to the weak singularity of type $t^{\gamma+\alpha}$ near $t=0$.
  The solution of this problem for $\alpha=0.4$ is shown in Figure~\ref{fig:sol}
  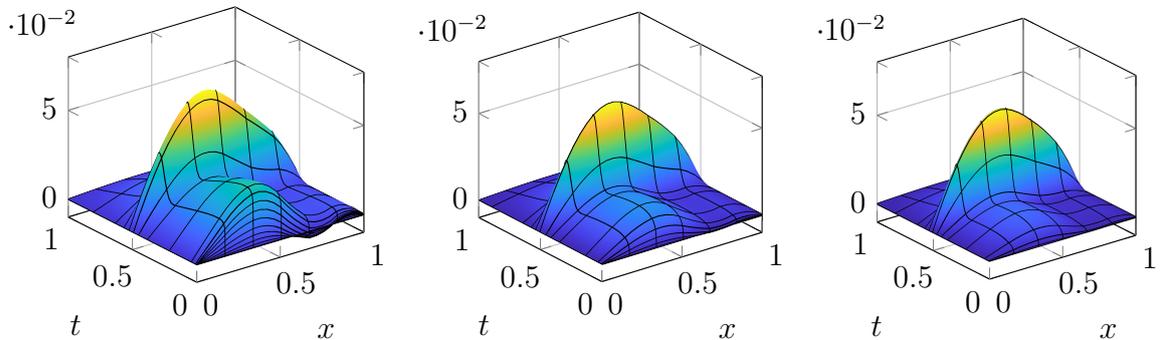
\begin{figure}[htb]
    \begin{center}
      \input{surf1}
      \input{surf2}
      \input{surf3}
    \end{center}\vspace{-0.3cm}
    \caption{Solution of Example~\ref{Ex2} with $\alpha=0.4$ and $\gamma\in\{0,\,0.2,\,0.4\}$ from left to right\label{fig:sol}}
  \end{figure}
  for $0\leq t\leq 1$. Note that similar problems posed in $\Omega=[0,1]^d$ for $d\in\{2,3\}$ were also investigated with similar results.
  For the spatial discretisation we use the same method as before and apply piecewise quadratic elements on a  spatial grid of
   10 cells.

  \subsection{Experiments on the reliability and convergence rates with Example~\ref{Ex1}}\label{ssec_numerics}

  Let us start by investigating the reliability of our error estimator using Example~\ref{Ex1}, for which the exact solution is available. For the mesh adaptation process we  use
  $Q_1=1.2$ unless specified otherwise.
  Recall that for this problem, in view of Lemma~\ref{lem_Linfty_new} combined with
  Remark~\ref{rem_lem_Linfty_new1}(ii), one can employ $\omega=\frac18\lambda$ for any $\lambda\ge 0$
  when measuring the error in the $L_\infty(\Omega)$ norm (see also Remark~\ref{rem_alg_oemga} for this case).
  When the error is measured  in the $L_2(\Omega)$ norm, the principal eigenvalue $\lambda=\pi^2$ of the operator $\LL=-\partial_x^2$ on $\Omega=(0,1)$ satisfies
  Theorem~\ref{the_L2}, so will be used in $\RR_0$ and $\RR_1$.

  Figures \ref{fig:R0LinfLinf_04}--\ref{fig:R1Linf_04} show loglog graphs of various values of the algorithm's $TOL$ and the corresponding actual errors vs. the corresponding numbers of time steps $M$ (recall that $M$ is automatically chosen by the algorithm for a prescribed value of $TOL$).
  The L1 method,  the L1-2 method, and a few collocation methods up to order $m=8$ are considered.
  Given a method of order $q$, we expect convergence rates of order $q-\alpha$ (see Remark~\ref{rem_apriori}).
  Hence, we also show the slopes for $M^{-(2-\alpha)}$ and $M^{-(5-\alpha)}$, which are, respectively, expected to have a good agreement with the error curves for
  the L1 method and the collocation method of order $m=4$. (The slope for $m=8$ is not given, as the prescribed tolerance $TOL$ is attained while
  we remain in a preasymptotic regime, with very few time steps $M$.)

  We start with the residual barrier  $\RR_0$, with $\lambda=\pi^2$ and $\omega=\lambda/8$,
   and measure the error in the $L_\infty(0,T;\, L_\infty(\Omega))$ norm;
  see Figure~\ref{fig:R0LinfLinf_04}.
  \begin{figure}[htb]
    \begin{center}
      \input{src/R0_Linf_Linf_04_pi2} 
    \end{center}\vspace{-0.6cm}
    \caption{$L_\infty(0,T;\, L_\infty(\Omega))$ errors for various methods vs. number of time steps $M$ for Example~\ref{Ex1}, $\alpha=0.4$,
             residual barrier $\RR_0$ with $\lambda=\pi^2$ and $\omega=\lambda/8$\label{fig:R0LinfLinf_04}}
  \end{figure}
  We clearly observe a tight bounding of the errors by the prescribed values of $TOL$ and therefore a good demonstration of the reliability
  of the estimator. Furthermore, we  observe convergence orders of almost $\ord{M^{-(q-\alpha)}}$, where
  $q=2$ for the L1 method, $q=3$ for the L1-2 method, and $q=m+1$, where
  $m$ is the polynomial degree
  used in the definition of the collocation methods in Section~\ref{ssec_collocation}.  Note that
  there are no theoretical a-priori error estimates in the literature for the considered collocation methods in the context of our problem~\eqref{problem}, while our adaptive algorithm yields reliable computed solutions with optimal convergence rates.

  Figure~\ref{fig:R0LinfLinf_04_0}
  \begin{figure}[htb]
    \begin{center}
      \input{src/R0_Linf_Linf_04_0} 
    \end{center}\vspace{-0.6cm}
    \caption{$L_\infty(0,T;\, L_\infty(\Omega))$ errors for various methods vs. number of time steps $M$ for Example~\ref{Ex1}, $\alpha=0.4$,
             residual barrier $\RR_0$ with $\lambda=\omega=0$\label{fig:R0LinfLinf_04_0}}
  \end{figure}
  shows the results for the same problem, but now we choose $\lambda=0$ and $\omega=0$. We observe a tighter fit
  of the error to the corresponding $TOL$, but at the same time for the lower-order methods more time steps were required by the algorithm in comparison to the previous
  choice of $\lambda$ and $\omega$.

  Changing the spatial norm to $L_2(\Omega)$, we observe a similar behaviour in the adaptivity and convergence; see Figure~\ref{fig:R0LinfL2_04}.
  \begin{figure}[htb]
    \begin{center}
      \input{src/R0_Linf_L2_04} 
    \end{center}\vspace{-0.6cm}
    \caption{$L_\infty(0,T;\, L_2(\Omega))$ errors for various methods vs. number of time steps $M$ for Example~\ref{Ex1}, $\alpha=0.4$,
             residual barrier $\RR_0$ with $\lambda=\pi^2$
             \label{fig:R0LinfL2_04}}
  \end{figure}

  {
  For smaller values of $\alpha$ the singularity at $t=0$ is stronger
   (assuming it is of type $t^\alpha$, as discussed in Remark~\ref{rem_apriori}),
   so we observe a stronger initial mesh refinement, as is also illustrated by Figure~\ref{fig:sampling}. 
  At the same time,  Figure~\ref{fig:R0LinfLinf_01}}
  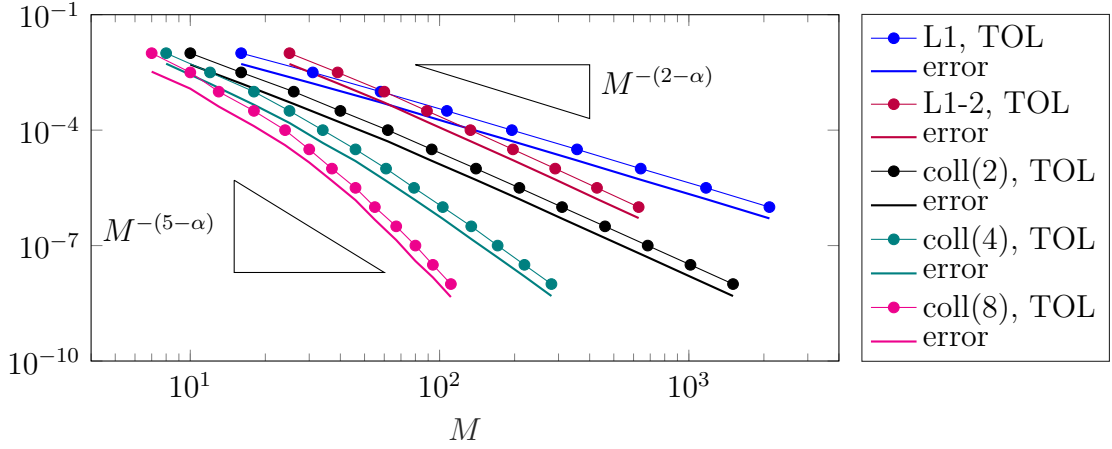
\begin{figure}[htb]
    \begin{center}
      \input{src/R0_Linf_Linf_01_pi2} 
    \end{center}\vspace{-0.6cm}
    \caption{$L_\infty(0,T;\, L_\infty(\Omega))$ errors for various methods vs. number of time steps $M$ for Example~\ref{Ex1}, $\alpha=0.1$,
             residual barrier $\RR_0$ with $\lambda=\pi^2$ and $\omega=\lambda/8$\label{fig:R0LinfLinf_01}}
  \end{figure}
  {
  shows  that the mesh adaptation process works similarly well for  $\alpha=0.1$.
  But for really small $\alpha$ we run into numerical issues. For example, $\alpha=0.01$ and $TOL=10^{-3}$ for the collocation method with $m=4$
  yield the first time step $\tau_1\approx 1.2\cdot 10^{-270}$
    (which is consistent with $\tau_1=M^{-r}$ for the optimal graded mesh with $r=(m+1-\alpha)/\alpha$;
  see, e.g., \cite{Kopteva_Meng}). 
  For smaller values of $TOL$, higher $m$, and/or smaller $\alpha$,
  the size of $\tau_1$ becomes numerically zero, as the smallest positive number in MatLab is $2^{-1074}\approx 5\cdot 10^{-324}$.
  In other words, for very small $\alpha$, double precision is no longer sufficient to represent the first time step.}

  For higher values of $\alpha$ the singularity at $t=0$ is weaker, but the residuals become more singular (see Figure~\ref{fig:sampling}).
  Figure~\ref{fig:R0LinfLinf_08}
  \begin{figure}[htb]
    \begin{center}
      \input{src/R0_Linf_Linf_08_pi2} 
    \end{center}\vspace{-0.6cm}
    \caption{$L_\infty(0,T;\, L_\infty(\Omega))$ errors for various methods vs. number of time steps $M$ for Example~\ref{Ex1}, $\alpha=0.8$,
             residual barrier $\RR_0$ with $\lambda=\pi^2$ and $\omega=\lambda/8$\label{fig:R0LinfLinf_08}}
  \end{figure}
  shows in the case $\alpha=0.8$ that the mesh adaptation process works similarly well in this regime. In fact, we tested the algorithm for values of up to $\alpha=0.999$
  and observed consistently good convergence behaviour.

  With the help of the second residual barrier  $\RR_1$, we can bound the error at a given final time, here $T=1$,
  while employing a weaker mesh refinement (as the resulting error is guaranteed to bounded by $TOL\cdot t^{\alpha-1}$).
  Figure~\ref{fig:R1Linf_04}
  \begin{figure}[htb]
    \begin{center}
      \input{src/R1_Linf_04_pi2} 
    \end{center}\vspace{-0.6cm}
    \caption{$L_\infty(\Omega)$ errors at $T=1$ for various methods vs. number of time steps $M$ for Example~\ref{Ex1}, $\alpha=0.4$,
             residual barrier $\RR_1$ with $\lambda=\pi^2$ and $\omega=\lambda/8$
             \label{fig:R1Linf_04}}
  \end{figure}
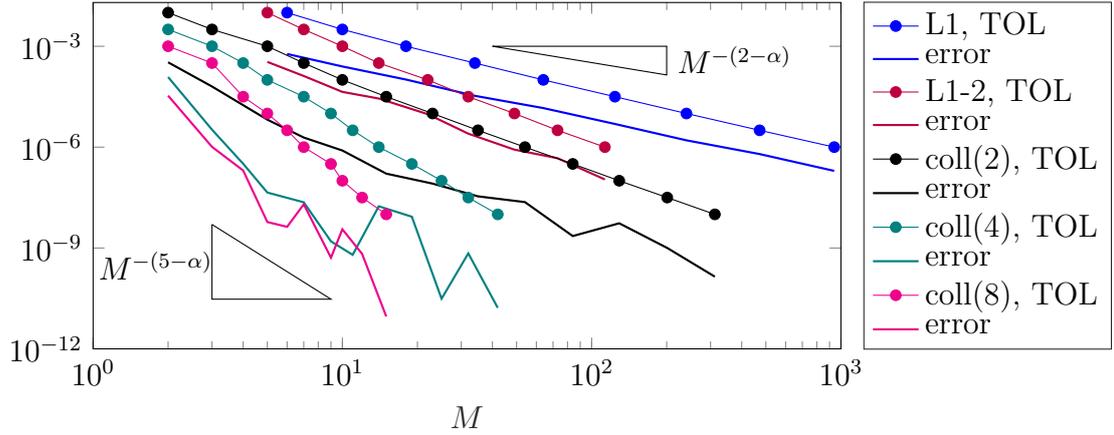
  shows the results for $\alpha=0.4$. We observe, that the error behaviour is not as smooth as for the other estimator
  for higher-order methods. This is partially caused by $Q_1=1.2$, while
   we remain in a preasymptotic regime, with very few time steps $M$ required by the adaptive algorithm
   ($Q_1$ closer to $1$ would produce smother error curves, but would require more iterations; see Figure~\ref{fig:Qcosts} below).


  \subsection{Experiments with Example~\ref{Ex2}. Algorithm parameters}
  The purpose of this section is twofold. First, experiments with
  Example~\ref{Ex2}, with an unknown solution that exhibits an initial weak singularity at $t=0$ (depending on $\gamma$)
  and a localised Gaussian pulse near $t=0.5$ (see Figure~\ref{fig:sol}) will illustrate that our algorithm is capable of adapting the temporal mesh
  to various solution singularities.
  Second, we will numerically investigate the parameters of the adaptive algorithm, in view of computational costs vs. the resulting errors.
  Thus, throughout this section, we
apply our algorithm to Example~\ref{Ex2} using the residual barrier  $\RR_0$ with the $L_\infty(\Omega)$ norm and
  $\lambda=\pi^2$, $w=\lambda/8$.

  \subsubsection*{Adaptivity for various solution singularities}
  Set $\alpha=0.4$ and $\gamma=0$ in Example~\ref{Ex2}. The adaptive time stepping was applied with $TOL=10^{-4}$, $Q_1=1.2$,
  for the collocation methods of order $m\in\{1,2,4,8\}$ (which includes the L1-method for $m=1$), with
  the generated time steps shown in Figure~\ref{fig:deltat}.
  \begin{figure}[htb]
    \begin{center}
      \input{src/meshsizes} 
    \end{center}\vspace{-0.6cm}
    \caption{Time steps vs. time  for 4 collocation methods applied to Example~\ref{Ex2}, $\alpha=0.4$, $\gamma=0$, $Q_1=1.2$, $TOL=10^{-4}$,
    residual barrier $\RR_0$ with $L_\infty(\Omega)$ norm and  $\lambda=\pi^2$, $w=\lambda/8$
    \label{fig:deltat}}
  \end{figure}
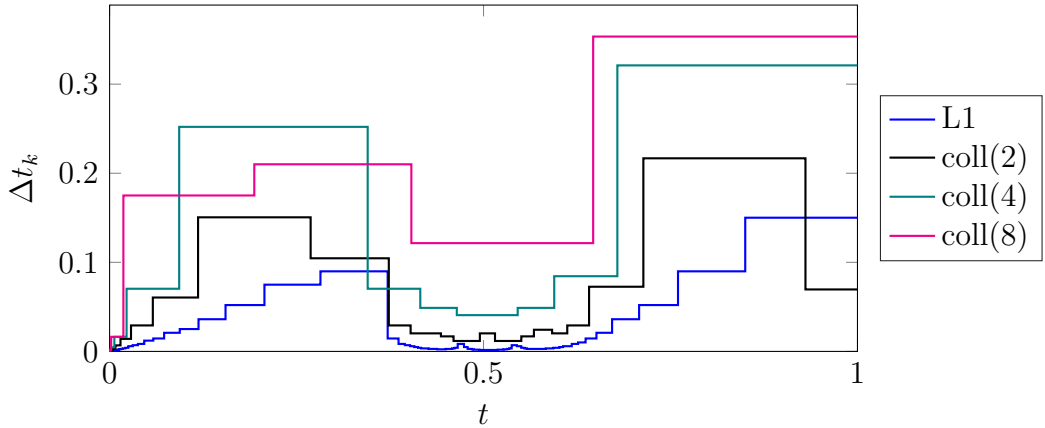
  We see, that indeed the meshes are refined near the two problematic points with a finer mesh for lower-order methods like the L1 method.
  We also considered $\alpha=0.8$ and varied $\gamma\in[0,\alpha]$.
  The adaptivity to the initial singularity of type $t^{\gamma+\alpha}$ as $\gamma$ changes is clearly shown in Figure~\ref{fig:meshgamma}.
  {Note also that  when the solution is of type $t^\alpha$,
  the adaptive temporal mesh becomes  similar to the optimal graded mesh, described in Remark~\ref{rem_apriori}, as is more clearly shown in
   \cite[Fig.~1]{Kopteva22}.}
  \begin{figure}[bht]
    \begin{center}
      \input{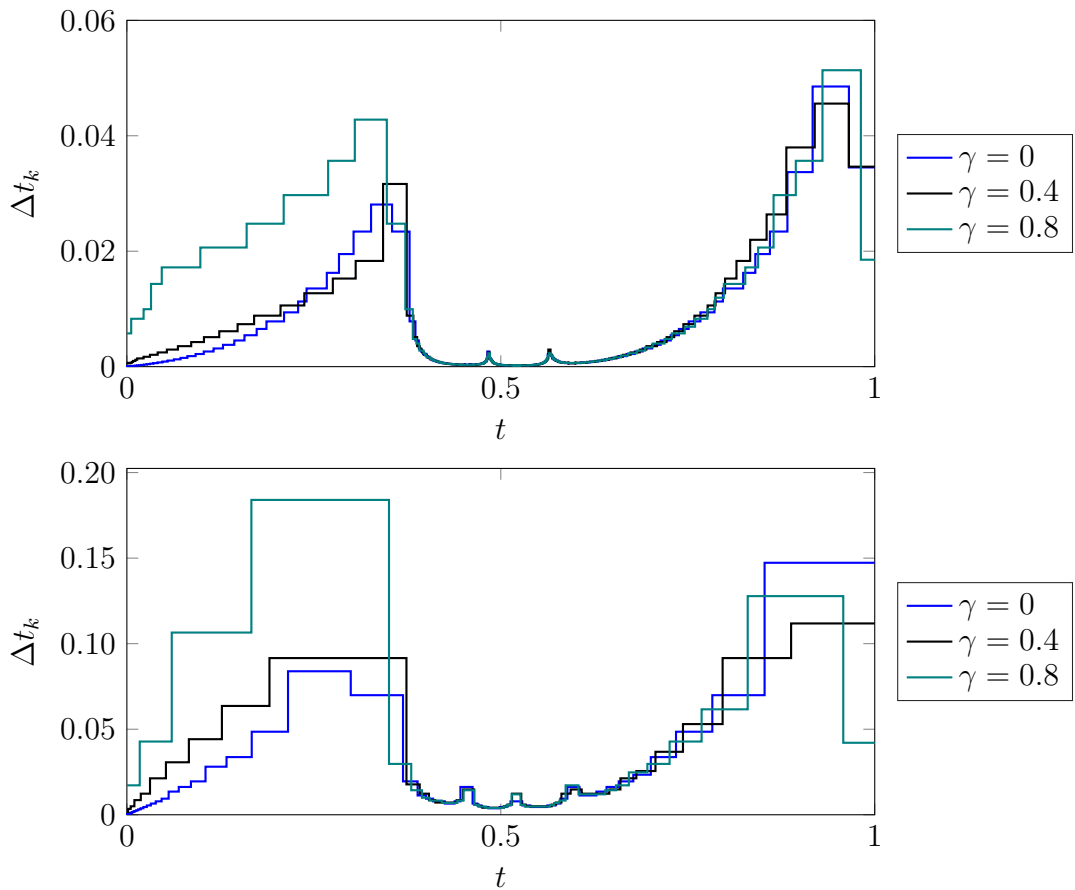}\\ 
      \input{src/meshsizes_2} 
    \end{center}\vspace{-0.6cm}
    \caption{Time steps vs. time for Example~\ref{Ex2} for varying $\gamma$ and the L1 method (top) and coll(2) (bottom), $\alpha=0.8$, $Q_1=1.2$, $TOL=10^{-4}$,
    residual barrier $\RR_0$ with $L_\infty(\Omega)$ norm and  $\lambda=\pi^2$, $w=\lambda/8$ \label{fig:meshgamma}}
  \end{figure}

  \subsubsection*{Parameter $Q_1$}
  Next, we want to investigate the influence of the value of $Q_1>1$ on the number of time steps.
  In Figure~\ref{fig:NvsQ}
  \begin{figure}[htb]
    \begin{center}
      \input{src/NvsQ}
    \end{center}\vspace{-0.6cm}
    \caption{Number of time intervals $M$ vs. $Q_1$ for 4 collocation methods applied to Example~\ref{Ex2}, $\alpha=0.4$, $\gamma=0$, $TOL=10^{-4}$,
    residual barrier $\RR_0$ with $L_\infty(\Omega)$ norm and  $\lambda=\pi^2$, $w=\lambda/8$\label{fig:NvsQ}}
  \end{figure}
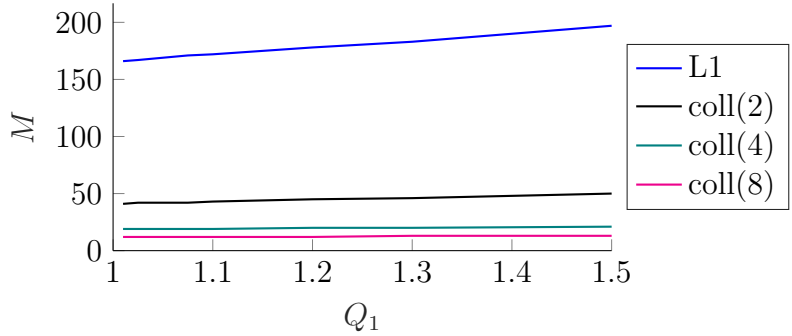
  we see for $m\in\{1,2,4,8\}$ the number of time steps $M$ in the adaptively-generated mesh for varying values of $Q_1$.
  We observe that the number of time steps increases with increasing $Q_1$, but very moderately in general, and even more so for higher-order methods.
  Thus, for an optimal mesh a relatively small value of $Q_1>1$ should be taken.

  On the other hand, smaller values of $Q_1$ may lead to many iterations and, therefore, higher computational costs, as shown in Figure~\ref{fig:Qcosts}.
  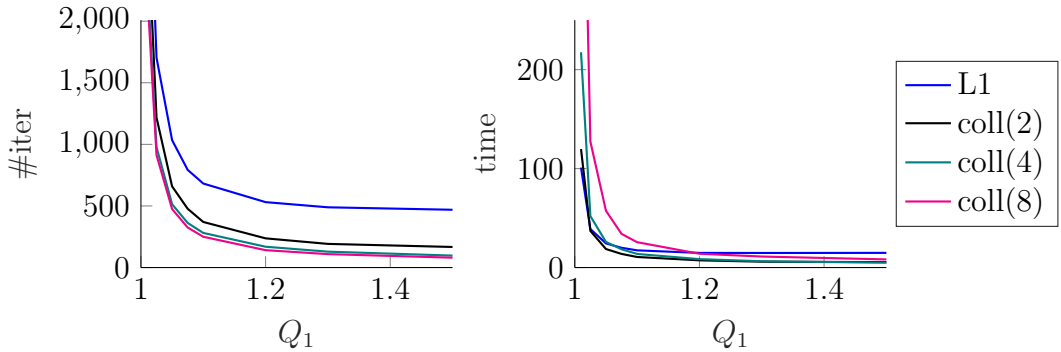
\begin{figure}[htb]
    \begin{center}
      \input{src/itervsQ} 
      \input{src/timevsQ} 
    \end{center}\vspace{-0.6cm}
    \caption{Computational costs vs. $Q_1$ for 4 collocation methods applied to Example~\ref{Ex2}, $\alpha=0.4$, $\gamma=0$, $TOL=10^{-4}$,
    residual barrier $\RR_0$ with $L_\infty(\Omega)$ norm and  $\lambda=\pi^2$, $w=\lambda/8$
\label{fig:Qcosts}}
  \end{figure}
  We observe a drastic increase of algorithm iterations, and, hence, computational time and costs for as $Q_1$ becomes close to $1$.
  \begin{itemize}
    \item We conclude that a small value of $Q_1=1.1$ or $Q_1=1.2$ seems to be a good compromise between computational
          costs and quality of the mesh.
    \item A higher order method leads to a much smaller number of time steps at similar costs for adapting the mesh.
  \end{itemize}
  Figure~\ref{fig:TOLcosts}
  \begin{figure}[htb]
    \begin{center}
      \input{src/itervsTOL} 
      \input{src/timevsTOL} 
    \end{center}\vspace{-0.6cm}
    \caption{Computational costs vs. $TOL$ for 4 collocation methods applied to Example~\ref{Ex2}, $\alpha=0.4$, $\gamma=0$, $Q_1=1.2$,
    residual barrier $\RR_0$ with $L_\infty(\Omega)$ norm and  $\lambda=\pi^2$, $w=\lambda/8$\label{fig:TOLcosts}}
  \end{figure}
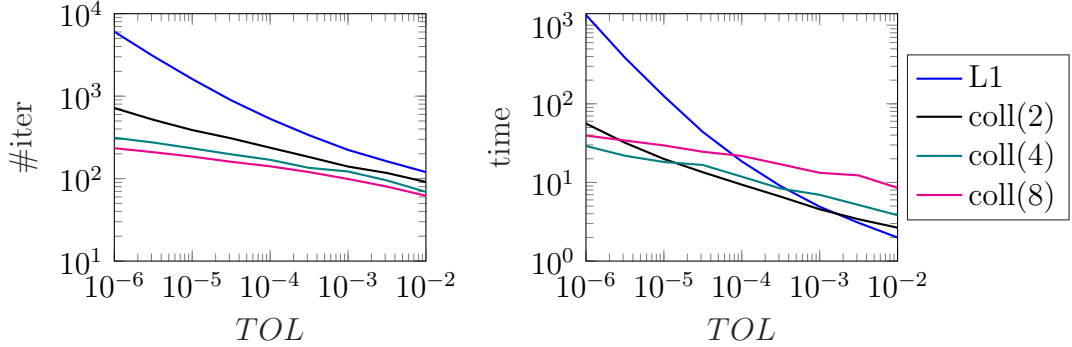
  shows the number of iterations and the corresponding computational time for varying values of $TOL$ and fixed $Q_1=1.2$. As to be expected,  smaller
  values of $TOL$ yield smaller errors, but lead to higher computational costs. This is even stronger observable for the lowest-order method.
  \begin{itemize}
    \item For a given value of $TOL$, higher-order method are less costly.
  \end{itemize}

  Finally, Figure~\ref{fig:res}
  \begin{figure}[htb]
    \begin{center}
      \input{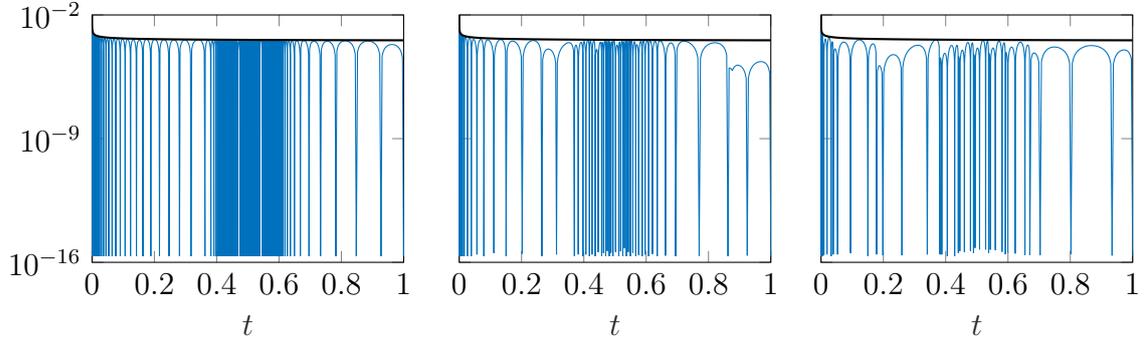}\hspace{2pt} 
      \input{src/R0_res2} 
      \input{src/R0_res4} 
    \end{center}\vspace{-0.6cm}
    \caption{Residuals  and the algorithm's residual bounds
    for collocation methods of order $m\in\{1,2,4\}$ (left to right)
   applied to Example~\ref{Ex2}, $\alpha=0.4$, $\gamma=0$, $Q_1=1.2$, $TOL=10^{-4}$,
    residual barrier $\RR_0$ with $L_\infty(\Omega)$ norm and  $\lambda=\pi^2$, $w=\lambda/8$\label{fig:res}}
  \end{figure}
  shows the behaviour of the residuals. The upper bound of~\eqref{barrier_bounds_alg_a}, imposed on the residual by the algorithm, is depicted in black,
  and, as expected, all residuals are below this bound.

  \section{Conclusions}

  Time-fractional parabolic equations with a Caputo time derivative were considered.
  For such equations, we have reviewed and generalized the a-posteriori error estimates from \cite{Kopteva22}, and
  improved the earlier time stepping algorithm based on this theory.
  A number of temporal discretizations were considered, including the L1 method, an L1-2 method, and continuous collocation methods of arbitrary order.
  A stable and efficient implementation of the resulting algorithm was described,
 which is essential in the context of higher-order methods.
 It was demonstrated that high-order methods (of order up to as high as 8)
    exhibit a huge improvement in the accuracy when the time steps are chosen adaptively, and, furthermore,
 adaptive temporal meshes yield optimal convergence rates  in the presence of various solution singularities.

  \section*{Funding}
  The second author was partially supported by  Science Foundation Ireland under Grant number 18/CRT/6049.

 \bibliographystyle{plain}
\bibliography{lit}

\end{document}

%% file: src/errors_eq_err.tex
%
%
\begin{tikzpicture}

\begin{axis}[%
width=0.3\textwidth,
height=0.16\textwidth,
scale only axis,
xmode=log,
xmin=1,
xmax=100,
xminorticks=true,
xlabel style={font=\color{white!15!black}},
xlabel={$M$},
ymode=log,
ymin=1e-8,
ymax=1,
yminorticks=true,
axis background/.style={fill=white},
]
\addplot [color=blue,thick]
  table[row sep=crcr]{%
2	0.075509\\
4	0.05088\\
8	0.037281\\
16	0.028076\\
32	0.021359\\
64	0.016312\\
};

\addplot [color=black,thick]
  table[row sep=crcr]{%
2	0.040224\\
4	0.030707\\
8	0.023471\\
16	0.017961\\
32	0.013759\\
64	0.010547\\
};

\addplot [color=teal,thick]
  table[row sep=crcr]{%
2	0.025876\\
4	0.019836\\
8	0.015232\\
16	0.011713\\
32	0.0090149\\
64	0.0069408\\
};

\addplot [color=magenta,thick]
  table[row sep=crcr]{%
2	0.011153\\
4	0.0087687\\
8	0.0069173\\
16	0.0054678\\
32	0.0043231\\
64	0.0034127\\
};

\addplot [color=purple,thick]
  table[row sep=crcr]{%
2	0.075509\\
4	0.05088\\
8	0.037281\\
16	0.028076\\
32	0.021359\\
64	0.016312\\
};

\addplot [color=black]
  table[row sep=crcr]{%
5e0   1e-03\\
5e0   4e-04\\
5e1   4e-04\\
5e0   1e-03\\
};
\node[color=black,right] at (1.5e0,5e-4) {$M^{-\alpha}$};

\end{axis}

\end{tikzpicture}%

%% file: src/errors_grad_err_pi2.tex
%
%
\begin{tikzpicture}

\begin{axis}[%
width=0.3\textwidth,
height=0.16\textwidth,
scale only axis,
xmode=log,
xmin=1,
xmax=100,
xminorticks=true,
xlabel style={font=\color{white!15!black}},
xlabel={$M$},
ymode=log,
ymin=1e-8,
ymax=1,
yminorticks=true,
yticklabels={},
axis background/.style={fill=white},
legend columns=1, 
legend style={at={(1.03,0.5)}, anchor=west, legend cell align=left, align=left, draw=white!15!black}
]
\addplot [color=blue,thick]
  table[row sep=crcr]{%
10	0.0048259\\
19	0.0014227\\
37	0.00043905\\
72	0.00013622\\
};
\addlegendentry{L1, error}

\addplot [color=purple,thick]
  table[row sep=crcr]{%
9	0.0046293\\
14	0.0014426\\
23	0.00041466\\
35	0.00013931\\
55	4.7361e-05\\
84	1.4227e-05\\
};
\addlegendentry{L1-2, error}

\addplot [color=black,thick]
  table[row sep=crcr]{%
4	0.0041742\\
7	0.0014139\\
11	0.00045195\\
17	0.00013164\\
26	3.8771e-05\\
40	1.2383e-05\\
63	3.9541e-06\\
97	1.1737e-06\\
};
\addlegendentry{coll(2), error}

\addplot [color=teal,thick]
  table[row sep=crcr]{%
3	0.0041239\\
4	0.0012828\\
6	0.00038772\\
9	0.00012446\\
12	3.9804e-05\\
16	1.1821e-05\\
21	4.0097e-06\\
29	1.223e-06\\
38	3.8003e-07\\
49	1.2133e-07\\
64	3.874e-08\\
84	1.2214e-08\\
};
\addlegendentry{coll(4), error}

\addplot [color=magenta,thick]
  table[row sep=crcr]{%
2	0.0034436\\
3	0.00098743\\
5	0.00031753\\
6	0.00010165\\
8	3.3698e-05\\
10	9.694e-06\\
12	3.0965e-06\\
15	9.8878e-07\\
18	3.1176e-07\\
21	9.954e-08\\
25	3.1781e-08\\
30	1.0147e-08\\
};
\addlegendentry{coll(8), error}


\addplot [color=black]
  table[row sep=crcr]{%
2e1   1e-02\\
6e1   1e-02\\
6e1   1.7e-03\\
2e1   1e-02\\
};
\node[color=black,right] at (2e1,1e-1) {$M^{-(2-\alpha)}$};

\addplot [color=black]
  table[row sep=crcr]{%
5e0   1e-05\\
5e0   6e-08\\
15e0  6e-08\\
5e0   1e-05\\
};
\node[color=black,right] at (1e0,6e-7) {$M^{-(5-\alpha)}$};

\end{axis}

\end{tikzpicture}%

%% file: src/residualsL1.tex
%
%
\definecolor{mycolor1}{rgb}{0.00000,0.44700,0.74100}%
\begin{tikzpicture}

\begin{axis}[%
width=0.6\textwidth,
height=0.15\textwidth,
scale only axis,
xmin=0,xmax=1,
xlabel style={font=\color{white!15!black}},
xlabel={$t$},
xtick={0,0.25,0.5,0.75,1},
xticklabels={0,$\frac{1}{4}$,$\frac{1}{2}$,$\frac{3}{4}$,1},
ymode=log,
ymin=1e-16,ymax=1,
yminorticks=true,
axis background/.style={fill=white}
]
\addplot [color=mycolor1, forget plot,thick]
  table[row sep=crcr]{%
0	1\\
2.5e-11	1\\
8e-10	1\\
6.075e-09	1\\
2.56e-08	1\\
7.8125e-08	1\\
1.944e-07	1\\
4.2017e-07	1\\
8.192e-07	1\\
1.4762e-06	0.99999\\
2.5e-06	0.99999\\
4.0263e-06	0.99998\\
6.2208e-06	0.99998\\
9.2823e-06	0.99996\\
1.3446e-05	0.99995\\
1.8984e-05	0.99992\\
2.6214e-05	0.9999\\
3.5496e-05	0.99986\\
4.7239e-05	0.99981\\
6.1902e-05	0.99975\\
8e-05	0.99968\\
0.0001021	0.99959\\
0.00012884	0.99948\\
0.00016091	0.99936\\
0.00019907	0.9992\\
0.00024414	0.99902\\
0.00029703	0.99881\\
0.00035872	0.99857\\
0.00043026	0.99828\\
0.00051278	0.99795\\
0.0006075	0.99757\\
0.00071573	0.99714\\
0.00083886	0.99664\\
0.00097838	0.99609\\
0.0011359	0.99546\\
0.001313	0.99475\\
0.0015117	0.99395\\
0.0017336	0.99307\\
0.0019809	0.99208\\
0.0022556	0.99098\\
0.00256	0.98976\\
0.0028964	0.98841\\
0.0032673	0.98693\\
0.0036752	0.9853\\
0.0041229	0.98351\\
0.0046132	0.98155\\
0.0051491	0.9794\\
0.0057336	0.97707\\
0.0063701	0.97452\\
0.0070619	0.97175\\
0.0078125	0.96875\\
0.0086256	0.9655\\
0.0095051	0.96198\\
0.010455	0.95818\\
0.011479	0.95408\\
0.012582	0.94967\\
0.013768	0.94493\\
0.015042	0.93983\\
0.016409	0.93436\\
0.017873	0.92851\\
0.01944	0.92224\\
0.021115	0.91554\\
0.022903	0.90839\\
0.024811	0.90076\\
0.026844	0.89263\\
0.029007	0.88397\\
0.031308	0.87477\\
0.033753	0.86499\\
0.036348	0.85461\\
0.039101	0.8436\\
0.042017	0.83193\\
0.045106	0.81958\\
0.048373	0.80651\\
0.051827	0.79269\\
0.055475	0.7781\\
0.059326	0.7627\\
0.063388	0.74645\\
0.06767	0.72932\\
0.072179	0.71128\\
0.076926	0.69229\\
0.08192	0.67232\\
0.08717	0.65132\\
0.092685	0.62926\\
0.098476	0.6061\\
0.10455	0.58179\\
0.11093	0.55629\\
0.11761	0.52957\\
0.12461	0.50158\\
0.13193	0.47227\\
0.1396	0.44159\\
0.14762	0.40951\\
0.15601	0.37597\\
0.16477	0.34092\\
0.17392	0.30431\\
0.18348	0.2661\\
0.19345	0.22622\\
0.20384	0.18463\\
0.21468	0.14127\\
0.22598	0.096079\\
0.23775	0.04901\\
0.25	2.2204e-16\\
0.25	0.0022187\\
0.25	0.0044373\\
0.25	0.006656\\
0.25	0.0088746\\
0.25	0.011093\\
0.25	0.013312\\
0.25	0.01553\\
0.25	0.017749\\
0.25	0.019967\\
0.25	0.022184\\
0.25	0.024401\\
0.25001	0.026618\\
0.25001	0.028834\\
0.25001	0.031048\\
0.25002	0.033262\\
0.25003	0.035473\\
0.25004	0.037683\\
0.25005	0.03989\\
0.25006	0.042095\\
0.25008	0.044296\\
0.2501	0.046493\\
0.25013	0.048686\\
0.25016	0.050874\\
0.2502	0.053056\\
0.25024	0.055231\\
0.2503	0.057398\\
0.25036	0.059557\\
0.25043	0.061707\\
0.25051	0.063846\\
0.25061	0.065973\\
0.25072	0.068087\\
0.25084	0.070187\\
0.25098	0.072271\\
0.25114	0.074338\\
0.25131	0.076386\\
0.25151	0.078413\\
0.25173	0.080417\\
0.25198	0.082397\\
0.25226	0.084351\\
0.25256	0.086277\\
0.2529	0.088171\\
0.25327	0.090033\\
0.25368	0.091859\\
0.25412	0.093646\\
0.25461	0.095393\\
0.25515	0.097097\\
0.25573	0.098754\\
0.25637	0.10036\\
0.25706	0.10192\\
0.25781	0.10342\\
0.25863	0.10485\\
0.25951	0.10623\\
0.26045	0.10754\\
0.26148	0.10878\\
0.26258	0.10995\\
0.26377	0.11103\\
0.26504	0.11204\\
0.26641	0.11296\\
0.26787	0.11378\\
0.26944	0.11452\\
0.27111	0.11515\\
0.2729	0.11567\\
0.27481	0.11609\\
0.27684	0.11639\\
0.27901	0.11658\\
0.28131	0.11663\\
0.28375	0.11656\\
0.28635	0.11635\\
0.2891	0.116\\
0.29202	0.1155\\
0.29511	0.11485\\
0.29837	0.11404\\
0.30183	0.11307\\
0.30548	0.11192\\
0.30933	0.1106\\
0.31339	0.1091\\
0.31767	0.1074\\
0.32218	0.10551\\
0.32693	0.10342\\
0.33192	0.10111\\
0.33717	0.098596\\
0.34268	0.095855\\
0.34848	0.092884\\
0.35455	0.089676\\
0.36093	0.086225\\
0.36761	0.082522\\
0.37461	0.078561\\
0.38193	0.074333\\
0.3896	0.069833\\
0.39762	0.065051\\
0.40601	0.05998\\
0.41477	0.054613\\
0.42392	0.048941\\
0.43348	0.042955\\
0.44345	0.036649\\
0.45384	0.030013\\
0.46468	0.023039\\
0.47598	0.015718\\
0.48775	0.0080413\\
0.5	2.2204e-16\\
0.5	0.00044392\\
0.5	0.00088784\\
0.5	0.0013318\\
0.5	0.0017757\\
0.5	0.0022196\\
0.5	0.0026635\\
0.5	0.0031074\\
0.5	0.0035512\\
0.5	0.003995\\
0.5	0.0044388\\
0.5	0.0048825\\
0.50001	0.005326\\
0.50001	0.0057695\\
0.50001	0.0062127\\
0.50002	0.0066558\\
0.50003	0.0070985\\
0.50004	0.007541\\
0.50005	0.007983\\
0.50006	0.0084246\\
0.50008	0.0088656\\
0.5001	0.009306\\
0.50013	0.0097457\\
0.50016	0.010184\\
0.5002	0.010622\\
0.50024	0.011059\\
0.5003	0.011495\\
0.50036	0.011929\\
0.50043	0.012361\\
0.50051	0.012792\\
0.50061	0.013221\\
0.50072	0.013647\\
0.50084	0.014072\\
0.50098	0.014493\\
0.50114	0.014912\\
0.50131	0.015328\\
0.50151	0.01574\\
0.50173	0.016148\\
0.50198	0.016553\\
0.50226	0.016953\\
0.50256	0.017348\\
0.5029	0.017738\\
0.50327	0.018122\\
0.50368	0.018501\\
0.50412	0.018873\\
0.50461	0.019238\\
0.50515	0.019596\\
0.50573	0.019946\\
0.50637	0.020287\\
0.50706	0.02062\\
0.50781	0.020942\\
0.50863	0.021255\\
0.50951	0.021556\\
0.51045	0.021846\\
0.51148	0.022124\\
0.51258	0.022389\\
0.51377	0.022639\\
0.51504	0.022875\\
0.51641	0.023096\\
0.51787	0.0233\\
0.51944	0.023486\\
0.52111	0.023655\\
0.5229	0.023803\\
0.52481	0.023932\\
0.52684	0.024039\\
0.52901	0.024123\\
0.53131	0.024183\\
0.53375	0.024218\\
0.53635	0.024226\\
0.5391	0.024207\\
0.54202	0.024159\\
0.54511	0.024079\\
0.54837	0.023968\\
0.55183	0.023823\\
0.55548	0.023642\\
0.55933	0.023425\\
0.56339	0.023168\\
0.56767	0.022872\\
0.57218	0.022532\\
0.57693	0.022149\\
0.58192	0.021719\\
0.58717	0.021241\\
0.59268	0.020713\\
0.59848	0.020133\\
0.60455	0.019498\\
0.61093	0.018806\\
0.61761	0.018056\\
0.62461	0.017244\\
0.63193	0.016369\\
0.6396	0.015427\\
0.64762	0.014418\\
0.65601	0.013337\\
0.66477	0.012183\\
0.67392	0.010954\\
0.68348	0.0096454\\
0.69345	0.0082561\\
0.70384	0.0067831\\
0.71468	0.0052237\\
0.72598	0.0035752\\
0.73775	0.0018349\\
0.75	2.2204e-16\\
0.75	0.00010237\\
0.75	0.00020474\\
0.75	0.00030711\\
0.75	0.00040948\\
0.75	0.00051185\\
0.75	0.00061421\\
0.75	0.00071658\\
0.75	0.00081894\\
0.75	0.00092129\\
0.75	0.0010236\\
0.75	0.0011259\\
0.75001	0.0012283\\
0.75001	0.0013305\\
0.75001	0.0014328\\
0.75002	0.001535\\
0.75003	0.0016371\\
0.75004	0.0017392\\
0.75005	0.0018412\\
0.75006	0.0019432\\
0.75008	0.002045\\
0.7501	0.0021467\\
0.75013	0.0022482\\
0.75016	0.0023496\\
0.7502	0.0024509\\
0.75024	0.0025519\\
0.7503	0.0026526\\
0.75036	0.0027531\\
0.75043	0.0028533\\
0.75051	0.0029532\\
0.75061	0.0030527\\
0.75072	0.0031518\\
0.75084	0.0032504\\
0.75098	0.0033486\\
0.75114	0.0034462\\
0.75131	0.0035432\\
0.75151	0.0036395\\
0.75173	0.0037351\\
0.75198	0.0038299\\
0.75226	0.0039239\\
0.75256	0.004017\\
0.7529	0.0041091\\
0.75327	0.0042001\\
0.75368	0.0042899\\
0.75412	0.0043785\\
0.75461	0.0044658\\
0.75515	0.0045516\\
0.75573	0.0046359\\
0.75637	0.0047185\\
0.75706	0.0047993\\
0.75781	0.0048782\\
0.75863	0.0049551\\
0.75951	0.0050298\\
0.76045	0.0051022\\
0.76148	0.0051721\\
0.76258	0.0052394\\
0.76377	0.0053038\\
0.76504	0.0053652\\
0.76641	0.0054234\\
0.76787	0.0054782\\
0.76944	0.0055294\\
0.77111	0.0055767\\
0.7729	0.00562\\
0.77481	0.0056588\\
0.77684	0.005693\\
0.77901	0.0057223\\
0.78131	0.0057464\\
0.78375	0.005765\\
0.78635	0.0057776\\
0.7891	0.0057841\\
0.79202	0.0057839\\
0.79511	0.0057767\\
0.79837	0.0057621\\
0.80183	0.0057397\\
0.80548	0.0057091\\
0.80933	0.0056696\\
0.81339	0.0056209\\
0.81767	0.0055624\\
0.82218	0.0054937\\
0.82693	0.005414\\
0.83192	0.0053229\\
0.83717	0.0052198\\
0.84268	0.0051039\\
0.84848	0.0049747\\
0.85455	0.0048315\\
0.86093	0.0046735\\
0.86761	0.0045001\\
0.87461	0.0043106\\
0.88193	0.0041041\\
0.8896	0.0038798\\
0.89762	0.003637\\
0.90601	0.0033749\\
0.91477	0.0030925\\
0.92392	0.0027891\\
0.93348	0.0024638\\
0.94345	0.0021156\\
0.95384	0.0017437\\
0.96468	0.0013471\\
0.97598	0.00092491\\
0.98775	0.0004762\\
1	2.2204e-16\\
};
\end{axis}
\end{tikzpicture}%

%% file: src/residualsL12.tex
%
%
\definecolor{mycolor1}{rgb}{0.00000,0.44700,0.74100}%
\begin{tikzpicture}

\begin{axis}[%
width=0.6\textwidth,
height=0.15\textwidth,
scale only axis,
xmin=0,xmax=1,
xlabel style={font=\color{white!15!black}},
xlabel={$t$},
xtick={0,0.25,0.5,0.75,1},
xticklabels={0,$\frac{1}{4}$,$\frac{1}{2}$,$\frac{3}{4}$,1},
ymode=log,
ymin=1e-16,ymax=1,
yminorticks=true,
axis background/.style={fill=white}
]
\addplot [color=mycolor1, forget plot,thick]
  table[row sep=crcr]{%
0	1\\
2.5e-11	1\\
8e-10	1\\
6.075e-09	1\\
2.56e-08	1\\
7.8125e-08	1\\
1.944e-07	1\\
4.2017e-07	1\\
8.192e-07	1\\
1.4762e-06	0.99999\\
2.5e-06	0.99999\\
4.0263e-06	0.99998\\
6.2208e-06	0.99998\\
9.2823e-06	0.99996\\
1.3446e-05	0.99995\\
1.8984e-05	0.99992\\
2.6214e-05	0.9999\\
3.5496e-05	0.99986\\
4.7239e-05	0.99981\\
6.1902e-05	0.99975\\
8e-05	0.99968\\
0.0001021	0.99959\\
0.00012884	0.99948\\
0.00016091	0.99936\\
0.00019907	0.9992\\
0.00024414	0.99902\\
0.00029703	0.99881\\
0.00035872	0.99857\\
0.00043026	0.99828\\
0.00051278	0.99795\\
0.0006075	0.99757\\
0.00071573	0.99714\\
0.00083886	0.99664\\
0.00097838	0.99609\\
0.0011359	0.99546\\
0.001313	0.99475\\
0.0015117	0.99395\\
0.0017336	0.99307\\
0.0019809	0.99208\\
0.0022556	0.99098\\
0.00256	0.98976\\
0.0028964	0.98841\\
0.0032673	0.98693\\
0.0036752	0.9853\\
0.0041229	0.98351\\
0.0046132	0.98155\\
0.0051491	0.9794\\
0.0057336	0.97707\\
0.0063701	0.97452\\
0.0070619	0.97175\\
0.0078125	0.96875\\
0.0086256	0.9655\\
0.0095051	0.96198\\
0.010455	0.95818\\
0.011479	0.95408\\
0.012582	0.94967\\
0.013768	0.94493\\
0.015042	0.93983\\
0.016409	0.93436\\
0.017873	0.92851\\
0.01944	0.92224\\
0.021115	0.91554\\
0.022903	0.90839\\
0.024811	0.90076\\
0.026844	0.89263\\
0.029007	0.88397\\
0.031308	0.87477\\
0.033753	0.86499\\
0.036348	0.85461\\
0.039101	0.8436\\
0.042017	0.83193\\
0.045106	0.81958\\
0.048373	0.80651\\
0.051827	0.79269\\
0.055475	0.7781\\
0.059326	0.7627\\
0.063388	0.74645\\
0.06767	0.72932\\
0.072179	0.71128\\
0.076926	0.69229\\
0.08192	0.67232\\
0.08717	0.65132\\
0.092685	0.62926\\
0.098476	0.6061\\
0.10455	0.58179\\
0.11093	0.55629\\
0.11761	0.52957\\
0.12461	0.50158\\
0.13193	0.47227\\
0.1396	0.44159\\
0.14762	0.40951\\
0.15601	0.37597\\
0.16477	0.34092\\
0.17392	0.30431\\
0.18348	0.2661\\
0.19345	0.22622\\
0.20384	0.18463\\
0.21468	0.14127\\
0.22598	0.096079\\
0.23775	0.04901\\
0.25	2.2204e-16\\
0.25	0.0010199\\
0.25	0.0020398\\
0.25	0.0030597\\
0.25	0.0040795\\
0.25	0.0050993\\
0.25	0.006119\\
0.25	0.0071385\\
0.25	0.0081576\\
0.25	0.0091763\\
0.25	0.010194\\
0.25	0.011211\\
0.25001	0.012227\\
0.25001	0.013241\\
0.25001	0.014253\\
0.25002	0.015263\\
0.25003	0.01627\\
0.25004	0.017272\\
0.25005	0.018271\\
0.25006	0.019264\\
0.25008	0.020251\\
0.2501	0.021231\\
0.25013	0.022203\\
0.25016	0.023166\\
0.2502	0.024118\\
0.25024	0.025058\\
0.2503	0.025985\\
0.25036	0.026897\\
0.25043	0.027792\\
0.25051	0.028669\\
0.25061	0.029525\\
0.25072	0.030359\\
0.25084	0.031169\\
0.25098	0.031952\\
0.25114	0.032706\\
0.25131	0.033428\\
0.25151	0.034116\\
0.25173	0.034768\\
0.25198	0.03538\\
0.25226	0.035949\\
0.25256	0.036472\\
0.2529	0.036948\\
0.25327	0.037371\\
0.25368	0.037739\\
0.25412	0.038048\\
0.25461	0.038296\\
0.25515	0.038479\\
0.25573	0.038592\\
0.25637	0.038633\\
0.25706	0.038598\\
0.25781	0.038483\\
0.25863	0.038286\\
0.25951	0.038001\\
0.26045	0.037626\\
0.26148	0.037158\\
0.26258	0.036593\\
0.26377	0.035928\\
0.26504	0.03516\\
0.26641	0.034287\\
0.26787	0.033306\\
0.26944	0.035857\\
0.27111	0.038663\\
0.2729	0.04161\\
0.27481	0.044697\\
0.27684	0.047922\\
0.27901	0.051283\\
0.28131	0.054775\\
0.28375	0.058392\\
0.28635	0.062127\\
0.2891	0.065971\\
0.29202	0.069911\\
0.29511	0.073935\\
0.29837	0.078026\\
0.30183	0.082165\\
0.30548	0.08633\\
0.30933	0.090496\\
0.31339	0.094632\\
0.31767	0.098706\\
0.32218	0.10268\\
0.32693	0.10651\\
0.33192	0.11015\\
0.33717	0.11355\\
0.34268	0.11665\\
0.34848	0.11937\\
0.35455	0.12166\\
0.36093	0.12342\\
0.36761	0.12456\\
0.37461	0.125\\
0.38193	0.12462\\
0.3896	0.12329\\
0.39762	0.12091\\
0.40601	0.11731\\
0.41477	0.11235\\
0.42392	0.10585\\
0.43348	0.097644\\
0.44345	0.087522\\
0.45384	0.07527\\
0.46468	0.060655\\
0.47598	0.043424\\
0.48775	0.023304\\
0.5	2.2204e-16\\
0.5	0.00066964\\
0.5	0.0013393\\
0.5	0.0020089\\
0.5	0.0026786\\
0.5	0.0033482\\
0.5	0.0040178\\
0.5	0.0046874\\
0.5	0.005357\\
0.5	0.0060265\\
0.5	0.0066959\\
0.5	0.0073653\\
0.50001	0.0080345\\
0.50001	0.0087036\\
0.50001	0.0093725\\
0.50002	0.010041\\
0.50003	0.01071\\
0.50004	0.011378\\
0.50005	0.012045\\
0.50006	0.012713\\
0.50008	0.01338\\
0.5001	0.014046\\
0.50013	0.014712\\
0.50016	0.015378\\
0.5002	0.016043\\
0.50024	0.016707\\
0.5003	0.01737\\
0.50036	0.018033\\
0.50043	0.018695\\
0.50051	0.019357\\
0.50061	0.020017\\
0.50072	0.020677\\
0.50084	0.021336\\
0.50098	0.021994\\
0.50114	0.022652\\
0.50131	0.023309\\
0.50151	0.023966\\
0.50173	0.024621\\
0.50198	0.025277\\
0.50226	0.025932\\
0.50256	0.026586\\
0.5029	0.027241\\
0.50327	0.027896\\
0.50368	0.02855\\
0.50412	0.029205\\
0.50461	0.02986\\
0.50515	0.030515\\
0.50573	0.031172\\
0.50637	0.031829\\
0.50706	0.032486\\
0.50781	0.033145\\
0.50863	0.033805\\
0.50951	0.034466\\
0.51045	0.035128\\
0.51148	0.035791\\
0.51258	0.036455\\
0.51377	0.03712\\
0.51504	0.037785\\
0.51641	0.03845\\
0.51787	0.039116\\
0.51944	0.03978\\
0.52111	0.040443\\
0.5229	0.041103\\
0.52481	0.04176\\
0.52684	0.042411\\
0.52901	0.043056\\
0.53131	0.043692\\
0.53375	0.044318\\
0.53635	0.04493\\
0.5391	0.045526\\
0.54202	0.046102\\
0.54511	0.046654\\
0.54837	0.047178\\
0.55183	0.047669\\
0.55548	0.048122\\
0.55933	0.048528\\
0.56339	0.048883\\
0.56767	0.049177\\
0.57218	0.049401\\
0.57693	0.049546\\
0.58192	0.049601\\
0.58717	0.049552\\
0.59268	0.049387\\
0.59848	0.049091\\
0.60455	0.048647\\
0.61093	0.048036\\
0.61761	0.04724\\
0.62461	0.046237\\
0.63193	0.045001\\
0.6396	0.043509\\
0.64762	0.041731\\
0.65601	0.039637\\
0.66477	0.037193\\
0.67392	0.034363\\
0.68348	0.031108\\
0.69345	0.027384\\
0.70384	0.023147\\
0.71468	0.018345\\
0.72598	0.012926\\
0.73775	0.0068321\\
0.75	2.2204e-16\\
0.75	0.00030444\\
0.75	0.00060887\\
0.75	0.00091331\\
0.75	0.0012177\\
0.75	0.0015222\\
0.75	0.0018266\\
0.75	0.002131\\
0.75	0.0024353\\
0.75	0.0027396\\
0.75	0.0030438\\
0.75	0.0033479\\
0.75001	0.0036519\\
0.75001	0.0039556\\
0.75001	0.0042592\\
0.75002	0.0045624\\
0.75003	0.0048654\\
0.75004	0.0051679\\
0.75005	0.0054699\\
0.75006	0.0057714\\
0.75008	0.0060722\\
0.7501	0.0063723\\
0.75013	0.0066716\\
0.75016	0.0069698\\
0.7502	0.007267\\
0.75024	0.007563\\
0.7503	0.0078577\\
0.75036	0.0081509\\
0.75043	0.0084424\\
0.75051	0.0087322\\
0.75061	0.0090201\\
0.75072	0.0093058\\
0.75084	0.0095892\\
0.75098	0.0098702\\
0.75114	0.010149\\
0.75131	0.010424\\
0.75151	0.010696\\
0.75173	0.010966\\
0.75198	0.011231\\
0.75226	0.011493\\
0.75256	0.011751\\
0.7529	0.012005\\
0.75327	0.012255\\
0.75368	0.0125\\
0.75412	0.01274\\
0.75461	0.012976\\
0.75515	0.013206\\
0.75573	0.01343\\
0.75637	0.013648\\
0.75706	0.013861\\
0.75781	0.014067\\
0.75863	0.014266\\
0.75951	0.014459\\
0.76045	0.014644\\
0.76148	0.014822\\
0.76258	0.014992\\
0.76377	0.015154\\
0.76504	0.015308\\
0.76641	0.015452\\
0.76787	0.015588\\
0.76944	0.015714\\
0.77111	0.01583\\
0.7729	0.015935\\
0.77481	0.01603\\
0.77684	0.016113\\
0.77901	0.016185\\
0.78131	0.016244\\
0.78375	0.01629\\
0.78635	0.016322\\
0.7891	0.01634\\
0.79202	0.016342\\
0.79511	0.016329\\
0.79837	0.016298\\
0.80183	0.016249\\
0.80548	0.016181\\
0.80933	0.016093\\
0.81339	0.015983\\
0.81767	0.01585\\
0.82218	0.015692\\
0.82693	0.015508\\
0.83192	0.015295\\
0.83717	0.015052\\
0.84268	0.014776\\
0.84848	0.014465\\
0.85455	0.014116\\
0.86093	0.013727\\
0.86761	0.013293\\
0.87461	0.012812\\
0.88193	0.01228\\
0.8896	0.011693\\
0.89762	0.011046\\
0.90601	0.010334\\
0.91477	0.0095521\\
0.92392	0.0086949\\
0.93348	0.0077562\\
0.94345	0.0067293\\
0.95384	0.0056069\\
0.96468	0.0043815\\
0.97598	0.0030446\\
0.98775	0.0015873\\
1	2.2204e-16\\
};
\end{axis}
\end{tikzpicture}%

%% file: src/residualsColl1.tex
%
%
\definecolor{mycolor1}{rgb}{0.00000,0.44700,0.74100}%
\begin{tikzpicture}

\begin{axis}[%
width=0.3\textwidth,
height=0.15\textwidth,
scale only axis,
xmin=0,xmax=1,
xlabel style={font=\color{white!15!black}},
xtick={0,0.25,0.5,0.75,1},
xticklabels={},
ymode=log,
ymin=1e-16,ymax=1,
yminorticks=true,
axis background/.style={fill=white}
]
\addplot [color=mycolor1, forget plot,thick]
  table[row sep=crcr]{%
0	1\\
2.5e-11	1\\
8e-10	1\\
6.075e-09	1\\
2.56e-08	1\\
7.8125e-08	1\\
1.944e-07	1\\
4.2017e-07	1\\
8.192e-07	1\\
1.4762e-06	0.99999\\
2.5e-06	0.99999\\
4.0263e-06	0.99998\\
6.2208e-06	0.99998\\
9.2823e-06	0.99996\\
1.3446e-05	0.99995\\
1.8984e-05	0.99992\\
2.6214e-05	0.9999\\
3.5496e-05	0.99986\\
4.7239e-05	0.99981\\
6.1902e-05	0.99975\\
8e-05	0.99968\\
0.0001021	0.99959\\
0.00012884	0.99948\\
0.00016091	0.99936\\
0.00019907	0.9992\\
0.00024414	0.99902\\
0.00029703	0.99881\\
0.00035872	0.99857\\
0.00043026	0.99828\\
0.00051278	0.99795\\
0.0006075	0.99757\\
0.00071573	0.99714\\
0.00083886	0.99664\\
0.00097838	0.99609\\
0.0011359	0.99546\\
0.001313	0.99475\\
0.0015117	0.99395\\
0.0017336	0.99307\\
0.0019809	0.99208\\
0.0022556	0.99098\\
0.00256	0.98976\\
0.0028964	0.98841\\
0.0032673	0.98693\\
0.0036752	0.9853\\
0.0041229	0.98351\\
0.0046132	0.98155\\
0.0051491	0.9794\\
0.0057336	0.97707\\
0.0063701	0.97452\\
0.0070619	0.97175\\
0.0078125	0.96875\\
0.0086256	0.9655\\
0.0095051	0.96198\\
0.010455	0.95818\\
0.011479	0.95408\\
0.012582	0.94967\\
0.013768	0.94493\\
0.015042	0.93983\\
0.016409	0.93436\\
0.017873	0.92851\\
0.01944	0.92224\\
0.021115	0.91554\\
0.022903	0.90839\\
0.024811	0.90076\\
0.026844	0.89263\\
0.029007	0.88397\\
0.031308	0.87477\\
0.033753	0.86499\\
0.036348	0.85461\\
0.039101	0.8436\\
0.042017	0.83193\\
0.045106	0.81958\\
0.048373	0.80651\\
0.051827	0.79269\\
0.055475	0.7781\\
0.059326	0.7627\\
0.063388	0.74645\\
0.06767	0.72932\\
0.072179	0.71128\\
0.076926	0.69229\\
0.08192	0.67232\\
0.08717	0.65132\\
0.092685	0.62926\\
0.098476	0.6061\\
0.10455	0.58179\\
0.11093	0.55629\\
0.11761	0.52957\\
0.12461	0.50158\\
0.13193	0.47227\\
0.1396	0.44159\\
0.14762	0.40951\\
0.15601	0.37597\\
0.16477	0.34092\\
0.17392	0.30431\\
0.18348	0.2661\\
0.19345	0.22622\\
0.20384	0.18463\\
0.21468	0.14127\\
0.22598	0.096079\\
0.23775	0.04901\\
0.25	2.2204e-16\\
0.25	0.0022187\\
0.25	0.0044373\\
0.25	0.006656\\
0.25	0.0088746\\
0.25	0.011093\\
0.25	0.013312\\
0.25	0.01553\\
0.25	0.017749\\
0.25	0.019967\\
0.25	0.022184\\
0.25	0.024401\\
0.25001	0.026618\\
0.25001	0.028834\\
0.25001	0.031048\\
0.25002	0.033262\\
0.25003	0.035473\\
0.25004	0.037683\\
0.25005	0.03989\\
0.25006	0.042095\\
0.25008	0.044296\\
0.2501	0.046493\\
0.25013	0.048686\\
0.25016	0.050874\\
0.2502	0.053056\\
0.25024	0.055231\\
0.2503	0.057398\\
0.25036	0.059557\\
0.25043	0.061707\\
0.25051	0.063846\\
0.25061	0.065973\\
0.25072	0.068087\\
0.25084	0.070187\\
0.25098	0.072271\\
0.25114	0.074338\\
0.25131	0.076386\\
0.25151	0.078413\\
0.25173	0.080417\\
0.25198	0.082397\\
0.25226	0.084351\\
0.25256	0.086277\\
0.2529	0.088171\\
0.25327	0.090033\\
0.25368	0.091859\\
0.25412	0.093646\\
0.25461	0.095393\\
0.25515	0.097097\\
0.25573	0.098754\\
0.25637	0.10036\\
0.25706	0.10192\\
0.25781	0.10342\\
0.25863	0.10485\\
0.25951	0.10623\\
0.26045	0.10754\\
0.26148	0.10878\\
0.26258	0.10995\\
0.26377	0.11103\\
0.26504	0.11204\\
0.26641	0.11296\\
0.26787	0.11378\\
0.26944	0.11452\\
0.27111	0.11515\\
0.2729	0.11567\\
0.27481	0.11609\\
0.27684	0.11639\\
0.27901	0.11658\\
0.28131	0.11663\\
0.28375	0.11656\\
0.28635	0.11635\\
0.2891	0.116\\
0.29202	0.1155\\
0.29511	0.11485\\
0.29837	0.11404\\
0.30183	0.11307\\
0.30548	0.11192\\
0.30933	0.1106\\
0.31339	0.1091\\
0.31767	0.1074\\
0.32218	0.10551\\
0.32693	0.10342\\
0.33192	0.10111\\
0.33717	0.098596\\
0.34268	0.095855\\
0.34848	0.092884\\
0.35455	0.089676\\
0.36093	0.086225\\
0.36761	0.082522\\
0.37461	0.078561\\
0.38193	0.074333\\
0.3896	0.069833\\
0.39762	0.065051\\
0.40601	0.05998\\
0.41477	0.054613\\
0.42392	0.048941\\
0.43348	0.042955\\
0.44345	0.036649\\
0.45384	0.030013\\
0.46468	0.023039\\
0.47598	0.015718\\
0.48775	0.0080413\\
0.5	2.2204e-16\\
0.5	0.00044392\\
0.5	0.00088784\\
0.5	0.0013318\\
0.5	0.0017757\\
0.5	0.0022196\\
0.5	0.0026635\\
0.5	0.0031074\\
0.5	0.0035512\\
0.5	0.003995\\
0.5	0.0044388\\
0.5	0.0048825\\
0.50001	0.005326\\
0.50001	0.0057695\\
0.50001	0.0062127\\
0.50002	0.0066558\\
0.50003	0.0070985\\
0.50004	0.007541\\
0.50005	0.007983\\
0.50006	0.0084246\\
0.50008	0.0088656\\
0.5001	0.009306\\
0.50013	0.0097457\\
0.50016	0.010184\\
0.5002	0.010622\\
0.50024	0.011059\\
0.5003	0.011495\\
0.50036	0.011929\\
0.50043	0.012361\\
0.50051	0.012792\\
0.50061	0.013221\\
0.50072	0.013647\\
0.50084	0.014072\\
0.50098	0.014493\\
0.50114	0.014912\\
0.50131	0.015328\\
0.50151	0.01574\\
0.50173	0.016148\\
0.50198	0.016553\\
0.50226	0.016953\\
0.50256	0.017348\\
0.5029	0.017738\\
0.50327	0.018122\\
0.50368	0.018501\\
0.50412	0.018873\\
0.50461	0.019238\\
0.50515	0.019596\\
0.50573	0.019946\\
0.50637	0.020287\\
0.50706	0.02062\\
0.50781	0.020942\\
0.50863	0.021255\\
0.50951	0.021556\\
0.51045	0.021846\\
0.51148	0.022124\\
0.51258	0.022389\\
0.51377	0.022639\\
0.51504	0.022875\\
0.51641	0.023096\\
0.51787	0.0233\\
0.51944	0.023486\\
0.52111	0.023655\\
0.5229	0.023803\\
0.52481	0.023932\\
0.52684	0.024039\\
0.52901	0.024123\\
0.53131	0.024183\\
0.53375	0.024218\\
0.53635	0.024226\\
0.5391	0.024207\\
0.54202	0.024159\\
0.54511	0.024079\\
0.54837	0.023968\\
0.55183	0.023823\\
0.55548	0.023642\\
0.55933	0.023425\\
0.56339	0.023168\\
0.56767	0.022872\\
0.57218	0.022532\\
0.57693	0.022149\\
0.58192	0.021719\\
0.58717	0.021241\\
0.59268	0.020713\\
0.59848	0.020133\\
0.60455	0.019498\\
0.61093	0.018806\\
0.61761	0.018056\\
0.62461	0.017244\\
0.63193	0.016369\\
0.6396	0.015427\\
0.64762	0.014418\\
0.65601	0.013337\\
0.66477	0.012183\\
0.67392	0.010954\\
0.68348	0.0096454\\
0.69345	0.0082561\\
0.70384	0.0067831\\
0.71468	0.0052237\\
0.72598	0.0035752\\
0.73775	0.0018349\\
0.75	2.2204e-16\\
0.75	0.00010237\\
0.75	0.00020474\\
0.75	0.00030711\\
0.75	0.00040948\\
0.75	0.00051185\\
0.75	0.00061421\\
0.75	0.00071658\\
0.75	0.00081894\\
0.75	0.00092129\\
0.75	0.0010236\\
0.75	0.0011259\\
0.75001	0.0012283\\
0.75001	0.0013305\\
0.75001	0.0014328\\
0.75002	0.001535\\
0.75003	0.0016371\\
0.75004	0.0017392\\
0.75005	0.0018412\\
0.75006	0.0019432\\
0.75008	0.002045\\
0.7501	0.0021467\\
0.75013	0.0022482\\
0.75016	0.0023496\\
0.7502	0.0024509\\
0.75024	0.0025519\\
0.7503	0.0026526\\
0.75036	0.0027531\\
0.75043	0.0028533\\
0.75051	0.0029532\\
0.75061	0.0030527\\
0.75072	0.0031518\\
0.75084	0.0032504\\
0.75098	0.0033486\\
0.75114	0.0034462\\
0.75131	0.0035432\\
0.75151	0.0036395\\
0.75173	0.0037351\\
0.75198	0.0038299\\
0.75226	0.0039239\\
0.75256	0.004017\\
0.7529	0.0041091\\
0.75327	0.0042001\\
0.75368	0.0042899\\
0.75412	0.0043785\\
0.75461	0.0044658\\
0.75515	0.0045516\\
0.75573	0.0046359\\
0.75637	0.0047185\\
0.75706	0.0047993\\
0.75781	0.0048782\\
0.75863	0.0049551\\
0.75951	0.0050298\\
0.76045	0.0051022\\
0.76148	0.0051721\\
0.76258	0.0052394\\
0.76377	0.0053038\\
0.76504	0.0053652\\
0.76641	0.0054234\\
0.76787	0.0054782\\
0.76944	0.0055294\\
0.77111	0.0055767\\
0.7729	0.00562\\
0.77481	0.0056588\\
0.77684	0.005693\\
0.77901	0.0057223\\
0.78131	0.0057464\\
0.78375	0.005765\\
0.78635	0.0057776\\
0.7891	0.0057841\\
0.79202	0.0057839\\
0.79511	0.0057767\\
0.79837	0.0057621\\
0.80183	0.0057397\\
0.80548	0.0057091\\
0.80933	0.0056696\\
0.81339	0.0056209\\
0.81767	0.0055624\\
0.82218	0.0054937\\
0.82693	0.005414\\
0.83192	0.0053229\\
0.83717	0.0052198\\
0.84268	0.0051039\\
0.84848	0.0049747\\
0.85455	0.0048315\\
0.86093	0.0046735\\
0.86761	0.0045001\\
0.87461	0.0043106\\
0.88193	0.0041041\\
0.8896	0.0038798\\
0.89762	0.003637\\
0.90601	0.0033749\\
0.91477	0.0030925\\
0.92392	0.0027891\\
0.93348	0.0024638\\
0.94345	0.0021156\\
0.95384	0.0017437\\
0.96468	0.0013471\\
0.97598	0.00092491\\
0.98775	0.0004762\\
1	2.2204e-16\\
};
\end{axis}
\end{tikzpicture}%

%% file: src/residualsColl2.tex
%
%
\definecolor{mycolor1}{rgb}{0.00000,0.44700,0.74100}%
\begin{tikzpicture}

\begin{axis}[%
width=0.3\textwidth,
height=0.15\textwidth,
scale only axis,
xmin=0,xmax=1,
xlabel style={font=\color{white!15!black}},
xtick={0,0.25,0.5,0.75,1},
xticklabels={},
ymode=log,
ymin=1e-16,ymax=1,
yminorticks=true,
yticklabels={},
axis background/.style={fill=white}
]
\addplot [color=mycolor1, forget plot,thick]
  table[row sep=crcr]{%
0	1\\
7.8125e-13	1\\
2.5e-11	1\\
1.8984e-10	1\\
8e-10	1\\
2.4414e-09	1\\
6.075e-09	1\\
1.313e-08	1\\
2.56e-08	1\\
4.6132e-08	1\\
7.8125e-08	1\\
1.2582e-07	1\\
1.944e-07	1\\
2.9007e-07	1\\
4.2017e-07	0.99999\\
5.9326e-07	0.99999\\
8.192e-07	0.99999\\
1.1093e-06	0.99999\\
1.4762e-06	0.99998\\
1.9345e-06	0.99998\\
2.5e-06	0.99997\\
3.1907e-06	0.99996\\
4.0263e-06	0.99995\\
5.0284e-06	0.99994\\
6.2208e-06	0.99993\\
7.6294e-06	0.99991\\
9.2823e-06	0.99989\\
1.121e-05	0.99987\\
1.3446e-05	0.99984\\
1.6024e-05	0.99981\\
1.8984e-05	0.99977\\
2.2367e-05	0.99973\\
2.6214e-05	0.99969\\
3.0575e-05	0.99963\\
3.5496e-05	0.99957\\
4.1033e-05	0.99951\\
4.7239e-05	0.99943\\
5.4175e-05	0.99935\\
6.1902e-05	0.99926\\
7.0488e-05	0.99915\\
8e-05	0.99904\\
9.0513e-05	0.99891\\
0.0001021	0.99878\\
0.00011485	0.99862\\
0.00012884	0.99845\\
0.00014416	0.99827\\
0.00016091	0.99807\\
0.00017918	0.99785\\
0.00019907	0.99761\\
0.00022068	0.99735\\
0.00024414	0.99707\\
0.00026955	0.99677\\
0.00029703	0.99644\\
0.00032672	0.99608\\
0.00035872	0.9957\\
0.00039319	0.99529\\
0.00043026	0.99484\\
0.00047007	0.99437\\
0.00051278	0.99386\\
0.00055853	0.99331\\
0.0006075	0.99272\\
0.00065984	0.9921\\
0.00071573	0.99143\\
0.00077534	0.99072\\
0.00083886	0.98996\\
0.00090648	0.98915\\
0.00097838	0.98829\\
0.0010548	0.98738\\
0.0011359	0.98641\\
0.0012219	0.98538\\
0.001313	0.9843\\
0.0014096	0.98315\\
0.0015117	0.98193\\
0.0016196	0.98065\\
0.0017336	0.97929\\
0.0018539	0.97786\\
0.0019809	0.97636\\
0.0021147	0.97477\\
0.0022556	0.9731\\
0.002404	0.97134\\
0.00256	0.96949\\
0.0027241	0.96755\\
0.0028964	0.96551\\
0.0030774	0.96337\\
0.0032673	0.96113\\
0.0034664	0.95879\\
0.0036752	0.95633\\
0.0038939	0.95376\\
0.0041229	0.95107\\
0.0043625	0.94826\\
0.0046132	0.94532\\
0.0048753	0.94226\\
0.0051491	0.93906\\
0.0054351	0.93572\\
0.0057336	0.93225\\
0.0060452	0.92863\\
0.0063701	0.92486\\
0.0067089	0.92093\\
0.0070619	0.91685\\
0.0074296	0.91261\\
0.0078125	0.9082\\
0.008211	0.90363\\
0.0086256	0.89887\\
0.0090568	0.89394\\
0.0095051	0.88883\\
0.0099709	0.88353\\
0.010455	0.87804\\
0.010957	0.87235\\
0.011479	0.86647\\
0.01202	0.86038\\
0.012582	0.85408\\
0.013165	0.84757\\
0.013768	0.84085\\
0.014394	0.8339\\
0.015042	0.82673\\
0.015714	0.81934\\
0.016409	0.81171\\
0.017129	0.80385\\
0.017873	0.79575\\
0.018643	0.7874\\
0.01944	0.77881\\
0.020264	0.76998\\
0.021115	0.76089\\
0.021995	0.75155\\
0.022903	0.74195\\
0.023842	0.73209\\
0.024811	0.72197\\
0.025811	0.71158\\
0.026844	0.70094\\
0.027909	0.69002\\
0.029007	0.67884\\
0.03014	0.66739\\
0.031308	0.65567\\
0.032512	0.64368\\
0.033753	0.63142\\
0.035032	0.61889\\
0.036348	0.6061\\
0.037704	0.59304\\
0.039101	0.57971\\
0.040538	0.56613\\
0.042017	0.55229\\
0.04354	0.53819\\
0.045106	0.52384\\
0.046716	0.50924\\
0.048373	0.4944\\
0.050076	0.47933\\
0.051827	0.46403\\
0.053626	0.44851\\
0.055475	0.43278\\
0.057375	0.41684\\
0.059326	0.40071\\
0.06133	0.3844\\
0.063388	0.36792\\
0.065501	0.35128\\
0.06767	0.3345\\
0.069895	0.31759\\
0.072179	0.30056\\
0.074523	0.28344\\
0.076926	0.26625\\
0.079392	0.249\\
0.08192	0.23171\\
0.084512	0.21441\\
0.08717	0.19712\\
0.089893	0.17987\\
0.092685	0.16268\\
0.095545	0.14558\\
0.098476	0.12861\\
0.10148	0.11179\\
0.10455	0.095167\\
0.1077	0.078767\\
0.11093	0.062633\\
0.11423	0.046804\\
0.11761	0.031322\\
0.12107	0.016233\\
0.12461	0.0015841\\
0.125	2.2204e-16\\
0.12823	0.012576\\
0.13193	0.026194\\
0.13572	0.039216\\
0.1396	0.051583\\
0.14357	0.063237\\
0.14762	0.074113\\
0.15177	0.084145\\
0.15601	0.093264\\
0.16034	0.1014\\
0.16477	0.10847\\
0.1693	0.1144\\
0.17392	0.1191\\
0.17865	0.12249\\
0.18348	0.12448\\
0.18841	0.12497\\
0.19345	0.12387\\
0.19859	0.12106\\
0.20384	0.11645\\
0.20921	0.10992\\
0.21468	0.10135\\
0.22027	0.090628\\
0.22598	0.077617\\
0.2318	0.062189\\
0.23775	0.044206\\
0.24381	0.023526\\
0.25	2.2204e-16\\
0.25	0.00016798\\
0.25	0.00033596\\
0.25	0.00050394\\
0.25	0.00067192\\
0.25	0.0008399\\
0.25	0.0010079\\
0.25	0.0011758\\
0.25	0.0013438\\
0.25	0.0015118\\
0.25	0.0016797\\
0.25	0.0018476\\
0.25	0.0020155\\
0.25	0.0021833\\
0.25	0.0023511\\
0.25	0.0025188\\
0.25	0.0026865\\
0.25	0.002854\\
0.25	0.0030215\\
0.25	0.0031888\\
0.25	0.003356\\
0.25	0.003523\\
0.25	0.0036898\\
0.25001	0.0038564\\
0.25001	0.0040227\\
0.25001	0.0041887\\
0.25001	0.0043545\\
0.25001	0.0045198\\
0.25001	0.0046848\\
0.25002	0.0048494\\
0.25002	0.0050134\\
0.25002	0.005177\\
0.25003	0.00534\\
0.25003	0.0055024\\
0.25004	0.0056641\\
0.25004	0.005825\\
0.25005	0.0059852\\
0.25005	0.0061446\\
0.25006	0.0063031\\
0.25007	0.0064606\\
0.25008	0.006617\\
0.25009	0.0067724\\
0.2501	0.0069266\\
0.25011	0.0070796\\
0.25013	0.0072313\\
0.25014	0.0073817\\
0.25016	0.0075307\\
0.25018	0.0076781\\
0.2502	0.0078239\\
0.25022	0.007968\\
0.25024	0.0081104\\
0.25027	0.0082508\\
0.2503	0.0083894\\
0.25033	0.0085259\\
0.25036	0.0086602\\
0.25039	0.0087924\\
0.25043	0.0089223\\
0.25047	0.0090497\\
0.25051	0.0091748\\
0.25056	0.0092974\\
0.25061	0.0094175\\
0.25066	0.0095348\\
0.25072	0.0096492\\
0.25078	0.0097608\\
0.25084	0.0098692\\
0.25091	0.0099746\\
0.25098	0.010077\\
0.25105	0.010176\\
0.25114	0.010272\\
0.25122	0.010364\\
0.25131	0.010453\\
0.25141	0.010538\\
0.25151	0.01062\\
0.25162	0.010697\\
0.25173	0.010772\\
0.25185	0.010842\\
0.25198	0.010908\\
0.25211	0.010971\\
0.25226	0.011029\\
0.2524	0.011082\\
0.25256	0.011132\\
0.25272	0.011178\\
0.2529	0.011219\\
0.25308	0.011256\\
0.25327	0.011288\\
0.25347	0.011317\\
0.25368	0.01134\\
0.25389	0.011359\\
0.25412	0.011373\\
0.25436	0.011384\\
0.25461	0.011389\\
0.25488	0.011389\\
0.25515	0.011385\\
0.25544	0.011377\\
0.25573	0.011363\\
0.25605	0.011346\\
0.25637	0.011324\\
0.25671	0.011296\\
0.25706	0.011264\\
0.25743	0.011226\\
0.25781	0.011184\\
0.25821	0.011136\\
0.25863	0.011083\\
0.25906	0.011025\\
0.25951	0.010964\\
0.25997	0.010897\\
0.26045	0.010825\\
0.26096	0.010749\\
0.26148	0.010667\\
0.26202	0.010581\\
0.26258	0.01049\\
0.26316	0.010393\\
0.26377	0.010294\\
0.26439	0.010189\\
0.26504	0.010081\\
0.26571	0.0099671\\
0.26641	0.0098491\\
0.26713	0.0097266\\
0.26787	0.0095998\\
0.26864	0.0094694\\
0.26944	0.0093356\\
0.27026	0.0091976\\
0.27111	0.0090557\\
0.27199	0.0089098\\
0.2729	0.0087601\\
0.27384	0.0086067\\
0.27481	0.0084507\\
0.27581	0.0082918\\
0.27684	0.0081295\\
0.27791	0.0079639\\
0.27901	0.0077953\\
0.28014	0.0076236\\
0.28131	0.0074496\\
0.28251	0.0072739\\
0.28375	0.0070957\\
0.28503	0.006915\\
0.28635	0.0067321\\
0.2877	0.0065471\\
0.2891	0.0063604\\
0.29054	0.006173\\
0.29202	0.005984\\
0.29354	0.0057935\\
0.29511	0.0056017\\
0.29672	0.0054087\\
0.29837	0.0052152\\
0.30008	0.0050218\\
0.30183	0.0048277\\
0.30363	0.0046332\\
0.30548	0.0044385\\
0.30737	0.0042437\\
0.30933	0.0040496\\
0.31133	0.0038563\\
0.31339	0.0036635\\
0.3155	0.0034715\\
0.31767	0.0032803\\
0.3199	0.0030904\\
0.32218	0.0029024\\
0.32452	0.002716\\
0.32693	0.0025314\\
0.32939	0.0023487\\
0.33192	0.0021683\\
0.33451	0.0019904\\
0.33717	0.0018155\\
0.33989	0.0016435\\
0.34268	0.0014746\\
0.34555	0.001309\\
0.34848	0.001147\\
0.35148	0.00098903\\
0.35455	0.00083524\\
0.3577	0.00068578\\
0.36093	0.00054091\\
0.36423	0.00040094\\
0.36761	0.00026613\\
0.37107	0.00013683\\
0.37461	1.3246e-05\\
0.375	2.2204e-16\\
0.37823	0.00010432\\
0.38193	0.00021556\\
0.38572	0.00032014\\
0.3896	0.00041772\\
0.39357	0.00050811\\
0.39762	0.00059086\\
0.40177	0.00066561\\
0.40601	0.00073199\\
0.41034	0.00078961\\
0.41477	0.00083807\\
0.4193	0.00087721\\
0.42392	0.0009064\\
0.42865	0.00092521\\
0.43348	0.00093317\\
0.43841	0.00092983\\
0.44345	0.00091472\\
0.44859	0.00088748\\
0.45384	0.00084755\\
0.45921	0.00079429\\
0.46468	0.00072716\\
0.47027	0.00064558\\
0.47598	0.00054898\\
0.4818	0.00043676\\
0.48775	0.00030837\\
0.49381	0.00016301\\
0.5	2.2204e-16\\
0.5	3.6423e-06\\
0.5	7.2814e-06\\
0.5	1.0922e-05\\
0.5	1.4563e-05\\
0.5	1.8204e-05\\
0.5	2.1844e-05\\
0.5	2.5485e-05\\
0.5	2.9126e-05\\
0.5	3.2767e-05\\
0.5	3.6407e-05\\
0.5	4.0049e-05\\
0.5	4.369e-05\\
0.5	4.7331e-05\\
0.5	5.0972e-05\\
0.5	5.4614e-05\\
0.5	5.8256e-05\\
0.5	6.1899e-05\\
0.5	6.5542e-05\\
0.5	6.9185e-05\\
0.5	7.2829e-05\\
0.5	7.6474e-05\\
0.5	8.0119e-05\\
0.50001	8.3765e-05\\
0.50001	8.7413e-05\\
0.50001	9.1061e-05\\
0.50001	9.471e-05\\
0.50001	9.8361e-05\\
0.50001	0.00010201\\
0.50002	0.00010567\\
0.50002	0.00010932\\
0.50002	0.00011298\\
0.50003	0.00011664\\
0.50003	0.0001203\\
0.50004	0.00012396\\
0.50004	0.00012763\\
0.50005	0.0001313\\
0.50005	0.00013497\\
0.50006	0.00013865\\
0.50007	0.00014232\\
0.50008	0.000146\\
0.50009	0.00014969\\
0.5001	0.00015337\\
0.50011	0.00015707\\
0.50013	0.00016076\\
0.50014	0.00016446\\
0.50016	0.00016816\\
0.50018	0.00017187\\
0.5002	0.00017558\\
0.50022	0.0001793\\
0.50024	0.00018302\\
0.50027	0.00018674\\
0.5003	0.00019047\\
0.50033	0.0001942\\
0.50036	0.00019794\\
0.50039	0.00020169\\
0.50043	0.00020544\\
0.50047	0.0002092\\
0.50051	0.00021296\\
0.50056	0.00021673\\
0.50061	0.0002205\\
0.50066	0.00022427\\
0.50072	0.00022806\\
0.50078	0.00023184\\
0.50084	0.00023563\\
0.50091	0.00023942\\
0.50098	0.00024322\\
0.50105	0.00024702\\
0.50114	0.00025082\\
0.50122	0.00025463\\
0.50131	0.00025845\\
0.50141	0.00026227\\
0.50151	0.00026609\\
0.50162	0.00026991\\
0.50173	0.00027373\\
0.50185	0.00027757\\
0.50198	0.0002814\\
0.50211	0.00028524\\
0.50226	0.00028908\\
0.5024	0.00029293\\
0.50256	0.00029677\\
0.50272	0.00030062\\
0.5029	0.00030447\\
0.50308	0.00030832\\
0.50327	0.00031217\\
0.50347	0.00031601\\
0.50368	0.00031986\\
0.50389	0.00032369\\
0.50412	0.00032752\\
0.50436	0.00033134\\
0.50461	0.00033515\\
0.50488	0.00033896\\
0.50515	0.00034274\\
0.50544	0.00034651\\
0.50573	0.00035026\\
0.50605	0.00035399\\
0.50637	0.0003577\\
0.50671	0.00036138\\
0.50706	0.00036503\\
0.50743	0.00036864\\
0.50781	0.00037222\\
0.50821	0.00037575\\
0.50863	0.00037924\\
0.50906	0.00038268\\
0.50951	0.00038607\\
0.50997	0.00038939\\
0.51045	0.00039265\\
0.51096	0.00039584\\
0.51148	0.00039895\\
0.51202	0.00040199\\
0.51258	0.00040493\\
0.51316	0.00040777\\
0.51377	0.00041052\\
0.51439	0.00041315\\
0.51504	0.00041567\\
0.51571	0.00041807\\
0.51641	0.00042032\\
0.51713	0.00042245\\
0.51787	0.00042455\\
0.51864	0.00042654\\
0.51944	0.00042837\\
0.52026	0.00043003\\
0.52111	0.0004315\\
0.52199	0.00043276\\
0.5229	0.00043381\\
0.52384	0.00043463\\
0.52481	0.00043522\\
0.52581	0.00043555\\
0.52684	0.00043562\\
0.52791	0.00043541\\
0.52901	0.00043492\\
0.53014	0.00043412\\
0.53131	0.00043302\\
0.53251	0.00043159\\
0.53375	0.00042982\\
0.53503	0.0004277\\
0.53635	0.00042521\\
0.5377	0.00042236\\
0.5391	0.00041911\\
0.54054	0.00041547\\
0.54202	0.00041141\\
0.54354	0.00040693\\
0.54511	0.00040201\\
0.54672	0.00039665\\
0.54837	0.00039083\\
0.55008	0.00038455\\
0.55183	0.00037779\\
0.55363	0.00037054\\
0.55548	0.00036281\\
0.55737	0.00035457\\
0.55933	0.00034583\\
0.56133	0.00033658\\
0.56339	0.00032682\\
0.5655	0.00031654\\
0.56767	0.00030575\\
0.5699	0.00029444\\
0.57218	0.00028263\\
0.57452	0.00027031\\
0.57693	0.00025749\\
0.57939	0.00024418\\
0.58192	0.0002304\\
0.58451	0.00021616\\
0.58717	0.00020147\\
0.58989	0.00018636\\
0.59268	0.00017085\\
0.59555	0.00015497\\
0.59848	0.00013875\\
0.60148	0.00012222\\
0.60455	0.00010543\\
0.6077	8.8412e-05\\
0.61093	7.1225e-05\\
0.61423	5.3917e-05\\
0.61761	3.6549e-05\\
0.62107	1.9185e-05\\
0.62461	1.896e-06\\
0.625	2.2204e-16\\
0.62823	1.5242e-05\\
0.63193	3.2146e-05\\
0.63572	4.8725e-05\\
0.6396	6.4882e-05\\
0.64357	8.0512e-05\\
0.64762	9.5503e-05\\
0.65177	0.00010973\\
0.65601	0.00012307\\
0.66034	0.00013538\\
0.66477	0.00014652\\
0.6693	0.00015631\\
0.67392	0.00016461\\
0.67865	0.00017121\\
0.68348	0.00017595\\
0.68841	0.00017861\\
0.69345	0.00017898\\
0.69859	0.00017683\\
0.70384	0.00017193\\
0.70921	0.00016402\\
0.71468	0.00015283\\
0.72027	0.00013808\\
0.72598	0.00011947\\
0.7318	9.6701e-05\\
0.73775	6.9431e-05\\
0.74381	3.7318e-05\\
0.75	2.2204e-16\\
0.75	3.4041e-07\\
0.75	6.7713e-07\\
0.75	1.0156e-06\\
0.75	1.3542e-06\\
0.75	1.6928e-06\\
0.75	2.0314e-06\\
0.75	2.37e-06\\
0.75	2.7086e-06\\
0.75	3.0473e-06\\
0.75	3.3861e-06\\
0.75	3.725e-06\\
0.75	4.064e-06\\
0.75	4.4032e-06\\
0.75	4.7426e-06\\
0.75	5.0823e-06\\
0.75	5.4224e-06\\
0.75	5.7628e-06\\
0.75	6.1038e-06\\
0.75	6.4453e-06\\
0.75	6.7875e-06\\
0.75	7.1304e-06\\
0.75	7.4742e-06\\
0.75001	7.819e-06\\
0.75001	8.1649e-06\\
0.75001	8.512e-06\\
0.75001	8.8605e-06\\
0.75001	9.2106e-06\\
0.75001	9.5624e-06\\
0.75002	9.9161e-06\\
0.75002	1.0272e-05\\
0.75002	1.063e-05\\
0.75003	1.099e-05\\
0.75003	1.1354e-05\\
0.75004	1.172e-05\\
0.75004	1.2089e-05\\
0.75005	1.2461e-05\\
0.75005	1.2837e-05\\
0.75006	1.3217e-05\\
0.75007	1.3601e-05\\
0.75008	1.399e-05\\
0.75009	1.4383e-05\\
0.7501	1.4781e-05\\
0.75011	1.5184e-05\\
0.75013	1.5593e-05\\
0.75014	1.6007e-05\\
0.75016	1.6428e-05\\
0.75018	1.6855e-05\\
0.7502	1.7288e-05\\
0.75022	1.7729e-05\\
0.75024	1.8177e-05\\
0.75027	1.8633e-05\\
0.7503	1.9096e-05\\
0.75033	1.9568e-05\\
0.75036	2.0049e-05\\
0.75039	2.0538e-05\\
0.75043	2.1037e-05\\
0.75047	2.1545e-05\\
0.75051	2.2063e-05\\
0.75056	2.2592e-05\\
0.75061	2.3131e-05\\
0.75066	2.368e-05\\
0.75072	2.4241e-05\\
0.75078	2.4813e-05\\
0.75084	2.5397e-05\\
0.75091	2.5992e-05\\
0.75098	2.66e-05\\
0.75105	2.7221e-05\\
0.75114	2.7853e-05\\
0.75122	2.8499e-05\\
0.75131	2.9158e-05\\
0.75141	2.983e-05\\
0.75151	3.0515e-05\\
0.75162	3.1214e-05\\
0.75173	3.1927e-05\\
0.75185	3.2653e-05\\
0.75198	3.3394e-05\\
0.75211	3.4148e-05\\
0.75226	3.4916e-05\\
0.7524	3.5698e-05\\
0.75256	3.6494e-05\\
0.75272	3.7304e-05\\
0.7529	3.8128e-05\\
0.75308	3.8965e-05\\
0.75327	3.9816e-05\\
0.75347	4.0681e-05\\
0.75368	4.1558e-05\\
0.75389	4.2448e-05\\
0.75412	4.3351e-05\\
0.75436	4.4266e-05\\
0.75461	4.5193e-05\\
0.75488	4.6131e-05\\
0.75515	4.708e-05\\
0.75544	4.804e-05\\
0.75573	4.9009e-05\\
0.75605	4.9987e-05\\
0.75637	5.0973e-05\\
0.75671	5.1967e-05\\
0.75706	5.2968e-05\\
0.75743	5.3975e-05\\
0.75781	5.4986e-05\\
0.75821	5.6002e-05\\
0.75863	5.702e-05\\
0.75906	5.804e-05\\
0.75951	5.9061e-05\\
0.75997	6.0081e-05\\
0.76045	6.1099e-05\\
0.76096	6.2113e-05\\
0.76148	6.3122e-05\\
0.76202	6.4124e-05\\
0.76258	6.5119e-05\\
0.76316	6.6103e-05\\
0.76377	6.7075e-05\\
0.76439	6.8034e-05\\
0.76504	6.8977e-05\\
0.76571	6.9902e-05\\
0.76641	7.0807e-05\\
0.76713	7.169e-05\\
0.76787	7.2549e-05\\
0.76864	7.3381e-05\\
0.76944	7.4183e-05\\
0.77026	7.4954e-05\\
0.77111	7.569e-05\\
0.77199	7.639e-05\\
0.7729	7.7049e-05\\
0.77384	7.7666e-05\\
0.77481	7.8236e-05\\
0.77581	7.8759e-05\\
0.77684	7.9229e-05\\
0.77791	7.9645e-05\\
0.77901	8.0003e-05\\
0.78014	8.0299e-05\\
0.78131	8.0532e-05\\
0.78251	8.0696e-05\\
0.78375	8.0789e-05\\
0.78503	8.0808e-05\\
0.78635	8.075e-05\\
0.7877	8.061e-05\\
0.7891	8.0386e-05\\
0.79054	8.0074e-05\\
0.79202	7.9671e-05\\
0.79354	7.9173e-05\\
0.79511	7.8578e-05\\
0.79672	7.7882e-05\\
0.79837	7.7083e-05\\
0.80008	7.6176e-05\\
0.80183	7.5161e-05\\
0.80363	7.4033e-05\\
0.80548	7.279e-05\\
0.80737	7.1431e-05\\
0.80933	6.9952e-05\\
0.81133	6.8352e-05\\
0.81339	6.663e-05\\
0.8155	6.4784e-05\\
0.81767	6.2814e-05\\
0.8199	6.0718e-05\\
0.82218	5.8497e-05\\
0.82452	5.615e-05\\
0.82693	5.3679e-05\\
0.82939	5.1085e-05\\
0.83192	4.837e-05\\
0.83451	4.5535e-05\\
0.83717	4.2585e-05\\
0.83989	3.9523e-05\\
0.84268	3.6353e-05\\
0.84555	3.3081e-05\\
0.84848	2.9714e-05\\
0.85148	2.6257e-05\\
0.85455	2.2721e-05\\
0.8577	1.9113e-05\\
0.86093	1.5444e-05\\
0.86423	1.1727e-05\\
0.86761	7.9733e-06\\
0.87107	4.1977e-06\\
0.87461	4.1606e-07\\
0.875	2.2204e-16\\
0.87823	3.3545e-06\\
0.88193	7.095e-06\\
0.88572	1.0785e-05\\
0.8896	1.4402e-05\\
0.89357	1.7921e-05\\
0.89762	2.1317e-05\\
0.90177	2.4561e-05\\
0.90601	2.7622e-05\\
0.91034	3.0468e-05\\
0.91477	3.3062e-05\\
0.9193	3.5367e-05\\
0.92392	3.7342e-05\\
0.92865	3.8944e-05\\
0.93348	4.0126e-05\\
0.93841	4.0838e-05\\
0.94345	4.1028e-05\\
0.94859	4.064e-05\\
0.95384	3.9614e-05\\
0.95921	3.7887e-05\\
0.96468	3.5391e-05\\
0.97027	3.2056e-05\\
0.97598	2.7806e-05\\
0.9818	2.2562e-05\\
0.98775	1.6239e-05\\
0.99381	8.7498e-06\\
1	2.2204e-16\\
};
\end{axis}
\end{tikzpicture}%

%% file: src/residualsColl4.tex
%
%
\definecolor{mycolor1}{rgb}{0.00000,0.44700,0.74100}%
\begin{tikzpicture}

\begin{axis}[%
width=0.3\textwidth,
height=0.15\textwidth,
scale only axis,
xmin=0,xmax=1,
xlabel style={font=\color{white!15!black}},
xlabel={$t$},
xtick={0,0.25,0.5,0.75,1},
xticklabels={0,$\frac{1}{4}$,$\frac{1}{2}$,$\frac{3}{4}$,1},
ymode=log,
ymin=1e-16,ymax=1,
yminorticks=true,
axis background/.style={fill=white}
]
\addplot [color=mycolor1, forget plot,thick]
  table[row sep=crcr]{%
0	1\\
2.4414e-14	1\\
7.8125e-13	1\\
5.9326e-12	1\\
2.5e-11	1\\
7.6294e-11	1\\
1.8984e-10	1\\
4.1033e-10	1\\
8e-10	1\\
1.4416e-09	1\\
2.4414e-09	1\\
3.9319e-09	1\\
6.075e-09	1\\
9.0648e-09	1\\
1.313e-08	1\\
1.8539e-08	1\\
2.56e-08	1\\
3.4664e-08	1\\
4.6132e-08	1\\
6.0452e-08	1\\
7.8125e-08	1\\
9.9709e-08	1\\
1.2582e-07	1\\
1.5714e-07	0.99999\\
1.944e-07	0.99999\\
2.3842e-07	0.99999\\
2.9007e-07	0.99999\\
3.5032e-07	0.99999\\
4.2017e-07	0.99999\\
5.0076e-07	0.99998\\
5.9326e-07	0.99998\\
6.9895e-07	0.99998\\
8.192e-07	0.99997\\
9.5545e-07	0.99997\\
1.1093e-06	0.99996\\
1.2823e-06	0.99996\\
1.4762e-06	0.99995\\
1.693e-06	0.99994\\
1.9345e-06	0.99994\\
2.2027e-06	0.99993\\
2.5e-06	0.99992\\
2.8285e-06	0.99991\\
3.1907e-06	0.99989\\
3.5891e-06	0.99988\\
4.0263e-06	0.99987\\
4.5051e-06	0.99985\\
5.0284e-06	0.99983\\
5.5992e-06	0.99981\\
6.2208e-06	0.99979\\
6.8964e-06	0.99977\\
7.6294e-06	0.99975\\
8.4235e-06	0.99972\\
9.2823e-06	0.99969\\
1.021e-05	0.99966\\
1.121e-05	0.99963\\
1.2287e-05	0.99959\\
1.3446e-05	0.99955\\
1.469e-05	0.99951\\
1.6024e-05	0.99947\\
1.7454e-05	0.99942\\
1.8984e-05	0.99937\\
2.062e-05	0.99931\\
2.2367e-05	0.99925\\
2.4229e-05	0.99919\\
2.6214e-05	0.99913\\
2.8327e-05	0.99906\\
3.0575e-05	0.99898\\
3.2962e-05	0.9989\\
3.5496e-05	0.99882\\
3.8184e-05	0.99873\\
4.1033e-05	0.99863\\
4.4049e-05	0.99853\\
4.7239e-05	0.99843\\
5.0612e-05	0.99831\\
5.4175e-05	0.9982\\
5.7936e-05	0.99807\\
6.1902e-05	0.99794\\
6.6084e-05	0.9978\\
7.0488e-05	0.99765\\
7.5123e-05	0.9975\\
8e-05	0.99734\\
8.5127e-05	0.99717\\
9.0513e-05	0.99699\\
9.6168e-05	0.9968\\
0.0001021	0.9966\\
0.00010833	0.99639\\
0.00011485	0.99618\\
0.00012168	0.99595\\
0.00012884	0.99571\\
0.00013633	0.99546\\
0.00014416	0.9952\\
0.00015235	0.99493\\
0.00016091	0.99465\\
0.00016985	0.99435\\
0.00017918	0.99404\\
0.00018891	0.99372\\
0.00019907	0.99338\\
0.00020965	0.99303\\
0.00022068	0.99266\\
0.00023218	0.99228\\
0.00024414	0.99188\\
0.00025659	0.99147\\
0.00026955	0.99104\\
0.00028303	0.9906\\
0.00029703	0.99013\\
0.00031159	0.98965\\
0.00032672	0.98915\\
0.00034242	0.98863\\
0.00035872	0.98809\\
0.00037564	0.98753\\
0.00039319	0.98695\\
0.00041139	0.98635\\
0.00043026	0.98573\\
0.00044981	0.98508\\
0.00047007	0.98441\\
0.00049105	0.98372\\
0.00051278	0.98301\\
0.00053527	0.98226\\
0.00055853	0.9815\\
0.00058261	0.98071\\
0.0006075	0.97989\\
0.00063324	0.97904\\
0.00065984	0.97817\\
0.00068733	0.97726\\
0.00071573	0.97633\\
0.00074506	0.97537\\
0.00077534	0.97438\\
0.0008066	0.97336\\
0.00083886	0.9723\\
0.00087214	0.97121\\
0.00090648	0.97009\\
0.00094188	0.96893\\
0.00097838	0.96774\\
0.001016	0.96652\\
0.0010548	0.96525\\
0.0010947	0.96395\\
0.0011359	0.96262\\
0.0011783	0.96124\\
0.0012219	0.95982\\
0.0012668	0.95837\\
0.001313	0.95687\\
0.0013606	0.95533\\
0.0014096	0.95375\\
0.0014599	0.95213\\
0.0015117	0.95046\\
0.0015649	0.94875\\
0.0016196	0.94699\\
0.0016758	0.94518\\
0.0017336	0.94333\\
0.001793	0.94142\\
0.0018539	0.93947\\
0.0019166	0.93747\\
0.0019809	0.93542\\
0.0020469	0.93332\\
0.0021147	0.93116\\
0.0021842	0.92896\\
0.0022556	0.92669\\
0.0023288	0.92438\\
0.002404	0.922\\
0.002481	0.91957\\
0.00256	0.91708\\
0.002641	0.91454\\
0.0027241	0.91193\\
0.0028092	0.90927\\
0.0028964	0.90654\\
0.0029858	0.90376\\
0.0030774	0.90091\\
0.0031712	0.89799\\
0.0032673	0.89502\\
0.0033657	0.89197\\
0.0034664	0.88887\\
0.0035696	0.88569\\
0.0036752	0.88245\\
0.0037833	0.87914\\
0.0038939	0.87576\\
0.0040071	0.87232\\
0.0041229	0.8688\\
0.0042414	0.86521\\
0.0043625	0.86155\\
0.0044865	0.85781\\
0.0046132	0.85401\\
0.0047428	0.85012\\
0.0048753	0.84617\\
0.0050107	0.84214\\
0.0051491	0.83803\\
0.0052905	0.83385\\
0.0054351	0.82959\\
0.0055827	0.82525\\
0.0057336	0.82083\\
0.0058877	0.81634\\
0.0060452	0.81176\\
0.0062059	0.80711\\
0.0063701	0.80238\\
0.0065377	0.79756\\
0.0067089	0.79266\\
0.0068836	0.78769\\
0.0070619	0.78263\\
0.0072439	0.77749\\
0.0074296	0.77226\\
0.0076191	0.76696\\
0.0078125	0.76157\\
0.0080098	0.75609\\
0.008211	0.75054\\
0.0084163	0.7449\\
0.0086256	0.73918\\
0.0088391	0.73337\\
0.0090568	0.72748\\
0.0092788	0.72151\\
0.0095051	0.71545\\
0.0097358	0.70931\\
0.0099709	0.70309\\
0.010211	0.69678\\
0.010455	0.69039\\
0.010704	0.68392\\
0.010957	0.67737\\
0.011216	0.67074\\
0.011479	0.66402\\
0.011747	0.65723\\
0.01202	0.65035\\
0.012299	0.6434\\
0.012582	0.63637\\
0.012871	0.62926\\
0.013165	0.62207\\
0.013464	0.61481\\
0.013768	0.60747\\
0.014078	0.60006\\
0.014394	0.59258\\
0.014715	0.58502\\
0.015042	0.5774\\
0.015375	0.5697\\
0.015714	0.56194\\
0.016058	0.55411\\
0.016409	0.54621\\
0.016766	0.53826\\
0.017129	0.53024\\
0.017498	0.52216\\
0.017873	0.51402\\
0.018255	0.50583\\
0.018643	0.49759\\
0.019038	0.48929\\
0.01944	0.48094\\
0.019848	0.47254\\
0.020264	0.4641\\
0.020686	0.45562\\
0.021115	0.44709\\
0.021551	0.43853\\
0.021995	0.42993\\
0.022445	0.4213\\
0.022903	0.41264\\
0.023369	0.40395\\
0.023842	0.39524\\
0.024323	0.38651\\
0.024811	0.37776\\
0.025307	0.36899\\
0.025811	0.36021\\
0.026323	0.35143\\
0.026844	0.34264\\
0.027372	0.33384\\
0.027909	0.32506\\
0.028454	0.31627\\
0.029007	0.3075\\
0.029569	0.29874\\
0.03014	0.29\\
0.03072	0.28128\\
0.031308	0.27258\\
0.031906	0.26392\\
0.032512	0.25529\\
0.033128	0.24669\\
0.033753	0.23814\\
0.034388	0.22964\\
0.035032	0.22118\\
0.035685	0.21278\\
0.036348	0.20444\\
0.037021	0.19617\\
0.037704	0.18796\\
0.038398	0.17983\\
0.039101	0.17178\\
0.039814	0.16381\\
0.040538	0.15592\\
0.041273	0.14813\\
0.042017	0.14044\\
0.042773	0.13284\\
0.04354	0.12535\\
0.044317	0.11798\\
0.045106	0.11072\\
0.045905	0.10358\\
0.046716	0.09656\\
0.047539	0.089673\\
0.048373	0.082921\\
0.049219	0.076307\\
0.050076	0.069836\\
0.050945	0.063512\\
0.051827	0.05734\\
0.05272	0.051323\\
0.053626	0.045467\\
0.054544	0.039775\\
0.055475	0.034251\\
0.056419	0.028899\\
0.057375	0.023722\\
0.058344	0.018725\\
0.059326	0.01391\\
0.060322	0.0092812\\
0.06133	0.0048412\\
0.062352	0.00059295\\
0.0625	2.2204e-16\\
0.063388	0.0034607\\
0.064438	0.0073172\\
0.065501	0.010974\\
0.066578	0.01443\\
0.06767	0.017682\\
0.068775	0.020729\\
0.069895	0.02357\\
0.07103	0.026204\\
0.072179	0.028629\\
0.073344	0.030846\\
0.074523	0.032853\\
0.075717	0.034652\\
0.076926	0.036242\\
0.078151	0.037625\\
0.079392	0.038801\\
0.080648	0.039773\\
0.08192	0.040541\\
0.083208	0.041108\\
0.084512	0.041477\\
0.085833	0.041651\\
0.08717	0.041633\\
0.088523	0.041428\\
0.089893	0.041039\\
0.091281	0.040471\\
0.092685	0.039729\\
0.094106	0.03882\\
0.095545	0.037749\\
0.097002	0.036523\\
0.098476	0.035149\\
0.099968	0.033634\\
0.10148	0.031986\\
0.10301	0.030214\\
0.10455	0.028327\\
0.10612	0.026334\\
0.1077	0.024245\\
0.1093	0.022069\\
0.11093	0.019819\\
0.11257	0.017505\\
0.11423	0.015138\\
0.11591	0.01273\\
0.11761	0.010295\\
0.11933	0.0078434\\
0.12107	0.0053896\\
0.12283	0.0029466\\
0.12461	0.000528\\
0.125	2.2204e-16\\
0.12641	0.0018525\\
0.12823	0.0041807\\
0.13007	0.0064426\\
0.13193	0.0086238\\
0.13382	0.01071\\
0.13572	0.012687\\
0.13765	0.01454\\
0.1396	0.016256\\
0.14157	0.01782\\
0.14357	0.019219\\
0.14558	0.020439\\
0.14762	0.021468\\
0.14968	0.022293\\
0.15177	0.022903\\
0.15388	0.023287\\
0.15601	0.023436\\
0.15816	0.02334\\
0.16034	0.022992\\
0.16254	0.022385\\
0.16477	0.021516\\
0.16702	0.02038\\
0.1693	0.018977\\
0.1716	0.017308\\
0.17392	0.015376\\
0.17627	0.013186\\
0.17865	0.010747\\
0.18105	0.0080704\\
0.18348	0.0051711\\
0.18593	0.0020672\\
0.1875	2.2204e-16\\
0.18841	0.001219\\
0.19091	0.0046613\\
0.19345	0.0082289\\
0.196	0.011886\\
0.19859	0.015591\\
0.2012	0.019299\\
0.20384	0.022955\\
0.20651	0.026501\\
0.20921	0.029871\\
0.21193	0.032991\\
0.21468	0.035779\\
0.21746	0.038145\\
0.22027	0.039989\\
0.22311	0.041203\\
0.22598	0.041666\\
0.22888	0.041247\\
0.2318	0.039805\\
0.23476	0.037185\\
0.23775	0.033217\\
0.24076	0.027721\\
0.24381	0.020497\\
0.24689	0.011334\\
0.25	2.2204e-16\\
0.25	3.151e-05\\
0.25	6.3039e-05\\
0.25	9.4559e-05\\
0.25	0.00012608\\
0.25	0.0001576\\
0.25	0.00018912\\
0.25	0.00022064\\
0.25	0.00025216\\
0.25	0.00028368\\
0.25	0.0003152\\
0.25	0.00034671\\
0.25	0.00037823\\
0.25	0.00040975\\
0.25	0.00044126\\
0.25	0.00047278\\
0.25	0.00050429\\
0.25	0.0005358\\
0.25	0.00056731\\
0.25	0.00059882\\
0.25	0.00063032\\
0.25	0.00066182\\
0.25	0.00069332\\
0.25	0.00072481\\
0.25	0.0007563\\
0.25	0.00078778\\
0.25	0.00081925\\
0.25	0.00085072\\
0.25	0.00088218\\
0.25	0.00091363\\
0.25	0.00094507\\
0.25	0.0009765\\
0.25	0.0010079\\
0.25	0.0010393\\
0.25	0.0010707\\
0.25	0.0011021\\
0.25	0.0011334\\
0.25	0.0011648\\
0.25	0.0011961\\
0.25	0.0012274\\
0.25	0.0012587\\
0.25	0.0012899\\
0.25	0.0013211\\
0.25	0.0013523\\
0.25	0.0013835\\
0.25	0.0014146\\
0.25001	0.0014457\\
0.25001	0.0014768\\
0.25001	0.0015078\\
0.25001	0.0015388\\
0.25001	0.0015698\\
0.25001	0.0016007\\
0.25001	0.0016315\\
0.25001	0.0016623\\
0.25001	0.0016931\\
0.25001	0.0017238\\
0.25001	0.0017544\\
0.25001	0.001785\\
0.25002	0.0018155\\
0.25002	0.0018459\\
0.25002	0.0018763\\
0.25002	0.0019066\\
0.25002	0.0019368\\
0.25002	0.001967\\
0.25003	0.0019971\\
0.25003	0.002027\\
0.25003	0.0020569\\
0.25003	0.0020867\\
0.25004	0.0021164\\
0.25004	0.002146\\
0.25004	0.0021755\\
0.25004	0.0022049\\
0.25005	0.0022342\\
0.25005	0.0022634\\
0.25005	0.0022924\\
0.25006	0.0023214\\
0.25006	0.0023502\\
0.25007	0.0023788\\
0.25007	0.0024074\\
0.25008	0.0024358\\
0.25008	0.002464\\
0.25009	0.0024921\\
0.25009	0.00252\\
0.2501	0.0025478\\
0.2501	0.0025755\\
0.25011	0.0026029\\
0.25011	0.0026302\\
0.25012	0.0026573\\
0.25013	0.0026843\\
0.25014	0.002711\\
0.25014	0.0027376\\
0.25015	0.0027639\\
0.25016	0.0027901\\
0.25017	0.0028161\\
0.25018	0.0028418\\
0.25019	0.0028673\\
0.2502	0.0028927\\
0.25021	0.0029178\\
0.25022	0.0029426\\
0.25023	0.0029672\\
0.25024	0.0029916\\
0.25026	0.0030158\\
0.25027	0.0030397\\
0.25028	0.0030633\\
0.2503	0.0030867\\
0.25031	0.0031098\\
0.25033	0.0031326\\
0.25034	0.0031552\\
0.25036	0.0031774\\
0.25038	0.0031994\\
0.25039	0.0032211\\
0.25041	0.0032425\\
0.25043	0.0032636\\
0.25045	0.0032844\\
0.25047	0.0033049\\
0.25049	0.003325\\
0.25051	0.0033449\\
0.25054	0.0033643\\
0.25056	0.0033835\\
0.25058	0.0034023\\
0.25061	0.0034208\\
0.25063	0.0034389\\
0.25066	0.0034567\\
0.25069	0.0034741\\
0.25072	0.0034911\\
0.25075	0.0035078\\
0.25078	0.0035241\\
0.25081	0.00354\\
0.25084	0.0035555\\
0.25087	0.0035706\\
0.25091	0.0035854\\
0.25094	0.0035997\\
0.25098	0.0036136\\
0.25102	0.0036271\\
0.25105	0.0036402\\
0.25109	0.0036528\\
0.25114	0.0036651\\
0.25118	0.0036769\\
0.25122	0.0036882\\
0.25127	0.0036992\\
0.25131	0.0037096\\
0.25136	0.0037197\\
0.25141	0.0037293\\
0.25146	0.0037385\\
0.25151	0.0037472\\
0.25156	0.0037555\\
0.25162	0.0037634\\
0.25168	0.0037707\\
0.25173	0.0037776\\
0.25179	0.0037839\\
0.25185	0.0037898\\
0.25192	0.0037952\\
0.25198	0.0038001\\
0.25205	0.0038045\\
0.25211	0.0038085\\
0.25218	0.0038118\\
0.25226	0.0038147\\
0.25233	0.0038171\\
0.2524	0.003819\\
0.25248	0.0038203\\
0.25256	0.0038211\\
0.25264	0.0038214\\
0.25272	0.0038211\\
0.25281	0.0038204\\
0.2529	0.0038191\\
0.25299	0.0038172\\
0.25308	0.0038148\\
0.25317	0.0038119\\
0.25327	0.0038084\\
0.25337	0.0038044\\
0.25347	0.0037998\\
0.25357	0.0037946\\
0.25368	0.003789\\
0.25378	0.0037827\\
0.25389	0.0037759\\
0.25401	0.0037686\\
0.25412	0.0037606\\
0.25424	0.0037522\\
0.25436	0.0037431\\
0.25449	0.0037335\\
0.25461	0.0037234\\
0.25474	0.0037126\\
0.25488	0.0037013\\
0.25501	0.0036896\\
0.25515	0.0036773\\
0.25529	0.0036645\\
0.25544	0.0036511\\
0.25558	0.0036371\\
0.25573	0.0036226\\
0.25589	0.0036074\\
0.25605	0.0035918\\
0.25621	0.0035755\\
0.25637	0.0035588\\
0.25654	0.0035414\\
0.25671	0.0035235\\
0.25688	0.003505\\
0.25706	0.003486\\
0.25724	0.0034664\\
0.25743	0.0034463\\
0.25762	0.0034256\\
0.25781	0.0034043\\
0.25801	0.0033826\\
0.25821	0.0033603\\
0.25842	0.0033374\\
0.25863	0.003314\\
0.25884	0.0032901\\
0.25906	0.0032656\\
0.25928	0.0032407\\
0.25951	0.0032152\\
0.25974	0.0031892\\
0.25997	0.0031626\\
0.26021	0.0031356\\
0.26045	0.0031081\\
0.2607	0.0030801\\
0.26096	0.0030516\\
0.26122	0.0030226\\
0.26148	0.0029933\\
0.26175	0.0029635\\
0.26202	0.0029332\\
0.2623	0.0029025\\
0.26258	0.0028713\\
0.26287	0.0028397\\
0.26316	0.0028077\\
0.26346	0.0027752\\
0.26377	0.0027423\\
0.26408	0.002709\\
0.26439	0.0026753\\
0.26472	0.0026412\\
0.26504	0.0026067\\
0.26538	0.0025718\\
0.26571	0.0025366\\
0.26606	0.002501\\
0.26641	0.0024651\\
0.26677	0.0024289\\
0.26713	0.0023923\\
0.2675	0.0023554\\
0.26787	0.0023182\\
0.26825	0.0022808\\
0.26864	0.002243\\
0.26904	0.002205\\
0.26944	0.0021668\\
0.26985	0.0021283\\
0.27026	0.0020896\\
0.27069	0.0020507\\
0.27111	0.0020116\\
0.27155	0.0019723\\
0.27199	0.0019329\\
0.27245	0.0018934\\
0.2729	0.0018537\\
0.27337	0.001814\\
0.27384	0.0017741\\
0.27432	0.0017342\\
0.27481	0.0016942\\
0.27531	0.0016541\\
0.27581	0.001614\\
0.27632	0.0015738\\
0.27684	0.0015337\\
0.27737	0.0014936\\
0.27791	0.0014535\\
0.27845	0.0014135\\
0.27901	0.0013735\\
0.27957	0.0013336\\
0.28014	0.0012938\\
0.28072	0.0012542\\
0.28131	0.0012147\\
0.28191	0.0011754\\
0.28251	0.0011362\\
0.28313	0.0010973\\
0.28375	0.0010586\\
0.28439	0.0010201\\
0.28503	0.00098189\\
0.28569	0.00094397\\
0.28635	0.00090637\\
0.28702	0.00086909\\
0.2877	0.00083215\\
0.2884	0.00079559\\
0.2891	0.00075942\\
0.28981	0.00072366\\
0.29054	0.00068837\\
0.29127	0.00065352\\
0.29202	0.00061914\\
0.29277	0.00058526\\
0.29354	0.00055188\\
0.29432	0.00051904\\
0.29511	0.00048675\\
0.29591	0.00045502\\
0.29672	0.00042389\\
0.29754	0.00039337\\
0.29837	0.00036348\\
0.29922	0.00033424\\
0.30008	0.00030567\\
0.30095	0.00027778\\
0.30183	0.0002506\\
0.30272	0.00022414\\
0.30363	0.00019841\\
0.30454	0.00017344\\
0.30548	0.00014924\\
0.30642	0.00012582\\
0.30737	0.0001032\\
0.30834	8.1397e-05\\
0.30933	6.042e-05\\
0.31032	4.0282e-05\\
0.31133	2.0995e-05\\
0.31235	2.5694e-06\\
0.3125	2.2204e-16\\
0.31339	1.4985e-05\\
0.31444	3.166e-05\\
0.3155	4.7451e-05\\
0.31658	6.2348e-05\\
0.31767	7.6346e-05\\
0.31878	8.944e-05\\
0.3199	0.00010162\\
0.32103	0.0001129\\
0.32218	0.00012326\\
0.32334	0.00013271\\
0.32452	0.00014124\\
0.32572	0.00014887\\
0.32693	0.00015559\\
0.32815	0.00016141\\
0.32939	0.00016634\\
0.33065	0.00017038\\
0.33192	0.00017355\\
0.33321	0.00017585\\
0.33451	0.0001773\\
0.33583	0.00017792\\
0.33717	0.00017771\\
0.33852	0.00017671\\
0.33989	0.00017492\\
0.34128	0.00017238\\
0.34268	0.0001691\\
0.34411	0.00016511\\
0.34555	0.00016044\\
0.347	0.00015512\\
0.34848	0.00014918\\
0.34997	0.00014264\\
0.35148	0.00013556\\
0.35301	0.00012796\\
0.35455	0.00011988\\
0.35612	0.00011137\\
0.3577	0.00010246\\
0.3593	9.3202e-05\\
0.36093	8.3639e-05\\
0.36257	7.382e-05\\
0.36423	6.3794e-05\\
0.36591	5.3611e-05\\
0.36761	4.3323e-05\\
0.36933	3.2985e-05\\
0.37107	2.265e-05\\
0.37283	1.2375e-05\\
0.37461	2.2159e-06\\
0.375	2.2204e-16\\
0.37641	7.7688e-06\\
0.37823	1.7521e-05\\
0.38007	2.6982e-05\\
0.38193	3.6092e-05\\
0.38382	4.4793e-05\\
0.38572	5.3026e-05\\
0.38765	6.0731e-05\\
0.3896	6.7851e-05\\
0.39157	7.4329e-05\\
0.39357	8.011e-05\\
0.39558	8.5139e-05\\
0.39762	8.9365e-05\\
0.39968	9.274e-05\\
0.40177	9.5215e-05\\
0.40388	9.6749e-05\\
0.40601	9.7303e-05\\
0.40816	9.6841e-05\\
0.41034	9.5336e-05\\
0.41254	9.2761e-05\\
0.41477	8.9101e-05\\
0.41702	8.4344e-05\\
0.4193	7.8489e-05\\
0.4216	7.154e-05\\
0.42392	6.3512e-05\\
0.42627	5.4432e-05\\
0.42865	4.4336e-05\\
0.43105	3.3274e-05\\
0.43348	2.1307e-05\\
0.43593	8.5125e-06\\
0.4375	2.2204e-16\\
0.43841	5.0168e-06\\
0.44091	1.9171e-05\\
0.44345	3.3824e-05\\
0.446	4.8826e-05\\
0.44859	6.4009e-05\\
0.4512	7.9182e-05\\
0.45384	9.4129e-05\\
0.45651	0.00010861\\
0.45921	0.00012235\\
0.46193	0.00013505\\
0.46468	0.00014637\\
0.46746	0.00015596\\
0.47027	0.00016341\\
0.47311	0.00016827\\
0.47598	0.00017006\\
0.47888	0.00016826\\
0.4818	0.00016229\\
0.48476	0.00015152\\
0.48775	0.00013528\\
0.49076	0.00011283\\
0.49381	8.3385e-05\\
0.49689	4.6082e-05\\
0.5	2.2204e-16\\
0.5	2.7272e-07\\
0.5	5.2121e-07\\
0.5	7.7987e-07\\
0.5	1.0404e-06\\
0.5	1.3005e-06\\
0.5	1.5605e-06\\
0.5	1.8207e-06\\
0.5	2.0808e-06\\
0.5	2.3409e-06\\
0.5	2.601e-06\\
0.5	2.8611e-06\\
0.5	3.1211e-06\\
0.5	3.3812e-06\\
0.5	3.6413e-06\\
0.5	3.9014e-06\\
0.5	4.1615e-06\\
0.5	4.4215e-06\\
0.5	4.6816e-06\\
0.5	4.9416e-06\\
0.5	5.2016e-06\\
0.5	5.4616e-06\\
0.5	5.7216e-06\\
0.5	5.9816e-06\\
0.5	6.2415e-06\\
0.5	6.5014e-06\\
0.5	6.7613e-06\\
0.5	7.0212e-06\\
0.5	7.281e-06\\
0.5	7.5408e-06\\
0.5	7.8005e-06\\
0.5	8.0602e-06\\
0.5	8.3198e-06\\
0.5	8.5793e-06\\
0.5	8.8388e-06\\
0.5	9.0982e-06\\
0.5	9.3576e-06\\
0.5	9.6168e-06\\
0.5	9.8759e-06\\
0.5	1.0135e-05\\
0.5	1.0394e-05\\
0.5	1.0653e-05\\
0.5	1.0911e-05\\
0.5	1.117e-05\\
0.5	1.1428e-05\\
0.5	1.1687e-05\\
0.50001	1.1945e-05\\
0.50001	1.2203e-05\\
0.50001	1.246e-05\\
0.50001	1.2718e-05\\
0.50001	1.2975e-05\\
0.50001	1.3232e-05\\
0.50001	1.3489e-05\\
0.50001	1.3746e-05\\
0.50001	1.4002e-05\\
0.50001	1.4258e-05\\
0.50001	1.4514e-05\\
0.50001	1.477e-05\\
0.50002	1.5025e-05\\
0.50002	1.528e-05\\
0.50002	1.5534e-05\\
0.50002	1.5788e-05\\
0.50002	1.6042e-05\\
0.50002	1.6295e-05\\
0.50003	1.6548e-05\\
0.50003	1.6801e-05\\
0.50003	1.7053e-05\\
0.50003	1.7304e-05\\
0.50004	1.7555e-05\\
0.50004	1.7806e-05\\
0.50004	1.8056e-05\\
0.50004	1.8305e-05\\
0.50005	1.8554e-05\\
0.50005	1.8803e-05\\
0.50005	1.905e-05\\
0.50006	1.9297e-05\\
0.50006	1.9544e-05\\
0.50007	1.9789e-05\\
0.50007	2.0034e-05\\
0.50008	2.0279e-05\\
0.50008	2.0522e-05\\
0.50009	2.0765e-05\\
0.50009	2.1007e-05\\
0.5001	2.1248e-05\\
0.5001	2.1488e-05\\
0.50011	2.1727e-05\\
0.50011	2.1965e-05\\
0.50012	2.2203e-05\\
0.50013	2.2439e-05\\
0.50014	2.2674e-05\\
0.50014	2.2909e-05\\
0.50015	2.3142e-05\\
0.50016	2.3374e-05\\
0.50017	2.3605e-05\\
0.50018	2.3835e-05\\
0.50019	2.4064e-05\\
0.5002	2.4291e-05\\
0.50021	2.4517e-05\\
0.50022	2.4742e-05\\
0.50023	2.4966e-05\\
0.50024	2.5188e-05\\
0.50026	2.5409e-05\\
0.50027	2.5628e-05\\
0.50028	2.5846e-05\\
0.5003	2.6062e-05\\
0.50031	2.6277e-05\\
0.50033	2.649e-05\\
0.50034	2.6702e-05\\
0.50036	2.6912e-05\\
0.50038	2.712e-05\\
0.50039	2.7327e-05\\
0.50041	2.7532e-05\\
0.50043	2.7735e-05\\
0.50045	2.7936e-05\\
0.50047	2.8135e-05\\
0.50049	2.8333e-05\\
0.50051	2.8528e-05\\
0.50054	2.8721e-05\\
0.50056	2.8913e-05\\
0.50058	2.9102e-05\\
0.50061	2.9289e-05\\
0.50063	2.9474e-05\\
0.50066	2.9657e-05\\
0.50069	2.9837e-05\\
0.50072	3.0015e-05\\
0.50075	3.0191e-05\\
0.50078	3.0365e-05\\
0.50081	3.0536e-05\\
0.50084	3.0704e-05\\
0.50087	3.087e-05\\
0.50091	3.1034e-05\\
0.50094	3.1194e-05\\
0.50098	3.1353e-05\\
0.50102	3.1508e-05\\
0.50105	3.1661e-05\\
0.50109	3.1811e-05\\
0.50114	3.1958e-05\\
0.50118	3.2102e-05\\
0.50122	3.2243e-05\\
0.50127	3.2381e-05\\
0.50131	3.2517e-05\\
0.50136	3.2649e-05\\
0.50141	3.2778e-05\\
0.50146	3.2904e-05\\
0.50151	3.3026e-05\\
0.50156	3.3146e-05\\
0.50162	3.3262e-05\\
0.50168	3.3374e-05\\
0.50173	3.3483e-05\\
0.50179	3.3589e-05\\
0.50185	3.3691e-05\\
0.50192	3.379e-05\\
0.50198	3.3885e-05\\
0.50205	3.3977e-05\\
0.50211	3.4064e-05\\
0.50218	3.4148e-05\\
0.50226	3.4228e-05\\
0.50233	3.4305e-05\\
0.5024	3.4377e-05\\
0.50248	3.4445e-05\\
0.50256	3.451e-05\\
0.50264	3.4571e-05\\
0.50272	3.4628e-05\\
0.50281	3.4681e-05\\
0.5029	3.4729e-05\\
0.50299	3.4774e-05\\
0.50308	3.4814e-05\\
0.50317	3.485e-05\\
0.50327	3.4881e-05\\
0.50337	3.4908e-05\\
0.50347	3.4931e-05\\
0.50357	3.4949e-05\\
0.50368	3.4962e-05\\
0.50378	3.4971e-05\\
0.50389	3.4975e-05\\
0.50401	3.4975e-05\\
0.50412	3.497e-05\\
0.50424	3.496e-05\\
0.50436	3.4945e-05\\
0.50449	3.4925e-05\\
0.50461	3.49e-05\\
0.50474	3.487e-05\\
0.50488	3.4836e-05\\
0.50501	3.4796e-05\\
0.50515	3.4751e-05\\
0.50529	3.4701e-05\\
0.50544	3.4646e-05\\
0.50558	3.4585e-05\\
0.50573	3.452e-05\\
0.50589	3.4449e-05\\
0.50605	3.4373e-05\\
0.50621	3.4291e-05\\
0.50637	3.4204e-05\\
0.50654	3.4112e-05\\
0.50671	3.4014e-05\\
0.50688	3.391e-05\\
0.50706	3.3802e-05\\
0.50724	3.3687e-05\\
0.50743	3.3567e-05\\
0.50762	3.3442e-05\\
0.50781	3.331e-05\\
0.50801	3.3174e-05\\
0.50821	3.3031e-05\\
0.50842	3.2883e-05\\
0.50863	3.2729e-05\\
0.50884	3.257e-05\\
0.50906	3.2405e-05\\
0.50928	3.2234e-05\\
0.50951	3.2057e-05\\
0.50974	3.1875e-05\\
0.50997	3.1687e-05\\
0.51021	3.1493e-05\\
0.51045	3.1293e-05\\
0.5107	3.1088e-05\\
0.51096	3.0877e-05\\
0.51122	3.066e-05\\
0.51148	3.0437e-05\\
0.51175	3.0209e-05\\
0.51202	2.9976e-05\\
0.5123	2.9738e-05\\
0.51258	2.9494e-05\\
0.51287	2.9244e-05\\
0.51316	2.8989e-05\\
0.51346	2.8728e-05\\
0.51377	2.8462e-05\\
0.51408	2.819e-05\\
0.51439	2.7912e-05\\
0.51472	2.7629e-05\\
0.51504	2.7341e-05\\
0.51538	2.7047e-05\\
0.51571	2.6749e-05\\
0.51606	2.6444e-05\\
0.51641	2.6135e-05\\
0.51677	2.5821e-05\\
0.51713	2.5501e-05\\
0.5175	2.5177e-05\\
0.51787	2.4848e-05\\
0.51825	2.4514e-05\\
0.51864	2.4175e-05\\
0.51904	2.3831e-05\\
0.51944	2.3483e-05\\
0.51985	2.3131e-05\\
0.52026	2.2774e-05\\
0.52069	2.2413e-05\\
0.52111	2.2048e-05\\
0.52155	2.1679e-05\\
0.52199	2.1306e-05\\
0.52245	2.093e-05\\
0.5229	2.055e-05\\
0.52337	2.0166e-05\\
0.52384	1.9779e-05\\
0.52432	1.9389e-05\\
0.52481	1.8995e-05\\
0.52531	1.8599e-05\\
0.52581	1.82e-05\\
0.52632	1.7799e-05\\
0.52684	1.7396e-05\\
0.52737	1.6991e-05\\
0.52791	1.6584e-05\\
0.52845	1.6175e-05\\
0.52901	1.5765e-05\\
0.52957	1.5353e-05\\
0.53014	1.4939e-05\\
0.53072	1.4525e-05\\
0.53131	1.411e-05\\
0.53191	1.3694e-05\\
0.53251	1.3278e-05\\
0.53313	1.2862e-05\\
0.53375	1.2446e-05\\
0.53439	1.203e-05\\
0.53503	1.1615e-05\\
0.53569	1.1201e-05\\
0.53635	1.0787e-05\\
0.53702	1.0375e-05\\
0.5377	9.9651e-06\\
0.5384	9.5566e-06\\
0.5391	9.1503e-06\\
0.53981	8.7464e-06\\
0.54054	8.3452e-06\\
0.54127	7.947e-06\\
0.54202	7.552e-06\\
0.54277	7.1606e-06\\
0.54354	6.773e-06\\
0.54432	6.3895e-06\\
0.54511	6.0104e-06\\
0.54591	5.6361e-06\\
0.54672	5.2668e-06\\
0.54754	4.903e-06\\
0.54837	4.5447e-06\\
0.54922	4.1923e-06\\
0.55008	3.8461e-06\\
0.55095	3.5062e-06\\
0.55183	3.1731e-06\\
0.55272	2.847e-06\\
0.55363	2.5283e-06\\
0.55454	2.2171e-06\\
0.55548	1.9138e-06\\
0.55642	1.6186e-06\\
0.55737	1.3319e-06\\
0.55834	1.0538e-06\\
0.55933	7.8474e-07\\
0.56032	5.2486e-07\\
0.56133	2.7443e-07\\
0.56235	3.3693e-08\\
0.5625	2.2204e-16\\
0.56339	1.9712e-07\\
0.56444	4.1779e-07\\
0.5655	6.2811e-07\\
0.56658	8.2787e-07\\
0.56767	1.0169e-06\\
0.56878	1.195e-06\\
0.5699	1.362e-06\\
0.57103	1.5178e-06\\
0.57218	1.6623e-06\\
0.57334	1.7953e-06\\
0.57452	1.9168e-06\\
0.57572	2.0267e-06\\
0.57693	2.125e-06\\
0.57815	2.2114e-06\\
0.57939	2.2862e-06\\
0.58065	2.3491e-06\\
0.58192	2.4004e-06\\
0.58321	2.4399e-06\\
0.58451	2.4678e-06\\
0.58583	2.4842e-06\\
0.58717	2.4892e-06\\
0.58852	2.483e-06\\
0.58989	2.4657e-06\\
0.59128	2.4375e-06\\
0.59268	2.3987e-06\\
0.59411	2.3495e-06\\
0.59555	2.2902e-06\\
0.597	2.2212e-06\\
0.59848	2.1428e-06\\
0.59997	2.0554e-06\\
0.60148	1.9595e-06\\
0.60301	1.8554e-06\\
0.60455	1.7437e-06\\
0.60612	1.625e-06\\
0.6077	1.4996e-06\\
0.6093	1.3684e-06\\
0.61093	1.2319e-06\\
0.61257	1.0907e-06\\
0.61423	9.4558e-07\\
0.61591	7.9715e-07\\
0.61761	6.4621e-07\\
0.61933	4.9355e-07\\
0.62107	3.3998e-07\\
0.62283	1.8633e-07\\
0.62461	3.3469e-08\\
0.625	2.2204e-16\\
0.62641	1.1771e-07\\
0.62823	2.663e-07\\
0.63007	4.1137e-07\\
0.63193	5.5198e-07\\
0.63382	6.8717e-07\\
0.63572	8.1598e-07\\
0.63765	9.3743e-07\\
0.6396	1.0506e-06\\
0.64157	1.1544e-06\\
0.64357	1.248e-06\\
0.64558	1.3304e-06\\
0.64762	1.4007e-06\\
0.64968	1.458e-06\\
0.65177	1.5015e-06\\
0.65388	1.5303e-06\\
0.65601	1.5438e-06\\
0.65816	1.5412e-06\\
0.66034	1.5218e-06\\
0.66254	1.4853e-06\\
0.66477	1.431e-06\\
0.66702	1.3587e-06\\
0.6693	1.2682e-06\\
0.6716	1.1594e-06\\
0.67392	1.0324e-06\\
0.67627	8.8745e-07\\
0.67865	7.2501e-07\\
0.68105	5.4573e-07\\
0.68348	3.5049e-07\\
0.68593	1.4044e-07\\
0.6875	2.2204e-16\\
0.68841	8.3013e-08\\
0.69091	3.1816e-07\\
0.69345	5.6297e-07\\
0.696	8.1504e-07\\
0.69859	1.0716e-06\\
0.7012	1.3295e-06\\
0.70384	1.585e-06\\
0.70651	1.834e-06\\
0.70921	2.0719e-06\\
0.71193	2.2935e-06\\
0.71468	2.4929e-06\\
0.71746	2.6638e-06\\
0.72027	2.799e-06\\
0.72311	2.8905e-06\\
0.72598	2.9297e-06\\
0.72888	2.9068e-06\\
0.7318	2.8116e-06\\
0.73476	2.6324e-06\\
0.73775	2.3568e-06\\
0.74076	1.9712e-06\\
0.74381	1.4608e-06\\
0.74689	8.0952e-07\\
0.75	2.2204e-16\\
0.75	3.252e-08\\
0.75	2.0622e-08\\
0.75	2.7368e-08\\
0.75	3.7542e-08\\
0.75	4.6994e-08\\
0.75	5.6198e-08\\
0.75	6.563e-08\\
0.75	7.5027e-08\\
0.75	8.4418e-08\\
0.75	9.3821e-08\\
0.75	1.0319e-07\\
0.75	1.1257e-07\\
0.75	1.2195e-07\\
0.75	1.3134e-07\\
0.75	1.4072e-07\\
0.75	1.501e-07\\
0.75	1.5949e-07\\
0.75	1.6887e-07\\
0.75	1.7826e-07\\
0.75	1.8764e-07\\
0.75	1.9703e-07\\
0.75	2.0642e-07\\
0.75	2.158e-07\\
0.75	2.2519e-07\\
0.75	2.3459e-07\\
0.75	2.4398e-07\\
0.75	2.5337e-07\\
0.75	2.6277e-07\\
0.75	2.7217e-07\\
0.75	2.8157e-07\\
0.75	2.9098e-07\\
0.75	3.0039e-07\\
0.75	3.098e-07\\
0.75	3.1921e-07\\
0.75	3.2863e-07\\
0.75	3.3805e-07\\
0.75	3.4748e-07\\
0.75	3.5691e-07\\
0.75	3.6635e-07\\
0.75	3.758e-07\\
0.75	3.8524e-07\\
0.75	3.947e-07\\
0.75	4.0416e-07\\
0.75	4.1363e-07\\
0.75	4.2311e-07\\
0.75001	4.326e-07\\
0.75001	4.4209e-07\\
0.75001	4.5159e-07\\
0.75001	4.611e-07\\
0.75001	4.7063e-07\\
0.75001	4.8016e-07\\
0.75001	4.897e-07\\
0.75001	4.9926e-07\\
0.75001	5.0882e-07\\
0.75001	5.184e-07\\
0.75001	5.2799e-07\\
0.75001	5.376e-07\\
0.75002	5.4722e-07\\
0.75002	5.5685e-07\\
0.75002	5.665e-07\\
0.75002	5.7616e-07\\
0.75002	5.8584e-07\\
0.75002	5.9554e-07\\
0.75003	6.0525e-07\\
0.75003	6.1499e-07\\
0.75003	6.2474e-07\\
0.75003	6.3451e-07\\
0.75004	6.4429e-07\\
0.75004	6.541e-07\\
0.75004	6.6393e-07\\
0.75004	6.7378e-07\\
0.75005	6.8366e-07\\
0.75005	6.9355e-07\\
0.75005	7.0347e-07\\
0.75006	7.1341e-07\\
0.75006	7.2337e-07\\
0.75007	7.3336e-07\\
0.75007	7.4337e-07\\
0.75008	7.5341e-07\\
0.75008	7.6348e-07\\
0.75009	7.7357e-07\\
0.75009	7.8369e-07\\
0.7501	7.9383e-07\\
0.7501	8.04e-07\\
0.75011	8.142e-07\\
0.75011	8.2443e-07\\
0.75012	8.3469e-07\\
0.75013	8.4497e-07\\
0.75014	8.5529e-07\\
0.75014	8.6563e-07\\
0.75015	8.7601e-07\\
0.75016	8.8641e-07\\
0.75017	8.9684e-07\\
0.75018	9.0731e-07\\
0.75019	9.178e-07\\
0.7502	9.2833e-07\\
0.75021	9.3889e-07\\
0.75022	9.4948e-07\\
0.75023	9.6009e-07\\
0.75024	9.7074e-07\\
0.75026	9.8142e-07\\
0.75027	9.9213e-07\\
0.75028	1.0029e-06\\
0.7503	1.0136e-06\\
0.75031	1.0245e-06\\
0.75033	1.0353e-06\\
0.75034	1.0461e-06\\
0.75036	1.057e-06\\
0.75038	1.068e-06\\
0.75039	1.0789e-06\\
0.75041	1.0899e-06\\
0.75043	1.1009e-06\\
0.75045	1.1119e-06\\
0.75047	1.123e-06\\
0.75049	1.134e-06\\
0.75051	1.1451e-06\\
0.75054	1.1563e-06\\
0.75056	1.1674e-06\\
0.75058	1.1786e-06\\
0.75061	1.1898e-06\\
0.75063	1.201e-06\\
0.75066	1.2122e-06\\
0.75069	1.2235e-06\\
0.75072	1.2348e-06\\
0.75075	1.246e-06\\
0.75078	1.2573e-06\\
0.75081	1.2686e-06\\
0.75084	1.2799e-06\\
0.75087	1.2913e-06\\
0.75091	1.3026e-06\\
0.75094	1.3139e-06\\
0.75098	1.3253e-06\\
0.75102	1.3366e-06\\
0.75105	1.3479e-06\\
0.75109	1.3593e-06\\
0.75114	1.3706e-06\\
0.75118	1.3819e-06\\
0.75122	1.3932e-06\\
0.75127	1.4045e-06\\
0.75131	1.4158e-06\\
0.75136	1.4271e-06\\
0.75141	1.4383e-06\\
0.75146	1.4495e-06\\
0.75151	1.4607e-06\\
0.75156	1.4719e-06\\
0.75162	1.483e-06\\
0.75168	1.4941e-06\\
0.75173	1.5051e-06\\
0.75179	1.5161e-06\\
0.75185	1.5271e-06\\
0.75192	1.538e-06\\
0.75198	1.5489e-06\\
0.75205	1.5597e-06\\
0.75211	1.5704e-06\\
0.75218	1.5811e-06\\
0.75226	1.5917e-06\\
0.75233	1.6022e-06\\
0.7524	1.6127e-06\\
0.75248	1.6231e-06\\
0.75256	1.6333e-06\\
0.75264	1.6435e-06\\
0.75272	1.6536e-06\\
0.75281	1.6636e-06\\
0.7529	1.6735e-06\\
0.75299	1.6833e-06\\
0.75308	1.693e-06\\
0.75317	1.7025e-06\\
0.75327	1.7119e-06\\
0.75337	1.7212e-06\\
0.75347	1.7304e-06\\
0.75357	1.7394e-06\\
0.75368	1.7482e-06\\
0.75378	1.7569e-06\\
0.75389	1.7655e-06\\
0.75401	1.7739e-06\\
0.75412	1.7821e-06\\
0.75424	1.7901e-06\\
0.75436	1.798e-06\\
0.75449	1.8056e-06\\
0.75461	1.8131e-06\\
0.75474	1.8203e-06\\
0.75488	1.8274e-06\\
0.75501	1.8342e-06\\
0.75515	1.8408e-06\\
0.75529	1.8472e-06\\
0.75544	1.8533e-06\\
0.75558	1.8592e-06\\
0.75573	1.8649e-06\\
0.75589	1.8703e-06\\
0.75605	1.8754e-06\\
0.75621	1.8803e-06\\
0.75637	1.8849e-06\\
0.75654	1.8892e-06\\
0.75671	1.8932e-06\\
0.75688	1.8969e-06\\
0.75706	1.9003e-06\\
0.75724	1.9034e-06\\
0.75743	1.9062e-06\\
0.75762	1.9086e-06\\
0.75781	1.9107e-06\\
0.75801	1.9125e-06\\
0.75821	1.9139e-06\\
0.75842	1.9149e-06\\
0.75863	1.9156e-06\\
0.75884	1.916e-06\\
0.75906	1.9159e-06\\
0.75928	1.9155e-06\\
0.75951	1.9147e-06\\
0.75974	1.9134e-06\\
0.75997	1.9118e-06\\
0.76021	1.9098e-06\\
0.76045	1.9073e-06\\
0.7607	1.9044e-06\\
0.76096	1.9011e-06\\
0.76122	1.8974e-06\\
0.76148	1.8932e-06\\
0.76175	1.8885e-06\\
0.76202	1.8834e-06\\
0.7623	1.8779e-06\\
0.76258	1.8718e-06\\
0.76287	1.8653e-06\\
0.76316	1.8583e-06\\
0.76346	1.8509e-06\\
0.76377	1.8429e-06\\
0.76408	1.8345e-06\\
0.76439	1.8255e-06\\
0.76472	1.8161e-06\\
0.76504	1.8061e-06\\
0.76538	1.7957e-06\\
0.76571	1.7847e-06\\
0.76606	1.7732e-06\\
0.76641	1.7612e-06\\
0.76677	1.7487e-06\\
0.76713	1.7357e-06\\
0.7675	1.7221e-06\\
0.76787	1.708e-06\\
0.76825	1.6934e-06\\
0.76864	1.6783e-06\\
0.76904	1.6626e-06\\
0.76944	1.6464e-06\\
0.76985	1.6297e-06\\
0.77026	1.6124e-06\\
0.77069	1.5947e-06\\
0.77111	1.5764e-06\\
0.77155	1.5576e-06\\
0.77199	1.5383e-06\\
0.77245	1.5184e-06\\
0.7729	1.4981e-06\\
0.77337	1.4773e-06\\
0.77384	1.4559e-06\\
0.77432	1.4341e-06\\
0.77481	1.4118e-06\\
0.77531	1.389e-06\\
0.77581	1.3658e-06\\
0.77632	1.342e-06\\
0.77684	1.3179e-06\\
0.77737	1.2933e-06\\
0.77791	1.2682e-06\\
0.77845	1.2428e-06\\
0.77901	1.2169e-06\\
0.77957	1.1907e-06\\
0.78014	1.164e-06\\
0.78072	1.137e-06\\
0.78131	1.1096e-06\\
0.78191	1.0819e-06\\
0.78251	1.0539e-06\\
0.78313	1.0256e-06\\
0.78375	9.9695e-07\\
0.78439	9.6806e-07\\
0.78503	9.3892e-07\\
0.78569	9.0954e-07\\
0.78635	8.7996e-07\\
0.78702	8.502e-07\\
0.7877	8.2026e-07\\
0.7884	7.9018e-07\\
0.7891	7.5998e-07\\
0.78981	7.2969e-07\\
0.79054	6.9933e-07\\
0.79127	6.6892e-07\\
0.79202	6.385e-07\\
0.79277	6.0808e-07\\
0.79354	5.777e-07\\
0.79432	5.4739e-07\\
0.79511	5.1717e-07\\
0.79591	4.8708e-07\\
0.79672	4.5714e-07\\
0.79754	4.2739e-07\\
0.79837	3.9785e-07\\
0.79922	3.6857e-07\\
0.80008	3.3956e-07\\
0.80095	3.1087e-07\\
0.80183	2.8252e-07\\
0.80272	2.5455e-07\\
0.80363	2.27e-07\\
0.80454	1.9989e-07\\
0.80548	1.7326e-07\\
0.80642	1.4715e-07\\
0.80737	1.2158e-07\\
0.80834	9.6595e-08\\
0.80933	7.2225e-08\\
0.81032	4.8504e-08\\
0.81133	2.5464e-08\\
0.81235	3.1391e-09\\
0.8125	2.2204e-16\\
0.81339	1.8439e-08\\
0.81444	3.9239e-08\\
0.8155	5.9229e-08\\
0.81658	7.8378e-08\\
0.81767	9.6657e-08\\
0.81878	1.1404e-07\\
0.8199	1.3049e-07\\
0.82103	1.4599e-07\\
0.82218	1.6052e-07\\
0.82334	1.7404e-07\\
0.82452	1.8653e-07\\
0.82572	1.9798e-07\\
0.82693	2.0837e-07\\
0.82815	2.1767e-07\\
0.82939	2.2588e-07\\
0.83065	2.3298e-07\\
0.83192	2.3895e-07\\
0.83321	2.438e-07\\
0.83451	2.4751e-07\\
0.83583	2.5008e-07\\
0.83717	2.5151e-07\\
0.83852	2.518e-07\\
0.83989	2.5096e-07\\
0.84128	2.49e-07\\
0.84268	2.4592e-07\\
0.84411	2.4175e-07\\
0.84555	2.365e-07\\
0.847	2.302e-07\\
0.84848	2.2287e-07\\
0.84997	2.1455e-07\\
0.85148	2.0526e-07\\
0.85301	1.9505e-07\\
0.85455	1.8395e-07\\
0.85612	1.7202e-07\\
0.8577	1.5931e-07\\
0.8593	1.4587e-07\\
0.86093	1.3177e-07\\
0.86257	1.1707e-07\\
0.86423	1.0183e-07\\
0.86591	8.6134e-08\\
0.86761	7.0059e-08\\
0.86933	5.3687e-08\\
0.87107	3.7105e-08\\
0.87283	2.0403e-08\\
0.87461	3.677e-09\\
0.875	2.2204e-16\\
0.87641	1.2974e-08\\
0.87823	2.9448e-08\\
0.88007	4.5639e-08\\
0.88193	6.1437e-08\\
0.88382	7.6731e-08\\
0.88572	9.1407e-08\\
0.88765	1.0535e-07\\
0.8896	1.1844e-07\\
0.89157	1.3056e-07\\
0.89357	1.4159e-07\\
0.89558	1.5141e-07\\
0.89762	1.5992e-07\\
0.89968	1.6698e-07\\
0.90177	1.7249e-07\\
0.90388	1.7634e-07\\
0.90601	1.7844e-07\\
0.90816	1.7867e-07\\
0.91034	1.7696e-07\\
0.91254	1.7322e-07\\
0.91477	1.6739e-07\\
0.91702	1.594e-07\\
0.9193	1.4922e-07\\
0.9216	1.3682e-07\\
0.92392	1.2219e-07\\
0.92627	1.0534e-07\\
0.92865	8.6308e-08\\
0.93105	6.5153e-08\\
0.93348	4.1965e-08\\
0.93593	1.6863e-08\\
0.9375	2.2204e-16\\
0.93841	9.9958e-09\\
0.94091	3.8419e-08\\
0.94345	6.8172e-08\\
0.946	9.8973e-08\\
0.94859	1.3049e-07\\
0.9512	1.6234e-07\\
0.95384	1.9408e-07\\
0.95651	2.2519e-07\\
0.95921	2.551e-07\\
0.96193	2.8315e-07\\
0.96468	3.0861e-07\\
0.96746	3.3065e-07\\
0.97027	3.4835e-07\\
0.97311	3.6068e-07\\
0.97598	3.6652e-07\\
0.97888	3.6462e-07\\
0.9818	3.5358e-07\\
0.98476	3.319e-07\\
0.98775	2.9792e-07\\
0.99076	2.4981e-07\\
0.99381	1.856e-07\\
0.99689	1.0311e-07\\
1	2.2204e-16\\
};
\end{axis}
\end{tikzpicture}%

%% file: src/residualsColl8.tex
%
%
\definecolor{mycolor1}{rgb}{0.00000,0.44700,0.74100}%
\begin{tikzpicture}

\begin{axis}[%
width=0.3\textwidth,
height=0.15\textwidth,
scale only axis,
xmin=0,xmax=1,
xlabel style={font=\color{white!15!black}},
xlabel={$t$},
xtick={0,0.25,0.5,0.75,1},
xticklabels={0,$\frac{1}{4}$,$\frac{1}{2}$,$\frac{3}{4}$,1},
ymode=log,
ymin=1e-16,ymax=1,
yminorticks=true,
yticklabels={},
axis background/.style={fill=white}
]
\addplot [color=mycolor1, forget plot,thick]
  table[row sep=crcr]{%
0	1\\
3.1397e-13	1\\
1.0047e-11	1\\
7.6294e-11	1\\
3.215e-10	1\\
9.8115e-10	1\\
2.4414e-09	1\\
5.2768e-09	1\\
1.0288e-08	1\\
1.8539e-08	1\\
3.1397e-08	1\\
5.0565e-08	1\\
7.8125e-08	0.99999\\
1.1657e-07	0.99999\\
1.6886e-07	0.99999\\
2.3842e-07	0.99998\\
3.2922e-07	0.99997\\
4.4579e-07	0.99996\\
5.9326e-07	0.99995\\
7.7741e-07	0.99993\\
1.0047e-06	0.99991\\
1.2823e-06	0.99989\\
1.6181e-06	0.99986\\
2.0208e-06	0.99982\\
2.5e-06	0.99978\\
3.0661e-06	0.99973\\
3.7304e-06	0.99968\\
4.5051e-06	0.99961\\
5.4035e-06	0.99953\\
6.4398e-06	0.99944\\
7.6294e-06	0.99934\\
8.9886e-06	0.99922\\
1.0535e-05	0.99908\\
1.2287e-05	0.99893\\
1.4265e-05	0.99876\\
1.649e-05	0.99857\\
1.8984e-05	0.99835\\
2.1772e-05	0.99811\\
2.4877e-05	0.99784\\
2.8327e-05	0.99754\\
3.215e-05	0.99721\\
3.6375e-05	0.99684\\
4.1033e-05	0.99644\\
4.6156e-05	0.99599\\
5.1778e-05	0.9955\\
5.7936e-05	0.99497\\
6.4666e-05	0.99439\\
7.2007e-05	0.99375\\
8e-05	0.99306\\
8.8688e-05	0.99231\\
9.8115e-05	0.9915\\
0.00010833	0.99061\\
0.00011937	0.98966\\
0.0001313	0.98863\\
0.00014416	0.98752\\
0.00015801	0.98633\\
0.00017291	0.98505\\
0.00018891	0.98368\\
0.00020607	0.9822\\
0.00022446	0.98063\\
0.00024414	0.97894\\
0.00026518	0.97715\\
0.00028764	0.97523\\
0.00031159	0.97319\\
0.00033712	0.97102\\
0.00036429	0.96871\\
0.00039319	0.96626\\
0.00042389	0.96367\\
0.00045649	0.96092\\
0.00049105	0.95801\\
0.00052768	0.95493\\
0.00056647	0.95169\\
0.0006075	0.94826\\
0.00065088	0.94465\\
0.00069669	0.94085\\
0.00074506	0.93684\\
0.00079607	0.93264\\
0.00084984	0.92822\\
0.00090648	0.92359\\
0.00096609	0.91873\\
0.0010288	0.91364\\
0.0010947	0.90831\\
0.001164	0.90274\\
0.0012367	0.89693\\
0.001313	0.89085\\
0.0013931	0.88452\\
0.001477	0.87792\\
0.0015649	0.87104\\
0.0016569	0.86389\\
0.0017532	0.85645\\
0.0018539	0.84873\\
0.0019593	0.84071\\
0.0020693	0.8324\\
0.0021842	0.82379\\
0.0023042	0.81487\\
0.0024294	0.80565\\
0.00256	0.79612\\
0.0026961	0.78628\\
0.002838	0.77612\\
0.0029858	0.76566\\
0.0031397	0.75487\\
0.0032998	0.74378\\
0.0034664	0.73237\\
0.0036397	0.72065\\
0.0038199	0.70863\\
0.0040071	0.6963\\
0.0042016	0.68366\\
0.0044035	0.67073\\
0.0046132	0.65752\\
0.0048308	0.64401\\
0.0050565	0.63023\\
0.0052905	0.61619\\
0.0055332	0.60188\\
0.0057846	0.58733\\
0.0060452	0.57254\\
0.006315	0.55753\\
0.0065944	0.54231\\
0.0068836	0.52689\\
0.0071828	0.5113\\
0.0074924	0.49554\\
0.0078125	0.47963\\
0.0081435	0.4636\\
0.0084856	0.44746\\
0.0088391	0.43124\\
0.0092043	0.41495\\
0.0095815	0.39862\\
0.0099709	0.38228\\
0.010373	0.36593\\
0.010788	0.34962\\
0.011216	0.33337\\
0.011657	0.3172\\
0.012113	0.30114\\
0.012582	0.28522\\
0.013066	0.26946\\
0.013565	0.2539\\
0.014078	0.23856\\
0.014608	0.22346\\
0.015153	0.20865\\
0.015714	0.19414\\
0.016291	0.17996\\
0.016886	0.16613\\
0.017498	0.1527\\
0.018127	0.13967\\
0.018774	0.12707\\
0.01944	0.11493\\
0.020124	0.10327\\
0.020828	0.092106\\
0.021551	0.081457\\
0.022294	0.071339\\
0.023058	0.061767\\
0.023842	0.052754\\
0.024647	0.044308\\
0.025474	0.036437\\
0.026323	0.029147\\
0.027195	0.022439\\
0.028089	0.016313\\
0.029007	0.010765\\
0.029949	0.0057899\\
0.030915	0.0013777\\
0.03125	2.2204e-16\\
0.031906	0.0024834\\
0.032922	0.0058082\\
0.033964	0.0086144\\
0.035032	0.010922\\
0.036126	0.012756\\
0.037248	0.014139\\
0.038398	0.015102\\
0.039575	0.015673\\
0.040782	0.015885\\
0.042017	0.015771\\
0.043283	0.015367\\
0.044579	0.014709\\
0.045905	0.013833\\
0.047264	0.012776\\
0.048654	0.011576\\
0.050076	0.010269\\
0.051532	0.0088912\\
0.053021	0.0074767\\
0.054544	0.0060584\\
0.056103	0.0046669\\
0.057696	0.0033305\\
0.059326	0.0020746\\
0.060993	0.00092115\\
0.0625	2.2204e-16\\
0.062696	0.00011105\\
0.064438	0.0010071\\
0.066218	0.0017561\\
0.068037	0.0023511\\
0.069895	0.0027893\\
0.071795	0.003072\\
0.073735	0.0032047\\
0.075717	0.0031964\\
0.077741	0.0030598\\
0.079809	0.0028109\\
0.08192	0.0024682\\
0.084076	0.0020525\\
0.086276	0.0015862\\
0.088523	0.0010923\\
0.090816	0.00059426\\
0.093157	0.00011461\\
0.09375	2.2204e-16\\
0.095545	0.00032547\\
0.097983	0.00070707\\
0.10047	0.0010143\\
0.10301	0.0012351\\
0.10559	0.0013614\\
0.10823	0.0013903\\
0.11093	0.0013237\\
0.11367	0.0011688\\
0.11647	0.00093786\\
0.11933	0.00064775\\
0.12224	0.0003194\\
0.125	2.2204e-16\\
0.1252	2.3177e-05\\
0.12823	0.0003541\\
0.13131	0.00064714\\
0.13445	0.00087736\\
0.13765	0.0010231\\
0.14091	0.0010677\\
0.14424	0.0010018\\
0.14762	0.00082478\\
0.15107	0.00054627\\
0.15458	0.00018707\\
0.15625	2.2204e-16\\
0.15816	0.00022082\\
0.16181	0.00063502\\
0.16552	0.0010049\\
0.1693	0.0012752\\
0.17314	0.0013911\\
0.17706	0.0013051\\
0.18105	0.0009854\\
0.18511	0.00042564\\
0.1875	2.2204e-16\\
0.18924	0.00034476\\
0.19345	0.0012507\\
0.19772	0.0021636\\
0.20208	0.0028981\\
0.20651	0.0032173\\
0.21102	0.0028508\\
0.21561	0.0015326\\
0.21875	2.2204e-16\\
0.22027	0.00093186\\
0.22502	0.0045576\\
0.22985	0.0090162\\
0.23476	0.013395\\
0.23976	0.015857\\
0.24483	0.013157\\
0.25	2.2204e-16\\
0.25	2.5406e-05\\
0.25	5.0812e-05\\
0.25	7.6219e-05\\
0.25	0.00010162\\
0.25	0.00012703\\
0.25	0.00015244\\
0.25	0.00017784\\
0.25	0.00020324\\
0.25	0.00022864\\
0.25	0.00025404\\
0.25	0.00027943\\
0.25	0.00030481\\
0.25	0.00033019\\
0.25	0.00035555\\
0.25	0.0003809\\
0.25	0.00040624\\
0.25	0.00043155\\
0.25	0.00045685\\
0.25	0.00048211\\
0.25	0.00050735\\
0.25	0.00053255\\
0.25	0.00055771\\
0.25	0.00058282\\
0.25	0.00060788\\
0.25	0.00063288\\
0.25	0.00065781\\
0.25	0.00068267\\
0.25001	0.00070745\\
0.25001	0.00073214\\
0.25001	0.00075674\\
0.25001	0.00078122\\
0.25001	0.00080559\\
0.25001	0.00082983\\
0.25001	0.00085394\\
0.25002	0.00087789\\
0.25002	0.00090169\\
0.25002	0.00092532\\
0.25002	0.00094876\\
0.25003	0.00097201\\
0.25003	0.00099505\\
0.25004	0.0010179\\
0.25004	0.0010405\\
0.25005	0.0010628\\
0.25005	0.0010849\\
0.25006	0.0011067\\
0.25006	0.0011282\\
0.25007	0.0011494\\
0.25008	0.0011702\\
0.25009	0.0011908\\
0.2501	0.001211\\
0.25011	0.0012308\\
0.25012	0.0012502\\
0.25013	0.0012692\\
0.25014	0.0012878\\
0.25016	0.001306\\
0.25017	0.0013237\\
0.25019	0.0013409\\
0.25021	0.0013577\\
0.25022	0.0013739\\
0.25024	0.0013896\\
0.25027	0.0014048\\
0.25029	0.0014195\\
0.25031	0.0014335\\
0.25034	0.001447\\
0.25036	0.0014598\\
0.25039	0.0014721\\
0.25042	0.0014837\\
0.25046	0.0014946\\
0.25049	0.0015049\\
0.25053	0.0015144\\
0.25057	0.0015233\\
0.25061	0.0015314\\
0.25065	0.0015388\\
0.2507	0.0015455\\
0.25075	0.0015514\\
0.2508	0.0015565\\
0.25085	0.0015608\\
0.25091	0.0015643\\
0.25097	0.0015669\\
0.25103	0.0015688\\
0.25109	0.0015698\\
0.25116	0.0015699\\
0.25124	0.0015692\\
0.25131	0.0015676\\
0.25139	0.0015651\\
0.25148	0.0015617\\
0.25156	0.0015575\\
0.25166	0.0015523\\
0.25175	0.0015462\\
0.25185	0.0015391\\
0.25196	0.0015312\\
0.25207	0.0015223\\
0.25218	0.0015125\\
0.2523	0.0015018\\
0.25243	0.0014902\\
0.25256	0.0014776\\
0.2527	0.0014641\\
0.25284	0.0014498\\
0.25299	0.0014345\\
0.25314	0.0014184\\
0.2533	0.0014013\\
0.25347	0.0013834\\
0.25364	0.0013646\\
0.25382	0.001345\\
0.25401	0.0013244\\
0.2542	0.0013031\\
0.2544	0.0012809\\
0.25461	0.001258\\
0.25483	0.0012343\\
0.25506	0.0012098\\
0.25529	0.0011846\\
0.25553	0.0011588\\
0.25578	0.0011322\\
0.25605	0.0011051\\
0.25631	0.0010773\\
0.25659	0.001049\\
0.25688	0.0010201\\
0.25718	0.0009908\\
0.25749	0.00096103\\
0.25781	0.00093086\\
0.25814	0.00090034\\
0.25849	0.00086951\\
0.25884	0.00083842\\
0.2592	0.00080712\\
0.25958	0.00077567\\
0.25997	0.00074411\\
0.26037	0.0007125\\
0.26079	0.00068091\\
0.26122	0.00064937\\
0.26166	0.00061796\\
0.26211	0.00058672\\
0.26258	0.00055573\\
0.26307	0.00052503\\
0.26356	0.00049469\\
0.26408	0.00046476\\
0.26461	0.00043531\\
0.26515	0.00040639\\
0.26571	0.00037806\\
0.26629	0.00035037\\
0.26689	0.00032338\\
0.2675	0.00029715\\
0.26813	0.00027172\\
0.26877	0.00024714\\
0.26944	0.00022345\\
0.27012	0.00020071\\
0.27083	0.00017894\\
0.27155	0.00015819\\
0.27229	0.00013848\\
0.27306	0.00011985\\
0.27384	0.00010232\\
0.27465	8.5895e-05\\
0.27547	7.0604e-05\\
0.27632	5.645e-05\\
0.27719	4.3437e-05\\
0.27809	3.1562e-05\\
0.27901	2.0817e-05\\
0.27995	1.119e-05\\
0.28091	2.6612e-06\\
0.28125	2.2204e-16\\
0.28191	4.7945e-06\\
0.28292	1.1207e-05\\
0.28396	1.6613e-05\\
0.28503	2.1052e-05\\
0.28613	2.4571e-05\\
0.28725	2.7221e-05\\
0.2884	2.9057e-05\\
0.28958	3.0139e-05\\
0.29078	3.0529e-05\\
0.29202	3.0293e-05\\
0.29328	2.9499e-05\\
0.29458	2.8218e-05\\
0.29591	2.6522e-05\\
0.29726	2.4481e-05\\
0.29865	2.2169e-05\\
0.30008	1.9654e-05\\
0.30153	1.7007e-05\\
0.30302	1.4293e-05\\
0.30454	1.1575e-05\\
0.3061	8.9109e-06\\
0.3077	6.3555e-06\\
0.30933	3.9565e-06\\
0.31099	1.7557e-06\\
0.3125	2.2204e-16\\
0.3127	2.1155e-07\\
0.31444	1.9174e-06\\
0.31622	3.3414e-06\\
0.31804	4.4709e-06\\
0.3199	5.3011e-06\\
0.32179	5.8352e-06\\
0.32373	6.0836e-06\\
0.32572	6.0644e-06\\
0.32774	5.8021e-06\\
0.32981	5.3271e-06\\
0.33192	4.6751e-06\\
0.33408	3.8856e-06\\
0.33628	3.0011e-06\\
0.33852	2.0656e-06\\
0.34082	1.1232e-06\\
0.34316	2.165e-07\\
0.34375	2.2204e-16\\
0.34555	6.1448e-07\\
0.34798	1.3343e-06\\
0.35047	1.913e-06\\
0.35301	2.3282e-06\\
0.35559	2.565e-06\\
0.35823	2.6181e-06\\
0.36093	2.4915e-06\\
0.36367	2.1988e-06\\
0.36647	1.7635e-06\\
0.36933	1.2174e-06\\
0.37224	5.9998e-07\\
0.375	2.2204e-16\\
0.3752	4.3517e-08\\
0.37823	6.6454e-07\\
0.38131	1.2139e-06\\
0.38445	1.645e-06\\
0.38765	1.9173e-06\\
0.39091	2e-06\\
0.39424	1.8757e-06\\
0.39762	1.5436e-06\\
0.40107	1.0219e-06\\
0.40458	3.4979e-07\\
0.40625	2.2204e-16\\
0.40816	4.1272e-07\\
0.41181	1.1864e-06\\
0.41552	1.8766e-06\\
0.4193	2.3804e-06\\
0.42314	2.5956e-06\\
0.42706	2.4341e-06\\
0.43105	1.8371e-06\\
0.43511	7.9319e-07\\
0.4375	2.2204e-16\\
0.43924	6.4222e-07\\
0.44345	2.3289e-06\\
0.44772	4.0271e-06\\
0.45208	5.3922e-06\\
0.45651	5.9839e-06\\
0.46102	5.3002e-06\\
0.46561	2.8482e-06\\
0.46875	2.2204e-16\\
0.47027	1.7312e-06\\
0.47502	8.4638e-06\\
0.47985	1.6738e-05\\
0.48476	2.4859e-05\\
0.48976	2.9416e-05\\
0.49483	2.44e-05\\
0.5	2.2204e-16\\
0.5	8.8292e-08\\
0.5	1.758e-07\\
0.5	2.6349e-07\\
0.5	3.5127e-07\\
0.5	4.3908e-07\\
0.5	5.2691e-07\\
0.5	6.1471e-07\\
0.5	7.0252e-07\\
0.5	7.9031e-07\\
0.5	8.7809e-07\\
0.5	9.6586e-07\\
0.5	1.0536e-06\\
0.5	1.1413e-06\\
0.5	1.229e-06\\
0.5	1.3166e-06\\
0.5	1.4042e-06\\
0.5	1.4917e-06\\
0.5	1.5791e-06\\
0.5	1.6665e-06\\
0.5	1.7537e-06\\
0.5	1.8408e-06\\
0.5	1.9277e-06\\
0.5	2.0145e-06\\
0.5	2.1012e-06\\
0.5	2.1876e-06\\
0.5	2.2738e-06\\
0.5	2.3597e-06\\
0.50001	2.4453e-06\\
0.50001	2.5307e-06\\
0.50001	2.6157e-06\\
0.50001	2.7003e-06\\
0.50001	2.7845e-06\\
0.50001	2.8683e-06\\
0.50001	2.9516e-06\\
0.50002	3.0344e-06\\
0.50002	3.1166e-06\\
0.50002	3.1983e-06\\
0.50002	3.2793e-06\\
0.50003	3.3596e-06\\
0.50003	3.4392e-06\\
0.50004	3.5181e-06\\
0.50004	3.5961e-06\\
0.50005	3.6733e-06\\
0.50005	3.7496e-06\\
0.50006	3.8249e-06\\
0.50006	3.8991e-06\\
0.50007	3.9724e-06\\
0.50008	4.0445e-06\\
0.50009	4.1154e-06\\
0.5001	4.1851e-06\\
0.50011	4.2535e-06\\
0.50012	4.3206e-06\\
0.50013	4.3862e-06\\
0.50014	4.4505e-06\\
0.50016	4.5132e-06\\
0.50017	4.5743e-06\\
0.50019	4.6338e-06\\
0.50021	4.6916e-06\\
0.50022	4.7476e-06\\
0.50024	4.8019e-06\\
0.50027	4.8543e-06\\
0.50029	4.9047e-06\\
0.50031	4.9532e-06\\
0.50034	4.9996e-06\\
0.50036	5.0439e-06\\
0.50039	5.086e-06\\
0.50042	5.1259e-06\\
0.50046	5.1635e-06\\
0.50049	5.1988e-06\\
0.50053	5.2317e-06\\
0.50057	5.2622e-06\\
0.50061	5.2901e-06\\
0.50065	5.3155e-06\\
0.5007	5.3383e-06\\
0.50075	5.3585e-06\\
0.5008	5.3759e-06\\
0.50085	5.3906e-06\\
0.50091	5.4024e-06\\
0.50097	5.4115e-06\\
0.50103	5.4176e-06\\
0.50109	5.4208e-06\\
0.50116	5.4211e-06\\
0.50124	5.4183e-06\\
0.50131	5.4126e-06\\
0.50139	5.4037e-06\\
0.50148	5.3918e-06\\
0.50156	5.3767e-06\\
0.50166	5.3585e-06\\
0.50175	5.3372e-06\\
0.50185	5.3126e-06\\
0.50196	5.2849e-06\\
0.50207	5.254e-06\\
0.50218	5.2199e-06\\
0.5023	5.1825e-06\\
0.50243	5.142e-06\\
0.50256	5.0983e-06\\
0.5027	5.0514e-06\\
0.50284	5.0013e-06\\
0.50299	4.9481e-06\\
0.50314	4.8918e-06\\
0.5033	4.8324e-06\\
0.50347	4.7699e-06\\
0.50364	4.7044e-06\\
0.50382	4.636e-06\\
0.50401	4.5646e-06\\
0.5042	4.4905e-06\\
0.5044	4.4135e-06\\
0.50461	4.3338e-06\\
0.50483	4.2515e-06\\
0.50506	4.1666e-06\\
0.50529	4.0793e-06\\
0.50553	3.9896e-06\\
0.50578	3.8976e-06\\
0.50605	3.8036e-06\\
0.50631	3.7076e-06\\
0.50659	3.6097e-06\\
0.50688	3.51e-06\\
0.50718	3.4087e-06\\
0.50749	3.3059e-06\\
0.50781	3.2017e-06\\
0.50814	3.0964e-06\\
0.50849	2.99e-06\\
0.50884	2.8827e-06\\
0.5092	2.7748e-06\\
0.50958	2.6663e-06\\
0.50997	2.5575e-06\\
0.51037	2.4486e-06\\
0.51079	2.3397e-06\\
0.51122	2.2311e-06\\
0.51166	2.1229e-06\\
0.51211	2.0153e-06\\
0.51258	1.9086e-06\\
0.51307	1.803e-06\\
0.51356	1.6986e-06\\
0.51408	1.5956e-06\\
0.51461	1.4943e-06\\
0.51515	1.3949e-06\\
0.51571	1.2974e-06\\
0.51629	1.2023e-06\\
0.51689	1.1095e-06\\
0.5175	1.0194e-06\\
0.51813	9.3203e-07\\
0.51877	8.4761e-07\\
0.51944	7.6628e-07\\
0.52012	6.882e-07\\
0.52083	6.1349e-07\\
0.52155	5.4227e-07\\
0.52229	4.7466e-07\\
0.52306	4.1075e-07\\
0.52384	3.506e-07\\
0.52465	2.943e-07\\
0.52547	2.4188e-07\\
0.52632	1.9336e-07\\
0.52719	1.4877e-07\\
0.52809	1.0809e-07\\
0.52901	7.1282e-08\\
0.52995	3.8313e-08\\
0.53091	9.1102e-09\\
0.53125	2.2204e-16\\
0.53191	1.6411e-08\\
0.53292	3.8357e-08\\
0.53396	5.6851e-08\\
0.53503	7.2034e-08\\
0.53613	8.4065e-08\\
0.53725	9.3121e-08\\
0.5384	9.939e-08\\
0.53958	1.0308e-07\\
0.54078	1.044e-07\\
0.54202	1.0358e-07\\
0.54328	1.0085e-07\\
0.54458	9.6464e-08\\
0.54591	9.0654e-08\\
0.54726	8.3671e-08\\
0.54865	7.5758e-08\\
0.55008	6.7158e-08\\
0.55153	5.8106e-08\\
0.55302	4.8827e-08\\
0.55454	3.9537e-08\\
0.5561	3.0435e-08\\
0.5577	2.1705e-08\\
0.55933	1.351e-08\\
0.56099	5.9948e-09\\
0.5625	2.2204e-16\\
0.5627	7.2224e-10\\
0.56444	6.5454e-09\\
0.56622	1.1405e-08\\
0.56804	1.5259e-08\\
0.5699	1.8091e-08\\
0.57179	1.9911e-08\\
0.57373	2.0757e-08\\
0.57572	2.069e-08\\
0.57774	1.9793e-08\\
0.57981	1.8171e-08\\
0.58192	1.5945e-08\\
0.58408	1.3251e-08\\
0.58628	1.0234e-08\\
0.58852	7.0432e-09\\
0.59082	3.8293e-09\\
0.59316	7.3808e-10\\
0.59375	2.2204e-16\\
0.59555	2.0946e-09\\
0.59798	4.5478e-09\\
0.60047	6.5201e-09\\
0.60301	7.9341e-09\\
0.60559	8.7406e-09\\
0.60823	8.9207e-09\\
0.61093	8.4885e-09\\
0.61367	7.4909e-09\\
0.61647	6.0072e-09\\
0.61933	4.1466e-09\\
0.62224	2.0435e-09\\
0.625	2.2204e-16\\
0.6252	1.482e-10\\
0.62823	2.263e-09\\
0.63131	4.1334e-09\\
0.63445	5.6009e-09\\
0.63765	6.5275e-09\\
0.64091	6.8086e-09\\
0.64424	6.385e-09\\
0.64762	5.254e-09\\
0.65107	3.4781e-09\\
0.65458	1.1904e-09\\
0.65625	2.2204e-16\\
0.65816	1.4045e-09\\
0.66181	4.037e-09\\
0.66552	6.3854e-09\\
0.6693	8.0989e-09\\
0.67314	8.8307e-09\\
0.67706	8.2807e-09\\
0.68105	6.2492e-09\\
0.68511	2.698e-09\\
0.6875	2.2204e-16\\
0.68924	2.1844e-09\\
0.69345	7.9205e-09\\
0.69772	1.3695e-08\\
0.70208	1.8337e-08\\
0.70651	2.0347e-08\\
0.71102	1.8021e-08\\
0.71561	9.6838e-09\\
0.71875	2.2204e-16\\
0.72027	5.8856e-09\\
0.72502	2.8773e-08\\
0.72985	5.6898e-08\\
0.73476	8.4499e-08\\
0.73976	9.9982e-08\\
0.74483	8.2929e-08\\
0.75	2.2204e-16\\
0.75	1.3801e-09\\
0.75	8.4215e-10\\
0.75	9.7291e-10\\
0.75	1.354e-09\\
0.75	1.6915e-09\\
0.75	2.0192e-09\\
0.75	2.3633e-09\\
0.75	2.6978e-09\\
0.75	3.0365e-09\\
0.75	3.3717e-09\\
0.75	3.7092e-09\\
0.75	4.0463e-09\\
0.75	4.3831e-09\\
0.75	4.7198e-09\\
0.75	5.0564e-09\\
0.75	5.3927e-09\\
0.75	5.7288e-09\\
0.75	6.0645e-09\\
0.75	6.3998e-09\\
0.75	6.7347e-09\\
0.75	7.0691e-09\\
0.75	7.403e-09\\
0.75	7.7362e-09\\
0.75	8.0686e-09\\
0.75	8.4003e-09\\
0.75	8.7311e-09\\
0.75	9.0608e-09\\
0.75001	9.3894e-09\\
0.75001	9.7168e-09\\
0.75001	1.0043e-08\\
0.75001	1.0367e-08\\
0.75001	1.069e-08\\
0.75001	1.1012e-08\\
0.75001	1.1331e-08\\
0.75002	1.1648e-08\\
0.75002	1.1963e-08\\
0.75002	1.2276e-08\\
0.75002	1.2586e-08\\
0.75003	1.2894e-08\\
0.75003	1.3199e-08\\
0.75004	1.35e-08\\
0.75004	1.3799e-08\\
0.75005	1.4094e-08\\
0.75005	1.4385e-08\\
0.75006	1.4673e-08\\
0.75006	1.4957e-08\\
0.75007	1.5236e-08\\
0.75008	1.5511e-08\\
0.75009	1.5782e-08\\
0.7501	1.6047e-08\\
0.75011	1.6308e-08\\
0.75012	1.6563e-08\\
0.75013	1.6813e-08\\
0.75014	1.7057e-08\\
0.75016	1.7295e-08\\
0.75017	1.7527e-08\\
0.75019	1.7753e-08\\
0.75021	1.7972e-08\\
0.75022	1.8184e-08\\
0.75024	1.8388e-08\\
0.75027	1.8586e-08\\
0.75029	1.8776e-08\\
0.75031	1.8959e-08\\
0.75034	1.9133e-08\\
0.75036	1.9299e-08\\
0.75039	1.9457e-08\\
0.75042	1.9606e-08\\
0.75046	1.9746e-08\\
0.75049	1.9878e-08\\
0.75053	2e-08\\
0.75057	2.0113e-08\\
0.75061	2.0215e-08\\
0.75065	2.0308e-08\\
0.7507	2.0391e-08\\
0.75075	2.0464e-08\\
0.7508	2.0526e-08\\
0.75085	2.0578e-08\\
0.75091	2.0618e-08\\
0.75097	2.0648e-08\\
0.75103	2.0667e-08\\
0.75109	2.0674e-08\\
0.75116	2.067e-08\\
0.75124	2.0655e-08\\
0.75131	2.0628e-08\\
0.75139	2.0589e-08\\
0.75148	2.0538e-08\\
0.75156	2.0475e-08\\
0.75166	2.0401e-08\\
0.75175	2.0314e-08\\
0.75185	2.0215e-08\\
0.75196	2.0104e-08\\
0.75207	1.9981e-08\\
0.75218	1.9846e-08\\
0.7523	1.9698e-08\\
0.75243	1.9538e-08\\
0.75256	1.9366e-08\\
0.7527	1.9183e-08\\
0.75284	1.8987e-08\\
0.75299	1.8779e-08\\
0.75314	1.8559e-08\\
0.7533	1.8328e-08\\
0.75347	1.8086e-08\\
0.75364	1.7832e-08\\
0.75382	1.7567e-08\\
0.75401	1.7291e-08\\
0.7542	1.7004e-08\\
0.7544	1.6707e-08\\
0.75461	1.64e-08\\
0.75483	1.6083e-08\\
0.75506	1.5757e-08\\
0.75529	1.5421e-08\\
0.75553	1.5077e-08\\
0.75578	1.4725e-08\\
0.75605	1.4364e-08\\
0.75631	1.3996e-08\\
0.75659	1.3621e-08\\
0.75688	1.324e-08\\
0.75718	1.2853e-08\\
0.75749	1.246e-08\\
0.75781	1.2063e-08\\
0.75814	1.1661e-08\\
0.75849	1.1256e-08\\
0.75884	1.0848e-08\\
0.7592	1.0438e-08\\
0.75958	1.0025e-08\\
0.75997	9.6125e-09\\
0.76037	9.1993e-09\\
0.76079	8.7867e-09\\
0.76122	8.3752e-09\\
0.76166	7.9658e-09\\
0.76211	7.5591e-09\\
0.76258	7.1559e-09\\
0.76307	6.757e-09\\
0.76356	6.363e-09\\
0.76408	5.9748e-09\\
0.76461	5.5932e-09\\
0.76515	5.2187e-09\\
0.76571	4.8523e-09\\
0.76629	4.4945e-09\\
0.76689	4.146e-09\\
0.7675	3.8076e-09\\
0.76813	3.4798e-09\\
0.76877	3.1633e-09\\
0.76944	2.8586e-09\\
0.77012	2.5662e-09\\
0.77083	2.2866e-09\\
0.77155	2.0203e-09\\
0.77229	1.7677e-09\\
0.77306	1.529e-09\\
0.77384	1.3046e-09\\
0.77465	1.0946e-09\\
0.77547	8.9925e-10\\
0.77632	7.1859e-10\\
0.77719	5.5263e-10\\
0.77809	4.0133e-10\\
0.77901	2.6456e-10\\
0.77995	1.4214e-10\\
0.78091	3.3784e-11\\
0.78125	2.2204e-16\\
0.78191	6.0833e-11\\
0.78292	1.4212e-10\\
0.78396	2.1055e-10\\
0.78503	2.6668e-10\\
0.78613	3.111e-10\\
0.78725	3.4448e-10\\
0.7884	3.6752e-10\\
0.78958	3.8101e-10\\
0.79078	3.8573e-10\\
0.79202	3.8256e-10\\
0.79328	3.7234e-10\\
0.79458	3.56e-10\\
0.79591	3.3442e-10\\
0.79726	3.0854e-10\\
0.79865	2.7925e-10\\
0.80008	2.4745e-10\\
0.80153	2.1401e-10\\
0.80302	1.7977e-10\\
0.80454	1.4551e-10\\
0.8061	1.1196e-10\\
0.8077	7.9816e-11\\
0.80933	4.9663e-11\\
0.81099	2.2028e-11\\
0.8125	2.2204e-16\\
0.8127	2.6527e-12\\
0.81444	2.4032e-11\\
0.81622	4.1859e-11\\
0.81804	5.5981e-11\\
0.8199	6.6343e-11\\
0.82179	7.299e-11\\
0.82373	7.6061e-11\\
0.82572	7.5784e-11\\
0.82774	7.247e-11\\
0.82981	6.6505e-11\\
0.83192	5.8336e-11\\
0.83408	4.8462e-11\\
0.83628	3.7412e-11\\
0.83852	2.5737e-11\\
0.84082	1.3988e-11\\
0.84316	2.6952e-12\\
0.84375	2.2204e-16\\
0.84555	7.6455e-12\\
0.84798	1.6593e-11\\
0.85047	2.378e-11\\
0.85301	2.8926e-11\\
0.85559	3.1853e-11\\
0.85823	3.2497e-11\\
0.86093	3.091e-11\\
0.86367	2.7267e-11\\
0.86647	2.1858e-11\\
0.86933	1.5082e-11\\
0.87224	7.4298e-12\\
0.875	2.2204e-16\\
0.8752	5.3907e-13\\
0.87823	8.2218e-12\\
0.88131	1.5011e-11\\
0.88445	2.0332e-11\\
0.88765	2.3687e-11\\
0.89091	2.4697e-11\\
0.89424	2.3153e-11\\
0.89762	1.9044e-11\\
0.90107	1.2602e-11\\
0.90458	4.3115e-12\\
0.90625	2.2204e-16\\
0.90816	5.0856e-12\\
0.91181	1.4611e-11\\
0.91552	2.3101e-11\\
0.9193	2.9289e-11\\
0.92314	3.1922e-11\\
0.92706	2.9923e-11\\
0.93105	2.2574e-11\\
0.93511	9.7422e-12\\
0.9375	2.2204e-16\\
0.93924	7.8849e-12\\
0.94345	2.8578e-11\\
0.94772	4.9398e-11\\
0.95208	6.6112e-11\\
0.95651	7.3334e-11\\
0.96102	6.4927e-11\\
0.96561	3.4876e-11\\
0.96875	2.2204e-16\\
0.97027	2.119e-11\\
0.97502	1.0355e-10\\
0.97985	2.0469e-10\\
0.98476	3.0387e-10\\
0.98976	3.5942e-10\\
0.99483	2.9801e-10\\
1	2.2204e-16\\
};
\end{axis}
\end{tikzpicture}%

%% file: src/bubbles_4_01.tex
%
%
\begin{tikzpicture}

\begin{axis}[%
width=0.17\textwidth,
height=0.15\textwidth,
scale only axis,
scaled y ticks=true,
xmin=0,
xmax=1.2e-57,
ymin=0,
scaled y ticks=base 10:1,
axis background/.style={fill=white}
]
\addplot [color=blue, mark=o, mark options={solid, blue}, forget plot,thick]
  table[row sep=crcr]{%
7.7371e-62	0\\
6.9579e-60	0.31724\\
2.1922e-59	0.31425\\
4.3014e-59	0.26966\\
6.9429e-59	0.21113\\
1.0067e-58	0.15071\\
1.3639e-58	0.095535\\
1.7632e-58	0.049958\\
2.2025e-58	0.016267\\
2.6801e-58	0.005005\\
3.1944e-58	0.01482\\
3.7443e-58	0.015374\\
4.3285e-58	0.009764\\
4.9461e-58	0.0015831\\
5.5962e-58	0.0055785\\
6.2781e-58	0.0087322\\
6.9909e-58	0.0061949\\
7.7342e-58	0.001597\\
8.5071e-58	0.011159\\
9.3092e-58	0.014909\\
1.014e-57	0\\
};
\end{axis}
\end{tikzpicture}%

%% file: src/bubbles_4_04.tex
%
%
\begin{tikzpicture}

\begin{axis}[%
width=0.17\textwidth,
height=0.15\textwidth,
scale only axis,
xmin=0,
ymin=0,
axis background/.style={fill=white}
]
\addplot [color=blue, mark=o, mark options={solid, blue}, forget plot,thick]
  table[row sep=crcr]{%
2.5026e-13	0\\
2.7127e-13	0.019061\\
3.1698e-13	0.031223\\
3.814e-13	0.034505\\
4.6208e-13	0.031232\\
5.575e-13	0.024332\\
6.666e-13	0.016271\\
7.8856e-13	0.0087862\\
9.2274e-13	0.0029132\\
1.0686e-12	0.00090418\\
1.2257e-12	0.0026832\\
1.3936e-12	0.0027768\\
1.5721e-12	0.0017536\\
1.7607e-12	0.00028206\\
1.9593e-12	0.00098432\\
2.1675e-12	0.0015241\\
2.3853e-12	0.0010687\\
2.6123e-12	0.00027213\\
2.8483e-12	0.0018774\\
3.0933e-12	0.0024761\\
3.3471e-12	0\\
};
\end{axis}
\end{tikzpicture}%

%% file: src/bubbles_4_08.tex
%
%
\begin{tikzpicture}

\begin{axis}[%
width=0.17\textwidth,
height=0.15\textwidth,
scale only axis,
xmin=3e-07,
xmax=2.5e-06,
ymin=0,
axis background/.style={fill=white}
]
\addplot [color=blue, mark=o, mark options={solid, blue}, forget plot,thick]
  table[row sep=crcr]{%
3.6864e-07	0\\
3.8155e-07	0.013323\\
4.0962e-07	0.014203\\
4.4919e-07	0.012664\\
4.9874e-07	0.010031\\
5.5735e-07	0.0071134\\
6.2436e-07	0.0044252\\
6.9927e-07	0.0022523\\
7.8169e-07	0.00070978\\
8.7128e-07	0.00021059\\
9.6776e-07	0.00059988\\
1.0709e-06	0.00059778\\
1.1805e-06	0.00036441\\
1.2964e-06	5.6696e-05\\
1.4183e-06	0.00019172\\
1.5463e-06	0.00028807\\
1.68e-06	0.00019626\\
1.8194e-06	4.8619e-05\\
1.9644e-06	0.00032664\\
2.1149e-06	0.00041994\\
2.2707e-06	0\\
};
\end{axis}
\end{tikzpicture}%

%% file: src/bubbles_4_099.tex
%
%
\begin{tikzpicture}

\begin{axis}[%
width=0.17\textwidth,
height=0.15\textwidth,
scale only axis,
xmin=5e-06,
xmax=4e-05,
ymin=0,
axis background/.style={fill=white}
]
\addplot [color=blue, mark=o, mark options={solid, blue}, forget plot,thick]
  table[row sep=crcr]{%
9.216e-06	0\\
9.4027e-06	0.00060849\\
9.8089e-06	0.00052619\\
1.0381e-05	0.00042102\\
1.1098e-05	0.00031192\\
1.1946e-05	0.00021147\\
1.2916e-05	0.00012737\\
1.3999e-05	6.3264e-05\\
1.5192e-05	1.9554e-05\\
1.6488e-05	5.709e-06\\
1.7884e-05	1.6039e-05\\
1.9376e-05	1.5788e-05\\
2.0962e-05	9.5171e-06\\
2.2638e-05	1.4654e-06\\
2.4403e-05	4.9068e-06\\
2.6253e-05	7.3041e-06\\
2.8188e-05	4.9316e-06\\
3.0205e-05	1.2111e-06\\
3.2303e-05	8.0676e-06\\
3.448e-05	1.0286e-05\\
3.6735e-05	0\\
};
\end{axis}
\end{tikzpicture}%

%% file: surf1.tex
\begin{tikzpicture}
	 \begin{axis}[
		width = 0.4\textwidth,
	 	 grid,
	 	 enlargelimits = false,
	 	 xmin = 0.000000,
	 	 xmax = 1.000000,
	 	 ymin = 0.000000,
	 	 ymax = 1.000000,
	 	 zmin = -0.010000,
	 	 zmax = 0.080000,
	 	 xlabel = {$x$},
	 	 ylabel = {$t$},
	 	 zlabel = {},
	 ]

	 	 \addplot3 graphics[
	 	 	 points={
	 	 	 	 (0.000000,0.000000,-0.010000) => (316.000000,62.000000)
	 	 	 	 (0.000000,1.000000,0.070000) => (89.000000,431.000000)
	 	 	 	 (1.000000,0.000000,0.070000) => (611.000000,404.000000)
	 	 	 	 (1.000000,1.000000,-0.010000) => (384.000000,267.000000)
	 	 }]{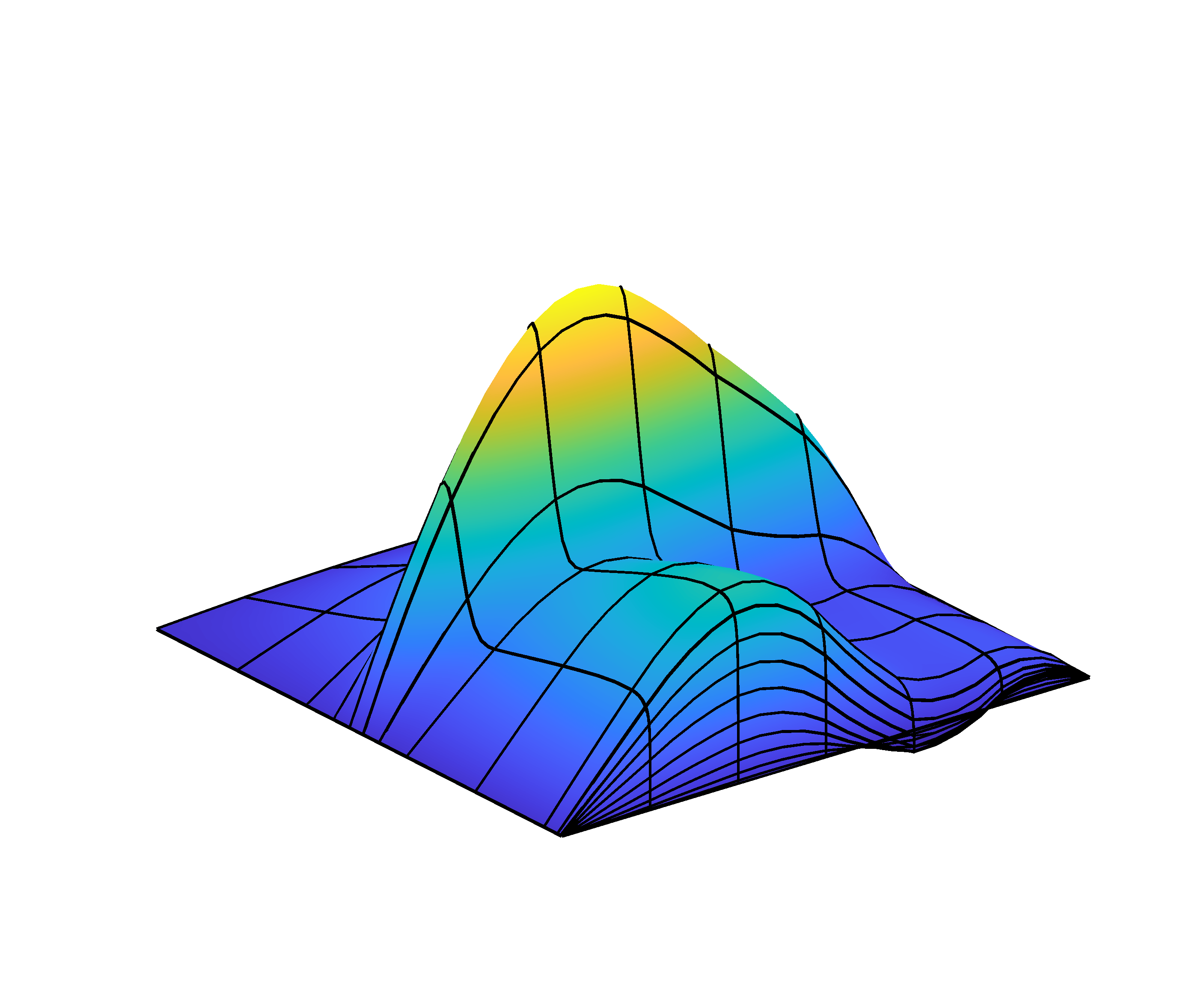};

	 \end{axis}
\end{tikzpicture}

%% file: surf2.tex
\begin{tikzpicture}
	 \begin{axis}[
		width = 0.4\textwidth,
	 	 grid,
	 	 enlargelimits = false,
	 	 xmin = 0.000000,
	 	 xmax = 1.000000,
	 	 ymin = 0.000000,
	 	 ymax = 1.000000,
	 	 zmin = -0.010000,
	 	 zmax = 0.080000,
	 	 xlabel = {$x$},
	 	 ylabel = {$t$},
	 	 zlabel = {},
	 ]

	 	 \addplot3 graphics[
	 	 	 points={
	 	 	 	 (0.000000,0.000000,-0.000562) => (312.000000,63.000000)
	 	 	 	 (0.000000,1.000000,0.080000) => (88.000000,437.000000)
	 	 	 	 (1.000000,0.000000,0.080000) => (604.000000,410.000000)
	 	 	 	 (1.000000,1.000000,-0.000562) => (380.000000,271.000000)
	 	 }]{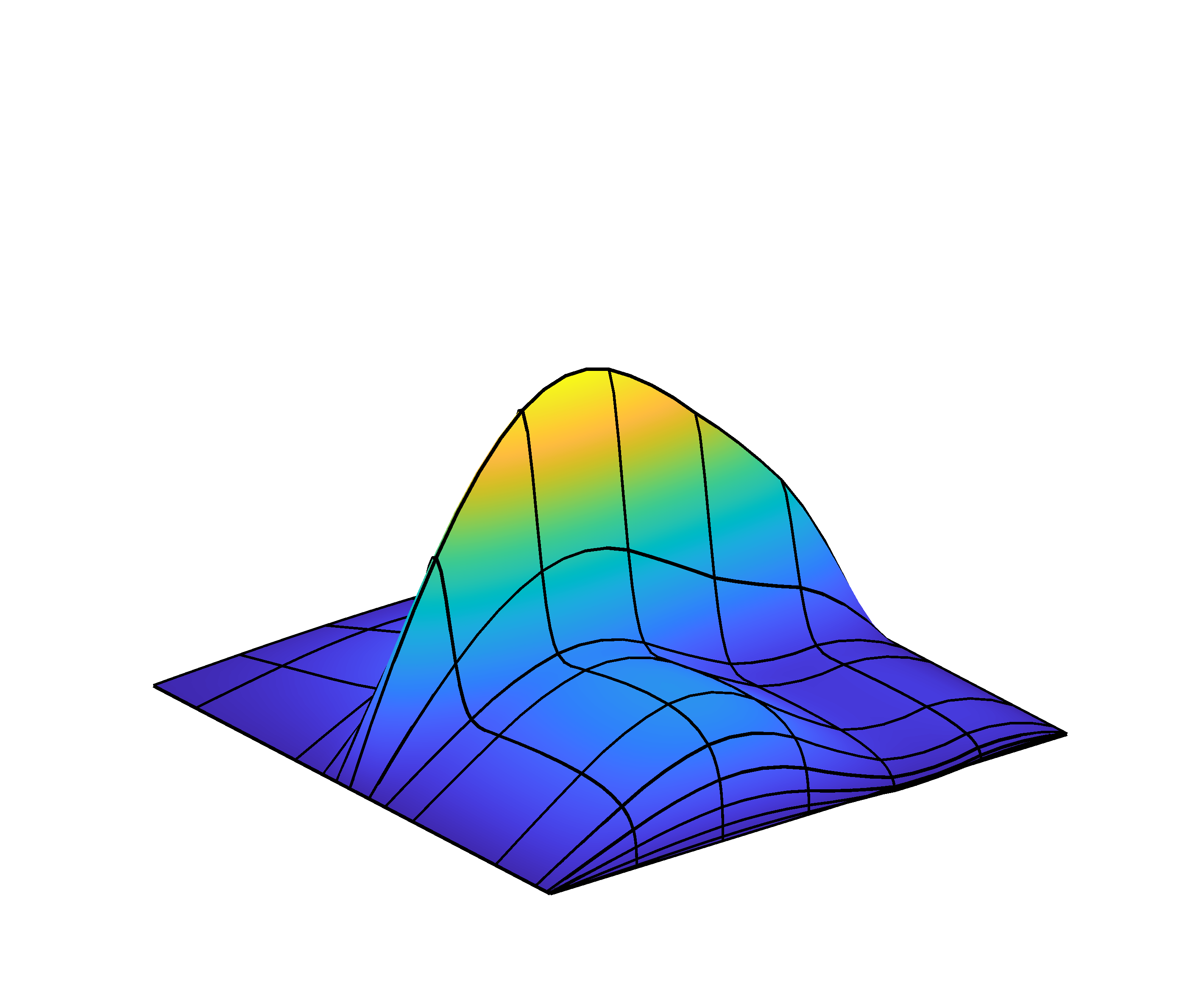};

	 \end{axis}
\end{tikzpicture}

%% file: surf3.tex
\begin{tikzpicture}
	 \begin{axis}[
		width = 0.4\textwidth,
	 	 grid,
	 	 enlargelimits = false,
	 	 xmin = 0.000000,
	 	 xmax = 1.000000,
	 	 ymin = 0.000000,
	 	 ymax = 1.000000,
	 	 zmin = -0.010000,
	 	 zmax = 0.080000,
	 	 xlabel = {$x$},
	 	 ylabel = {$t$},
	 	 zlabel = {},
	 ]

	 	 \addplot3 graphics[
	 	 	 points={
	 	 	 	 (0.000000,0.000000,-0.010000) => (316.000000,62.000000)
	 	 	 	 (0.000000,1.000000,0.060000) => (89.000000,431.000000)
	 	 	 	 (1.000000,0.000000,0.060000) => (611.000000,404.000000)
	 	 	 	 (1.000000,1.000000,-0.010000) => (384.000000,267.000000)
	 	 }]{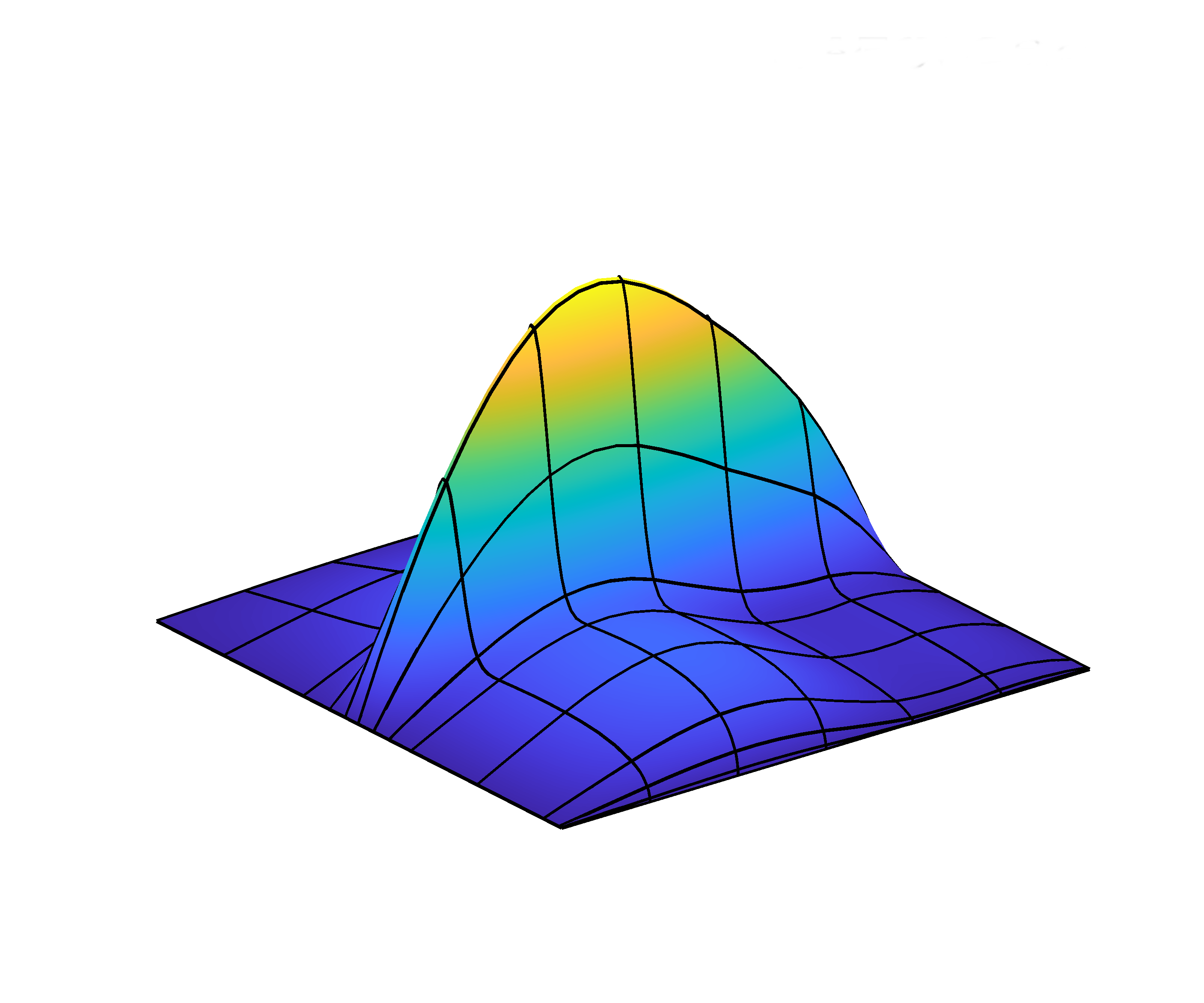};

	 \end{axis}
\end{tikzpicture}

%% file: src/R0_Linf_Linf_04_pi2.tex
%
%
\begin{tikzpicture}

\begin{axis}[%
width=0.6\textwidth,
height=0.28\textwidth,
scale only axis,
xmode=log,
xmin=1,
xmax=8000,
xminorticks=true,
xlabel style={font=\color{white!15!black}},
xlabel={$M$},
ymode=log,
ymin=1e-09,
ymax=0.02,
yminorticks=true,
axis background/.style={fill=white},
legend columns=1, 
legend style={at={(1.03,0.5)}, anchor=west, legend cell align=left, align=left, draw=white!15!black}
]
\addplot [color=blue, mark=*, mark options={solid, blue}]
  table[row sep=crcr]{%
10	0.01\\
19	0.0031623\\
37	0.001\\
72	0.00031623\\
142	0.0001\\
280	3.1623e-05\\
557	1e-05\\
1119	3.1623e-06\\
2248	1e-06\\
};
\addlegendentry{L1, TOL}
\addplot [color=blue,thick]
  table[row sep=crcr]{%
10	0.0048259\\
19	0.0014227\\
37	0.00043905\\
72	0.00013622\\
142	4.1712e-05\\
280	1.2821e-05\\
557	3.9117e-06\\
1119	1.1969e-06\\
2248	3.7773e-07\\
};
\addlegendentry{error}

\addplot [color=purple, mark=*, mark options={solid, purple}]
  table[row sep=crcr]{%
9	0.01\\
14	0.0031623\\
23	0.001\\
35	0.00031623\\
55	0.0001\\
84	3.1623e-05\\
130	1e-05\\
199	3.1623e-06\\
307	1e-06\\
475	3.1623e-07\\
734	1e-07\\
};
\addlegendentry{L1-2, TOL}
\addplot [color=purple,thick]
  table[row sep=crcr]{%
9	0.0046293\\
14	0.0014426\\
23	0.00041466\\
35	0.00013931\\
55	4.7361e-05\\
84	1.4227e-05\\
130	4.4404e-06\\
199	1.4006e-06\\
307	4.352e-07\\
475	1.3093e-07\\
734	4.0995e-08\\
};
\addlegendentry{error}

\addplot [color=black, mark=*, mark options={solid, black}]
  table[row sep=crcr]{%
4	0.01\\
7	0.0031623\\
11	0.001\\
17	0.00031623\\
26	0.0001\\
40	3.1623e-05\\
63	1e-05\\
97	3.1623e-06\\
151	1e-06\\
234	3.1623e-07\\
365	1e-07\\
568	3.1623e-08\\
881	1e-08\\
};
\addlegendentry{coll(2), TOL}
\addplot [color=black,thick]
  table[row sep=crcr]{%
4	0.0041742\\
7	0.0014139\\
11	0.00045195\\
17	0.00013164\\
26	3.8771e-05\\
40	1.2383e-05\\
63	3.9541e-06\\
97	1.1737e-06\\
151	3.7475e-07\\
234	1.1815e-07\\
365	3.7724e-08\\
568	1.2044e-08\\
881	3.8455e-09\\
};
\addlegendentry{error}
\addplot [color=teal, mark=*, mark options={solid, teal}]
  table[row sep=crcr]{%
3	0.01\\
4	0.0031623\\
6	0.001\\
9	0.00031623\\
12	0.0001\\
16	3.1623e-05\\
21	1e-05\\
29	3.1623e-06\\
38	1e-06\\
49	3.1623e-07\\
64	1e-07\\
84	3.1623e-08\\
107	1e-08\\
};
\addlegendentry{coll(4), TOL}
\addplot [color=teal,thick]
  table[row sep=crcr]{%
3	0.0041239\\
4	0.0012828\\
6	0.00038772\\
9	0.00012446\\
12	3.9804e-05\\
16	1.1821e-05\\
21	4.0097e-06\\
29	1.223e-06\\
38	3.8003e-07\\
49	1.2133e-07\\
64	3.874e-08\\
84	1.2214e-08\\
107	3.8997e-09\\
};
\addlegendentry{error}
\addplot [color=magenta, mark=*, mark options={solid, magenta}]
  table[row sep=crcr]{%
2	0.01\\
3	0.0031623\\
5	0.001\\
6	0.00031623\\
8	0.0001\\
10	3.1623e-05\\
12	1e-05\\
15	3.1623e-06\\
18	1e-06\\
21	3.1623e-07\\
25	1e-07\\
30	3.1623e-08\\
35	1e-08\\
};
\addlegendentry{coll(8), TOL}
\addplot [color=magenta,thick]
  table[row sep=crcr]{%
2	0.0034436\\
3	0.00098743\\
5	0.00031753\\
6	0.00010165\\
8	3.3698e-05\\
10	9.694e-06\\
12	3.0965e-06\\
15	9.8878e-07\\
18	3.1176e-07\\
21	9.954e-08\\
25	3.1781e-08\\
30	1.0147e-08\\
35	3.1992e-09\\
};
\addlegendentry{error}

\addplot [color=black]
  table[row sep=crcr]{%
4e1   5e-03\\
2e2   5e-03\\
2e2   4e-04\\
4e1   5e-03\\
};
\node[color=black,right] at (2e2,2e-3) {$M^{-(2-\alpha)}$};
\addplot [color=black]
  table[row sep=crcr]{%
5e0   5e-06\\
5e0   3e-08\\
15e0   3e-08\\
5e0   5e-06\\
};
\node[color=black,right] at (1e0,3e-7) {$M^{-(5-\alpha)}$};
\end{axis}

\end{tikzpicture}%

%% file: src/R0_Linf_Linf_04_0.tex
%
%
\begin{tikzpicture}

\begin{axis}[%
width=0.6\textwidth,
height=0.28\textwidth,
scale only axis,
xmode=log,
xmin=1,
xmax=8000,
xminorticks=true,
xlabel style={font=\color{white!15!black}},
xlabel={$M$},
ymode=log,
ymin=1e-09,
ymax=0.02,
yminorticks=true,
axis background/.style={fill=white},
legend columns=1, 
legend style={at={(1.03,0.5)}, anchor=west, legend cell align=left, align=left, draw=white!15!black}
]
\addplot [color=blue, mark=*, mark options={solid, blue}]
  table[row sep=crcr]{%
18	0.01\\
34	0.0031623\\
63	0.001\\
122	0.00031623\\
240	0.0001\\
474	3.1623e-05\\
951	1e-05\\
1891	3.1623e-06\\
3848	1e-06\\
};
\addlegendentry{L1, TOL}
\addplot [color=blue,thick]
  table[row sep=crcr]{%
18	0.0059732\\
34	0.0022801\\
63	0.00078689\\
122	0.00027198\\
240	8.7036e-05\\
474	2.7809e-05\\
951	8.7699e-06\\
1891	2.6054e-06\\
3848	8.3196e-07\\
};
\addlegendentry{error}

\addplot [color=purple, mark=*, mark options={solid, purple}]
  table[row sep=crcr]{%
11	0.01\\
17	0.0031623\\
27	0.001\\
41	0.00031623\\
62	0.0001\\
96	3.1623e-05\\
147	1e-05\\
226	3.1623e-06\\
348	1e-06\\
534	3.1623e-07\\
830	1e-07\\
};
\addlegendentry{L1-2, TOL}
\addplot [color=purple,thick]
  table[row sep=crcr]{%
11	0.0059732\\
17	0.0022801\\
27	0.00078689\\
41	0.00027198\\
62	8.7036e-05\\
96	2.7809e-05\\
147	8.7699e-06\\
226	2.6033e-06\\
348	8.3119e-07\\
534	2.6538e-07\\
830	8.4731e-08\\
};
\addlegendentry{error}

\addplot [color=black, mark=*, mark options={solid, black}]
  table[row sep=crcr]{%
5	0.01\\
8	0.0031623\\
12	0.001\\
19	0.00031623\\
29	0.0001\\
46	3.1623e-05\\
71	1e-05\\
110	3.1623e-06\\
170	1e-06\\
265	3.1623e-07\\
410	1e-07\\
636	3.1623e-08\\
987	1e-08\\
};
\addlegendentry{coll(2), TOL}
\addplot [color=black,thick]
  table[row sep=crcr]{%
5	0.0059038\\
8	0.0023943\\
12	0.0008389\\
19	0.00026662\\
29	8.5341e-05\\
46	2.7269e-05\\
71	8.7088e-06\\
110	2.7807e-06\\
170	8.7675e-07\\
265	2.7993e-07\\
410	8.309e-08\\
636	2.6529e-08\\
987	8.47e-09\\
};
\addlegendentry{error}
\addplot [color=teal, mark=*, mark options={solid, teal}]
  table[row sep=crcr]{%
3	0.01\\
5	0.0031623\\
6	0.001\\
9	0.00031623\\
12	0.0001\\
16	3.1623e-05\\
21	1e-05\\
28	3.1623e-06\\
37	1e-06\\
49	3.1623e-07\\
63	1e-07\\
82	3.1623e-08\\
107	1e-08\\
};
\addlegendentry{coll(4), TOL}
\addplot [color=teal,thick]
  table[row sep=crcr]{%
3	0.0058193\\
5	0.0023925\\
6	0.00084602\\
9	0.0002733\\
12	8.7588e-05\\
16	2.7649e-05\\
21	8.8309e-06\\
28	2.8199e-06\\
37	8.3703e-07\\
49	2.6725e-07\\
63	8.426e-08\\
82	2.6902e-08\\
107	8.5894e-09\\
};
\addlegendentry{error}
\addplot [color=magenta, mark=*, mark options={solid, magenta}]
  table[row sep=crcr]{%
2	0.01\\
3	0.0031623\\
4	0.001\\
6	0.00031623\\
7	0.0001\\
9	3.1623e-05\\
11	1e-05\\
14	3.1623e-06\\
17	1e-06\\
20	3.1623e-07\\
24	1e-07\\
29	3.1623e-08\\
34	1e-08\\
};
\addlegendentry{coll(8), TOL}
\addplot [color=magenta,thick]
  table[row sep=crcr]{%
2	0.0040461\\
3	0.0018264\\
4	0.00062766\\
6	0.00022201\\
7	7.1632e-05\\
9	2.1335e-05\\
11	6.8185e-06\\
14	2.1505e-06\\
17	6.8666e-07\\
20	2.1924e-07\\
24	7e-08\\
29	2.207e-08\\
34	7.0465e-09\\
};
\addlegendentry{error}

\addplot [color=black]
  table[row sep=crcr]{%
4e1   5e-03\\
2e2   5e-03\\
2e2   4e-04\\
4e1   5e-03\\
};
\node[color=black,right] at (2e2,2e-3) {$M^{-(2-\alpha)}$};
\addplot [color=black]
  table[row sep=crcr]{%
5e0   5e-06\\
5e0   3e-08\\
15e0  3e-08\\
5e0   5e-06\\
};
\node[color=black,right] at (1e0,3e-7) {$M^{-(5-\alpha)}$};
\end{axis}

\end{tikzpicture}%

%% file: src/R0_Linf_L2_04.tex
%
%
\begin{tikzpicture}

\begin{axis}[%
width=0.6\textwidth,
height=0.28\textwidth,
scale only axis,
xmode=log,
xmin=1,
xmax=8000,
xminorticks=true,
xlabel style={font=\color{white!15!black}},
xlabel={$M$},
ymode=log,
ymin=1e-10,
ymax=0.1,
yminorticks=true,
axis background/.style={fill=white},
legend columns=1, 
legend style={at={(1.03,0.5)}, anchor=west, legend cell align=left, align=left, draw=white!15!black}
]
\addplot [color=blue, mark=*, mark options={solid, blue}]
  table[row sep=crcr]{%
25	0.01\\
49	0.0031623\\
94	0.001\\
183	0.00031623\\
360	0.0001\\
716	3.1623e-05\\
1430	1e-05\\
2873	3.1623e-06\\
5792	1e-06\\
};
\addlegendentry{L1, TOL}
\addplot [color=blue,thick]
  table[row sep=crcr]{%
25	0.00061665\\
49	0.00018457\\
94	5.9082e-05\\
183	1.8879e-05\\
360	6.0294e-06\\
716	1.9012e-06\\
1430	6.0702e-07\\
2873	1.9381e-07\\
5792	6.1879e-08\\
};
\addlegendentry{error}

\addplot [color=purple, mark=*, mark options={solid, purple}]
  table[row sep=crcr]{%
17	0.01\\
27	0.0031623\\
42	0.001\\
65	0.00031623\\
99	0.0001\\
152	3.1623e-05\\
233	1e-05\\
361	3.1623e-06\\
554	1e-06\\
860	3.1623e-07\\
};
\addlegendentry{L1-2, TOL}
\addplot [color=purple,thick]
  table[row sep=crcr]{%
17	0.00061987\\
27	0.00019643\\
42	6.1435e-05\\
65	1.9121e-05\\
99	6.0294e-06\\
152	1.9012e-06\\
233	6.0702e-07\\
361	1.9381e-07\\
554	6.1879e-08\\
860	1.951e-08\\
};
\addlegendentry{error}

\addplot [color=black, mark=*, mark options={solid, black}]
  table[row sep=crcr]{%
8	0.01\\
12	0.0031623\\
20	0.001\\
31	0.00031623\\
48	0.0001\\
74	3.1623e-05\\
114	1e-05\\
177	3.1623e-06\\
275	1e-06\\
426	3.1623e-07\\
662	1e-07\\
1028	3.1623e-08\\
1594	1e-08\\
};
\addlegendentry{coll(2), TOL}
\addplot [color=black,thick]
  table[row sep=crcr]{%
8	0.00061154\\
12	0.00019456\\
20	6.2307e-05\\
31	1.8513e-05\\
48	5.9125e-06\\
74	1.8879e-06\\
114	5.9525e-07\\
177	1.9005e-07\\
275	6.068e-08\\
426	1.9374e-08\\
662	6.1856e-09\\
1028	1.9502e-09\\
1594	6.2267e-10\\
};
\addlegendentry{error}
\addplot [color=teal, mark=*, mark options={solid, teal}]
  table[row sep=crcr]{%
5	0.01\\
7	0.0031623\\
10	0.001\\
13	0.00031623\\
18	0.0001\\
24	3.1623e-05\\
31	1e-05\\
41	3.1623e-06\\
54	1e-06\\
70	3.1623e-07\\
91	1e-07\\
117	3.1623e-08\\
152	1e-08\\
};
\addlegendentry{coll(4), TOL}
\addplot [color=teal,thick]
  table[row sep=crcr]{%
5	0.00061628\\
7	0.00019936\\
10	5.9451e-05\\
13	1.877e-05\\
18	5.9954e-06\\
24	1.9145e-06\\
31	6.1128e-07\\
41	1.9517e-07\\
54	6.1535e-08\\
70	1.9647e-08\\
91	6.2728e-09\\
117	2.0028e-09\\
152	6.3144e-10\\
};
\addlegendentry{error}
\addplot [color=magenta, mark=*, mark options={solid, magenta}]
  table[row sep=crcr]{%
4	0.01\\
5	0.0031623\\
7	0.001\\
8	0.00031623\\
10	0.0001\\
13	3.1623e-05\\
16	1e-05\\
19	3.1623e-06\\
23	1e-06\\
27	3.1623e-07\\
32	1e-07\\
38	3.1623e-08\\
44	1e-08\\
};
\addlegendentry{coll(8), TOL}
\addplot [color=magenta,thick]
  table[row sep=crcr]{%
4	0.0004882\\
5	0.00015045\\
7	4.982e-05\\
8	1.5575e-05\\
10	4.979e-06\\
13	1.5704e-06\\
16	5.0146e-07\\
19	1.6011e-07\\
23	5.1121e-08\\
27	1.6118e-08\\
32	4.7841e-09\\
38	1.5275e-09\\
44	4.8769e-10\\
};
\addlegendentry{error}

\addplot [color=black]
  table[row sep=crcr]{%
6e1   5e-03\\
3e2   5e-03\\
3e2   4e-04\\
6e1   5e-03\\
};
\node[color=black,right] at (3e2,2e-3) {$M^{-(2-\alpha)}$};
\addplot [color=black]
  table[row sep=crcr]{%
5e0   5e-07\\
5e0   3e-09\\
15e0   3e-09\\
5e0   5e-07\\
};
\node[color=black,right] at (1e0,3e-8) {$M^{-(5-\alpha)}$};

\end{axis}

\end{tikzpicture}%

%% file: src/R0_Linf_Linf_01_pi2.tex
%
%
\begin{tikzpicture}

\begin{axis}[%
width=0.6\textwidth,
height=0.28\textwidth,
scale only axis,
xmode=log,
xmin=4,
xmax=4000,
xminorticks=true,
xlabel style={font=\color{white!15!black}},
xlabel={$M$},
ymode=log,
ymin=1e-10,
ymax=0.1,
yminorticks=true,
axis background/.style={fill=white},
legend columns=1, 
legend style={at={(1.03,0.5)}, anchor=west, legend cell align=left, align=left, draw=white!15!black}
]
\addplot [color=blue, mark=*, mark options={solid, blue}]
  table[row sep=crcr]{%
16	0.01\\
31	0.0031623\\
58	0.001\\
107	0.00031623\\
195	0.0001\\
356	3.1623e-05\\
641	1e-05\\
1172	3.1623e-06\\
2105	1e-06\\
};
\addlegendentry{L1, TOL}
\addplot [color=blue,thick]
  table[row sep=crcr]{%
16	0.0052386\\
31	0.0016331\\
58	0.00052672\\
107	0.00016078\\
195	5.1956e-05\\
356	1.6308e-05\\
641	5.1416e-06\\
1172	1.615e-06\\
2105	5.0877e-07\\
};
\addlegendentry{error}

\addplot [color=purple, mark=*, mark options={solid, purple}]
  table[row sep=crcr]{%
25	0.01\\
39	0.0031623\\
60	0.001\\
89	0.00031623\\
133	0.0001\\
197	3.1623e-05\\
291	1e-05\\
428	3.1623e-06\\
629	1e-06\\
};
\addlegendentry{L1-2, TOL}
\addplot [color=purple,thick]
  table[row sep=crcr]{%
25	0.0052311\\
39	0.0016477\\
60	0.00052044\\
89	0.00016427\\
133	5.162e-05\\
197	1.6561e-05\\
291	5.2113e-06\\
428	1.6245e-06\\
629	5.1908e-07\\
};
\addlegendentry{error}

\addplot [color=black, mark=*, mark options={solid, black}]
  table[row sep=crcr]{%
10	0.01\\
16	0.0031623\\
26	0.001\\
40	0.00031623\\
62	0.0001\\
93	3.1623e-05\\
140	1e-05\\
209	3.1623e-06\\
310	1e-06\\
461	3.1623e-07\\
684	1e-07\\
1015	3.1623e-08\\
1504	1e-08\\
};
\addlegendentry{coll(2), TOL}
\addplot [color=black,thick]
  table[row sep=crcr]{%
10	0.005078\\
16	0.0016285\\
26	0.00048087\\
40	0.00015818\\
62	5.1112e-05\\
93	1.5974e-05\\
140	5.0889e-06\\
209	1.6113e-06\\
310	5.0171e-07\\
461	1.5743e-07\\
684	4.9743e-08\\
1015	1.5531e-08\\
1504	4.8861e-09\\
};
\addlegendentry{error}
\addplot [color=teal, mark=*, mark options={solid, teal}]
  table[row sep=crcr]{%
8	0.01\\
12	0.0031623\\
18	0.001\\
25	0.00031623\\
34	0.0001\\
46	3.1623e-05\\
61	1e-05\\
79	3.1623e-06\\
103	1e-06\\
134	3.1623e-07\\
171	1e-07\\
219	3.1623e-08\\
281	1e-08\\
};
\addlegendentry{coll(4), TOL}
\addplot [color=teal,thick]
  table[row sep=crcr]{%
8	0.0053311\\
12	0.0016122\\
18	0.00045742\\
25	0.00015382\\
34	4.6554e-05\\
46	1.5728e-05\\
61	4.828e-06\\
79	1.5956e-06\\
103	4.9743e-07\\
134	1.4782e-07\\
171	4.917e-08\\
219	1.5837e-08\\
281	4.8987e-09\\
};
\addlegendentry{error}
\addplot [color=magenta, mark=*, mark options={solid, magenta}]
  table[row sep=crcr]{%
7	0.01\\
10	0.0031623\\
13	0.001\\
18	0.00031623\\
24	0.0001\\
30	3.1623e-05\\
37	1e-05\\
46	3.1623e-06\\
55	1e-06\\
67	3.1623e-07\\
80	1e-07\\
94	3.1623e-08\\
111	1e-08\\
};
\addlegendentry{coll(8), TOL}
\addplot [color=magenta,thick]
  table[row sep=crcr]{%
7	0.0033088\\
10	0.0011961\\
13	0.00041123\\
18	0.00012738\\
24	4.0721e-05\\
30	1.4477e-05\\
37	4.8338e-06\\
46	1.5169e-06\\
55	4.6942e-07\\
67	1.3943e-07\\
80	3.9695e-08\\
94	1.4937e-08\\
111	4.6063e-09\\
};
\addlegendentry{error}

\addplot [color=black]
  table[row sep=crcr]{%
8e1   5e-03\\
4e2   5e-03\\
4e2   2e-04\\
8e1   5e-03\\
};
\node[color=black,right] at (4e2,2e-3) {$M^{-(2-\alpha)}$};
\addplot [color=black]
  table[row sep=crcr]{%
1.5e1   5e-06\\
1.5e1   2e-08\\
6e1     2e-08\\
1.5e1   5e-06\\
};
\node[color=black,right] at (4e0,3e-7) {$M^{-(5-\alpha)}$};

\end{axis}

\end{tikzpicture}%

%% file: src/R0_Linf_Linf_08_pi2.tex
%
%
\begin{tikzpicture}

\begin{axis}[%
width=0.6\textwidth,
height=0.28\textwidth,
scale only axis,
xmode=log,
xmin=0.8,
xmax=20000,
xminorticks=true,
xlabel style={font=\color{white!15!black}},
xlabel={$M$},
ymode=log,
ymin=1e-10,
ymax=0.1,
yminorticks=true,
axis background/.style={fill=white},
legend columns=1, 
legend style={at={(1.03,0.5)}, anchor=west, legend cell align=left, align=left, draw=white!15!black}
]
\addplot [color=blue, mark=*, mark options={solid, blue}]
  table[row sep=crcr]{%
9	0.01\\
21	0.0031623\\
53	0.001\\
138	0.00031623\\
364	0.0001\\
935	3.1623e-05\\
2509	1e-05\\
6385	3.1623e-06\\
17049	1e-06\\
};
\addlegendentry{L1, TOL}
\addplot [color=blue,thick]
  table[row sep=crcr]{%
9	0.0037816\\
21	0.0010996\\
53	0.00031532\\
138	9.6462e-05\\
364	2.8652e-05\\
935	9.0888e-06\\
2509	2.7847e-06\\
6385	8.9231e-07\\
17049	2.7834e-07\\
};
\addlegendentry{error}

\addplot [color=purple, mark=*, mark options={solid, purple}]
  table[row sep=crcr]{%
4	0.01\\
7	0.0031623\\
13	0.001\\
21	0.00031623\\
37	0.0001\\
62	3.1623e-05\\
105	1e-05\\
177	3.1623e-06\\
295	1e-06\\
};
\addlegendentry{L1-2, TOL}
\addplot [color=purple,thick]
  table[row sep=crcr]{%
4	0.0041692\\
7	0.001193\\
13	0.00040555\\
21	0.00013837\\
37	3.8056e-05\\
62	1.1618e-05\\
105	3.465e-06\\
177	1.0491e-06\\
295	3.3632e-07\\
};
\addlegendentry{error}

\addplot [color=black, mark=*, mark options={solid, black}]
  table[row sep=crcr]{%
2	0.01\\
3	0.0031623\\
6	0.001\\
11	0.00031623\\
19	0.0001\\
34	3.1623e-05\\
57	1e-05\\
99	3.1623e-06\\
171	1e-06\\
290	3.1623e-07\\
494	1e-07\\
837	3.1623e-08\\
1417	1e-08\\
};
\addlegendentry{coll(2), TOL}
\addplot [color=black,thick]
  table[row sep=crcr]{%
2	0.0033056\\
3	0.0011377\\
6	0.0002527\\
11	7.6388e-05\\
19	2.5088e-05\\
34	6.2786e-06\\
57	1.9608e-06\\
99	6.2619e-07\\
171	1.9994e-07\\
290	6.384e-08\\
494	2.0383e-08\\
837	6.3461e-09\\
1417	2.0262e-09\\
};
\addlegendentry{error}
\addplot [color=teal, mark=*, mark options={solid, teal}]
  table[row sep=crcr]{%
2	0.0031623\\
3	0.001\\
5	0.00031623\\
7	0.0001\\
10	3.1623e-05\\
14	1e-05\\
19	3.1623e-06\\
26	1e-06\\
35	3.1623e-07\\
49	1e-07\\
64	3.1623e-08\\
85	1e-08\\
};
\addlegendentry{coll(4), TOL}
\addplot [color=teal,thick]
  table[row sep=crcr]{%
2	0.00079041\\
3	0.00030739\\
5	6.1923e-05\\
7	1.9932e-05\\
10	6.2218e-06\\
14	1.9886e-06\\
19	6.3514e-07\\
26	2.0281e-07\\
35	6.4755e-08\\
49	2.0675e-08\\
64	6.6011e-09\\
85	1.8216e-09\\
};
\addlegendentry{error}
\addplot [color=magenta, mark=*, mark options={solid, magenta}]
  table[row sep=crcr]{%
1	0.01\\
2	0.0031623\\
3	0.00031623\\
4	0.0001\\
6	3.1623e-05\\
7	1e-05\\
9	3.1623e-06\\
10	1e-06\\
13	3.1623e-07\\
16	1e-07\\
19	3.1623e-08\\
22	1e-08\\
};
\addlegendentry{coll(8), TOL}
\addplot [color=magenta,thick]
  table[row sep=crcr]{%
1	0.0016723\\
2	0.00071136\\
3	6.7195e-05\\
4	1.877e-05\\
6	6.0116e-06\\
7	1.9213e-06\\
9	6.1362e-07\\
10	1.9594e-07\\
13	6.1005e-08\\
16	1.9478e-08\\
19	6.2189e-09\\
22	1.9856e-09\\
};
\addlegendentry{error}

\addplot [color=black]
  table[row sep=crcr]{%
4e1   5e-03\\
2e2   5e-03\\
2e2   7.2e-04\\
4e1   5e-03\\
};
\node[color=black,right] at (2e2,2e-3) {$M^{-(2-\alpha)}$};
\addplot [color=black]
  table[row sep=crcr]{%
5e0   5e-07\\
5e0   3e-09\\
15e0   3e-09\\
5e0   5e-07\\
};
\node[color=black,right] at (0.8e0,3e-8) {$M^{-(5-\alpha)}$};

\end{axis}

\end{tikzpicture}%

%% file: src/R1_Linf_04_pi2.tex
%
%
\begin{tikzpicture}

\begin{axis}[%
width=0.6\textwidth,
height=0.28\textwidth,
scale only axis,
xmode=log,
xmin=1,
xmax=1000,
xminorticks=true,
xlabel style={font=\color{white!15!black}},
xlabel={$M$},
ymode=log,
ymin=1e-12,
ymax=0.02,
yminorticks=true,
axis background/.style={fill=white},
legend columns=1, 
legend style={at={(1.03,0.5)}, anchor=west, legend cell align=left, align=left, draw=white!15!black}
]
\addplot [color=blue, mark=*, mark options={solid, blue}]
  table[row sep=crcr]{%
6	0.01\\
10	0.0031623\\
18	0.001\\
34	0.00031623\\
64	0.0001\\
124	3.1623e-05\\
240	1e-05\\
472	3.1623e-06\\
940	1e-06\\
};
\addlegendentry{L1, TOL}
\addplot [color=blue,thick]
  table[row sep=crcr]{%
6	0.00059084\\
10	0.00025054\\
18	0.00010232\\
34	3.3574e-05\\
64	1.449e-05\\
124	4.8828e-06\\
240	1.5917e-06\\
472	6.2133e-07\\
940	1.9625e-07\\
};
\addlegendentry{error}

\addplot [color=purple, mark=*, mark options={solid, purple}]
  table[row sep=crcr]{%
5	0.01\\
7	0.0031623\\
10	0.001\\
14	0.00031623\\
22	0.0001\\
32	3.1623e-05\\
49	1e-05\\
73	3.1623e-06\\
113	1e-06\\
};
\addlegendentry{L1-2, TOL}
\addplot [color=purple,thick]
  table[row sep=crcr]{%
5	0.000343\\
7	0.00013048\\
10	4.3957e-05\\
14	2.7867e-05\\
22	9.5489e-06\\
32	2.5112e-06\\
49	8.4158e-07\\
73	4.6774e-07\\
113	1.0909e-07\\
};
\addlegendentry{error}

\addplot [color=black, mark=*, mark options={solid, black}]
  table[row sep=crcr]{%
2	0.01\\
3	0.0031623\\
5	0.001\\
7	0.00031623\\
10	0.0001\\
15	3.1623e-05\\
23	1e-05\\
35	3.1623e-06\\
54	1e-06\\
84	3.1623e-07\\
129	1e-07\\
201	3.1623e-08\\
312	1e-08\\
};
\addlegendentry{coll(2), TOL}
\addplot [color=black,thick]
  table[row sep=crcr]{%
2	0.00033311\\
3	6.3519e-05\\
5	6.5629e-06\\
7	1.9284e-06\\
10	7.9248e-07\\
15	1.6256e-07\\
23	8.2253e-08\\
35	3.4353e-08\\
54	2.3159e-08\\
84	2.2609e-09\\
129	5.4284e-09\\
201	1.0011e-09\\
312	1.4035e-10\\
};
\addlegendentry{error}
\addplot [color=teal, mark=*, mark options={solid, teal}]
  table[row sep=crcr]{%
2	0.0031623\\
3	0.001\\
4	0.00031623\\
5	0.0001\\
7	3.1623e-05\\
9	1e-05\\
11	3.1623e-06\\
14	1e-06\\
19	3.1623e-07\\
25	1e-07\\
32	3.1623e-08\\
42	1e-08\\
};
\addlegendentry{coll(4), TOL}
\addplot [color=teal,thick]
  table[row sep=crcr]{%
2	0.00012123\\
3	3.3022e-06\\
4	3.1985e-07\\
5	4.4577e-08\\
7	2.2866e-08\\
9	1.5774e-09\\
11	6.2484e-10\\
14	1.7409e-08\\
19	8.5125e-09\\
25	3.0982e-11\\
32	6.8992e-10\\
42	1.6625e-11\\
};
\addlegendentry{error}
\addplot [color=magenta, mark=*, mark options={solid, magenta}]
  table[row sep=crcr]{%
2	0.001\\
3	0.00031623\\
4	3.1623e-05\\
5	1e-05\\
6	3.1623e-06\\
7	1e-06\\
9	3.1623e-07\\
10	1e-07\\
12	3.1623e-08\\
15	1e-08\\
};
\addlegendentry{coll(8), TOL}
\addplot [color=magenta,thick]
  table[row sep=crcr]{%
2	3.3871e-05\\
3	1.0186e-06\\
4	2.016e-07\\
5	5.9147e-09\\
6	4.258e-09\\
7	1.9876e-08\\
9	5.1411e-10\\
10	3.5997e-09\\
12	6.6889e-10\\
15	9.2087e-12\\
};
\addlegendentry{error}

\addplot [color=black]
  table[row sep=crcr]{%
4e1   1e-03\\
2e2   1e-03\\
2e2   1.4e-04\\
4e1   1e-03\\
};
\node[color=black,right] at (2e2,4e-4) {$M^{-(2-\alpha)}$};
\addplot [color=black]
  table[row sep=crcr]{%
3e0   5e-09\\
3e0   3e-11\\
9e0  3e-11\\
3e0   5e-09\\
};
\node[color=black,right] at (0.95e0,3e-10) {$M^{-(5-\alpha)}$};
\end{axis}

\end{tikzpicture}%

%% file: src/meshsizes.tex
%
%
%
\begin{tikzpicture}

\begin{axis}[%
width=0.6\textwidth,
height=0.28\textwidth,
scale only axis,
xmin=0,xmax=1,
xlabel={$t$},xtick={0,0.5,1},
ymin=0,
ylabel={$\Delta t_k$},
axis background/.style={fill=white},
legend style={at={(1.03,0.5)}, anchor=west, legend cell align=left, align=left, draw=white!15!black}
]
\addplot [color=blue,thick]
  table[row sep=crcr]{%
0	7.0448e-11\\
7.0448e-11	7.0448e-11\\
7.0448e-11	3.635e-10\\
4.3395e-10	3.635e-10\\
4.3395e-10	1.0854e-09\\
1.5193e-09	1.0854e-09\\
1.5193e-09	2.7008e-09\\
4.2202e-09	2.7008e-09\\
4.2202e-09	4.667e-09\\
8.8872e-09	4.667e-09\\
8.8872e-09	8.0646e-09\\
1.6952e-08	8.0646e-09\\
1.6952e-08	1.3936e-08\\
3.0887e-08	1.3936e-08\\
3.0887e-08	2.0067e-08\\
5.0955e-08	2.0067e-08\\
5.0955e-08	2.8897e-08\\
7.9852e-08	2.8897e-08\\
7.9852e-08	4.1612e-08\\
1.2146e-07	4.1612e-08\\
1.2146e-07	4.9934e-08\\
1.714e-07	4.9934e-08\\
1.714e-07	7.1905e-08\\
2.433e-07	7.1905e-08\\
2.433e-07	8.6286e-08\\
3.2959e-07	8.6286e-08\\
3.2959e-07	1.2425e-07\\
4.5384e-07	1.2425e-07\\
4.5384e-07	1.491e-07\\
6.0294e-07	1.491e-07\\
6.0294e-07	1.7892e-07\\
7.8186e-07	1.7892e-07\\
7.8186e-07	2.1471e-07\\
9.9657e-07	2.1471e-07\\
9.9657e-07	2.5765e-07\\
1.2542e-06	2.5765e-07\\
1.2542e-06	3.0918e-07\\
1.5634e-06	3.0918e-07\\
1.5634e-06	3.7101e-07\\
1.9344e-06	3.7101e-07\\
1.9344e-06	4.4521e-07\\
2.3796e-06	4.4521e-07\\
2.3796e-06	5.3426e-07\\
2.9139e-06	5.3426e-07\\
2.9139e-06	6.4111e-07\\
3.555e-06	6.4111e-07\\
3.555e-06	7.6933e-07\\
4.3243e-06	7.6933e-07\\
4.3243e-06	9.232e-07\\
5.2475e-06	9.232e-07\\
5.2475e-06	1.1078e-06\\
6.3554e-06	1.1078e-06\\
6.3554e-06	1.3294e-06\\
7.6848e-06	1.3294e-06\\
7.6848e-06	1.3294e-06\\
9.0142e-06	1.3294e-06\\
9.0142e-06	1.5953e-06\\
1.0609e-05	1.5953e-06\\
1.0609e-05	1.9143e-06\\
1.2524e-05	1.9143e-06\\
1.2524e-05	2.2972e-06\\
1.4821e-05	2.2972e-06\\
1.4821e-05	2.7567e-06\\
1.7578e-05	2.7567e-06\\
1.7578e-05	2.7567e-06\\
2.0334e-05	2.7567e-06\\
2.0334e-05	3.308e-06\\
2.3642e-05	3.308e-06\\
2.3642e-05	3.9696e-06\\
2.7612e-05	3.9696e-06\\
2.7612e-05	4.7635e-06\\
3.2375e-05	4.7635e-06\\
3.2375e-05	4.7635e-06\\
3.7139e-05	4.7635e-06\\
3.7139e-05	5.7162e-06\\
4.2855e-05	5.7162e-06\\
4.2855e-05	6.8594e-06\\
4.9715e-05	6.8594e-06\\
4.9715e-05	6.8594e-06\\
5.6574e-05	6.8594e-06\\
5.6574e-05	8.2313e-06\\
6.4805e-05	8.2313e-06\\
6.4805e-05	9.8776e-06\\
7.4683e-05	9.8776e-06\\
7.4683e-05	1.1853e-05\\
8.6536e-05	1.1853e-05\\
8.6536e-05	1.1853e-05\\
9.8389e-05	1.1853e-05\\
9.8389e-05	1.4224e-05\\
0.00011261	1.4224e-05\\
0.00011261	1.7068e-05\\
0.00012968	1.7068e-05\\
0.00012968	1.7068e-05\\
0.00014675	1.7068e-05\\
0.00014675	2.0482e-05\\
0.00016723	2.0482e-05\\
0.00016723	2.4579e-05\\
0.00019181	2.4579e-05\\
0.00019181	2.4579e-05\\
0.00021639	2.4579e-05\\
0.00021639	2.9494e-05\\
0.00024588	2.9494e-05\\
0.00024588	3.5393e-05\\
0.00028128	3.5393e-05\\
0.00028128	3.5393e-05\\
0.00031667	3.5393e-05\\
0.00031667	4.2472e-05\\
0.00035914	4.2472e-05\\
0.00035914	5.0966e-05\\
0.00041011	5.0966e-05\\
0.00041011	6.1159e-05\\
0.00047127	6.1159e-05\\
0.00047127	6.1159e-05\\
0.00053243	6.1159e-05\\
0.00053243	7.3391e-05\\
0.00060582	7.3391e-05\\
0.00060582	8.807e-05\\
0.00069389	8.807e-05\\
0.00069389	8.807e-05\\
0.00078196	8.807e-05\\
0.00078196	0.00010568\\
0.00088764	0.00010568\\
0.00088764	0.00012682\\
0.0010145	0.00012682\\
0.0010145	0.00015218\\
0.0011666	0.00015218\\
0.0011666	0.00015218\\
0.0013188	0.00015218\\
0.0013188	0.00018262\\
0.0015015	0.00018262\\
0.0015015	0.00021915\\
0.0017206	0.00021915\\
0.0017206	0.00026297\\
0.0019836	0.00026297\\
0.0019836	0.00031557\\
0.0022991	0.00031557\\
0.0022991	0.00031557\\
0.0026147	0.00031557\\
0.0026147	0.00037868\\
0.0029934	0.00037868\\
0.0029934	0.00045442\\
0.0034478	0.00045442\\
0.0034478	0.0005453\\
0.0039931	0.0005453\\
0.0039931	0.00065436\\
0.0046475	0.00065436\\
0.0046475	0.00078524\\
0.0054327	0.00078524\\
0.0054327	0.00094228\\
0.006375	0.00094228\\
0.006375	0.0011307\\
0.0075057	0.0011307\\
0.0075057	0.0013569\\
0.0088626	0.0013569\\
0.0088626	0.0016283\\
0.010491	0.0016283\\
0.010491	0.0019539\\
0.012445	0.0019539\\
0.012445	0.0023447\\
0.01479	0.0023447\\
0.01479	0.0028136\\
0.017603	0.0028136\\
0.017603	0.0033764\\
0.02098	0.0033764\\
0.02098	0.0040517\\
0.025031	0.0040517\\
0.025031	0.0058344\\
0.030866	0.0058344\\
0.030866	0.0070013\\
0.037867	0.0070013\\
0.037867	0.0084015\\
0.046268	0.0084015\\
0.046268	0.012098\\
0.058366	0.012098\\
0.058366	0.014518\\
0.072884	0.014518\\
0.072884	0.020906\\
0.09379	0.020906\\
0.09379	0.025087\\
0.11888	0.025087\\
0.11888	0.036125\\
0.155	0.036125\\
0.155	0.05202\\
0.20702	0.05202\\
0.20702	0.074909\\
0.28193	0.074909\\
0.28193	0.08989\\
0.37182	0.08989\\
0.37182	0.014518\\
0.38634	0.014518\\
0.38634	0.0084015\\
0.39474	0.0084015\\
0.39474	0.0070013\\
0.40174	0.0070013\\
0.40174	0.0058344\\
0.40758	0.0058344\\
0.40758	0.004862\\
0.41244	0.004862\\
0.41244	0.0040517\\
0.41649	0.0040517\\
0.41649	0.0033764\\
0.41987	0.0033764\\
0.41987	0.0033764\\
0.42324	0.0033764\\
0.42324	0.0033764\\
0.42662	0.0033764\\
0.42662	0.0028136\\
0.42943	0.0028136\\
0.42943	0.0028136\\
0.43225	0.0028136\\
0.43225	0.0028136\\
0.43506	0.0028136\\
0.43506	0.0028136\\
0.43787	0.0028136\\
0.43787	0.0023447\\
0.44022	0.0023447\\
0.44022	0.0023447\\
0.44256	0.0023447\\
0.44256	0.0023447\\
0.44491	0.0023447\\
0.44491	0.0023447\\
0.44725	0.0023447\\
0.44725	0.0023447\\
0.4496	0.0023447\\
0.4496	0.0028136\\
0.45241	0.0028136\\
0.45241	0.0028136\\
0.45522	0.0028136\\
0.45522	0.0028136\\
0.45804	0.0028136\\
0.45804	0.0033764\\
0.46141	0.0033764\\
0.46141	0.0040517\\
0.46547	0.0040517\\
0.46547	0.0084015\\
0.47387	0.0084015\\
0.47387	0.004862\\
0.47873	0.004862\\
0.47873	0.0028136\\
0.48154	0.0028136\\
0.48154	0.0028136\\
0.48436	0.0028136\\
0.48436	0.0023447\\
0.4867	0.0023447\\
0.4867	0.0019539\\
0.48865	0.0019539\\
0.48865	0.0019539\\
0.49061	0.0019539\\
0.49061	0.0019539\\
0.49256	0.0019539\\
0.49256	0.0019539\\
0.49452	0.0019539\\
0.49452	0.0016283\\
0.49614	0.0016283\\
0.49614	0.0016283\\
0.49777	0.0016283\\
0.49777	0.0016283\\
0.4994	0.0016283\\
0.4994	0.0016283\\
0.50103	0.0016283\\
0.50103	0.0016283\\
0.50266	0.0016283\\
0.50266	0.0016283\\
0.50429	0.0016283\\
0.50429	0.0016283\\
0.50591	0.0016283\\
0.50591	0.0016283\\
0.50754	0.0016283\\
0.50754	0.0016283\\
0.50917	0.0016283\\
0.50917	0.0016283\\
0.5108	0.0016283\\
0.5108	0.0016283\\
0.51243	0.0016283\\
0.51243	0.0016283\\
0.51406	0.0016283\\
0.51406	0.0016283\\
0.51568	0.0016283\\
0.51568	0.0016283\\
0.51731	0.0016283\\
0.51731	0.0016283\\
0.51894	0.0016283\\
0.51894	0.0019539\\
0.52089	0.0019539\\
0.52089	0.0019539\\
0.52285	0.0019539\\
0.52285	0.0019539\\
0.5248	0.0019539\\
0.5248	0.0019539\\
0.52676	0.0019539\\
0.52676	0.0023447\\
0.5291	0.0023447\\
0.5291	0.0023447\\
0.53145	0.0023447\\
0.53145	0.0028136\\
0.53426	0.0028136\\
0.53426	0.0033764\\
0.53764	0.0033764\\
0.53764	0.0070013\\
0.54464	0.0070013\\
0.54464	0.0058344\\
0.55047	0.0058344\\
0.55047	0.0040517\\
0.55452	0.0040517\\
0.55452	0.0033764\\
0.5579	0.0033764\\
0.5579	0.0028136\\
0.56071	0.0028136\\
0.56071	0.0028136\\
0.56353	0.0028136\\
0.56353	0.0028136\\
0.56634	0.0028136\\
0.56634	0.0028136\\
0.56915	0.0028136\\
0.56915	0.0028136\\
0.57197	0.0028136\\
0.57197	0.0028136\\
0.57478	0.0028136\\
0.57478	0.0028136\\
0.57759	0.0028136\\
0.57759	0.0028136\\
0.58041	0.0028136\\
0.58041	0.0028136\\
0.58322	0.0028136\\
0.58322	0.0033764\\
0.5866	0.0033764\\
0.5866	0.0033764\\
0.58997	0.0033764\\
0.58997	0.0033764\\
0.59335	0.0033764\\
0.59335	0.0040517\\
0.5974	0.0040517\\
0.5974	0.0040517\\
0.60145	0.0040517\\
0.60145	0.004862\\
0.60632	0.004862\\
0.60632	0.0058344\\
0.61215	0.0058344\\
0.61215	0.0058344\\
0.61799	0.0058344\\
0.61799	0.0084015\\
0.62639	0.0084015\\
0.62639	0.010082\\
0.63647	0.010082\\
0.63647	0.014518\\
0.65099	0.014518\\
0.65099	0.020906\\
0.67189	0.020906\\
0.67189	0.036125\\
0.70802	0.036125\\
0.70802	0.05202\\
0.76004	0.05202\\
0.76004	0.08989\\
0.84993	0.08989\\
0.84993	0.15007\\
1	0.15007\\
};
\addlegendentry{L1}
\addplot [color=black,thick]
  table[row sep=crcr]{%
0	2.4461e-10\\
2.4461e-10	2.4461e-10\\
2.4461e-10	3.1406e-09\\
3.3852e-09	3.1406e-09\\
3.3852e-09	1.9446e-08\\
2.2831e-08	1.9446e-08\\
2.2831e-08	6.9678e-08\\
9.2509e-08	6.9678e-08\\
9.2509e-08	2.0806e-07\\
3.0057e-07	2.0806e-07\\
3.0057e-07	5.1771e-07\\
8.1828e-07	5.1771e-07\\
8.1828e-07	1.0735e-06\\
1.8918e-06	1.0735e-06\\
1.8918e-06	2.2261e-06\\
4.1179e-06	2.2261e-06\\
4.1179e-06	4.616e-06\\
8.7339e-06	4.616e-06\\
8.7339e-06	7.9764e-06\\
1.671e-05	7.9764e-06\\
1.671e-05	1.3783e-05\\
3.0494e-05	1.3783e-05\\
3.0494e-05	2.3817e-05\\
5.4311e-05	2.3817e-05\\
5.4311e-05	4.1157e-05\\
9.5468e-05	4.1157e-05\\
9.5468e-05	7.1119e-05\\
0.00016659	7.1119e-05\\
0.00016659	0.00012289\\
0.00028948	0.00012289\\
0.00028948	0.00021236\\
0.00050184	0.00021236\\
0.00050184	0.00036696\\
0.00086879	0.00036696\\
0.00086879	0.0006341\\
0.0015029	0.0006341\\
0.0015029	0.0010957\\
0.0025986	0.0010957\\
0.0025986	0.0018934\\
0.004492	0.0018934\\
0.004492	0.0032718\\
0.0077639	0.0032718\\
0.0077639	0.0067844\\
0.014548	0.0067844\\
0.014548	0.014068\\
0.028617	0.014068\\
0.028617	0.029172\\
0.057788	0.029172\\
0.057788	0.060491\\
0.11828	0.060491\\
0.11828	0.15052\\
0.2688	0.15052\\
0.2688	0.10453\\
0.37333	0.10453\\
0.37333	0.029172\\
0.4025	0.029172\\
0.4025	0.020258\\
0.42276	0.020258\\
0.42276	0.020258\\
0.44302	0.020258\\
0.44302	0.016882\\
0.4599	0.016882\\
0.4599	0.011724\\
0.47162	0.011724\\
0.47162	0.011724\\
0.48335	0.011724\\
0.48335	0.011724\\
0.49507	0.011724\\
0.49507	0.020258\\
0.51533	0.020258\\
0.51533	0.011724\\
0.52705	0.011724\\
0.52705	0.011724\\
0.53877	0.011724\\
0.53877	0.011724\\
0.5505	0.011724\\
0.5505	0.016882\\
0.56738	0.016882\\
0.56738	0.02431\\
0.59169	0.02431\\
0.59169	0.020258\\
0.61195	0.020258\\
0.61195	0.029172\\
0.64112	0.029172\\
0.64112	0.072589\\
0.71371	0.072589\\
0.71371	0.21675\\
0.93046	0.21675\\
0.93046	0.069542\\
1	0.069542\\
};
\addlegendentry{coll(2)}
\addplot [color=teal,thick]
  table[row sep=crcr]{%
0	7.0779e-10\\
7.0779e-10	7.0779e-10\\
7.0779e-10	1.8844e-08\\
1.9551e-08	1.8844e-08\\
1.9551e-08	2.0161e-07\\
2.2117e-07	2.0161e-07\\
2.2117e-07	1.2483e-06\\
1.4695e-06	1.2483e-06\\
1.4695e-06	6.4412e-06\\
7.9107e-06	6.4412e-06\\
7.9107e-06	2.7696e-05\\
3.5607e-05	2.7696e-05\\
3.5607e-05	9.924e-05\\
0.00013485	9.924e-05\\
0.00013485	0.00035559\\
0.00049044	0.00035559\\
0.00049044	0.0012742\\
0.0017646	0.0012742\\
0.0017646	0.0045655\\
0.0063301	0.0045655\\
0.0063301	0.016359\\
0.022689	0.016359\\
0.022689	0.070341\\
0.09303	0.070341\\
0.09303	0.25205\\
0.34508	0.25205\\
0.34508	0.070341\\
0.41542	0.070341\\
0.41542	0.048848\\
0.46426	0.048848\\
0.46426	0.040707\\
0.50497	0.040707\\
0.50497	0.040707\\
0.54568	0.040707\\
0.54568	0.048848\\
0.59453	0.048848\\
0.59453	0.084409\\
0.67894	0.084409\\
0.67894	0.32106\\
1	0.32106\\
};
\addlegendentry{coll(4)}
\addplot [color=magenta,thick]
  table[row sep=crcr]{%
0	1.7612e-09\\
1.7612e-09	1.7612e-09\\
1.7612e-09	9.7229e-08\\
9.8991e-08	9.7229e-08\\
9.8991e-08	1.7976e-06\\
1.8966e-06	1.7976e-06\\
1.8966e-06	2.308e-05\\
2.4977e-05	2.308e-05\\
2.4977e-05	0.00020578\\
0.00023076	0.00020578\\
0.00023076	0.0018348\\
0.0020655	0.0018348\\
0.0020655	0.016359\\
0.018425	0.016359\\
0.018425	0.17503\\
0.19346	0.17503\\
0.19346	0.21004\\
0.40349	0.21004\\
0.40349	0.12155\\
0.52504	0.12155\\
0.52504	0.12155\\
0.64659	0.12155\\
0.64659	0.35341\\
1	0.35341\\
};
\addlegendentry{coll(8)}

\end{axis}

\end{tikzpicture}%

%% file: src/meshsizes_2.tex
%
%
%
\begin{tikzpicture}

\begin{axis}[%
width=0.6\textwidth,
height=0.28\textwidth,
scale only axis,
xmin=0,xmax=1,
xlabel={$t$},xtick={0,0.5,1},
ymin=0,
ylabel={$\Delta t_k$},yticklabels={,0,0.05,0.10,0.15,0.20},
axis background/.style={fill=white},
legend style={at={(1.03,0.5)}, anchor=west, legend cell align=left, align=left, draw=white!15!black}
]

\addplot [color=blue,thick]
  table[row sep=crcr]{%
0	1.911e-05\\
1.911e-05	1.911e-05\\
1.911e-05	4.7553e-05\\
6.6663e-05	4.7553e-05\\
6.6663e-05	8.2171e-05\\
0.00014883	8.2171e-05\\
0.00014883	0.00011833\\
0.00026716	0.00011833\\
0.00026716	0.00017039\\
0.00043755	0.00017039\\
0.00043755	0.00020447\\
0.00064202	0.00020447\\
0.00064202	0.00024536\\
0.00088738	0.00024536\\
0.00088738	0.00029443\\
0.0011818	0.00029443\\
0.0011818	0.00035332\\
0.0015351	0.00035332\\
0.0015351	0.00042398\\
0.0019591	0.00042398\\
0.0019591	0.00050878\\
0.0024679	0.00050878\\
0.0024679	0.00061054\\
0.0030784	0.00061054\\
0.0030784	0.00073264\\
0.0038111	0.00073264\\
0.0038111	0.00087917\\
0.0046902	0.00087917\\
0.0046902	0.001055\\
0.0057452	0.001055\\
0.0057452	0.001266\\
0.0070113	0.001266\\
0.0070113	0.0015192\\
0.0085305	0.0015192\\
0.0085305	0.0018231\\
0.010354	0.0018231\\
0.010354	0.0021877\\
0.012541	0.0021877\\
0.012541	0.0026252\\
0.015166	0.0026252\\
0.015166	0.0031502\\
0.018317	0.0031502\\
0.018317	0.0037803\\
0.022097	0.0037803\\
0.022097	0.0045363\\
0.026633	0.0045363\\
0.026633	0.0054436\\
0.032077	0.0054436\\
0.032077	0.0065323\\
0.038609	0.0065323\\
0.038609	0.0078388\\
0.046448	0.0078388\\
0.046448	0.0094065\\
0.055854	0.0094065\\
0.055854	0.013545\\
0.0694	0.013545\\
0.0694	0.016254\\
0.085654	0.016254\\
0.085654	0.019505\\
0.10516	0.019505\\
0.10516	0.028088\\
0.13325	0.028088\\
0.13325	0.033705\\
0.16695	0.033705\\
0.16695	0.048536\\
0.21549	0.048536\\
0.21549	0.08387\\
0.29936	0.08387\\
0.29936	0.069891\\
0.36925	0.069891\\
0.36925	0.019505\\
0.38875	0.019505\\
0.38875	0.011288\\
0.40004	0.011288\\
0.40004	0.0094065\\
0.40945	0.0094065\\
0.40945	0.0078388\\
0.41729	0.0078388\\
0.41729	0.0078388\\
0.42513	0.0078388\\
0.42513	0.0065323\\
0.43166	0.0065323\\
0.43166	0.0065323\\
0.43819	0.0065323\\
0.43819	0.0078388\\
0.44603	0.0078388\\
0.44603	0.016254\\
0.46228	0.016254\\
0.46228	0.0065323\\
0.46882	0.0065323\\
0.46882	0.0054436\\
0.47426	0.0054436\\
0.47426	0.0045363\\
0.4788	0.0045363\\
0.4788	0.0045363\\
0.48333	0.0045363\\
0.48333	0.0037803\\
0.48711	0.0037803\\
0.48711	0.0037803\\
0.49089	0.0037803\\
0.49089	0.0037803\\
0.49467	0.0037803\\
0.49467	0.0037803\\
0.49845	0.0037803\\
0.49845	0.0045363\\
0.50299	0.0045363\\
0.50299	0.0045363\\
0.50753	0.0045363\\
0.50753	0.0054436\\
0.51297	0.0054436\\
0.51297	0.0078388\\
0.52081	0.0078388\\
0.52081	0.0078388\\
0.52865	0.0078388\\
0.52865	0.0054436\\
0.53409	0.0054436\\
0.53409	0.0054436\\
0.53954	0.0054436\\
0.53954	0.0045363\\
0.54407	0.0045363\\
0.54407	0.0045363\\
0.54861	0.0045363\\
0.54861	0.0045363\\
0.55314	0.0045363\\
0.55314	0.0045363\\
0.55768	0.0045363\\
0.55768	0.0045363\\
0.56222	0.0045363\\
0.56222	0.0054436\\
0.56766	0.0054436\\
0.56766	0.0054436\\
0.5731	0.0054436\\
0.5731	0.0065323\\
0.57964	0.0065323\\
0.57964	0.0094065\\
0.58904	0.0094065\\
0.58904	0.016254\\
0.6053	0.016254\\
0.6053	0.011288\\
0.61659	0.011288\\
0.61659	0.011288\\
0.62787	0.011288\\
0.62787	0.013545\\
0.64142	0.013545\\
0.64142	0.016254\\
0.65767	0.016254\\
0.65767	0.019505\\
0.67718	0.019505\\
0.67718	0.023406\\
0.70059	0.023406\\
0.70059	0.033705\\
0.73429	0.033705\\
0.73429	0.048536\\
0.78283	0.048536\\
0.78283	0.069891\\
0.85272	0.069891\\
0.85272	0.14728\\
1	0.14728\\
};
\addlegendentry{$\gamma=0$}

\addplot [color=black,thick]
  table[row sep=crcr]{%
0	0.0016589\\
0.0016589	0.0016589\\
0.0016589	0.0034399\\
0.0050987	0.0034399\\
0.0050987	0.0049534\\
0.010052	0.0049534\\
0.010052	0.0085595\\
0.018612	0.0085595\\
0.018612	0.012326\\
0.030937	0.012326\\
0.030937	0.021299\\
0.052236	0.021299\\
0.052236	0.03067\\
0.082906	0.03067\\
0.082906	0.044165\\
0.12707	0.044165\\
0.12707	0.063597\\
0.19067	0.063597\\
0.19067	0.09158\\
0.28225	0.09158\\
0.28225	0.09158\\
0.37383	0.09158\\
0.37383	0.017749\\
0.39158	0.017749\\
0.39158	0.012326\\
0.4039	0.012326\\
0.4039	0.0085595\\
0.41246	0.0085595\\
0.41246	0.0071329\\
0.4196	0.0071329\\
0.4196	0.0071329\\
0.42673	0.0071329\\
0.42673	0.0071329\\
0.43386	0.0071329\\
0.43386	0.0071329\\
0.44099	0.0071329\\
0.44099	0.0085595\\
0.44955	0.0085595\\
0.44955	0.014791\\
0.46434	0.014791\\
0.46434	0.0059441\\
0.47029	0.0059441\\
0.47029	0.0049534\\
0.47524	0.0049534\\
0.47524	0.0041278\\
0.47937	0.0041278\\
0.47937	0.0041278\\
0.4835	0.0041278\\
0.4835	0.0041278\\
0.48763	0.0041278\\
0.48763	0.0041278\\
0.49175	0.0041278\\
0.49175	0.0041278\\
0.49588	0.0041278\\
0.49588	0.0041278\\
0.50001	0.0041278\\
0.50001	0.0041278\\
0.50414	0.0041278\\
0.50414	0.0049534\\
0.50909	0.0049534\\
0.50909	0.0059441\\
0.51503	0.0059441\\
0.51503	0.012326\\
0.52736	0.012326\\
0.52736	0.0059441\\
0.5333	0.0059441\\
0.5333	0.0049534\\
0.53826	0.0049534\\
0.53826	0.0049534\\
0.54321	0.0049534\\
0.54321	0.0049534\\
0.54816	0.0049534\\
0.54816	0.0049534\\
0.55312	0.0049534\\
0.55312	0.0049534\\
0.55807	0.0049534\\
0.55807	0.0049534\\
0.56302	0.0049534\\
0.56302	0.0049534\\
0.56798	0.0049534\\
0.56798	0.0059441\\
0.57392	0.0059441\\
0.57392	0.0071329\\
0.58105	0.0071329\\
0.58105	0.012326\\
0.59338	0.012326\\
0.59338	0.014791\\
0.60817	0.014791\\
0.60817	0.012326\\
0.6205	0.012326\\
0.6205	0.012326\\
0.63282	0.012326\\
0.63282	0.012326\\
0.64515	0.012326\\
0.64515	0.014791\\
0.65994	0.014791\\
0.65994	0.021299\\
0.68124	0.021299\\
0.68124	0.025558\\
0.7068	0.025558\\
0.7068	0.036804\\
0.7436	0.036804\\
0.7436	0.052998\\
0.7966	0.052998\\
0.7966	0.09158\\
0.88818	0.09158\\
0.88818	0.11182\\
1	0.11182\\
};
\addlegendentry{$\gamma=0.4$}

\addplot [color=teal,thick]
  table[row sep=crcr]{%
0	0.017199\\
0.017199	0.017199\\
0.017199	0.042797\\
0.059997	0.042797\\
0.059997	0.10649\\
0.16649	0.10649\\
0.16649	0.18402\\
0.35051	0.18402\\
0.35051	0.02972\\
0.38023	0.02972\\
0.38023	0.014333\\
0.39456	0.014333\\
0.39456	0.0099533\\
0.40452	0.0099533\\
0.40452	0.0082944\\
0.41281	0.0082944\\
0.41281	0.0082944\\
0.42111	0.0082944\\
0.42111	0.006912\\
0.42802	0.006912\\
0.42802	0.006912\\
0.43493	0.006912\\
0.43493	0.006912\\
0.44184	0.006912\\
0.44184	0.0082944\\
0.45014	0.0082944\\
0.45014	0.014333\\
0.46447	0.014333\\
0.46447	0.00576\\
0.47023	0.00576\\
0.47023	0.0048\\
0.47503	0.0048\\
0.47503	0.0048\\
0.47983	0.0048\\
0.47983	0.004\\
0.48383	0.004\\
0.48383	0.004\\
0.48783	0.004\\
0.48783	0.004\\
0.49183	0.004\\
0.49183	0.004\\
0.49583	0.004\\
0.49583	0.004\\
0.49983	0.004\\
0.49983	0.004\\
0.50383	0.004\\
0.50383	0.0048\\
0.50863	0.0048\\
0.50863	0.00576\\
0.51439	0.00576\\
0.51439	0.011944\\
0.52633	0.011944\\
0.52633	0.00576\\
0.53209	0.00576\\
0.53209	0.0048\\
0.53689	0.0048\\
0.53689	0.0048\\
0.54169	0.0048\\
0.54169	0.0048\\
0.54649	0.0048\\
0.54649	0.0048\\
0.55129	0.0048\\
0.55129	0.0048\\
0.55609	0.0048\\
0.55609	0.0048\\
0.56089	0.0048\\
0.56089	0.0048\\
0.56569	0.0048\\
0.56569	0.00576\\
0.57145	0.00576\\
0.57145	0.006912\\
0.57836	0.006912\\
0.57836	0.0082944\\
0.58666	0.0082944\\
0.58666	0.017199\\
0.60386	0.017199\\
0.60386	0.011944\\
0.6158	0.011944\\
0.6158	0.011944\\
0.62775	0.011944\\
0.62775	0.011944\\
0.63969	0.011944\\
0.63969	0.014333\\
0.65402	0.014333\\
0.65402	0.017199\\
0.67122	0.017199\\
0.67122	0.024767\\
0.69599	0.024767\\
0.69599	0.02972\\
0.72571	0.02972\\
0.72571	0.042797\\
0.76851	0.042797\\
0.76851	0.061628\\
0.83013	0.061628\\
0.83013	0.12779\\
0.95793	0.12779\\
0.95793	0.042073\\
1	0.042073\\
};
\addlegendentry{$\gamma=0.8$}

\end{axis}

\end{tikzpicture}%

%% file: src/NVsQ.tex
%
%
%
\begin{tikzpicture}

\begin{axis}[%
width=0.4\textwidth,
height=0.2\textwidth,
scale only axis,
xmin=1,xmax=1.5,
xlabel style={font=\color{white!15!black}},
xlabel={$Q_1$},
ymin=0,
ylabel style={font=\color{white!15!black}},
ylabel={$M$},
axis background/.style={fill=white},
axis x line*=bottom,
axis y line*=left,
legend style={at={(1.03,0.5)}, anchor=west, legend cell align=left, align=left, draw=white!15!black}
]
\addplot [color=blue,thick]
  table[row sep=crcr]{%
1.01	166\\
1.025	167\\
1.05	169\\
1.075	171\\
1.1	172\\
1.2	178\\
1.3	183\\
1.5	197\\
};
\addlegendentry{L1}
\addplot [color=black,thick]
  table[row sep=crcr]{%
1.01	41\\
1.025	42\\
1.05	42\\
1.075	42\\
1.1	43\\
1.2	45\\
1.3	46\\
1.5	50\\
};
\addlegendentry{coll(2)}
\addplot [color=teal,thick]
  table[row sep=crcr]{%
1.01	19\\
1.025	19\\
1.05	19\\
1.075	19\\
1.1	19\\
1.2	20\\
1.3	20\\
1.5	21\\
};
\addlegendentry{coll(4)}
\addplot [color=magenta,thick]
  table[row sep=crcr]{%
1.01	12\\
1.025	12\\
1.05	12\\
1.075	12\\
1.1	12\\
1.2	12\\
1.3	13\\
1.5	13\\
};
\addlegendentry{coll(8)}
\end{axis}
\end{tikzpicture}%

%% file: src/iterVsQ.tex
%
%
%
\begin{tikzpicture}

\begin{axis}[%
width=0.25\textwidth,
height=0.2\textwidth,
scale only axis,
xmin=1,
xmax=1.5,
xlabel style={font=\color{white!15!black}},
xlabel={$Q_1$},
ymin=0,
ymax=2008,
ylabel style={font=\color{white!15!black}},
ylabel={\#iter},
axis background/.style={fill=white},
axis x line*=bottom,
axis y line*=left,
legend style={legend cell align=left, align=left, draw=white!15!black}
]
\addplot [color=blue,thick]
  table[row sep=crcr]{%
1.01	3732\\
1.025	1703\\
1.05	1034\\
1.075	793\\
1.1	682\\
1.2	531\\
1.3	489\\
1.5	469\\
};
\addplot [color=black,thick]
  table[row sep=crcr]{%
1.01	2876\\
1.025	1217\\
1.05	659\\
1.075	477\\
1.1	370\\
1.2	237\\
1.3	192\\
1.5	167\\
};
\addplot [color=teal,thick]
  table[row sep=crcr]{%
1.01	2369\\
1.025	977\\
1.05	514\\
1.075	362\\
1.1	281\\
1.2	169\\
1.3	129\\
1.5	97\\
};
\addplot [color=magenta,thick]
  table[row sep=crcr]{%
1.01	2249\\
1.025	913\\
1.05	475\\
1.075	325\\
1.1	250\\
1.2	141\\
1.3	109\\
1.5	80\\
};
\end{axis}
\end{tikzpicture}%

%% file: src/timeVsQ.tex
%
%
%
\begin{tikzpicture}

\begin{axis}[%
width=0.25\textwidth,
height=0.2\textwidth,
scale only axis,
xmin=1,
xmax=1.5,
xlabel style={font=\color{white!15!black}},
xlabel={$Q_1$},
ymin=0,
ymax=250,
ylabel style={font=\color{white!15!black}},
ylabel={time},
axis background/.style={fill=white},
axis x line*=bottom,
axis y line*=left,
legend columns=1, 
legend style={at={(1.03,0.5)}, anchor=west, legend cell align=left, align=left, draw=white!15!black}
]
\addplot [color=blue,thick]
  table[row sep=crcr]{%
1.01	101.03\\
1.025	38.915\\
1.05	24.416\\
1.075	19.974\\
1.1	17.383\\
1.2	14.801\\
1.3	14.727\\
1.5	14.83\\
};
\addlegendentry{L1}
\addplot [color=black,thick]
  table[row sep=crcr]{%
1.01	119.74\\
1.025	37.063\\
1.05	18.709\\
1.075	13.777\\
1.1	10.678\\
1.2	7.2965\\
1.3	5.9161\\
1.5	5.4692\\
};
\addlegendentry{coll(2)}
\addplot [color=teal,thick]
  table[row sep=crcr]{%
1.01	217.58\\
1.025	52.002\\
1.05	25.905\\
1.075	18.477\\
1.1	13.746\\
1.2	8.5144\\
1.3	6.3815\\
1.5	4.9046\\
};
\addlegendentry{coll(4)}
\addplot [color=magenta,thick]
  table[row sep=crcr]{%
1.01	561.52\\
1.025	127.35\\
1.05	57.408\\
1.075	34.041\\
1.1	25.634\\
1.2	13.953\\
1.3	11.12\\
1.5	8.229\\
};
\addlegendentry{coll(8)}
\end{axis}
\end{tikzpicture}%

%% file: src/iterVsTOL.tex
%
%
%
\begin{tikzpicture}

\begin{axis}[%
width=0.25\textwidth,
height=0.2\textwidth,
scale only axis,
xmode=log,
xmin=1e-06,
xmax=0.01,
xminorticks=true,
xlabel style={font=\color{white!15!black}},
xlabel={$TOL$},
ymode=log,
ymin=10,
ymax=10000,
yminorticks=true,
ylabel style={font=\color{white!15!black}},
ylabel={\#iter},
axis background/.style={fill=white},
legend style={legend cell align=left, align=left, draw=white!15!black}
]
\addplot [color=blue,thick]
  table[row sep=crcr]{%
0.01	120\\
0.0031623	162\\
0.001	223\\
0.00031623	336\\
0.0001	531\\
3.1623e-05	891\\
1e-05	1614\\
3.1623e-06	3056\\
1e-06	6057\\
};
\addplot [color=black,thick]
  table[row sep=crcr]{%
0.01	91\\
0.0031623	117\\
0.001	140\\
0.00031623	183\\
0.0001	237\\
3.1623e-05	308\\
1e-05	388\\
3.1623e-06	517\\
1e-06	718\\
};
\addplot [color=teal,thick]
  table[row sep=crcr]{%
0.01	69\\
0.0031623	95\\
0.001	122\\
0.00031623	135\\
0.0001	169\\
3.1623e-05	197\\
1e-05	233\\
3.1623e-06	274\\
1e-06	313\\
};
\addplot [color=magenta,thick]
  table[row sep=crcr]{%
0.01	62\\
0.0031623	80\\
0.001	99\\
0.00031623	120\\
0.0001	141\\
3.1623e-05	160\\
1e-05	185\\
3.1623e-06	209\\
1e-06	234\\
};
\end{axis}
\end{tikzpicture}%

%% file: src/timeVsTOL.tex
%
%
%
\begin{tikzpicture}

\begin{axis}[%
width=0.25\textwidth,
height=0.2\textwidth,
scale only axis,
xmode=log,
xmin=1e-06,
xmax=0.01,
xminorticks=true,
xlabel style={font=\color{white!15!black}},
xlabel={$TOL$},
ymode=log,
ymin=1,
ymax=1400,
yminorticks=true,
ylabel style={font=\color{white!15!black}},
ylabel={time},
axis background/.style={fill=white},
legend columns=1, 
legend style={at={(1.03,0.5)}, anchor=west, legend cell align=left, align=left, draw=white!15!black}
]
\addplot [color=blue,thick]
  table[row sep=crcr]{%
0.01	1.9923\\
0.0031623	3.0637\\
0.001	4.9118\\
0.00031623	8.9588\\
0.0001	18.516\\
3.1623e-05	43.856\\
1e-05	125.38\\
3.1623e-06	386.2\\
1e-06	1349.4\\
};
\addlegendentry{L1}
\addplot [color=black,thick]
  table[row sep=crcr]{%
0.01	2.66\\
0.0031623	3.4007\\
0.001	4.5574\\
0.00031623	6.5949\\
0.0001	9.3772\\
3.1623e-05	13.489\\
1e-05	19.897\\
3.1623e-06	32.012\\
1e-06	56.086\\
};
\addlegendentry{coll(2)}
\addplot [color=teal,thick]
  table[row sep=crcr]{%
0.01	3.8328\\
0.0031623	5.1719\\
0.001	6.9649\\
0.00031623	8.357\\
0.0001	11.845\\
3.1623e-05	16.675\\
1e-05	18.205\\
3.1623e-06	21.85\\
1e-06	29.053\\
};
\addlegendentry{coll(4)}
\addplot [color=magenta,thick]
  table[row sep=crcr]{%
0.01	8.5316\\
0.0031623	12.298\\
0.001	13.267\\
0.00031623	17.012\\
0.0001	21.727\\
3.1623e-05	24.544\\
1e-05	29.562\\
3.1623e-06	34.15\\
1e-06	39.529\\
};
\addlegendentry{coll(8)}
\end{axis}

\end{tikzpicture}%

%% file: src/R0_res2.tex
%
%
\definecolor{mycolor1}{rgb}{0.00000,0.44700,0.74100}%
\begin{tikzpicture}

\begin{axis}[%
width=0.25\textwidth,
height=0.2\textwidth,
scale only axis,
unbounded coords=jump,
xmin=0,
xmax=1,
xlabel style={font=\color{white!15!black}},
xlabel={$t$},
ymode=log,
ymin=1e-16,
ymax=0.01,
yminorticks=true,
yticklabels={},
axis background/.style={fill=white}
]
\addplot [color=mycolor1, forget plot]
  table[row sep=crcr]{%
0	2.2204e-16\\
4e-06	0.0048479\\
8e-06	0.0053975\\
1.2e-05	0.0054597\\
1.6e-05	0.0053041\\
2e-05	0.0050294\\
2.4e-05	0.0046842\\
2.8e-05	0.0042968\\
3.2e-05	0.0038847\\
3.6e-05	0.0034602\\
4e-05	0.0030318\\
4.4e-05	0.0026059\\
4.8e-05	0.0021874\\
5.2e-05	0.0017801\\
5.6e-05	0.001387\\
6e-05	0.0010105\\
6.4e-05	0.00065279\\
6.8e-05	0.00031548\\
7.2e-05	2.2688e-16\\
7.6e-05	0.00029242\\
8e-05	0.00056074\\
8.4e-05	0.00080404\\
8.8e-05	0.0010215\\
9.2e-05	0.0012126\\
9.6e-05	0.0013765\\
0.0001	0.0015127\\
0.000104	0.0016209\\
0.000108	0.0017004\\
0.000112	0.0017511\\
0.000116	0.0017724\\
0.00012	0.0017641\\
0.000124	0.0017259\\
0.000128	0.0016576\\
0.000132	0.0015588\\
0.000136	0.0014295\\
0.00014	0.0012693\\
0.000144	0.0010781\\
0.000148	0.00085569\\
0.000152	0.00060198\\
0.000156	0.0003168\\
0.00016	2.2204e-16\\
0.0001896	0.00083829\\
0.0002192	0.0012927\\
0.0002488	0.0015563\\
0.0002784	0.0016954\\
0.000308	0.0017468\\
0.00033761	0.0017344\\
0.00036721	0.0016745\\
0.00039681	0.001579\\
0.00042641	0.0014572\\
0.00045601	0.0013162\\
0.00048561	0.0011679\\
0.00051521	0.0010126\\
0.00054481	0.0008481\\
0.00057441	0.0006781\\
0.00060401	0.00050573\\
0.00063362	0.00033372\\
0.00066322	0.00016444\\
0.00069282	2.3605e-16\\
0.00072242	0.00015774\\
0.00075202	0.00030713\\
0.00078162	0.0004467\\
0.00081122	0.0005751\\
0.00084082	0.00069114\\
0.00087042	0.00079374\\
0.00090002	0.0008819\\
0.00092963	0.00095472\\
0.00095923	0.0010114\\
0.00098883	0.0010511\\
0.0010184	0.0010733\\
0.001048	0.0010772\\
0.0010776	0.0010622\\
0.0011072	0.0010278\\
0.0011368	0.00097356\\
0.0011664	0.0008989\\
0.001196	0.00080342\\
0.0012256	0.00068669\\
0.0012552	0.00054834\\
0.0012848	0.000388\\
0.0013144	0.00020533\\
0.001344	2.2204e-16\\
0.0014462	0.00039725\\
0.0015484	0.00065903\\
0.0016506	0.00083517\\
0.0017528	0.00094727\\
0.001855	0.0010092\\
0.0019572	0.001031\\
0.0020594	0.0010204\\
0.0021616	0.00098362\\
0.0022638	0.00092573\\
0.0023659	0.00085104\\
0.0024681	0.00076319\\
0.0025703	0.00066532\\
0.0026725	0.00056016\\
0.0027747	0.0004501\\
0.0028769	0.00033728\\
0.0029791	0.00022357\\
0.0030813	0.00011064\\
0.0031835	2.5035e-16\\
0.0032857	0.00010698\\
0.0033879	0.00020906\\
0.00349	0.00030512\\
0.0035922	0.00039414\\
0.0036944	0.00047519\\
0.0037966	0.00054739\\
0.0038988	0.00060996\\
0.004001	0.00066217\\
0.0041032	0.00070334\\
0.0042054	0.00073285\\
0.0043076	0.00075012\\
0.0044098	0.0007546\\
0.004512	0.00074579\\
0.0046141	0.00072323\\
0.0047163	0.00068647\\
0.0048185	0.0006351\\
0.0049207	0.00056873\\
0.0050229	0.00048701\\
0.0051251	0.00038958\\
0.0052273	0.00027614\\
0.0053295	0.00014638\\
0.0054317	2.2204e-16\\
0.0056967	0.00027464\\
0.0059618	0.00046512\\
0.0062268	0.00059988\\
0.0064919	0.00069084\\
0.0067569	0.00074597\\
0.007022	0.00077133\\
0.0072871	0.00077179\\
0.0075521	0.00075144\\
0.0078172	0.00071377\\
0.0080822	0.00066181\\
0.0083473	0.00059823\\
0.0086123	0.0005254\\
0.0088774	0.00044545\\
0.0091425	0.00036029\\
0.0094075	0.00027165\\
0.0096726	0.00018112\\
0.0099376	9.0124e-05\\
0.010203	2.2204e-16\\
0.010468	8.804e-05\\
0.010733	0.00017287\\
0.010998	0.00025346\\
0.011263	0.00032883\\
0.011528	0.0003981\\
0.011793	0.00046042\\
0.012058	0.00051502\\
0.012323	0.00056116\\
0.012588	0.00059816\\
0.012853	0.00062538\\
0.013118	0.00064222\\
0.013383	0.0006481\\
0.013648	0.0006425\\
0.013913	0.0006249\\
0.014179	0.00059483\\
0.014444	0.00055184\\
0.014709	0.0004955\\
0.014974	0.0004254\\
0.015239	0.00034115\\
0.015504	0.0002424\\
0.015769	0.0001288\\
0.016034	2.2204e-16\\
0.016602	0.00019714\\
0.01717	0.00034132\\
0.017738	0.00044682\\
0.018307	0.00052068\\
0.018875	0.00056789\\
0.019443	0.00059236\\
0.020011	0.0005974\\
0.020579	0.00058583\\
0.021147	0.00056014\\
0.021716	0.00052254\\
0.022284	0.00047503\\
0.022852	0.00041944\\
0.02342	0.0003574\\
0.023988	0.00029045\\
0.024556	0.00021998\\
0.025125	0.00014729\\
0.025693	7.3592e-05\\
0.026261	2.2204e-16\\
0.026829	7.2437e-05\\
0.027397	0.00014274\\
0.027966	0.00020999\\
0.028534	0.00027333\\
0.029102	0.00033195\\
0.02967	0.00038508\\
0.030238	0.00043199\\
0.030806	0.00047202\\
0.031375	0.00050451\\
0.031943	0.00052885\\
0.032511	0.00054447\\
0.033079	0.0005508\\
0.033647	0.00054734\\
0.034215	0.00053357\\
0.034784	0.00050903\\
0.035352	0.00047327\\
0.03592	0.00042584\\
0.036488	0.00036634\\
0.037056	0.00029438\\
0.037624	0.00020957\\
0.038193	0.00011156\\
0.038761	2.2204e-16\\
0.039767	0.00012643\\
0.040774	0.00022441\\
0.04178	0.00029814\\
0.042787	0.00035128\\
0.043794	0.00038661\\
0.0448	0.00040643\\
0.045807	0.00041273\\
0.046813	0.00040728\\
0.04782	0.00039166\\
0.048826	0.00036732\\
0.049833	0.00033559\\
0.050839	0.0002977\\
0.051846	0.00025479\\
0.052853	0.00020792\\
0.053859	0.0001581\\
0.054866	0.00010626\\
0.055872	5.3284e-05\\
0.056879	2.7863e-16\\
0.057885	5.2805e-05\\
0.058892	0.00010439\\
0.059898	0.00015404\\
0.060905	0.0002011\\
0.061911	0.00024492\\
0.062918	0.0002849\\
0.063925	0.00032047\\
0.064931	0.00035106\\
0.065938	0.00037617\\
0.066944	0.00039527\\
0.067951	0.0004079\\
0.068957	0.00041359\\
0.069964	0.0004119\\
0.07097	0.00040241\\
0.071977	0.0003847\\
0.072984	0.0003584\\
0.07399	0.00032313\\
0.074997	0.00027852\\
0.076003	0.00022423\\
0.07701	0.00015993\\
0.078016	8.5286e-05\\
0.079023	2.2204e-16\\
0.080806	0.00011983\\
0.082589	0.00021098\\
0.084372	0.00027997\\
0.086155	0.00033009\\
0.087939	0.00036377\\
0.089722	0.00038304\\
0.091505	0.00038966\\
0.093288	0.0003852\\
0.095071	0.00037109\\
0.096854	0.00034864\\
0.098638	0.00031908\\
0.10042	0.00028353\\
0.1022	0.00024306\\
0.10399	0.00019867\\
0.10577	0.0001513\\
0.10755	0.00010185\\
0.10934	5.1145e-05\\
0.11112	3.3485e-16\\
0.1129	5.0827e-05\\
0.11469	0.00010061\\
0.11647	0.00014867\\
0.11825	0.00019433\\
0.12004	0.00023697\\
0.12182	0.00027599\\
0.1236	0.00031081\\
0.12539	0.00034087\\
0.12717	0.00036565\\
0.12895	0.00038464\\
0.13073	0.00039735\\
0.13252	0.00040331\\
0.1343	0.00040207\\
0.13608	0.00039319\\
0.13787	0.00037625\\
0.13965	0.00035086\\
0.14143	0.00031662\\
0.14322	0.00027316\\
0.145	0.00022011\\
0.14678	0.00015712\\
0.14857	8.386e-05\\
0.15035	2.2204e-16\\
0.15322	9.7096e-05\\
0.15609	0.00017347\\
0.15896	0.00023211\\
0.16184	0.0002753\\
0.16471	0.00030486\\
0.16758	0.00032232\\
0.17045	0.00032906\\
0.17332	0.00032634\\
0.1762	0.00031531\\
0.17907	0.00029704\\
0.18194	0.00027254\\
0.18481	0.00024275\\
0.18768	0.00020857\\
0.19055	0.00017084\\
0.19343	0.00013037\\
0.1963	8.7925e-05\\
0.19917	4.4236e-05\\
0.20204	3.0715e-16\\
0.20491	4.4115e-05\\
0.20779	8.7469e-05\\
0.21066	0.00012945\\
0.21353	0.00016947\\
0.2164	0.00020696\\
0.21927	0.00024139\\
0.22214	0.00027222\\
0.22502	0.00029895\\
0.22789	0.00032111\\
0.23076	0.00033822\\
0.23363	0.00034983\\
0.2365	0.00035551\\
0.23938	0.00035483\\
0.24225	0.0003474\\
0.24512	0.00033281\\
0.24799	0.00031069\\
0.25086	0.00028067\\
0.25373	0.0002424\\
0.25661	0.00019552\\
0.25948	0.00013971\\
0.26235	7.4641e-05\\
0.26522	2.2204e-16\\
0.26783	2.9701e-05\\
0.27044	5.5158e-05\\
0.27305	7.5677e-05\\
0.27566	9.1657e-05\\
0.27828	0.00010346\\
0.28089	0.00011142\\
0.2835	0.00011584\\
0.28611	0.00011699\\
0.28872	0.00011514\\
0.29133	0.00011051\\
0.29394	0.00010336\\
0.29655	9.3898e-05\\
0.29916	8.2337e-05\\
0.30177	6.8881e-05\\
0.30438	5.373e-05\\
0.30699	3.7075e-05\\
0.3096	1.9103e-05\\
0.31221	2.7842e-16\\
0.31483	2.0053e-05\\
0.31744	4.0873e-05\\
0.32005	6.2277e-05\\
0.32266	8.4076e-05\\
0.32527	0.00010607\\
0.32788	0.00012805\\
0.33049	0.00014978\\
0.3331	0.000171\\
0.33571	0.00019141\\
0.33832	0.00021065\\
0.34093	0.00022829\\
0.34354	0.0002438\\
0.34615	0.00025651\\
0.34877	0.00026559\\
0.35138	0.00027\\
0.35399	0.00026841\\
0.3566	0.00025911\\
0.35921	0.00024\\
0.36182	0.00020835\\
0.36443	0.0001608\\
0.36704	9.3109e-05\\
0.36965	2.2204e-16\\
0.37034	7.084e-05\\
0.37103	0.00011311\\
0.37171	0.00014417\\
0.3724	0.00016676\\
0.37309	0.00018228\\
0.37378	0.00019162\\
0.37446	0.00019543\\
0.37515	0.00019428\\
0.37584	0.00018863\\
0.37653	0.00017891\\
0.37721	0.00016551\\
0.3779	0.00014882\\
0.37859	0.00012921\\
0.37928	0.00010703\\
0.37996	8.2656e-05\\
0.38065	5.6446e-05\\
0.38134	2.877e-05\\
0.38203	2.7929e-16\\
0.38271	2.9482e-05\\
0.3834	5.9284e-05\\
0.38409	8.9008e-05\\
0.38478	0.00011824\\
0.38546	0.00014656\\
0.38615	0.00017352\\
0.38684	0.00019868\\
0.38753	0.00022156\\
0.38821	0.00024169\\
0.3889	0.00025856\\
0.38959	0.00027166\\
0.39028	0.00028045\\
0.39096	0.00028439\\
0.39165	0.00028289\\
0.39234	0.00027537\\
0.39303	0.00026122\\
0.39371	0.00023979\\
0.3944	0.00021045\\
0.39509	0.0001725\\
0.39578	0.00012527\\
0.39646	6.8013e-05\\
0.39715	2.2204e-16\\
0.39767	0.00011966\\
0.39818	0.00019099\\
0.3987	0.00024233\\
0.39922	0.00027869\\
0.39973	0.00030267\\
0.40025	0.00031603\\
0.40077	0.00032011\\
0.40128	0.00031599\\
0.4018	0.00030462\\
0.40232	0.00028686\\
0.40283	0.00026347\\
0.40335	0.00023518\\
0.40387	0.00020269\\
0.40438	0.00016667\\
0.4049	0.00012777\\
0.40541	8.6608e-05\\
0.40593	4.3814e-05\\
0.40645	2.757e-16\\
0.40696	4.4227e-05\\
0.40748	8.8263e-05\\
0.408	0.00013151\\
0.40851	0.00017336\\
0.40903	0.00021323\\
0.40955	0.00025051\\
0.41006	0.0002846\\
0.41058	0.0003149\\
0.4111	0.00034082\\
0.41161	0.00036174\\
0.41213	0.00037705\\
0.41265	0.00038616\\
0.41316	0.00038844\\
0.41368	0.00038329\\
0.4142	0.00037009\\
0.41471	0.00034821\\
0.41523	0.00031705\\
0.41575	0.00027597\\
0.41626	0.00022435\\
0.41678	0.00016157\\
0.41729	8.6996e-05\\
0.41781	2.2204e-16\\
0.41828	0.00012046\\
0.41875	0.00019634\\
0.41922	0.00025086\\
0.41969	0.00028904\\
0.42016	0.00031371\\
0.42063	0.00032688\\
0.4211	0.00033009\\
0.42157	0.00032465\\
0.42204	0.00031167\\
0.42251	0.00029217\\
0.42298	0.00026705\\
0.42345	0.00023717\\
0.42392	0.00020334\\
0.42439	0.00016629\\
0.42485	0.00012676\\
0.42532	8.5433e-05\\
0.42579	4.2967e-05\\
0.42626	2.7669e-16\\
0.42673	4.2849e-05\\
0.4272	8.4985e-05\\
0.42767	0.00012583\\
0.42814	0.00016481\\
0.42861	0.0002014\\
0.42908	0.00023506\\
0.42955	0.00026527\\
0.43002	0.00029153\\
0.43049	0.00031336\\
0.43096	0.0003303\\
0.43143	0.00034187\\
0.4319	0.00034764\\
0.43237	0.00034719\\
0.43284	0.0003401\\
0.43331	0.00032597\\
0.43378	0.00030443\\
0.43425	0.0002751\\
0.43472	0.00023763\\
0.43519	0.0001917\\
0.43565	0.00013698\\
0.43612	7.3177e-05\\
0.43659	2.2204e-16\\
0.43711	4.8484e-05\\
0.43763	5.9648e-05\\
0.43814	6.8591e-05\\
0.43866	7.8295e-05\\
0.43918	8.8191e-05\\
0.43969	9.7223e-05\\
0.44021	0.00010453\\
0.44073	0.00010899\\
0.44124	0.00011054\\
0.44176	0.00010916\\
0.44228	0.00010485\\
0.44279	9.7645e-05\\
0.44331	8.7624e-05\\
0.44383	7.4885e-05\\
0.44434	5.9564e-05\\
0.44486	4.1834e-05\\
0.44537	2.1898e-05\\
0.44589	3.7947e-16\\
0.44641	2.3583e-05\\
0.44692	4.8535e-05\\
0.44744	7.4505e-05\\
0.44796	0.0001011\\
0.44847	0.00012789\\
0.44899	0.00015441\\
0.44951	0.00018014\\
0.45002	0.00020454\\
0.45054	0.00022702\\
0.45106	0.00024695\\
0.45157	0.00026366\\
0.45209	0.00027642\\
0.45261	0.00028449\\
0.45312	0.00028706\\
0.45364	0.0002833\\
0.45416	0.00027232\\
0.45467	0.0002532\\
0.45519	0.00022496\\
0.45571	0.0001866\\
0.45622	0.00013706\\
0.45674	7.5236e-05\\
0.45725	2.2204e-16\\
0.45758	0.00010842\\
0.4579	0.00017052\\
0.45822	0.00021449\\
0.45854	0.00024514\\
0.45886	0.00026495\\
0.45918	0.00027555\\
0.4595	0.00027816\\
0.45982	0.00027378\\
0.46014	0.00026324\\
0.46046	0.0002473\\
0.46078	0.00022664\\
0.4611	0.0002019\\
0.46142	0.00017369\\
0.46174	0.00014258\\
0.46207	0.00010912\\
0.46239	7.3855e-05\\
0.46271	3.7308e-05\\
0.46303	3.3182e-16\\
0.46335	3.7558e-05\\
0.46367	7.4862e-05\\
0.46399	0.00011141\\
0.46431	0.0001467\\
0.46463	0.00018023\\
0.46495	0.00021152\\
0.46527	0.00024006\\
0.46559	0.00026537\\
0.46591	0.00028694\\
0.46623	0.00030428\\
0.46656	0.00031689\\
0.46688	0.00032427\\
0.4672	0.00032593\\
0.46752	0.00032137\\
0.46784	0.00031007\\
0.46816	0.00029154\\
0.46848	0.00026527\\
0.4688	0.00023076\\
0.46912	0.00018749\\
0.46944	0.00013494\\
0.46976	7.2621e-05\\
0.47008	2.2204e-16\\
0.47038	0.00012319\\
0.47067	0.00019637\\
0.47096	0.00024811\\
0.47125	0.00028391\\
0.47154	0.00030673\\
0.47183	0.00031858\\
0.47212	0.00032099\\
0.47242	0.00031521\\
0.47271	0.0003023\\
0.473	0.00028321\\
0.47329	0.00025881\\
0.47358	0.00022987\\
0.47387	0.00019715\\
0.47417	0.00016132\\
0.47446	0.00012307\\
0.47475	8.3026e-05\\
0.47504	4.1804e-05\\
0.47533	4.3784e-16\\
0.47562	4.1807e-05\\
0.47591	8.3053e-05\\
0.47621	0.00012319\\
0.4765	0.00016166\\
0.47679	0.00019795\\
0.47708	0.00023154\\
0.47737	0.0002619\\
0.47766	0.00028853\\
0.47796	0.00031093\\
0.47825	0.00032861\\
0.47854	0.00034108\\
0.47883	0.00034785\\
0.47912	0.00034846\\
0.47941	0.00034242\\
0.47971	0.00032928\\
0.48	0.00030857\\
0.48029	0.00027983\\
0.48058	0.00024261\\
0.48087	0.00019646\\
0.48116	0.00014093\\
0.48145	7.5591e-05\\
0.48175	2.2204e-16\\
0.48204	0.0001139\\
0.48233	0.00018386\\
0.48262	0.00023336\\
0.48291	0.00026747\\
0.4832	0.00028904\\
0.4835	0.00030002\\
0.48379	0.00030195\\
0.48408	0.00029607\\
0.48437	0.00028345\\
0.48466	0.00026504\\
0.48495	0.00024169\\
0.48524	0.00021419\\
0.48554	0.00018327\\
0.48583	0.00014961\\
0.48612	0.00011385\\
0.48641	7.6614e-05\\
0.4867	3.8476e-05\\
0.48699	3.4694e-16\\
0.48729	3.8276e-05\\
0.48758	7.5833e-05\\
0.48787	0.00011217\\
0.48816	0.0001468\\
0.48845	0.00017925\\
0.48874	0.00020907\\
0.48903	0.00023582\\
0.48933	0.00025905\\
0.48962	0.00027836\\
0.48991	0.00029334\\
0.4902	0.00030358\\
0.49049	0.0003087\\
0.49078	0.00030832\\
0.49108	0.00030208\\
0.49137	0.00028962\\
0.49166	0.00027059\\
0.49195	0.00024465\\
0.49224	0.00021146\\
0.49253	0.00017071\\
0.49283	0.00012209\\
0.49312	6.5283e-05\\
0.49341	2.2204e-16\\
0.49388	0.00015569\\
0.49435	0.00024741\\
0.49482	0.00030872\\
0.49529	0.00034764\\
0.49576	0.00036888\\
0.49623	0.00037577\\
0.4967	0.00037094\\
0.49716	0.00035654\\
0.49763	0.00033441\\
0.4981	0.00030615\\
0.49857	0.00027317\\
0.49904	0.00023671\\
0.49951	0.00019789\\
0.49998	0.00015773\\
0.50045	0.0001171\\
0.50092	7.6808e-05\\
0.50139	3.7567e-05\\
0.50186	3.9031e-16\\
0.50233	3.5345e-05\\
0.5028	6.7997e-05\\
0.50327	9.7555e-05\\
0.50374	0.00012369\\
0.50421	0.00014613\\
0.50468	0.00016469\\
0.50515	0.00017923\\
0.50562	0.00018966\\
0.50609	0.00019599\\
0.50656	0.00019825\\
0.50703	0.00019656\\
0.5075	0.00019105\\
0.50796	0.00018196\\
0.50843	0.00016955\\
0.5089	0.00015414\\
0.50937	0.00013611\\
0.50984	0.00011587\\
0.51031	9.3911e-05\\
0.51078	7.1998e-05\\
0.51125	5.4761e-05\\
0.51172	3.1048e-05\\
0.51219	2.2204e-16\\
0.51251	0.00013322\\
0.51283	0.00019519\\
0.51315	0.00023633\\
0.51347	0.00026349\\
0.51379	0.00027981\\
0.51412	0.00028717\\
0.51444	0.00028692\\
0.51476	0.00028009\\
0.51508	0.00026751\\
0.5154	0.00024993\\
0.51572	0.000228\\
0.51604	0.00020233\\
0.51636	0.00017348\\
0.51668	0.00014201\\
0.517	0.00010842\\
0.51732	7.3224e-05\\
0.51764	3.6922e-05\\
0.51796	4.2436e-16\\
0.51828	3.706e-05\\
0.51861	7.3778e-05\\
0.51893	0.00010968\\
0.51925	0.00014428\\
0.51957	0.00017712\\
0.51989	0.00020772\\
0.52021	0.0002356\\
0.52053	0.00026028\\
0.52085	0.00028129\\
0.52117	0.00029815\\
0.52149	0.00031038\\
0.52181	0.0003175\\
0.52213	0.00031902\\
0.52245	0.00031446\\
0.52277	0.00030332\\
0.5231	0.00028513\\
0.52342	0.00025938\\
0.52374	0.00022558\\
0.52406	0.00018324\\
0.52438	0.00013187\\
0.5247	7.0952e-05\\
0.52502	2.2204e-16\\
0.52531	0.00011768\\
0.5256	0.00018784\\
0.52589	0.00023744\\
0.52619	0.00027173\\
0.52648	0.00029356\\
0.52677	0.00030485\\
0.52706	0.0003071\\
0.52735	0.00030149\\
0.52764	0.00028907\\
0.52794	0.00027074\\
0.52823	0.00024733\\
0.52852	0.0002196\\
0.52881	0.00018828\\
0.5291	0.00015401\\
0.52939	0.00011745\\
0.52968	7.9207e-05\\
0.52998	3.9866e-05\\
0.53027	6.2908e-16\\
0.53056	3.9839e-05\\
0.53085	7.9113e-05\\
0.53114	0.0001173\\
0.53143	0.00015387\\
0.53173	0.00018834\\
0.53202	0.00022021\\
0.53231	0.00024899\\
0.5326	0.00027419\\
0.53289	0.00029536\\
0.53318	0.00031203\\
0.53347	0.00032374\\
0.53377	0.00033004\\
0.53406	0.00033048\\
0.53435	0.00032463\\
0.53464	0.00031204\\
0.53493	0.00029229\\
0.53522	0.00026496\\
0.53552	0.00022962\\
0.53581	0.00018587\\
0.5361	0.00013328\\
0.53639	7.1459e-05\\
0.53668	2.2204e-16\\
0.537	0.00013974\\
0.53732	0.00022255\\
0.53764	0.00028056\\
0.53797	0.00032018\\
0.53829	0.00034492\\
0.53861	0.00035717\\
0.53893	0.00035877\\
0.53925	0.00035121\\
0.53957	0.00033577\\
0.53989	0.00031358\\
0.54021	0.00028564\\
0.54053	0.00025289\\
0.54085	0.00021619\\
0.54117	0.00017634\\
0.54149	0.00013409\\
0.54181	9.0168e-05\\
0.54213	4.5253e-05\\
0.54246	3.86e-16\\
0.54278	4.4962e-05\\
0.5431	8.9028e-05\\
0.54342	0.00013161\\
0.54374	0.00017216\\
0.54406	0.00021012\\
0.54438	0.00024495\\
0.5447	0.00027616\\
0.54502	0.00030324\\
0.54534	0.00032571\\
0.54566	0.00034309\\
0.54598	0.00035494\\
0.5463	0.00036079\\
0.54662	0.00036023\\
0.54695	0.00035282\\
0.54727	0.00033816\\
0.54759	0.00031584\\
0.54791	0.00028547\\
0.54823	0.00024668\\
0.54855	0.0001991\\
0.54887	0.00014235\\
0.54919	7.6101e-05\\
0.54951	2.2204e-16\\
0.54994	0.0001373\\
0.55036	0.00022227\\
0.55079	0.00028144\\
0.55122	0.0003212\\
0.55165	0.00034527\\
0.55207	0.00035631\\
0.5525	0.00035636\\
0.55293	0.00034714\\
0.55335	0.0003301\\
0.55378	0.00030651\\
0.55421	0.00027752\\
0.55463	0.00024417\\
0.55506	0.00020738\\
0.55549	0.00016802\\
0.55591	0.00012689\\
0.55634	8.4728e-05\\
0.55677	4.2219e-05\\
0.55719	3.5041e-16\\
0.55762	4.1335e-05\\
0.55805	8.1233e-05\\
0.55848	0.00011918\\
0.5589	0.00015469\\
0.55933	0.00018731\\
0.55976	0.00021664\\
0.56018	0.00024228\\
0.56061	0.00026389\\
0.56104	0.00028111\\
0.56146	0.00029366\\
0.56189	0.00030125\\
0.56232	0.00030363\\
0.56274	0.00030056\\
0.56317	0.00029183\\
0.5636	0.00027726\\
0.56402	0.00025668\\
0.56445	0.00022994\\
0.56488	0.00019691\\
0.56531	0.00015748\\
0.56573	0.00011157\\
0.56616	5.9092e-05\\
0.56659	2.2204e-16\\
0.5671	0.00010573\\
0.56762	0.00014611\\
0.56814	0.00017219\\
0.56865	0.00018947\\
0.56917	0.00020011\\
0.56969	0.00020554\\
0.5702	0.00020713\\
0.57072	0.00020508\\
0.57123	0.00019942\\
0.57175	0.00019015\\
0.57227	0.00017714\\
0.57278	0.00016061\\
0.5733	0.00014071\\
0.57382	0.00011747\\
0.57433	9.132e-05\\
0.57485	6.2709e-05\\
0.57537	3.2108e-05\\
0.57588	2.5674e-16\\
0.5764	3.3122e-05\\
0.57692	6.6752e-05\\
0.57743	0.00010038\\
0.57795	0.00013347\\
0.57847	0.00016549\\
0.57898	0.00019591\\
0.5795	0.00022417\\
0.58002	0.00024971\\
0.58053	0.00027198\\
0.58105	0.0002904\\
0.58156	0.00030441\\
0.58208	0.00031342\\
0.5826	0.00031685\\
0.58311	0.00031412\\
0.58363	0.00030465\\
0.58415	0.00028784\\
0.58466	0.00026311\\
0.58518	0.00022987\\
0.5857	0.00018753\\
0.58621	0.00013549\\
0.58673	7.318e-05\\
0.58725	2.2204e-16\\
0.58776	0.00013182\\
0.58828	0.0002098\\
0.5888	0.00026482\\
0.58931	0.00030275\\
0.58983	0.00032674\\
0.59035	0.00033898\\
0.59086	0.00034112\\
0.59138	0.00033451\\
0.5919	0.00032034\\
0.59241	0.00029965\\
0.59293	0.00027337\\
0.59344	0.00024239\\
0.59396	0.0002075\\
0.59448	0.00016947\\
0.59499	0.00012903\\
0.59551	8.6868e-05\\
0.59603	4.3646e-05\\
0.59654	2.7756e-16\\
0.59706	4.3457e-05\\
0.59758	8.6132e-05\\
0.59809	0.00012745\\
0.59861	0.00016686\\
0.59913	0.00020382\\
0.59964	0.00023781\\
0.60016	0.00026831\\
0.60068	0.00029483\\
0.60119	0.0003169\\
0.60171	0.00033403\\
0.60223	0.00034577\\
0.60274	0.00035169\\
0.60326	0.00035133\\
0.60378	0.00034429\\
0.60429	0.00033015\\
0.60481	0.00030851\\
0.60532	0.00027898\\
0.60584	0.00024118\\
0.60636	0.00019473\\
0.60687	0.00013929\\
0.60739	7.449e-05\\
0.60791	2.2204e-16\\
0.60859	0.00012962\\
0.60928	0.00021226\\
0.60997	0.00027067\\
0.61066	0.00031057\\
0.61134	0.00033534\\
0.61203	0.00034741\\
0.61272	0.00034871\\
0.61341	0.00034082\\
0.61409	0.00032512\\
0.61478	0.00030282\\
0.61547	0.000275\\
0.61616	0.00024266\\
0.61684	0.0002067\\
0.61753	0.00016796\\
0.61822	0.00012722\\
0.61891	8.5201e-05\\
0.61959	4.2581e-05\\
0.62028	2.9533e-16\\
0.62097	4.1945e-05\\
0.62166	8.2688e-05\\
0.62234	0.0001217\\
0.62303	0.00015847\\
0.62372	0.00019253\\
0.62441	0.00022343\\
0.62509	0.00025075\\
0.62578	0.00027407\\
0.62647	0.00029303\\
0.62716	0.00030725\\
0.62784	0.00031639\\
0.62853	0.00032013\\
0.62922	0.00031816\\
0.62991	0.00031019\\
0.63059	0.00029594\\
0.63128	0.00027515\\
0.63197	0.00024758\\
0.63266	0.00021297\\
0.63334	0.00017112\\
0.63403	0.0001218\\
0.63472	6.4827e-05\\
0.63541	2.2204e-16\\
0.63688	0.00013079\\
0.63835	0.00021748\\
0.63983	0.00027699\\
0.6413	0.00031576\\
0.64278	0.00033795\\
0.64425	0.00034667\\
0.64572	0.00034433\\
0.6472	0.00033294\\
0.64867	0.00031419\\
0.65014	0.0002895\\
0.65162	0.00026013\\
0.65309	0.00022715\\
0.65456	0.00019152\\
0.65604	0.00015407\\
0.65751	0.00011557\\
0.65899	7.6667e-05\\
0.66046	3.7965e-05\\
0.66193	2.3766e-16\\
0.66341	3.6746e-05\\
0.66488	7.1835e-05\\
0.66635	0.00010487\\
0.66783	0.0001355\\
0.6693	0.00016338\\
0.67078	0.00018823\\
0.67225	0.00020975\\
0.67372	0.00022772\\
0.6752	0.00024187\\
0.67667	0.00025202\\
0.67814	0.00025795\\
0.67962	0.00025947\\
0.68109	0.00025643\\
0.68257	0.00024865\\
0.68404	0.00023599\\
0.68551	0.00021832\\
0.68699	0.00019549\\
0.68846	0.00016738\\
0.68993	0.00013388\\
0.69141	9.489e-05\\
0.69288	5.0294e-05\\
0.69435	2.2204e-16\\
0.69856	0.00011428\\
0.70276	0.00019518\\
0.70697	0.00025233\\
0.71117	0.00029075\\
0.71538	0.00031388\\
0.71958	0.00032434\\
0.72379	0.00032426\\
0.72799	0.00031541\\
0.7322	0.00029929\\
0.7364	0.00027721\\
0.74061	0.00025031\\
0.74481	0.0002196\\
0.74901	0.00018598\\
0.75322	0.00015027\\
0.75742	0.00011318\\
0.76163	7.5386e-05\\
0.76583	3.7476e-05\\
0.77004	2.3628e-16\\
0.77424	3.6541e-05\\
0.77845	7.1686e-05\\
0.78265	0.00010501\\
0.78686	0.00013612\\
0.79106	0.00016466\\
0.79527	0.00019029\\
0.79947	0.00021269\\
0.80367	0.00023158\\
0.80788	0.00024668\\
0.81208	0.00025772\\
0.81629	0.00026449\\
0.82049	0.00026675\\
0.8247	0.00026428\\
0.8289	0.00025689\\
0.83311	0.00024439\\
0.83731	0.00022661\\
0.84152	0.00020337\\
0.84572	0.00017451\\
0.84993	0.00013989\\
0.85413	9.935e-05\\
0.85833	5.2764e-05\\
0.86254	2.2204e-16\\
0.86598	9.8015e-06\\
0.86941	9.7274e-06\\
0.87285	8.6451e-06\\
0.87629	7.331e-06\\
0.87972	9.374e-06\\
0.88316	1.1543e-05\\
0.88659	1.3067e-05\\
0.89003	1.3973e-05\\
0.89347	1.4306e-05\\
0.8969	1.4115e-05\\
0.90034	1.3452e-05\\
0.90378	1.2369e-05\\
0.90721	1.092e-05\\
0.91065	9.1581e-06\\
0.91409	7.1352e-06\\
0.91752	4.9017e-06\\
0.92096	2.5071e-06\\
0.9244	2.4286e-16\\
0.92783	2.5726e-06\\
0.93127	5.1647e-06\\
0.93471	7.7316e-06\\
0.93814	1.023e-05\\
0.94158	1.2617e-05\\
0.94502	1.4852e-05\\
0.94845	1.6895e-05\\
0.95189	1.8707e-05\\
0.95533	2.0248e-05\\
0.95876	2.1483e-05\\
0.9622	2.2375e-05\\
0.96563	2.2889e-05\\
0.96907	2.299e-05\\
0.97251	2.2644e-05\\
0.97594	2.1819e-05\\
0.97938	2.0482e-05\\
0.98282	1.8601e-05\\
0.98625	1.6147e-05\\
0.98969	1.3089e-05\\
0.99313	9.3978e-06\\
0.99656	5.0441e-06\\
1	2.2204e-16\\
};
\addplot [color=black,thick, forget plot]
  table[row sep=crcr]{%
0	inf\\
4e-06	0.0099878\\
8e-06	0.007642\\
1.2e-05	0.0065428\\
1.6e-05	0.0058642\\
2e-05	0.0053891\\
2.4e-05	0.0050311\\
2.8e-05	0.0047482\\
3.2e-05	0.0045169\\
3.6e-05	0.0043228\\
4e-05	0.0041568\\
4.4e-05	0.0040125\\
4.8e-05	0.0038855\\
5.2e-05	0.0037725\\
5.6e-05	0.0036711\\
6e-05	0.0035794\\
6.4e-05	0.0034958\\
6.8e-05	0.0034192\\
7.2e-05	0.0033487\\
7.6e-05	0.0032835\\
8e-05	0.0032229\\
8.4e-05	0.0031664\\
8.8e-05	0.0031136\\
9.2e-05	0.003064\\
9.6e-05	0.0030173\\
0.0001	0.0029733\\
0.000104	0.0029317\\
0.000108	0.0028923\\
0.000112	0.0028548\\
0.000116	0.0028192\\
0.00012	0.0027853\\
0.000124	0.0027529\\
0.000128	0.002722\\
0.000132	0.0026923\\
0.000136	0.0026639\\
0.00014	0.0026367\\
0.000144	0.0026105\\
0.000148	0.0025853\\
0.000152	0.0025611\\
0.000156	0.0025377\\
0.00016	0.0025151\\
0.0001896	0.0023697\\
0.0002192	0.002253\\
0.0002488	0.0021566\\
0.0002784	0.0020749\\
0.000308	0.0020046\\
0.00033761	0.0019432\\
0.00036721	0.0018889\\
0.00039681	0.0018403\\
0.00042641	0.0017966\\
0.00045601	0.001757\\
0.00048561	0.0017208\\
0.00051521	0.0016876\\
0.00054481	0.0016569\\
0.00057441	0.0016285\\
0.00060401	0.0016021\\
0.00063362	0.0015774\\
0.00066322	0.0015543\\
0.00069282	0.0015325\\
0.00072242	0.0015121\\
0.00075202	0.0014928\\
0.00078162	0.0014745\\
0.00081122	0.0014572\\
0.00084082	0.0014407\\
0.00087042	0.001425\\
0.00090002	0.0014101\\
0.00092963	0.0013958\\
0.00095923	0.0013821\\
0.00098883	0.0013691\\
0.0010184	0.0013565\\
0.001048	0.0013445\\
0.0010776	0.0013329\\
0.0011072	0.0013218\\
0.0011368	0.001311\\
0.0011664	0.0013007\\
0.001196	0.0012907\\
0.0012256	0.0012811\\
0.0012552	0.0012718\\
0.0012848	0.0012627\\
0.0013144	0.001254\\
0.001344	0.0012455\\
0.0014462	0.0012182\\
0.0015484	0.0011935\\
0.0016506	0.0011709\\
0.0017528	0.0011503\\
0.001855	0.0011312\\
0.0019572	0.0011136\\
0.0020594	0.0010972\\
0.0021616	0.0010819\\
0.0022638	0.0010676\\
0.0023659	0.0010541\\
0.0024681	0.0010415\\
0.0025703	0.0010295\\
0.0026725	0.0010183\\
0.0027747	0.0010076\\
0.0028769	0.0009974\\
0.0029791	0.00098773\\
0.0030813	0.00097851\\
0.0031835	0.00096971\\
0.0032857	0.0009613\\
0.0033879	0.00095325\\
0.00349	0.00094553\\
0.0035922	0.00093812\\
0.0036944	0.000931\\
0.0037966	0.00092415\\
0.0038988	0.00091756\\
0.004001	0.0009112\\
0.0041032	0.00090506\\
0.0042054	0.00089914\\
0.0043076	0.00089341\\
0.0044098	0.00088787\\
0.004512	0.00088251\\
0.0046141	0.00087732\\
0.0047163	0.00087228\\
0.0048185	0.00086739\\
0.0049207	0.00086265\\
0.0050229	0.00085804\\
0.0051251	0.00085356\\
0.0052273	0.00084921\\
0.0053295	0.00084497\\
0.0054317	0.00084085\\
0.0056967	0.00083064\\
0.0059618	0.00082107\\
0.0062268	0.00081209\\
0.0064919	0.00080362\\
0.0067569	0.00079562\\
0.007022	0.00078805\\
0.0072871	0.00078087\\
0.0075521	0.00077405\\
0.0078172	0.00076755\\
0.0080822	0.00076136\\
0.0083473	0.00075544\\
0.0086123	0.00074978\\
0.0088774	0.00074436\\
0.0091425	0.00073916\\
0.0094075	0.00073417\\
0.0096726	0.00072937\\
0.0099376	0.00072475\\
0.010203	0.0007203\\
0.010468	0.00071601\\
0.010733	0.00071187\\
0.010998	0.00070787\\
0.011263	0.00070401\\
0.011528	0.00070027\\
0.011793	0.00069664\\
0.012058	0.00069313\\
0.012323	0.00068973\\
0.012588	0.00068642\\
0.012853	0.00068322\\
0.013118	0.0006801\\
0.013383	0.00067707\\
0.013648	0.00067412\\
0.013913	0.00067126\\
0.014179	0.00066847\\
0.014444	0.00066575\\
0.014709	0.00066309\\
0.014974	0.00066051\\
0.015239	0.00065799\\
0.015504	0.00065553\\
0.015769	0.00065313\\
0.016034	0.00065078\\
0.016602	0.00064593\\
0.01717	0.0006413\\
0.017738	0.00063689\\
0.018307	0.00063266\\
0.018875	0.00062862\\
0.019443	0.00062475\\
0.020011	0.00062103\\
0.020579	0.00061745\\
0.021147	0.00061401\\
0.021716	0.0006107\\
0.022284	0.00060751\\
0.022852	0.00060442\\
0.02342	0.00060145\\
0.023988	0.00059857\\
0.024556	0.00059579\\
0.025125	0.0005931\\
0.025693	0.00059049\\
0.026261	0.00058796\\
0.026829	0.0005855\\
0.027397	0.00058312\\
0.027966	0.0005808\\
0.028534	0.00057855\\
0.029102	0.00057636\\
0.02967	0.00057423\\
0.030238	0.00057216\\
0.030806	0.00057014\\
0.031375	0.00056817\\
0.031943	0.00056626\\
0.032511	0.00056439\\
0.033079	0.00056256\\
0.033647	0.00056078\\
0.034215	0.00055904\\
0.034784	0.00055734\\
0.035352	0.00055567\\
0.03592	0.00055405\\
0.036488	0.00055246\\
0.037056	0.0005509\\
0.037624	0.00054938\\
0.038193	0.00054789\\
0.038761	0.00054643\\
0.039767	0.00054392\\
0.040774	0.00054149\\
0.04178	0.00053914\\
0.042787	0.00053688\\
0.043794	0.00053468\\
0.0448	0.00053256\\
0.045807	0.0005305\\
0.046813	0.00052851\\
0.04782	0.00052657\\
0.048826	0.00052469\\
0.049833	0.00052287\\
0.050839	0.00052109\\
0.051846	0.00051936\\
0.052853	0.00051768\\
0.053859	0.00051605\\
0.054866	0.00051445\\
0.055872	0.0005129\\
0.056879	0.00051138\\
0.057885	0.0005099\\
0.058892	0.00050846\\
0.059898	0.00050705\\
0.060905	0.00050568\\
0.061911	0.00050433\\
0.062918	0.00050302\\
0.063925	0.00050174\\
0.064931	0.00050048\\
0.065938	0.00049925\\
0.066944	0.00049805\\
0.067951	0.00049687\\
0.068957	0.00049571\\
0.069964	0.00049458\\
0.07097	0.00049347\\
0.071977	0.00049239\\
0.072984	0.00049132\\
0.07399	0.00049027\\
0.074997	0.00048925\\
0.076003	0.00048824\\
0.07701	0.00048725\\
0.078016	0.00048628\\
0.079023	0.00048533\\
0.080806	0.00048368\\
0.082589	0.00048209\\
0.084372	0.00048054\\
0.086155	0.00047903\\
0.087939	0.00047757\\
0.089722	0.00047615\\
0.091505	0.00047477\\
0.093288	0.00047343\\
0.095071	0.00047212\\
0.096854	0.00047084\\
0.098638	0.0004696\\
0.10042	0.00046839\\
0.1022	0.00046721\\
0.10399	0.00046606\\
0.10577	0.00046493\\
0.10755	0.00046383\\
0.10934	0.00046276\\
0.11112	0.00046171\\
0.1129	0.00046068\\
0.11469	0.00045968\\
0.11647	0.0004587\\
0.11825	0.00045773\\
0.12004	0.00045679\\
0.12182	0.00045587\\
0.1236	0.00045497\\
0.12539	0.00045408\\
0.12717	0.00045321\\
0.12895	0.00045236\\
0.13073	0.00045153\\
0.13252	0.00045071\\
0.1343	0.00044991\\
0.13608	0.00044912\\
0.13787	0.00044834\\
0.13965	0.00044758\\
0.14143	0.00044683\\
0.14322	0.0004461\\
0.145	0.00044538\\
0.14678	0.00044467\\
0.14857	0.00044397\\
0.15035	0.00044329\\
0.15322	0.00044221\\
0.15609	0.00044115\\
0.15896	0.00044013\\
0.16184	0.00043913\\
0.16471	0.00043815\\
0.16758	0.0004372\\
0.17045	0.00043627\\
0.17332	0.00043536\\
0.1762	0.00043448\\
0.17907	0.00043361\\
0.18194	0.00043276\\
0.18481	0.00043193\\
0.18768	0.00043112\\
0.19055	0.00043033\\
0.19343	0.00042955\\
0.1963	0.00042879\\
0.19917	0.00042804\\
0.20204	0.00042731\\
0.20491	0.0004266\\
0.20779	0.00042589\\
0.21066	0.0004252\\
0.21353	0.00042453\\
0.2164	0.00042386\\
0.21927	0.00042321\\
0.22214	0.00042257\\
0.22502	0.00042194\\
0.22789	0.00042133\\
0.23076	0.00042072\\
0.23363	0.00042013\\
0.2365	0.00041954\\
0.23938	0.00041896\\
0.24225	0.0004184\\
0.24512	0.00041784\\
0.24799	0.00041729\\
0.25086	0.00041675\\
0.25373	0.00041622\\
0.25661	0.0004157\\
0.25948	0.00041519\\
0.26235	0.00041468\\
0.26522	0.00041418\\
0.26783	0.00041374\\
0.27044	0.0004133\\
0.27305	0.00041286\\
0.27566	0.00041243\\
0.27828	0.00041201\\
0.28089	0.00041159\\
0.2835	0.00041118\\
0.28611	0.00041077\\
0.28872	0.00041037\\
0.29133	0.00040998\\
0.29394	0.00040958\\
0.29655	0.0004092\\
0.29916	0.00040881\\
0.30177	0.00040844\\
0.30438	0.00040806\\
0.30699	0.0004077\\
0.3096	0.00040733\\
0.31221	0.00040697\\
0.31483	0.00040662\\
0.31744	0.00040626\\
0.32005	0.00040592\\
0.32266	0.00040557\\
0.32527	0.00040523\\
0.32788	0.0004049\\
0.33049	0.00040456\\
0.3331	0.00040424\\
0.33571	0.00040391\\
0.33832	0.00040359\\
0.34093	0.00040327\\
0.34354	0.00040296\\
0.34615	0.00040265\\
0.34877	0.00040234\\
0.35138	0.00040203\\
0.35399	0.00040173\\
0.3566	0.00040143\\
0.35921	0.00040114\\
0.36182	0.00040084\\
0.36443	0.00040056\\
0.36704	0.00040027\\
0.36965	0.00039998\\
0.37034	0.00039991\\
0.37103	0.00039984\\
0.37171	0.00039976\\
0.3724	0.00039969\\
0.37309	0.00039961\\
0.37378	0.00039954\\
0.37446	0.00039947\\
0.37515	0.0003994\\
0.37584	0.00039932\\
0.37653	0.00039925\\
0.37721	0.00039918\\
0.3779	0.00039911\\
0.37859	0.00039903\\
0.37928	0.00039896\\
0.37996	0.00039889\\
0.38065	0.00039882\\
0.38134	0.00039875\\
0.38203	0.00039868\\
0.38271	0.00039861\\
0.3834	0.00039853\\
0.38409	0.00039846\\
0.38478	0.00039839\\
0.38546	0.00039832\\
0.38615	0.00039825\\
0.38684	0.00039818\\
0.38753	0.00039811\\
0.38821	0.00039804\\
0.3889	0.00039797\\
0.38959	0.00039791\\
0.39028	0.00039784\\
0.39096	0.00039777\\
0.39165	0.0003977\\
0.39234	0.00039763\\
0.39303	0.00039756\\
0.39371	0.00039749\\
0.3944	0.00039743\\
0.39509	0.00039736\\
0.39578	0.00039729\\
0.39646	0.00039722\\
0.39715	0.00039716\\
0.39767	0.0003971\\
0.39818	0.00039705\\
0.3987	0.000397\\
0.39922	0.00039695\\
0.39973	0.0003969\\
0.40025	0.00039685\\
0.40077	0.0003968\\
0.40128	0.00039675\\
0.4018	0.0003967\\
0.40232	0.00039665\\
0.40283	0.0003966\\
0.40335	0.00039656\\
0.40387	0.00039651\\
0.40438	0.00039646\\
0.4049	0.00039641\\
0.40541	0.00039636\\
0.40593	0.00039631\\
0.40645	0.00039626\\
0.40696	0.00039621\\
0.40748	0.00039616\\
0.408	0.00039611\\
0.40851	0.00039607\\
0.40903	0.00039602\\
0.40955	0.00039597\\
0.41006	0.00039592\\
0.41058	0.00039587\\
0.4111	0.00039582\\
0.41161	0.00039578\\
0.41213	0.00039573\\
0.41265	0.00039568\\
0.41316	0.00039563\\
0.41368	0.00039558\\
0.4142	0.00039554\\
0.41471	0.00039549\\
0.41523	0.00039544\\
0.41575	0.00039539\\
0.41626	0.00039535\\
0.41678	0.0003953\\
0.41729	0.00039525\\
0.41781	0.0003952\\
0.41828	0.00039516\\
0.41875	0.00039512\\
0.41922	0.00039508\\
0.41969	0.00039503\\
0.42016	0.00039499\\
0.42063	0.00039495\\
0.4211	0.00039491\\
0.42157	0.00039486\\
0.42204	0.00039482\\
0.42251	0.00039478\\
0.42298	0.00039474\\
0.42345	0.0003947\\
0.42392	0.00039465\\
0.42439	0.00039461\\
0.42485	0.00039457\\
0.42532	0.00039453\\
0.42579	0.00039449\\
0.42626	0.00039444\\
0.42673	0.0003944\\
0.4272	0.00039436\\
0.42767	0.00039432\\
0.42814	0.00039428\\
0.42861	0.00039424\\
0.42908	0.0003942\\
0.42955	0.00039416\\
0.43002	0.00039411\\
0.43049	0.00039407\\
0.43096	0.00039403\\
0.43143	0.00039399\\
0.4319	0.00039395\\
0.43237	0.00039391\\
0.43284	0.00039387\\
0.43331	0.00039383\\
0.43378	0.00039379\\
0.43425	0.00039375\\
0.43472	0.00039371\\
0.43519	0.00039367\\
0.43565	0.00039363\\
0.43612	0.00039358\\
0.43659	0.00039354\\
0.43711	0.0003935\\
0.43763	0.00039346\\
0.43814	0.00039341\\
0.43866	0.00039337\\
0.43918	0.00039332\\
0.43969	0.00039328\\
0.44021	0.00039324\\
0.44073	0.00039319\\
0.44124	0.00039315\\
0.44176	0.00039311\\
0.44228	0.00039306\\
0.44279	0.00039302\\
0.44331	0.00039298\\
0.44383	0.00039293\\
0.44434	0.00039289\\
0.44486	0.00039285\\
0.44537	0.0003928\\
0.44589	0.00039276\\
0.44641	0.00039272\\
0.44692	0.00039267\\
0.44744	0.00039263\\
0.44796	0.00039259\\
0.44847	0.00039255\\
0.44899	0.0003925\\
0.44951	0.00039246\\
0.45002	0.00039242\\
0.45054	0.00039238\\
0.45106	0.00039233\\
0.45157	0.00039229\\
0.45209	0.00039225\\
0.45261	0.00039221\\
0.45312	0.00039216\\
0.45364	0.00039212\\
0.45416	0.00039208\\
0.45467	0.00039204\\
0.45519	0.000392\\
0.45571	0.00039196\\
0.45622	0.00039191\\
0.45674	0.00039187\\
0.45725	0.00039183\\
0.45758	0.0003918\\
0.4579	0.00039178\\
0.45822	0.00039175\\
0.45854	0.00039173\\
0.45886	0.0003917\\
0.45918	0.00039168\\
0.4595	0.00039165\\
0.45982	0.00039162\\
0.46014	0.0003916\\
0.46046	0.00039157\\
0.46078	0.00039155\\
0.4611	0.00039152\\
0.46142	0.0003915\\
0.46174	0.00039147\\
0.46207	0.00039145\\
0.46239	0.00039142\\
0.46271	0.0003914\\
0.46303	0.00039137\\
0.46335	0.00039135\\
0.46367	0.00039132\\
0.46399	0.00039129\\
0.46431	0.00039127\\
0.46463	0.00039124\\
0.46495	0.00039122\\
0.46527	0.00039119\\
0.46559	0.00039117\\
0.46591	0.00039114\\
0.46623	0.00039112\\
0.46656	0.00039109\\
0.46688	0.00039107\\
0.4672	0.00039104\\
0.46752	0.00039102\\
0.46784	0.00039099\\
0.46816	0.00039097\\
0.46848	0.00039094\\
0.4688	0.00039092\\
0.46912	0.00039089\\
0.46944	0.00039087\\
0.46976	0.00039084\\
0.47008	0.00039082\\
0.47038	0.0003908\\
0.47067	0.00039077\\
0.47096	0.00039075\\
0.47125	0.00039073\\
0.47154	0.00039071\\
0.47183	0.00039068\\
0.47212	0.00039066\\
0.47242	0.00039064\\
0.47271	0.00039062\\
0.473	0.0003906\\
0.47329	0.00039057\\
0.47358	0.00039055\\
0.47387	0.00039053\\
0.47417	0.00039051\\
0.47446	0.00039048\\
0.47475	0.00039046\\
0.47504	0.00039044\\
0.47533	0.00039042\\
0.47562	0.00039039\\
0.47591	0.00039037\\
0.47621	0.00039035\\
0.4765	0.00039033\\
0.47679	0.00039031\\
0.47708	0.00039028\\
0.47737	0.00039026\\
0.47766	0.00039024\\
0.47796	0.00039022\\
0.47825	0.0003902\\
0.47854	0.00039017\\
0.47883	0.00039015\\
0.47912	0.00039013\\
0.47941	0.00039011\\
0.47971	0.00039009\\
0.48	0.00039006\\
0.48029	0.00039004\\
0.48058	0.00039002\\
0.48087	0.00039\\
0.48116	0.00038998\\
0.48145	0.00038996\\
0.48175	0.00038993\\
0.48204	0.00038991\\
0.48233	0.00038989\\
0.48262	0.00038987\\
0.48291	0.00038985\\
0.4832	0.00038982\\
0.4835	0.0003898\\
0.48379	0.00038978\\
0.48408	0.00038976\\
0.48437	0.00038974\\
0.48466	0.00038972\\
0.48495	0.0003897\\
0.48524	0.00038967\\
0.48554	0.00038965\\
0.48583	0.00038963\\
0.48612	0.00038961\\
0.48641	0.00038959\\
0.4867	0.00038957\\
0.48699	0.00038954\\
0.48729	0.00038952\\
0.48758	0.0003895\\
0.48787	0.00038948\\
0.48816	0.00038946\\
0.48845	0.00038944\\
0.48874	0.00038942\\
0.48903	0.0003894\\
0.48933	0.00038937\\
0.48962	0.00038935\\
0.48991	0.00038933\\
0.4902	0.00038931\\
0.49049	0.00038929\\
0.49078	0.00038927\\
0.49108	0.00038925\\
0.49137	0.00038923\\
0.49166	0.0003892\\
0.49195	0.00038918\\
0.49224	0.00038916\\
0.49253	0.00038914\\
0.49283	0.00038912\\
0.49312	0.0003891\\
0.49341	0.00038908\\
0.49388	0.00038904\\
0.49435	0.00038901\\
0.49482	0.00038898\\
0.49529	0.00038894\\
0.49576	0.00038891\\
0.49623	0.00038887\\
0.4967	0.00038884\\
0.49716	0.00038881\\
0.49763	0.00038877\\
0.4981	0.00038874\\
0.49857	0.00038871\\
0.49904	0.00038867\\
0.49951	0.00038864\\
0.49998	0.00038861\\
0.50045	0.00038857\\
0.50092	0.00038854\\
0.50139	0.00038851\\
0.50186	0.00038847\\
0.50233	0.00038844\\
0.5028	0.00038841\\
0.50327	0.00038837\\
0.50374	0.00038834\\
0.50421	0.00038831\\
0.50468	0.00038828\\
0.50515	0.00038824\\
0.50562	0.00038821\\
0.50609	0.00038818\\
0.50656	0.00038815\\
0.50703	0.00038811\\
0.5075	0.00038808\\
0.50796	0.00038805\\
0.50843	0.00038801\\
0.5089	0.00038798\\
0.50937	0.00038795\\
0.50984	0.00038792\\
0.51031	0.00038788\\
0.51078	0.00038785\\
0.51125	0.00038782\\
0.51172	0.00038779\\
0.51219	0.00038776\\
0.51251	0.00038773\\
0.51283	0.00038771\\
0.51315	0.00038769\\
0.51347	0.00038767\\
0.51379	0.00038765\\
0.51412	0.00038762\\
0.51444	0.0003876\\
0.51476	0.00038758\\
0.51508	0.00038756\\
0.5154	0.00038754\\
0.51572	0.00038752\\
0.51604	0.00038749\\
0.51636	0.00038747\\
0.51668	0.00038745\\
0.517	0.00038743\\
0.51732	0.00038741\\
0.51764	0.00038739\\
0.51796	0.00038736\\
0.51828	0.00038734\\
0.51861	0.00038732\\
0.51893	0.0003873\\
0.51925	0.00038728\\
0.51957	0.00038726\\
0.51989	0.00038723\\
0.52021	0.00038721\\
0.52053	0.00038719\\
0.52085	0.00038717\\
0.52117	0.00038715\\
0.52149	0.00038713\\
0.52181	0.00038711\\
0.52213	0.00038708\\
0.52245	0.00038706\\
0.52277	0.00038704\\
0.5231	0.00038702\\
0.52342	0.000387\\
0.52374	0.00038698\\
0.52406	0.00038696\\
0.52438	0.00038693\\
0.5247	0.00038691\\
0.52502	0.00038689\\
0.52531	0.00038687\\
0.5256	0.00038685\\
0.52589	0.00038683\\
0.52619	0.00038681\\
0.52648	0.0003868\\
0.52677	0.00038678\\
0.52706	0.00038676\\
0.52735	0.00038674\\
0.52764	0.00038672\\
0.52794	0.0003867\\
0.52823	0.00038668\\
0.52852	0.00038666\\
0.52881	0.00038664\\
0.5291	0.00038662\\
0.52939	0.0003866\\
0.52968	0.00038658\\
0.52998	0.00038657\\
0.53027	0.00038655\\
0.53056	0.00038653\\
0.53085	0.00038651\\
0.53114	0.00038649\\
0.53143	0.00038647\\
0.53173	0.00038645\\
0.53202	0.00038643\\
0.53231	0.00038641\\
0.5326	0.0003864\\
0.53289	0.00038638\\
0.53318	0.00038636\\
0.53347	0.00038634\\
0.53377	0.00038632\\
0.53406	0.0003863\\
0.53435	0.00038628\\
0.53464	0.00038626\\
0.53493	0.00038624\\
0.53522	0.00038623\\
0.53552	0.00038621\\
0.53581	0.00038619\\
0.5361	0.00038617\\
0.53639	0.00038615\\
0.53668	0.00038613\\
0.537	0.00038611\\
0.53732	0.00038609\\
0.53764	0.00038607\\
0.53797	0.00038605\\
0.53829	0.00038603\\
0.53861	0.00038601\\
0.53893	0.00038599\\
0.53925	0.00038597\\
0.53957	0.00038595\\
0.53989	0.00038593\\
0.54021	0.00038591\\
0.54053	0.00038589\\
0.54085	0.00038587\\
0.54117	0.00038585\\
0.54149	0.00038582\\
0.54181	0.0003858\\
0.54213	0.00038578\\
0.54246	0.00038576\\
0.54278	0.00038574\\
0.5431	0.00038572\\
0.54342	0.0003857\\
0.54374	0.00038568\\
0.54406	0.00038566\\
0.54438	0.00038564\\
0.5447	0.00038562\\
0.54502	0.0003856\\
0.54534	0.00038558\\
0.54566	0.00038556\\
0.54598	0.00038554\\
0.5463	0.00038552\\
0.54662	0.0003855\\
0.54695	0.00038548\\
0.54727	0.00038546\\
0.54759	0.00038544\\
0.54791	0.00038542\\
0.54823	0.0003854\\
0.54855	0.00038538\\
0.54887	0.00038536\\
0.54919	0.00038534\\
0.54951	0.00038532\\
0.54994	0.0003853\\
0.55036	0.00038527\\
0.55079	0.00038524\\
0.55122	0.00038522\\
0.55165	0.00038519\\
0.55207	0.00038516\\
0.5525	0.00038514\\
0.55293	0.00038511\\
0.55335	0.00038508\\
0.55378	0.00038506\\
0.55421	0.00038503\\
0.55463	0.00038501\\
0.55506	0.00038498\\
0.55549	0.00038495\\
0.55591	0.00038493\\
0.55634	0.0003849\\
0.55677	0.00038487\\
0.55719	0.00038485\\
0.55762	0.00038482\\
0.55805	0.0003848\\
0.55848	0.00038477\\
0.5589	0.00038475\\
0.55933	0.00038472\\
0.55976	0.00038469\\
0.56018	0.00038467\\
0.56061	0.00038464\\
0.56104	0.00038462\\
0.56146	0.00038459\\
0.56189	0.00038456\\
0.56232	0.00038454\\
0.56274	0.00038451\\
0.56317	0.00038449\\
0.5636	0.00038446\\
0.56402	0.00038444\\
0.56445	0.00038441\\
0.56488	0.00038439\\
0.56531	0.00038436\\
0.56573	0.00038433\\
0.56616	0.00038431\\
0.56659	0.00038428\\
0.5671	0.00038425\\
0.56762	0.00038422\\
0.56814	0.00038419\\
0.56865	0.00038416\\
0.56917	0.00038413\\
0.56969	0.0003841\\
0.5702	0.00038407\\
0.57072	0.00038404\\
0.57123	0.00038401\\
0.57175	0.00038398\\
0.57227	0.00038395\\
0.57278	0.00038392\\
0.5733	0.00038389\\
0.57382	0.00038386\\
0.57433	0.00038383\\
0.57485	0.0003838\\
0.57537	0.00038377\\
0.57588	0.00038374\\
0.5764	0.00038371\\
0.57692	0.00038368\\
0.57743	0.00038365\\
0.57795	0.00038362\\
0.57847	0.00038359\\
0.57898	0.00038356\\
0.5795	0.00038353\\
0.58002	0.0003835\\
0.58053	0.00038347\\
0.58105	0.00038344\\
0.58156	0.00038341\\
0.58208	0.00038338\\
0.5826	0.00038335\\
0.58311	0.00038332\\
0.58363	0.00038329\\
0.58415	0.00038326\\
0.58466	0.00038323\\
0.58518	0.0003832\\
0.5857	0.00038317\\
0.58621	0.00038314\\
0.58673	0.00038311\\
0.58725	0.00038308\\
0.58776	0.00038306\\
0.58828	0.00038303\\
0.5888	0.000383\\
0.58931	0.00038297\\
0.58983	0.00038294\\
0.59035	0.00038291\\
0.59086	0.00038288\\
0.59138	0.00038285\\
0.5919	0.00038282\\
0.59241	0.00038279\\
0.59293	0.00038277\\
0.59344	0.00038274\\
0.59396	0.00038271\\
0.59448	0.00038268\\
0.59499	0.00038265\\
0.59551	0.00038262\\
0.59603	0.00038259\\
0.59654	0.00038256\\
0.59706	0.00038254\\
0.59758	0.00038251\\
0.59809	0.00038248\\
0.59861	0.00038245\\
0.59913	0.00038242\\
0.59964	0.00038239\\
0.60016	0.00038236\\
0.60068	0.00038234\\
0.60119	0.00038231\\
0.60171	0.00038228\\
0.60223	0.00038225\\
0.60274	0.00038222\\
0.60326	0.0003822\\
0.60378	0.00038217\\
0.60429	0.00038214\\
0.60481	0.00038211\\
0.60532	0.00038208\\
0.60584	0.00038206\\
0.60636	0.00038203\\
0.60687	0.000382\\
0.60739	0.00038197\\
0.60791	0.00038194\\
0.60859	0.00038191\\
0.60928	0.00038187\\
0.60997	0.00038183\\
0.61066	0.0003818\\
0.61134	0.00038176\\
0.61203	0.00038172\\
0.61272	0.00038169\\
0.61341	0.00038165\\
0.61409	0.00038161\\
0.61478	0.00038158\\
0.61547	0.00038154\\
0.61616	0.0003815\\
0.61684	0.00038147\\
0.61753	0.00038143\\
0.61822	0.00038139\\
0.61891	0.00038136\\
0.61959	0.00038132\\
0.62028	0.00038129\\
0.62097	0.00038125\\
0.62166	0.00038121\\
0.62234	0.00038118\\
0.62303	0.00038114\\
0.62372	0.00038111\\
0.62441	0.00038107\\
0.62509	0.00038103\\
0.62578	0.000381\\
0.62647	0.00038096\\
0.62716	0.00038093\\
0.62784	0.00038089\\
0.62853	0.00038086\\
0.62922	0.00038082\\
0.62991	0.00038079\\
0.63059	0.00038075\\
0.63128	0.00038072\\
0.63197	0.00038068\\
0.63266	0.00038065\\
0.63334	0.00038061\\
0.63403	0.00038058\\
0.63472	0.00038054\\
0.63541	0.00038051\\
0.63688	0.00038043\\
0.63835	0.00038036\\
0.63983	0.00038028\\
0.6413	0.00038021\\
0.64278	0.00038014\\
0.64425	0.00038006\\
0.64572	0.00037999\\
0.6472	0.00037992\\
0.64867	0.00037984\\
0.65014	0.00037977\\
0.65162	0.0003797\\
0.65309	0.00037963\\
0.65456	0.00037956\\
0.65604	0.00037948\\
0.65751	0.00037941\\
0.65899	0.00037934\\
0.66046	0.00037927\\
0.66193	0.0003792\\
0.66341	0.00037913\\
0.66488	0.00037906\\
0.66635	0.00037899\\
0.66783	0.00037892\\
0.6693	0.00037885\\
0.67078	0.00037878\\
0.67225	0.00037871\\
0.67372	0.00037864\\
0.6752	0.00037857\\
0.67667	0.00037851\\
0.67814	0.00037844\\
0.67962	0.00037837\\
0.68109	0.0003783\\
0.68257	0.00037823\\
0.68404	0.00037817\\
0.68551	0.0003781\\
0.68699	0.00037803\\
0.68846	0.00037796\\
0.68993	0.0003779\\
0.69141	0.00037783\\
0.69288	0.00037777\\
0.69435	0.0003777\\
0.69856	0.00037751\\
0.70276	0.00037733\\
0.70697	0.00037714\\
0.71117	0.00037696\\
0.71538	0.00037678\\
0.71958	0.0003766\\
0.72379	0.00037642\\
0.72799	0.00037624\\
0.7322	0.00037607\\
0.7364	0.00037589\\
0.74061	0.00037572\\
0.74481	0.00037555\\
0.74901	0.00037538\\
0.75322	0.00037521\\
0.75742	0.00037504\\
0.76163	0.00037488\\
0.76583	0.00037471\\
0.77004	0.00037455\\
0.77424	0.00037439\\
0.77845	0.00037423\\
0.78265	0.00037407\\
0.78686	0.00037391\\
0.79106	0.00037375\\
0.79527	0.00037359\\
0.79947	0.00037344\\
0.80367	0.00037329\\
0.80788	0.00037313\\
0.81208	0.00037298\\
0.81629	0.00037283\\
0.82049	0.00037268\\
0.8247	0.00037253\\
0.8289	0.00037238\\
0.83311	0.00037224\\
0.83731	0.00037209\\
0.84152	0.00037195\\
0.84572	0.00037181\\
0.84993	0.00037166\\
0.85413	0.00037152\\
0.85833	0.00037138\\
0.86254	0.00037124\\
0.86598	0.00037113\\
0.86941	0.00037102\\
0.87285	0.0003709\\
0.87629	0.00037079\\
0.87972	0.00037068\\
0.88316	0.00037057\\
0.88659	0.00037046\\
0.89003	0.00037035\\
0.89347	0.00037025\\
0.8969	0.00037014\\
0.90034	0.00037003\\
0.90378	0.00036992\\
0.90721	0.00036982\\
0.91065	0.00036971\\
0.91409	0.00036961\\
0.91752	0.0003695\\
0.92096	0.0003694\\
0.9244	0.0003693\\
0.92783	0.00036919\\
0.93127	0.00036909\\
0.93471	0.00036899\\
0.93814	0.00036889\\
0.94158	0.00036879\\
0.94502	0.00036869\\
0.94845	0.00036859\\
0.95189	0.00036849\\
0.95533	0.00036839\\
0.95876	0.00036829\\
0.9622	0.00036819\\
0.96563	0.0003681\\
0.96907	0.000368\\
0.97251	0.0003679\\
0.97594	0.00036781\\
0.97938	0.00036771\\
0.98282	0.00036762\\
0.98625	0.00036752\\
0.98969	0.00036743\\
0.99313	0.00036734\\
0.99656	0.00036724\\
1	0.00036715\\
};
\end{axis}
\end{tikzpicture}%

%% file: src/R0_res4.tex
%
%
\definecolor{mycolor1}{rgb}{0.00000,0.44700,0.74100}%
\begin{tikzpicture}

\begin{axis}[%
width=0.25\textwidth,
height=0.2\textwidth,
scale only axis,
unbounded coords=jump,
xmin=0,
xmax=1,
xlabel style={font=\color{white!15!black}},
xlabel={$t$},
ymode=log,
ymin=1e-16,
ymax=0.01,
yminorticks=true,
yticklabels={},
axis background/.style={fill=white}
]
\addplot [color=mycolor1, forget plot]
  table[row sep=crcr]{%
0	2.2204e-16\\
5.1875e-06	0.0042064\\
1.0375e-05	0.0039727\\
1.5562e-05	0.0033646\\
2.075e-05	0.002652\\
2.5937e-05	0.0019274\\
3.1125e-05	0.0012319\\
3.6312e-05	0.00058605\\
4.15e-05	2.2508e-16\\
4.6687e-05	0.00052143\\
5.1875e-05	0.00097666\\
5.7062e-05	0.0013661\\
6.225e-05	0.0016915\\
6.7437e-05	0.0019553\\
7.2625e-05	0.0021607\\
7.7812e-05	0.002311\\
8.3e-05	0.0024098\\
8.8187e-05	0.002461\\
9.3375e-05	0.0024684\\
9.8562e-05	0.0024358\\
0.00010375	0.0023671\\
0.00010894	0.002266\\
0.00011412	0.0021362\\
0.00011931	0.0019814\\
0.0001245	0.001805\\
0.00012969	0.0016106\\
0.00013487	0.0014012\\
0.00014006	0.0011802\\
0.00014525	0.00095042\\
0.00015044	0.00071486\\
0.00015562	0.00047624\\
0.00016081	0.00023714\\
0.000166	2.2204e-16\\
0.00017119	0.00023287\\
0.00017637	0.00045932\\
0.00018156	0.00067737\\
0.00018675	0.00088516\\
0.00019194	0.001081\\
0.00019712	0.0012634\\
0.00020231	0.0014309\\
0.0002075	0.0015824\\
0.00021269	0.0017167\\
0.00021787	0.001833\\
0.00022306	0.0019306\\
0.00022825	0.0020087\\
0.00023344	0.0020671\\
0.00023862	0.0021053\\
0.00024381	0.0021234\\
0.000249	0.0021213\\
0.00025419	0.0020993\\
0.00025937	0.0020577\\
0.00026456	0.001997\\
0.00026975	0.001918\\
0.00027494	0.0018214\\
0.00028012	0.0017082\\
0.00028531	0.0015797\\
0.0002905	0.0014371\\
0.00029569	0.0012819\\
0.00030087	0.0011157\\
0.00030606	0.0009404\\
0.00031125	0.00075784\\
0.00031644	0.00057018\\
0.00032162	0.00037967\\
0.00032681	0.00018875\\
0.000332	2.2204e-16\\
0.00033719	0.00018383\\
0.00034237	0.00035984\\
0.00034756	0.00052497\\
0.00035275	0.00067599\\
0.00035794	0.00080953\\
0.00036312	0.00092205\\
0.00036831	0.0010099\\
0.0003735	0.0010691\\
0.00037869	0.0010958\\
0.00038387	0.0010857\\
0.00038906	0.0010345\\
0.00039425	0.00093777\\
0.00039944	0.00079083\\
0.00040462	0.00058889\\
0.00040981	0.00032699\\
0.000415	2.2925e-16\\
0.00048301	0.00034138\\
0.00055101	0.00048132\\
0.00061902	0.00050171\\
0.00068703	0.00045279\\
0.00075504	0.00036292\\
0.00082305	0.00025009\\
0.00089105	0.00012634\\
0.00095906	2.2964e-16\\
0.0010271	0.00012311\\
0.0010951	0.00023893\\
0.0011631	0.00034554\\
0.0012311	0.00044406\\
0.0012991	0.00053054\\
0.0013671	0.00060383\\
0.0014351	0.00066325\\
0.0015031	0.00070852\\
0.0015711	0.00073969\\
0.0016391	0.00075704\\
0.0017071	0.00076106\\
0.0017752	0.0007524\\
0.0018432	0.00073184\\
0.0019112	0.00070024\\
0.0019792	0.00065855\\
0.0020472	0.00060776\\
0.0021152	0.00054892\\
0.0021832	0.00048306\\
0.0022512	0.00041127\\
0.0023192	0.0003346\\
0.0023872	0.00025411\\
0.0024552	0.00017084\\
0.0025232	8.5808e-05\\
0.0025913	2.2204e-16\\
0.0026593	8.5633e-05\\
0.0027273	0.00017018\\
0.0027953	0.00025277\\
0.0028633	0.00033258\\
0.0029313	0.00040883\\
0.0029993	0.00048082\\
0.0030673	0.00054787\\
0.0031353	0.00060938\\
0.0032033	0.00066479\\
0.0032713	0.00071364\\
0.0033393	0.00075548\\
0.0034073	0.00078997\\
0.0034754	0.00081681\\
0.0035434	0.00083578\\
0.0036114	0.00084672\\
0.0036794	0.00084954\\
0.0037474	0.00084423\\
0.0038154	0.00083085\\
0.0038834	0.0008095\\
0.0039514	0.00078041\\
0.0040194	0.00074383\\
0.0040874	0.00070011\\
0.0041554	0.00064967\\
0.0042234	0.00059301\\
0.0042914	0.00053069\\
0.0043595	0.00046336\\
0.0044275	0.00039175\\
0.0044955	0.00031665\\
0.0045635	0.00023893\\
0.0046315	0.00015955\\
0.0046995	7.9538e-05\\
0.0047675	2.2204e-16\\
0.0048355	7.7879e-05\\
0.0049035	0.00015283\\
0.0049715	0.00022353\\
0.0050395	0.00028853\\
0.0051075	0.00034635\\
0.0051756	0.00039541\\
0.0052436	0.00043405\\
0.0053116	0.00046053\\
0.0053796	0.00047304\\
0.0054476	0.00046967\\
0.0055156	0.00044847\\
0.0055836	0.00040736\\
0.0056516	0.00034421\\
0.0057196	0.00025681\\
0.0057876	0.00014287\\
0.0058556	2.5082e-16\\
0.0062716	0.00012422\\
0.0066875	0.00019538\\
0.0071034	0.00021764\\
0.0075194	0.00020633\\
0.0079353	0.00017206\\
0.0083512	0.00012257\\
0.0087671	6.3715e-05\\
0.0091831	2.5064e-16\\
0.009599	6.5104e-05\\
0.010015	0.00012891\\
0.010431	0.00018936\\
0.010847	0.00024491\\
0.011263	0.00029439\\
0.011679	0.00033701\\
0.012095	0.00037222\\
0.012511	0.00039972\\
0.012926	0.00041939\\
0.013342	0.00043126\\
0.013758	0.00043551\\
0.014174	0.0004324\\
0.01459	0.0004223\\
0.015006	0.00040564\\
0.015422	0.0003829\\
0.015838	0.00035462\\
0.016254	0.00032137\\
0.01667	0.00028372\\
0.017086	0.0002423\\
0.017502	0.00019771\\
0.017918	0.00015057\\
0.018334	0.00010151\\
0.018749	5.1116e-05\\
0.019165	2.7984e-16\\
0.019581	5.1261e-05\\
0.019997	0.0001021\\
0.020413	0.00015199\\
0.020829	0.00020041\\
0.021245	0.00024686\\
0.021661	0.0002909\\
0.022077	0.0003321\\
0.022493	0.00037006\\
0.022909	0.00040442\\
0.023325	0.00043488\\
0.023741	0.00046114\\
0.024157	0.00048296\\
0.024572	0.00050014\\
0.024988	0.00051252\\
0.025404	0.00051998\\
0.02582	0.00052245\\
0.026236	0.00051989\\
0.026652	0.00051231\\
0.027068	0.00049979\\
0.027484	0.00048242\\
0.0279	0.00046036\\
0.028316	0.00043381\\
0.028732	0.00040301\\
0.029148	0.00036826\\
0.029564	0.00032992\\
0.02998	0.00028836\\
0.030396	0.00024404\\
0.030811	0.00019745\\
0.031227	0.00014914\\
0.031643	9.9797e-05\\
0.032059	4.9851e-05\\
0.032475	2.2204e-16\\
0.032891	4.9001e-05\\
0.033307	9.6344e-05\\
0.033723	0.00014117\\
0.034139	0.00018255\\
0.034555	0.00021952\\
0.034971	0.00025104\\
0.035387	0.00027604\\
0.035803	0.00029336\\
0.036219	0.00030182\\
0.036634	0.00030015\\
0.03705	0.00028704\\
0.037466	0.00026113\\
0.037882	0.00022098\\
0.038298	0.00016511\\
0.038714	9.1986e-05\\
0.03913	2.7935e-16\\
0.040867	6.1938e-05\\
0.042605	0.00010341\\
0.044342	0.00011935\\
0.04608	0.00011618\\
0.047817	9.898e-05\\
0.049555	7.1806e-05\\
0.051292	3.7927e-05\\
0.05303	2.7842e-16\\
0.054767	3.9815e-05\\
0.056504	7.9765e-05\\
0.058242	0.00011844\\
0.059979	0.00015471\\
0.061717	0.0001877\\
0.063454	0.00021674\\
0.065192	0.00024133\\
0.066929	0.00026114\\
0.068667	0.00027597\\
0.070404	0.00028572\\
0.072141	0.00029041\\
0.073879	0.00029011\\
0.075616	0.00028499\\
0.077354	0.00027527\\
0.079091	0.00026123\\
0.080829	0.00024316\\
0.082566	0.00022144\\
0.084304	0.00019641\\
0.086041	0.00016849\\
0.087778	0.00013808\\
0.089516	0.00010559\\
0.091253	7.1467e-05\\
0.092991	3.6128e-05\\
0.094728	3.3445e-16\\
0.096466	3.6495e-05\\
0.098203	7.2945e-05\\
0.099941	0.00010895\\
0.10168	0.00014412\\
0.10342	0.00017808\\
0.10515	0.00021049\\
0.10689	0.00024101\\
0.10863	0.00026932\\
0.11037	0.00029515\\
0.1121	0.00031823\\
0.11384	0.00033833\\
0.11558	0.00035524\\
0.11731	0.00036879\\
0.11905	0.00037883\\
0.12079	0.00038524\\
0.12253	0.00038795\\
0.12426	0.00038691\\
0.126	0.0003821\\
0.12774	0.00037355\\
0.12948	0.00036131\\
0.13121	0.00034548\\
0.13295	0.00032619\\
0.13469	0.00030361\\
0.13643	0.00027796\\
0.13816	0.00024947\\
0.1399	0.00021844\\
0.14164	0.00018519\\
0.14338	0.00015009\\
0.14511	0.00011355\\
0.14685	7.602e-05\\
0.14859	3.7992e-05\\
0.15033	2.2204e-16\\
0.15206	3.738e-05\\
0.1538	7.353e-05\\
0.15554	0.00010779\\
0.15728	0.00013945\\
0.15901	0.00016776\\
0.16075	0.00019193\\
0.16249	0.00021113\\
0.16423	0.00022447\\
0.16596	0.00023103\\
0.1677	0.00022984\\
0.16944	0.00021989\\
0.17118	0.00020011\\
0.17291	0.00016941\\
0.17465	0.00012663\\
0.17639	7.057e-05\\
0.17813	2.794e-16\\
0.18067	6.6932e-06\\
0.18321	1.0947e-05\\
0.18576	1.3041e-05\\
0.1883	1.3149e-05\\
0.19084	1.1604e-05\\
0.19339	8.7144e-06\\
0.19593	4.7613e-06\\
0.19848	2.4985e-16\\
0.20102	5.3474e-06\\
0.20356	1.1146e-05\\
0.20611	1.7179e-05\\
0.20865	2.3252e-05\\
0.21119	2.9191e-05\\
0.21374	3.4838e-05\\
0.21628	4.0058e-05\\
0.21883	4.4729e-05\\
0.22137	4.8747e-05\\
0.22391	5.2022e-05\\
0.22646	5.4478e-05\\
0.229	5.6055e-05\\
0.23155	5.6704e-05\\
0.23409	5.6387e-05\\
0.23663	5.5082e-05\\
0.23918	5.2774e-05\\
0.24172	4.9461e-05\\
0.24426	4.5151e-05\\
0.24681	3.9862e-05\\
0.24935	3.3621e-05\\
0.2519	2.6464e-05\\
0.25444	1.8438e-05\\
0.25698	9.5959e-06\\
0.25953	2.498e-16\\
0.26207	1.028e-05\\
0.26461	2.1165e-05\\
0.26716	3.2571e-05\\
0.2697	4.4405e-05\\
0.27225	5.6569e-05\\
0.27479	6.8955e-05\\
0.27733	8.1452e-05\\
0.27988	9.3941e-05\\
0.28242	0.0001063\\
0.28496	0.00011839\\
0.28751	0.00013008\\
0.29005	0.00014122\\
0.2926	0.00015168\\
0.29514	0.00016129\\
0.29768	0.0001699\\
0.30023	0.00017735\\
0.30277	0.00018347\\
0.30532	0.00018809\\
0.30786	0.00019103\\
0.3104	0.00019212\\
0.31295	0.00019119\\
0.31549	0.00018804\\
0.31803	0.0001825\\
0.32058	0.0001744\\
0.32312	0.00016357\\
0.32567	0.00014983\\
0.32821	0.00013305\\
0.33075	0.00011308\\
0.3333	8.9837e-05\\
0.33584	6.324e-05\\
0.33838	3.3277e-05\\
0.34093	2.7839e-16\\
0.34347	3.645e-05\\
0.34602	7.5816e-05\\
0.34856	0.00011769\\
0.3511	0.00016147\\
0.35365	0.0002063\\
0.35619	0.00025103\\
0.35873	0.0002941\\
0.36128	0.00033347\\
0.36382	0.00036651\\
0.36637	0.00038983\\
0.36891	0.00039919\\
0.37145	0.00038925\\
0.374	0.0003534\\
0.37654	0.00028351\\
0.37908	0.0001697\\
0.38163	2.7842e-16\\
0.38237	3.299e-05\\
0.3831	3.6096e-05\\
0.38384	3.3959e-05\\
0.38458	2.9242e-05\\
0.38531	2.3007e-05\\
0.38605	1.5825e-05\\
0.38679	8.0679e-06\\
0.38752	2.7842e-16\\
0.38826	8.1695e-06\\
0.389	1.6268e-05\\
0.38973	2.4146e-05\\
0.39047	3.1673e-05\\
0.39121	3.8731e-05\\
0.39194	4.5213e-05\\
0.39268	5.1023e-05\\
0.39342	5.6074e-05\\
0.39416	6.0288e-05\\
0.39489	6.3598e-05\\
0.39563	6.5943e-05\\
0.39637	6.7273e-05\\
0.3971	6.755e-05\\
0.39784	6.6742e-05\\
0.39858	6.4829e-05\\
0.39931	6.1803e-05\\
0.40005	5.7664e-05\\
0.40079	5.2425e-05\\
0.40152	4.611e-05\\
0.40226	3.8752e-05\\
0.403	3.04e-05\\
0.40373	2.1109e-05\\
0.40447	1.0949e-05\\
0.40521	2.2204e-16\\
0.40594	1.1648e-05\\
0.40668	2.3893e-05\\
0.40742	3.6623e-05\\
0.40816	4.9719e-05\\
0.40889	6.3048e-05\\
0.40963	7.6474e-05\\
0.41037	8.9849e-05\\
0.4111	0.00010302\\
0.41184	0.00011583\\
0.41258	0.00012812\\
0.41331	0.00013972\\
0.41405	0.00015047\\
0.41479	0.00016019\\
0.41552	0.00016872\\
0.41626	0.00017591\\
0.417	0.00018158\\
0.41773	0.00018561\\
0.41847	0.00018784\\
0.41921	0.00018815\\
0.41994	0.00018643\\
0.42068	0.00018258\\
0.42142	0.00017654\\
0.42216	0.00016824\\
0.42289	0.00015767\\
0.42363	0.00014484\\
0.42437	0.00012978\\
0.4251	0.00011257\\
0.42584	9.3326e-05\\
0.42658	7.221e-05\\
0.42731	4.9433e-05\\
0.42805	2.5257e-05\\
0.42879	2.2204e-16\\
0.42952	2.5957e-05\\
0.43026	5.2168e-05\\
0.431	7.8117e-05\\
0.43173	0.00010321\\
0.43247	0.00012679\\
0.43321	0.00014809\\
0.43394	0.00016627\\
0.43468	0.0001804\\
0.43542	0.00018943\\
0.43616	0.00019223\\
0.43689	0.00018755\\
0.43763	0.00017403\\
0.43837	0.00015019\\
0.4391	0.00011441\\
0.43984	6.4973e-05\\
0.44058	3.2986e-16\\
0.44119	3.798e-05\\
0.44179	4.2994e-05\\
0.4424	4.1654e-05\\
0.44301	3.6803e-05\\
0.44362	2.9621e-05\\
0.44423	2.0791e-05\\
0.44484	1.0792e-05\\
0.44545	4.3143e-16\\
0.44606	1.1266e-05\\
0.44667	2.2726e-05\\
0.44728	3.4124e-05\\
0.44788	4.5228e-05\\
0.44849	5.5821e-05\\
0.4491	6.5705e-05\\
0.44971	7.4698e-05\\
0.45032	8.2633e-05\\
0.45093	8.9361e-05\\
0.45154	9.475e-05\\
0.45215	9.8685e-05\\
0.45276	0.00010107\\
0.45336	0.00010183\\
0.45397	0.0001009\\
0.45458	9.8238e-05\\
0.45519	9.3837e-05\\
0.4558	8.7699e-05\\
0.45641	7.9839e-05\\
0.45702	7.0295e-05\\
0.45763	5.9125e-05\\
0.45824	4.6408e-05\\
0.45885	3.2235e-05\\
0.45945	1.6721e-05\\
0.46006	3.7956e-16\\
0.46067	1.7776e-05\\
0.46128	3.6439e-05\\
0.46189	5.5808e-05\\
0.4625	7.5685e-05\\
0.46311	9.5863e-05\\
0.46372	0.00011612\\
0.46433	0.00013621\\
0.46493	0.00015592\\
0.46554	0.00017498\\
0.46615	0.00019315\\
0.46676	0.00021018\\
0.46737	0.00022582\\
0.46798	0.00023983\\
0.46859	0.00025196\\
0.4692	0.00026199\\
0.46981	0.00026971\\
0.47042	0.0002749\\
0.47102	0.00027739\\
0.47163	0.00027701\\
0.47224	0.00027363\\
0.47285	0.00026713\\
0.47346	0.00025744\\
0.47407	0.00024453\\
0.47468	0.00022839\\
0.47529	0.00020908\\
0.4759	0.00018668\\
0.47651	0.00016134\\
0.47711	0.00013328\\
0.47772	0.00010274\\
0.47833	7.0074e-05\\
0.47894	3.5669e-05\\
0.47955	4.4105e-16\\
0.48016	3.6379e-05\\
0.48077	7.2834e-05\\
0.48138	0.00010864\\
0.48199	0.00014299\\
0.48259	0.00017496\\
0.4832	0.00020356\\
0.48381	0.00022765\\
0.48442	0.00024602\\
0.48503	0.00025732\\
0.48564	0.0002601\\
0.48625	0.00025277\\
0.48686	0.00023362\\
0.48747	0.00020082\\
0.48808	0.00015238\\
0.48868	8.6193e-05\\
0.48929	3.3307e-16\\
0.4899	5.077e-05\\
0.49051	5.6073e-05\\
0.49112	5.357e-05\\
0.49173	4.6927e-05\\
0.49234	3.7571e-05\\
0.49295	2.6288e-05\\
0.49356	1.3622e-05\\
0.49417	6.5296e-16\\
0.49477	1.4211e-05\\
0.49538	2.8693e-05\\
0.49599	4.3172e-05\\
0.4966	5.7403e-05\\
0.49721	7.1106e-05\\
0.49782	8.402e-05\\
0.49843	9.5914e-05\\
0.49904	0.00010657\\
0.49965	0.00011573\\
0.50025	0.0001232\\
0.50086	0.0001288\\
0.50147	0.00013239\\
0.50208	0.00013384\\
0.50269	0.00013306\\
0.5033	0.00012996\\
0.50391	0.00012448\\
0.50452	0.00011662\\
0.50513	0.00010639\\
0.50574	9.3844e-05\\
0.50634	7.9052e-05\\
0.50695	6.2124e-05\\
0.50756	4.3194e-05\\
0.50817	2.2424e-05\\
0.50878	5.3235e-16\\
0.50939	2.3866e-05\\
0.51	4.8939e-05\\
0.51061	7.4956e-05\\
0.51122	0.00010163\\
0.51183	0.00012867\\
0.51243	0.00015576\\
0.51304	0.00018258\\
0.51365	0.00020878\\
0.51426	0.00023405\\
0.51487	0.00025804\\
0.51548	0.00028041\\
0.51609	0.00030084\\
0.5167	0.000319\\
0.51731	0.00033458\\
0.51791	0.00034729\\
0.51852	0.00035686\\
0.51913	0.00036302\\
0.51974	0.00036557\\
0.52035	0.00036431\\
0.52096	0.00035909\\
0.52157	0.0003498\\
0.52218	0.00033636\\
0.52279	0.00031876\\
0.5234	0.00029703\\
0.524	0.00027127\\
0.52461	0.00024163\\
0.52522	0.00020833\\
0.52583	0.00017167\\
0.52644	0.00013201\\
0.52705	8.9813e-05\\
0.52766	4.5601e-05\\
0.52827	5.3324e-16\\
0.52888	4.6273e-05\\
0.52949	9.2404e-05\\
0.53009	0.00013748\\
0.5307	0.00018047\\
0.53131	0.00022025\\
0.53192	0.00025557\\
0.53253	0.00028507\\
0.53314	0.00030727\\
0.53375	0.00032053\\
0.53436	0.00032314\\
0.53497	0.0003132\\
0.53557	0.00028871\\
0.53618	0.00024752\\
0.53679	0.00018732\\
0.5374	0.00010568\\
0.53801	3.3237e-16\\
0.53868	5.9395e-05\\
0.53935	6.4398e-05\\
0.54002	6.0677e-05\\
0.54069	5.2547e-05\\
0.54136	4.1652e-05\\
0.54203	2.8885e-05\\
0.5427	1.4846e-05\\
0.54337	5.1951e-16\\
0.54404	1.5261e-05\\
0.54471	3.0607e-05\\
0.54538	4.5762e-05\\
0.54605	6.0484e-05\\
0.54672	7.4491e-05\\
0.54739	8.7532e-05\\
0.54806	9.9387e-05\\
0.54873	0.00010986\\
0.5494	0.0001187\\
0.55007	0.00012574\\
0.55074	0.00013082\\
0.55141	0.00013383\\
0.55208	0.00013466\\
0.55275	0.00013326\\
0.55342	0.00012957\\
0.55409	0.00012356\\
0.55476	0.00011525\\
0.55543	0.00010468\\
0.5561	9.1935e-05\\
0.55677	7.7111e-05\\
0.55744	6.034e-05\\
0.55811	4.1775e-05\\
0.55878	2.1596e-05\\
0.55945	4.9202e-16\\
0.56012	2.2793e-05\\
0.56079	4.6545e-05\\
0.56146	7.1002e-05\\
0.56213	9.5887e-05\\
0.5628	0.00012091\\
0.56347	0.00014578\\
0.56414	0.00017019\\
0.5648	0.00019384\\
0.56547	0.00021643\\
0.56614	0.00023766\\
0.56681	0.00025724\\
0.56748	0.00027489\\
0.56815	0.00029032\\
0.56882	0.0003033\\
0.56949	0.00031358\\
0.57016	0.00032095\\
0.57083	0.00032521\\
0.5715	0.00032621\\
0.57217	0.00032382\\
0.57284	0.00031794\\
0.57351	0.00030851\\
0.57418	0.00029551\\
0.57485	0.00027897\\
0.57552	0.00025896\\
0.57619	0.0002356\\
0.57686	0.00020906\\
0.57753	0.00017957\\
0.5782	0.00014742\\
0.57887	0.00011294\\
0.57954	7.6549e-05\\
0.58021	3.8722e-05\\
0.58088	3.9517e-16\\
0.58155	3.9005e-05\\
0.58222	7.7607e-05\\
0.58289	0.00011505\\
0.58356	0.00015048\\
0.58423	0.000183\\
0.5849	0.00021159\\
0.58557	0.00023518\\
0.58624	0.0002526\\
0.58691	0.0002626\\
0.58758	0.00026382\\
0.58825	0.00025483\\
0.58892	0.00023411\\
0.58959	0.00020003\\
0.59026	0.00015087\\
0.59093	8.4835e-05\\
0.5916	2.8449e-16\\
0.59258	6.5582e-05\\
0.59356	7.3682e-05\\
0.59454	7.1098e-05\\
0.59552	6.2588e-05\\
0.5965	5.0165e-05\\
0.59748	3.5039e-05\\
0.59846	1.8084e-05\\
0.59945	3.5419e-16\\
0.60043	1.862e-05\\
0.60141	3.7261e-05\\
0.60239	5.5468e-05\\
0.60337	7.2842e-05\\
0.60435	8.9081e-05\\
0.60533	0.00010389\\
0.60631	0.00011704\\
0.60729	0.00012828\\
0.60827	0.00013742\\
0.60925	0.00014432\\
0.61023	0.00014887\\
0.61121	0.00015097\\
0.61219	0.00015056\\
0.61318	0.00014764\\
0.61416	0.00014224\\
0.61514	0.00013441\\
0.61612	0.00012425\\
0.6171	0.00011186\\
0.61808	9.7368e-05\\
0.61906	8.0946e-05\\
0.62004	6.278e-05\\
0.62102	4.3082e-05\\
0.622	2.2075e-05\\
0.62298	3.1522e-16\\
0.62396	2.2891e-05\\
0.62494	4.6339e-05\\
0.62593	7.0077e-05\\
0.62691	9.3834e-05\\
0.62789	0.00011734\\
0.62887	0.00014031\\
0.62985	0.00016248\\
0.63083	0.00018356\\
0.63181	0.00020331\\
0.63279	0.00022147\\
0.63377	0.00023781\\
0.63475	0.00025212\\
0.63573	0.0002642\\
0.63671	0.00027387\\
0.63769	0.00028097\\
0.63867	0.00028538\\
0.63966	0.00028698\\
0.64064	0.00028571\\
0.64162	0.0002815\\
0.6426	0.00027435\\
0.64358	0.00026427\\
0.64456	0.00025131\\
0.64554	0.00023554\\
0.64652	0.0002171\\
0.6475	0.00019613\\
0.64848	0.00017282\\
0.64946	0.00014743\\
0.65044	0.0001202\\
0.65142	9.1471e-05\\
0.65241	6.1586e-05\\
0.65339	3.0949e-05\\
0.65437	2.5347e-16\\
0.65535	3.0773e-05\\
0.65633	6.0839e-05\\
0.65731	8.9623e-05\\
0.65829	0.0001165\\
0.65927	0.00014081\\
0.66025	0.00016183\\
0.66123	0.0001788\\
0.66221	0.00019091\\
0.66319	0.00019731\\
0.66417	0.00019709\\
0.66515	0.0001893\\
0.66614	0.00017294\\
0.66712	0.00014695\\
0.6681	0.00011024\\
0.66908	6.1654e-05\\
0.67006	2.5327e-16\\
0.67418	3.2559e-05\\
0.67831	6.054e-05\\
0.68243	7.0947e-05\\
0.68656	6.8916e-05\\
0.69068	5.8285e-05\\
0.6948	4.1903e-05\\
0.69893	2.1924e-05\\
0.70305	2.5161e-16\\
0.70718	2.2592e-05\\
0.7113	4.4866e-05\\
0.71543	6.6065e-05\\
0.71955	8.5614e-05\\
0.72367	0.00010309\\
0.7278	0.00011818\\
0.73192	0.00013069\\
0.73605	0.0001405\\
0.74017	0.00014756\\
0.7443	0.00015186\\
0.74842	0.00015346\\
0.75254	0.00015247\\
0.75667	0.00014899\\
0.76079	0.00014318\\
0.76492	0.00013521\\
0.76904	0.00012527\\
0.77317	0.00011356\\
0.77729	0.00010029\\
0.78141	8.5666e-05\\
0.78554	6.9915e-05\\
0.78966	5.3256e-05\\
0.79379	3.5906e-05\\
0.79791	1.8084e-05\\
0.80204	2.949e-16\\
0.80616	1.8138e-05\\
0.81028	3.6129e-05\\
0.81441	5.3783e-05\\
0.81853	7.0915e-05\\
0.82266	8.7353e-05\\
0.82678	0.00010293\\
0.83091	0.00011751\\
0.83503	0.00013093\\
0.83915	0.00014308\\
0.84328	0.00015385\\
0.8474	0.00016312\\
0.85153	0.00017083\\
0.85565	0.00017688\\
0.85977	0.00018124\\
0.8639	0.00018386\\
0.86802	0.00018471\\
0.87215	0.00018378\\
0.87627	0.00018108\\
0.8804	0.00017662\\
0.88452	0.00017046\\
0.88864	0.00016264\\
0.89277	0.00015324\\
0.89689	0.00014233\\
0.90102	0.00013004\\
0.90514	0.00011648\\
0.90927	0.00010179\\
0.91339	8.6132e-05\\
0.91751	6.9676e-05\\
0.92164	5.2615e-05\\
0.92576	3.5161e-05\\
0.92989	1.7541e-05\\
0.93401	2.4308e-16\\
0.93814	1.7199e-05\\
0.94226	3.3774e-05\\
0.94638	4.9426e-05\\
0.95051	6.3839e-05\\
0.95463	7.6678e-05\\
0.95876	8.7588e-05\\
0.96288	9.6199e-05\\
0.96701	0.00010212\\
0.97113	0.00010495\\
0.97525	0.00010425\\
0.97938	9.9594e-05\\
0.9835	9.0506e-05\\
0.98763	7.6511e-05\\
0.99175	5.7109e-05\\
0.99588	3.1784e-05\\
1	2.2389e-16\\
};
\addplot [color=black,thick, forget plot]
  table[row sep=crcr]{%
0	inf\\
5.1875e-06	0.0090311\\
1.0375e-05	0.0069169\\
1.5562e-05	0.0059262\\
2.075e-05	0.0053147\\
2.5937e-05	0.0048865\\
3.1125e-05	0.0045639\\
3.6312e-05	0.0043089\\
4.15e-05	0.0041004\\
4.6687e-05	0.0039255\\
5.1875e-05	0.0037759\\
5.7062e-05	0.0036459\\
6.225e-05	0.0035314\\
6.7437e-05	0.0034296\\
7.2625e-05	0.0033382\\
7.7812e-05	0.0032555\\
8.3e-05	0.0031802\\
8.8187e-05	0.0031112\\
9.3375e-05	0.0030476\\
9.8562e-05	0.0029888\\
0.00010375	0.0029342\\
0.00010894	0.0028833\\
0.00011412	0.0028357\\
0.00011931	0.002791\\
0.0001245	0.002749\\
0.00012969	0.0027093\\
0.00013487	0.0026718\\
0.00014006	0.0026363\\
0.00014525	0.0026025\\
0.00015044	0.0025704\\
0.00015562	0.0025398\\
0.00016081	0.0025107\\
0.000166	0.0024828\\
0.00017119	0.0024561\\
0.00017637	0.0024305\\
0.00018156	0.0024059\\
0.00018675	0.0023823\\
0.00019194	0.0023596\\
0.00019712	0.0023378\\
0.00020231	0.0023167\\
0.0002075	0.0022964\\
0.00021269	0.0022768\\
0.00021787	0.0022578\\
0.00022306	0.0022395\\
0.00022825	0.0022217\\
0.00023344	0.0022045\\
0.00023862	0.0021878\\
0.00024381	0.0021717\\
0.000249	0.002156\\
0.00025419	0.0021407\\
0.00025937	0.0021259\\
0.00026456	0.0021115\\
0.00026975	0.0020975\\
0.00027494	0.0020838\\
0.00028012	0.0020706\\
0.00028531	0.0020576\\
0.0002905	0.002045\\
0.00029569	0.0020327\\
0.00030087	0.0020207\\
0.00030606	0.0020089\\
0.00031125	0.0019975\\
0.00031644	0.0019863\\
0.00032162	0.0019754\\
0.00032681	0.0019647\\
0.000332	0.0019542\\
0.00033719	0.001944\\
0.00034237	0.001934\\
0.00034756	0.0019242\\
0.00035275	0.0019146\\
0.00035794	0.0019052\\
0.00036312	0.001896\\
0.00036831	0.0018869\\
0.0003735	0.0018781\\
0.00037869	0.0018694\\
0.00038387	0.0018609\\
0.00038906	0.0018525\\
0.00039425	0.0018443\\
0.00039944	0.0018363\\
0.00040462	0.0018284\\
0.00040981	0.0018206\\
0.000415	0.001813\\
0.00048301	0.0017239\\
0.00055101	0.0016508\\
0.00061902	0.0015893\\
0.00068703	0.0015367\\
0.00075504	0.0014909\\
0.00082305	0.0014505\\
0.00089105	0.0014145\\
0.00095906	0.0013822\\
0.0010271	0.001353\\
0.0010951	0.0013263\\
0.0011631	0.0013019\\
0.0012311	0.0012793\\
0.0012991	0.0012585\\
0.0013671	0.0012391\\
0.0014351	0.0012211\\
0.0015031	0.0012042\\
0.0015711	0.0011883\\
0.0016391	0.0011734\\
0.0017071	0.0011593\\
0.0017752	0.001146\\
0.0018432	0.0011333\\
0.0019112	0.0011214\\
0.0019792	0.0011099\\
0.0020472	0.0010991\\
0.0021152	0.0010887\\
0.0021832	0.0010788\\
0.0022512	0.0010693\\
0.0023192	0.0010602\\
0.0023872	0.0010514\\
0.0024552	0.001043\\
0.0025232	0.001035\\
0.0025913	0.0010272\\
0.0026593	0.0010197\\
0.0027273	0.0010125\\
0.0027953	0.0010055\\
0.0028633	0.00099872\\
0.0029313	0.00099219\\
0.0029993	0.00098587\\
0.0030673	0.00097975\\
0.0031353	0.00097381\\
0.0032033	0.00096805\\
0.0032713	0.00096246\\
0.0033393	0.00095703\\
0.0034073	0.00095175\\
0.0034754	0.00094662\\
0.0035434	0.00094163\\
0.0036114	0.00093677\\
0.0036794	0.00093203\\
0.0037474	0.00092742\\
0.0038154	0.00092292\\
0.0038834	0.00091854\\
0.0039514	0.00091425\\
0.0040194	0.00091008\\
0.0040874	0.000906\\
0.0041554	0.00090201\\
0.0042234	0.00089811\\
0.0042914	0.0008943\\
0.0043595	0.00089058\\
0.0044275	0.00088693\\
0.0044955	0.00088336\\
0.0045635	0.00087987\\
0.0046315	0.00087645\\
0.0046995	0.0008731\\
0.0047675	0.00086981\\
0.0048355	0.00086659\\
0.0049035	0.00086344\\
0.0049715	0.00086034\\
0.0050395	0.00085731\\
0.0051075	0.00085432\\
0.0051756	0.0008514\\
0.0052436	0.00084853\\
0.0053116	0.00084571\\
0.0053796	0.00084294\\
0.0054476	0.00084022\\
0.0055156	0.00083754\\
0.0055836	0.00083491\\
0.0056516	0.00083233\\
0.0057196	0.00082979\\
0.0057876	0.00082729\\
0.0058556	0.00082483\\
0.0062716	0.00081062\\
0.0066875	0.00079767\\
0.0071034	0.00078581\\
0.0075194	0.00077487\\
0.0079353	0.00076476\\
0.0083512	0.00075535\\
0.0087671	0.00074659\\
0.0091831	0.00073838\\
0.009599	0.00073068\\
0.010015	0.00072344\\
0.010431	0.0007166\\
0.010847	0.00071014\\
0.011263	0.00070401\\
0.011679	0.00069819\\
0.012095	0.00069266\\
0.012511	0.00068738\\
0.012926	0.00068235\\
0.013342	0.00067753\\
0.013758	0.00067293\\
0.014174	0.00066851\\
0.01459	0.00066427\\
0.015006	0.0006602\\
0.015422	0.00065628\\
0.015838	0.00065251\\
0.016254	0.00064887\\
0.01667	0.00064536\\
0.017086	0.00064198\\
0.017502	0.0006387\\
0.017918	0.00063553\\
0.018334	0.00063247\\
0.018749	0.0006295\\
0.019165	0.00062662\\
0.019581	0.00062383\\
0.019997	0.00062112\\
0.020413	0.00061848\\
0.020829	0.00061592\\
0.021245	0.00061343\\
0.021661	0.00061101\\
0.022077	0.00060866\\
0.022493	0.00060636\\
0.022909	0.00060412\\
0.023325	0.00060194\\
0.023741	0.00059981\\
0.024157	0.00059774\\
0.024572	0.00059571\\
0.024988	0.00059373\\
0.025404	0.0005918\\
0.02582	0.00058991\\
0.026236	0.00058806\\
0.026652	0.00058626\\
0.027068	0.00058449\\
0.027484	0.00058276\\
0.0279	0.00058107\\
0.028316	0.00057941\\
0.028732	0.00057778\\
0.029148	0.00057619\\
0.029564	0.00057463\\
0.02998	0.0005731\\
0.030396	0.0005716\\
0.030811	0.00057012\\
0.031227	0.00056868\\
0.031643	0.00056726\\
0.032059	0.00056587\\
0.032475	0.0005645\\
0.032891	0.00056316\\
0.033307	0.00056184\\
0.033723	0.00056054\\
0.034139	0.00055927\\
0.034555	0.00055802\\
0.034971	0.00055678\\
0.035387	0.00055557\\
0.035803	0.00055438\\
0.036219	0.00055321\\
0.036634	0.00055205\\
0.03705	0.00055092\\
0.037466	0.0005498\\
0.037882	0.0005487\\
0.038298	0.00054762\\
0.038714	0.00054655\\
0.03913	0.0005455\\
0.040867	0.00054127\\
0.042605	0.00053728\\
0.044342	0.00053352\\
0.04608	0.00052996\\
0.047817	0.00052658\\
0.049555	0.00052337\\
0.051292	0.00052031\\
0.05303	0.00051739\\
0.054767	0.00051461\\
0.056504	0.00051194\\
0.058242	0.00050939\\
0.059979	0.00050694\\
0.061717	0.00050459\\
0.063454	0.00050233\\
0.065192	0.00050016\\
0.066929	0.00049806\\
0.068667	0.00049604\\
0.070404	0.00049409\\
0.072141	0.00049221\\
0.073879	0.00049039\\
0.075616	0.00048863\\
0.077354	0.00048692\\
0.079091	0.00048527\\
0.080829	0.00048366\\
0.082566	0.00048211\\
0.084304	0.0004806\\
0.086041	0.00047913\\
0.087778	0.0004777\\
0.089516	0.00047631\\
0.091253	0.00047496\\
0.092991	0.00047365\\
0.094728	0.00047237\\
0.096466	0.00047112\\
0.098203	0.0004699\\
0.099941	0.00046871\\
0.10168	0.00046756\\
0.10342	0.00046642\\
0.10515	0.00046532\\
0.10689	0.00046424\\
0.10863	0.00046318\\
0.11037	0.00046215\\
0.1121	0.00046114\\
0.11384	0.00046015\\
0.11558	0.00045918\\
0.11731	0.00045824\\
0.11905	0.00045731\\
0.12079	0.0004564\\
0.12253	0.00045551\\
0.12426	0.00045464\\
0.126	0.00045378\\
0.12774	0.00045294\\
0.12948	0.00045212\\
0.13121	0.00045131\\
0.13295	0.00045051\\
0.13469	0.00044973\\
0.13643	0.00044897\\
0.13816	0.00044821\\
0.1399	0.00044748\\
0.14164	0.00044675\\
0.14338	0.00044604\\
0.14511	0.00044533\\
0.14685	0.00044464\\
0.14859	0.00044396\\
0.15033	0.0004433\\
0.15206	0.00044264\\
0.1538	0.00044199\\
0.15554	0.00044136\\
0.15728	0.00044073\\
0.15901	0.00044011\\
0.16075	0.0004395\\
0.16249	0.00043891\\
0.16423	0.00043832\\
0.16596	0.00043773\\
0.1677	0.00043716\\
0.16944	0.0004366\\
0.17118	0.00043604\\
0.17291	0.00043549\\
0.17465	0.00043495\\
0.17639	0.00043442\\
0.17813	0.00043389\\
0.18067	0.00043314\\
0.18321	0.00043239\\
0.18576	0.00043167\\
0.1883	0.00043095\\
0.19084	0.00043025\\
0.19339	0.00042956\\
0.19593	0.00042889\\
0.19848	0.00042822\\
0.20102	0.00042757\\
0.20356	0.00042693\\
0.20611	0.0004263\\
0.20865	0.00042568\\
0.21119	0.00042508\\
0.21374	0.00042448\\
0.21628	0.00042389\\
0.21883	0.00042331\\
0.22137	0.00042274\\
0.22391	0.00042218\\
0.22646	0.00042163\\
0.229	0.00042109\\
0.23155	0.00042056\\
0.23409	0.00042003\\
0.23663	0.00041951\\
0.23918	0.000419\\
0.24172	0.0004185\\
0.24426	0.00041801\\
0.24681	0.00041752\\
0.24935	0.00041704\\
0.2519	0.00041656\\
0.25444	0.0004161\\
0.25698	0.00041563\\
0.25953	0.00041518\\
0.26207	0.00041473\\
0.26461	0.00041429\\
0.26716	0.00041385\\
0.2697	0.00041342\\
0.27225	0.000413\\
0.27479	0.00041258\\
0.27733	0.00041216\\
0.27988	0.00041175\\
0.28242	0.00041135\\
0.28496	0.00041095\\
0.28751	0.00041056\\
0.29005	0.00041017\\
0.2926	0.00040978\\
0.29514	0.00040941\\
0.29768	0.00040903\\
0.30023	0.00040866\\
0.30277	0.00040829\\
0.30532	0.00040793\\
0.30786	0.00040757\\
0.3104	0.00040722\\
0.31295	0.00040687\\
0.31549	0.00040653\\
0.31803	0.00040618\\
0.32058	0.00040585\\
0.32312	0.00040551\\
0.32567	0.00040518\\
0.32821	0.00040485\\
0.33075	0.00040453\\
0.3333	0.00040421\\
0.33584	0.0004039\\
0.33838	0.00040358\\
0.34093	0.00040327\\
0.34347	0.00040297\\
0.34602	0.00040266\\
0.34856	0.00040236\\
0.3511	0.00040206\\
0.35365	0.00040177\\
0.35619	0.00040148\\
0.35873	0.00040119\\
0.36128	0.0004009\\
0.36382	0.00040062\\
0.36637	0.00040034\\
0.36891	0.00040006\\
0.37145	0.00039979\\
0.374	0.00039952\\
0.37654	0.00039925\\
0.37908	0.00039898\\
0.38163	0.00039872\\
0.38237	0.00039864\\
0.3831	0.00039857\\
0.38384	0.00039849\\
0.38458	0.00039841\\
0.38531	0.00039834\\
0.38605	0.00039826\\
0.38679	0.00039819\\
0.38752	0.00039811\\
0.38826	0.00039804\\
0.389	0.00039796\\
0.38973	0.00039789\\
0.39047	0.00039782\\
0.39121	0.00039774\\
0.39194	0.00039767\\
0.39268	0.0003976\\
0.39342	0.00039752\\
0.39416	0.00039745\\
0.39489	0.00039738\\
0.39563	0.0003973\\
0.39637	0.00039723\\
0.3971	0.00039716\\
0.39784	0.00039709\\
0.39858	0.00039702\\
0.39931	0.00039694\\
0.40005	0.00039687\\
0.40079	0.0003968\\
0.40152	0.00039673\\
0.40226	0.00039666\\
0.403	0.00039659\\
0.40373	0.00039652\\
0.40447	0.00039645\\
0.40521	0.00039638\\
0.40594	0.00039631\\
0.40668	0.00039624\\
0.40742	0.00039617\\
0.40816	0.0003961\\
0.40889	0.00039603\\
0.40963	0.00039596\\
0.41037	0.00039589\\
0.4111	0.00039582\\
0.41184	0.00039575\\
0.41258	0.00039569\\
0.41331	0.00039562\\
0.41405	0.00039555\\
0.41479	0.00039548\\
0.41552	0.00039541\\
0.41626	0.00039535\\
0.417	0.00039528\\
0.41773	0.00039521\\
0.41847	0.00039514\\
0.41921	0.00039508\\
0.41994	0.00039501\\
0.42068	0.00039494\\
0.42142	0.00039488\\
0.42216	0.00039481\\
0.42289	0.00039475\\
0.42363	0.00039468\\
0.42437	0.00039461\\
0.4251	0.00039455\\
0.42584	0.00039448\\
0.42658	0.00039442\\
0.42731	0.00039435\\
0.42805	0.00039429\\
0.42879	0.00039422\\
0.42952	0.00039416\\
0.43026	0.00039409\\
0.431	0.00039403\\
0.43173	0.00039396\\
0.43247	0.0003939\\
0.43321	0.00039384\\
0.43394	0.00039377\\
0.43468	0.00039371\\
0.43542	0.00039365\\
0.43616	0.00039358\\
0.43689	0.00039352\\
0.43763	0.00039346\\
0.43837	0.00039339\\
0.4391	0.00039333\\
0.43984	0.00039327\\
0.44058	0.00039321\\
0.44119	0.00039315\\
0.44179	0.0003931\\
0.4424	0.00039305\\
0.44301	0.000393\\
0.44362	0.00039295\\
0.44423	0.0003929\\
0.44484	0.00039285\\
0.44545	0.0003928\\
0.44606	0.00039275\\
0.44667	0.00039269\\
0.44728	0.00039264\\
0.44788	0.00039259\\
0.44849	0.00039254\\
0.4491	0.00039249\\
0.44971	0.00039244\\
0.45032	0.00039239\\
0.45093	0.00039234\\
0.45154	0.00039229\\
0.45215	0.00039224\\
0.45276	0.00039219\\
0.45336	0.00039214\\
0.45397	0.0003921\\
0.45458	0.00039205\\
0.45519	0.000392\\
0.4558	0.00039195\\
0.45641	0.0003919\\
0.45702	0.00039185\\
0.45763	0.0003918\\
0.45824	0.00039175\\
0.45885	0.0003917\\
0.45945	0.00039165\\
0.46006	0.00039161\\
0.46067	0.00039156\\
0.46128	0.00039151\\
0.46189	0.00039146\\
0.4625	0.00039141\\
0.46311	0.00039136\\
0.46372	0.00039132\\
0.46433	0.00039127\\
0.46493	0.00039122\\
0.46554	0.00039117\\
0.46615	0.00039113\\
0.46676	0.00039108\\
0.46737	0.00039103\\
0.46798	0.00039098\\
0.46859	0.00039094\\
0.4692	0.00039089\\
0.46981	0.00039084\\
0.47042	0.00039079\\
0.47102	0.00039075\\
0.47163	0.0003907\\
0.47224	0.00039065\\
0.47285	0.00039061\\
0.47346	0.00039056\\
0.47407	0.00039051\\
0.47468	0.00039047\\
0.47529	0.00039042\\
0.4759	0.00039037\\
0.47651	0.00039033\\
0.47711	0.00039028\\
0.47772	0.00039024\\
0.47833	0.00039019\\
0.47894	0.00039014\\
0.47955	0.0003901\\
0.48016	0.00039005\\
0.48077	0.00039001\\
0.48138	0.00038996\\
0.48199	0.00038992\\
0.48259	0.00038987\\
0.4832	0.00038982\\
0.48381	0.00038978\\
0.48442	0.00038973\\
0.48503	0.00038969\\
0.48564	0.00038964\\
0.48625	0.0003896\\
0.48686	0.00038955\\
0.48747	0.00038951\\
0.48808	0.00038947\\
0.48868	0.00038942\\
0.48929	0.00038938\\
0.4899	0.00038933\\
0.49051	0.00038929\\
0.49112	0.00038924\\
0.49173	0.0003892\\
0.49234	0.00038915\\
0.49295	0.00038911\\
0.49356	0.00038907\\
0.49417	0.00038902\\
0.49477	0.00038898\\
0.49538	0.00038894\\
0.49599	0.00038889\\
0.4966	0.00038885\\
0.49721	0.0003888\\
0.49782	0.00038876\\
0.49843	0.00038872\\
0.49904	0.00038867\\
0.49965	0.00038863\\
0.50025	0.00038859\\
0.50086	0.00038854\\
0.50147	0.0003885\\
0.50208	0.00038846\\
0.50269	0.00038842\\
0.5033	0.00038837\\
0.50391	0.00038833\\
0.50452	0.00038829\\
0.50513	0.00038824\\
0.50574	0.0003882\\
0.50634	0.00038816\\
0.50695	0.00038812\\
0.50756	0.00038808\\
0.50817	0.00038803\\
0.50878	0.00038799\\
0.50939	0.00038795\\
0.51	0.00038791\\
0.51061	0.00038786\\
0.51122	0.00038782\\
0.51183	0.00038778\\
0.51243	0.00038774\\
0.51304	0.0003877\\
0.51365	0.00038766\\
0.51426	0.00038761\\
0.51487	0.00038757\\
0.51548	0.00038753\\
0.51609	0.00038749\\
0.5167	0.00038745\\
0.51731	0.00038741\\
0.51791	0.00038737\\
0.51852	0.00038733\\
0.51913	0.00038728\\
0.51974	0.00038724\\
0.52035	0.0003872\\
0.52096	0.00038716\\
0.52157	0.00038712\\
0.52218	0.00038708\\
0.52279	0.00038704\\
0.5234	0.000387\\
0.524	0.00038696\\
0.52461	0.00038692\\
0.52522	0.00038688\\
0.52583	0.00038684\\
0.52644	0.0003868\\
0.52705	0.00038676\\
0.52766	0.00038672\\
0.52827	0.00038668\\
0.52888	0.00038664\\
0.52949	0.0003866\\
0.53009	0.00038656\\
0.5307	0.00038652\\
0.53131	0.00038648\\
0.53192	0.00038644\\
0.53253	0.0003864\\
0.53314	0.00038636\\
0.53375	0.00038632\\
0.53436	0.00038628\\
0.53497	0.00038624\\
0.53557	0.0003862\\
0.53618	0.00038616\\
0.53679	0.00038612\\
0.5374	0.00038609\\
0.53801	0.00038605\\
0.53868	0.000386\\
0.53935	0.00038596\\
0.54002	0.00038592\\
0.54069	0.00038588\\
0.54136	0.00038583\\
0.54203	0.00038579\\
0.5427	0.00038575\\
0.54337	0.00038571\\
0.54404	0.00038566\\
0.54471	0.00038562\\
0.54538	0.00038558\\
0.54605	0.00038554\\
0.54672	0.0003855\\
0.54739	0.00038545\\
0.54806	0.00038541\\
0.54873	0.00038537\\
0.5494	0.00038533\\
0.55007	0.00038529\\
0.55074	0.00038525\\
0.55141	0.0003852\\
0.55208	0.00038516\\
0.55275	0.00038512\\
0.55342	0.00038508\\
0.55409	0.00038504\\
0.55476	0.000385\\
0.55543	0.00038496\\
0.5561	0.00038492\\
0.55677	0.00038488\\
0.55744	0.00038483\\
0.55811	0.00038479\\
0.55878	0.00038475\\
0.55945	0.00038471\\
0.56012	0.00038467\\
0.56079	0.00038463\\
0.56146	0.00038459\\
0.56213	0.00038455\\
0.5628	0.00038451\\
0.56347	0.00038447\\
0.56414	0.00038443\\
0.5648	0.00038439\\
0.56547	0.00038435\\
0.56614	0.00038431\\
0.56681	0.00038427\\
0.56748	0.00038423\\
0.56815	0.00038419\\
0.56882	0.00038415\\
0.56949	0.00038411\\
0.57016	0.00038407\\
0.57083	0.00038403\\
0.5715	0.00038399\\
0.57217	0.00038395\\
0.57284	0.00038391\\
0.57351	0.00038387\\
0.57418	0.00038384\\
0.57485	0.0003838\\
0.57552	0.00038376\\
0.57619	0.00038372\\
0.57686	0.00038368\\
0.57753	0.00038364\\
0.5782	0.0003836\\
0.57887	0.00038356\\
0.57954	0.00038352\\
0.58021	0.00038349\\
0.58088	0.00038345\\
0.58155	0.00038341\\
0.58222	0.00038337\\
0.58289	0.00038333\\
0.58356	0.00038329\\
0.58423	0.00038326\\
0.5849	0.00038322\\
0.58557	0.00038318\\
0.58624	0.00038314\\
0.58691	0.0003831\\
0.58758	0.00038307\\
0.58825	0.00038303\\
0.58892	0.00038299\\
0.58959	0.00038295\\
0.59026	0.00038291\\
0.59093	0.00038288\\
0.5916	0.00038284\\
0.59258	0.00038278\\
0.59356	0.00038273\\
0.59454	0.00038268\\
0.59552	0.00038262\\
0.5965	0.00038257\\
0.59748	0.00038251\\
0.59846	0.00038246\\
0.59945	0.0003824\\
0.60043	0.00038235\\
0.60141	0.0003823\\
0.60239	0.00038224\\
0.60337	0.00038219\\
0.60435	0.00038214\\
0.60533	0.00038208\\
0.60631	0.00038203\\
0.60729	0.00038198\\
0.60827	0.00038192\\
0.60925	0.00038187\\
0.61023	0.00038182\\
0.61121	0.00038177\\
0.61219	0.00038171\\
0.61318	0.00038166\\
0.61416	0.00038161\\
0.61514	0.00038156\\
0.61612	0.0003815\\
0.6171	0.00038145\\
0.61808	0.0003814\\
0.61906	0.00038135\\
0.62004	0.0003813\\
0.62102	0.00038125\\
0.622	0.0003812\\
0.62298	0.00038114\\
0.62396	0.00038109\\
0.62494	0.00038104\\
0.62593	0.00038099\\
0.62691	0.00038094\\
0.62789	0.00038089\\
0.62887	0.00038084\\
0.62985	0.00038079\\
0.63083	0.00038074\\
0.63181	0.00038069\\
0.63279	0.00038064\\
0.63377	0.00038059\\
0.63475	0.00038054\\
0.63573	0.00038049\\
0.63671	0.00038044\\
0.63769	0.00038039\\
0.63867	0.00038034\\
0.63966	0.00038029\\
0.64064	0.00038024\\
0.64162	0.00038019\\
0.6426	0.00038014\\
0.64358	0.0003801\\
0.64456	0.00038005\\
0.64554	0.00038\\
0.64652	0.00037995\\
0.6475	0.0003799\\
0.64848	0.00037985\\
0.64946	0.0003798\\
0.65044	0.00037976\\
0.65142	0.00037971\\
0.65241	0.00037966\\
0.65339	0.00037961\\
0.65437	0.00037956\\
0.65535	0.00037952\\
0.65633	0.00037947\\
0.65731	0.00037942\\
0.65829	0.00037937\\
0.65927	0.00037933\\
0.66025	0.00037928\\
0.66123	0.00037923\\
0.66221	0.00037919\\
0.66319	0.00037914\\
0.66417	0.00037909\\
0.66515	0.00037905\\
0.66614	0.000379\\
0.66712	0.00037895\\
0.6681	0.00037891\\
0.66908	0.00037886\\
0.67006	0.00037881\\
0.67418	0.00037862\\
0.67831	0.00037843\\
0.68243	0.00037824\\
0.68656	0.00037805\\
0.69068	0.00037786\\
0.6948	0.00037768\\
0.69893	0.0003775\\
0.70305	0.00037731\\
0.70718	0.00037713\\
0.7113	0.00037695\\
0.71543	0.00037678\\
0.71955	0.0003766\\
0.72367	0.00037642\\
0.7278	0.00037625\\
0.73192	0.00037608\\
0.73605	0.00037591\\
0.74017	0.00037574\\
0.7443	0.00037557\\
0.74842	0.0003754\\
0.75254	0.00037524\\
0.75667	0.00037507\\
0.76079	0.00037491\\
0.76492	0.00037475\\
0.76904	0.00037459\\
0.77317	0.00037443\\
0.77729	0.00037427\\
0.78141	0.00037411\\
0.78554	0.00037396\\
0.78966	0.0003738\\
0.79379	0.00037365\\
0.79791	0.0003735\\
0.80204	0.00037335\\
0.80616	0.00037319\\
0.81028	0.00037305\\
0.81441	0.0003729\\
0.81853	0.00037275\\
0.82266	0.0003726\\
0.82678	0.00037246\\
0.83091	0.00037232\\
0.83503	0.00037217\\
0.83915	0.00037203\\
0.84328	0.00037189\\
0.8474	0.00037175\\
0.85153	0.00037161\\
0.85565	0.00037147\\
0.85977	0.00037133\\
0.8639	0.0003712\\
0.86802	0.00037106\\
0.87215	0.00037093\\
0.87627	0.00037079\\
0.8804	0.00037066\\
0.88452	0.00037053\\
0.88864	0.0003704\\
0.89277	0.00037027\\
0.89689	0.00037014\\
0.90102	0.00037001\\
0.90514	0.00036988\\
0.90927	0.00036975\\
0.91339	0.00036963\\
0.91751	0.0003695\\
0.92164	0.00036938\\
0.92576	0.00036925\\
0.92989	0.00036913\\
0.93401	0.00036901\\
0.93814	0.00036889\\
0.94226	0.00036877\\
0.94638	0.00036865\\
0.95051	0.00036853\\
0.95463	0.00036841\\
0.95876	0.00036829\\
0.96288	0.00036817\\
0.96701	0.00036806\\
0.97113	0.00036794\\
0.97525	0.00036783\\
0.97938	0.00036771\\
0.9835	0.0003676\\
0.98763	0.00036749\\
0.99175	0.00036737\\
0.99588	0.00036726\\
1	0.00036715\\
};
\end{axis}
\end{tikzpicture}%